\documentclass[11pt]{article}
\usepackage{amsmath, amssymb, amsthm, esint, hyperref, cite}
\usepackage{geometry, mathabx}
\newcommand{\bd}{\boldsymbol}

\usepackage{xcolor}
\usepackage{graphicx} 

\def\R{{\mathbb R}}

\theoremstyle{definition}

\newtheorem{theorem}{Theorem}[section]

\newtheorem{definition}{Definition}[section]
\newtheorem{proposition}{Proposition}[section]
\newtheorem{remark}{Remark}[section]
\newtheorem{lemma}{Lemma}[section]
\begin{document}

\title{\sc{A Regularized Interface Method for Fluid-Poroelastic Structure Interaction Problems with Nonlinear Geometric Coupling}}
%\title{\sc{Singular limit of a fluid-poroelastic structure interaction problem with nonlinear geometric coupling}}
%\title{\sc{Well-posedness of a nonlinearly coupled fluid-poroviscoelastic structure interaction problem}}
\author{Jeffrey Kuan, Sun\v{c}ica \v{C}ani\'{c}, Boris Muha}
\maketitle
\if 1 = 0
{\bf{Corresponding author information:}} Sun\v{c}ica \v{C}ani\'{c}, Department of Mathematics, University of California, Berkeley, canics@berkeley.edu
\fi
\begin{abstract}
\if 1 = 0
We study a benchmark fluid-poroelastic structure interaction (FPSI) problem coupling an incompressible, viscous Newtonian fluid, modeled by the Navier-Stokes equations, with a bulk poroelastic medium modeled by the Biot system, with both phases occupying domains of the same spatial dimension. The fluid and structure are in direct contact along a moving interface determined by the trace of the poroelastic displacement. Both subproblems evolve on a priori unknown domains and exhibit strong geometric nonlinearities. To address the lack of existence theory in this \emph{nonlinear, moving domain setting}, we introduce a \emph{regularized interface method}: we regularize the Biot displacement %\bd{\eta}$ 
by spatial convolution at scale $\delta>0$ to define smooth, regularized moving domains and interface, and we modify the weak formulation so that it remains energy-consistent with the original problem.
For each fixed $\delta>0$, we prove existence of a weak solution to the resulting {\emph{nonlinearly coupled}} regularized interface problem by first inserting a thin plate of thickness $h>0$ at the interface and establishing existence for the model with the plate--involving a time-discretization via Lie operator splitting scheme, uniform a priori bounds, and Aubin-Lions compactness on moving domains. We then pass to the singular limit $h\to0$ using uniform-in-$h$ estimates to recover a {\emph{regularized interface weak solution}}. The regularization introduced in this manuscript is essential to maintain uniform geometric control of the moving interface and to accommodate vector-valued structural displacements. 

An accompanying work will show {\emph{weak-classical consistency}}, recovering classical solutions to the non-regularized problem in the limit as $\delta\to0$, when such classical solutions exist.
\fi

\if 1 = 0
We introduce a new {\emph{regularized interface method}} to study the existence of weak solutions for nonlinear moving boundary problems with {\emph{low geometric boundary regularity}}. As a benchmark, we consider a fluid-poroelastic structure interaction (FPSI) problem coupling two physically rich systems: the Navier-Stokes equations for an incompressible viscous fluid and the Biot system for a bulk poroelastic medium. These are coupled across a moving interface with low regularity and strong geometric nonlinearities. The fluid and structure occupy domains of the same spatial dimension, with the interface determined by the trace of the poroelastic displacement, so that the domains are a priori unknown. Despite its fundamental importance and broad engineering applications, no existence theory has been available in this nonlinear moving-domain setting, primarily because the lack of interface regularity precludes even the formulation of a weak solution framework.
\fi

We introduce a new {\emph{regularized interface method}} for proving existence of weak solutions to nonlinear moving boundary problems with {\emph{low-regularity interfaces.}}
 As a benchmark, we study a fluid-poroelastic structure interaction (FPSI) problem coupling the Navier-Stokes equations for an incompressible viscous fluid with the Biot system for a bulk poroelastic medium. The two phases occupy domains of the same spatial dimension, separated by a moving interface defined by the trace of the poroelastic displacement, which exhibits low regularity and strong geometric nonlinearities. Despite its importance in applications, no existence theory has been available for this nonlinear moving-domain setting, primarily because the lack of interface regularity precludes even the formulation of a weak solution framework.

To address this gap, we (1) introduce a regularization of the Biot displacement via spatial convolution at scale $\delta > 0$, which defines regularized moving domains and interface, and (2) modify the weak formulation in a way that preserves energy consistency with the original problem. For each fixed $\delta > 0$, we prove existence of a weak solution to the resulting regularized interface problem. The proof strategy involves inserting a thin plate of thickness $h > 0$ at the interface, applying a time-discretization via a Lie operator splitting scheme, establishing uniform a priori bounds, and employing Aubin-Lions compactness on moving domains. The analysis is particularly involved, partly because the thin plate allows displacements in all spatial directions, which may lead to the formation of certain geometric singularities not present when only scalar (transverse) displacements are considered. Passing to the singular limit $h \to 0$ with uniform-in-h estimates and Aubin--Lions-type compactness arguments yields a regularized interface weak solution.
The regularization introduced in this manuscript is essential to maintain uniform geometric control of the moving interface and to accommodate vector-valued structural displacements. 

An accompanying work will prove {\emph{weak-classical consistency}}, showing that in the limit $\delta \to 0$, the regularized weak solutions recover classical solutions to the non-regularized problem when such classical solutions exist.

\end{abstract}

\section{Introduction}

In this manuscript, we study a benchmark fluid-poroelastic structure interaction (FPSI) problem involving the coupling between an incompressible, viscous Newtonian fluid and a bulk poroelastic structure, both occupying domains of the same spatial dimension. Our focus is on a benchmark configuration in two spatial dimensions. The fluid is modeled by the Navier-Stokes equations, while the poroelastic structure is governed by the Biot system.
We are particularly interested in a nonlinearly coupled setting in which the interface between the fluid and the poroelastic structure is determined by the trace of the displacement of the thick poroelastic medium in direct contact with the fluid.
In this setting, the fluid and the Biot problems are each posed on a priori unknown moving domains, and hence create complex geometric nonlinearities.

The central objective of this manuscript is to study the existence of finite-energy solutions for this problem. To the best of our knowledge, no such result currently exists for FPSI or FSI problems in which a thick structure occupying a domain of the same dimension as the fluid, is in direct contact with the fluid, without a separating interface possessing mass and elastic energy, and where the elastic properties of the structure are modeled by the classical second-order elasticity equations.

This lack of existence theory stems from several key mathematical challenges. First, the interface is defined by the trace of the displacement field of the poroelastic structure, which, in general, belongs only to $H^{1/2}$. This regularity is insufficient to define the interface as the graph of a continuous function, thereby obstructing even the geometric description of the fluid domain. Second, for moving-boundary FPSI problems, the weak formulation itself becomes problematic: key geometric quantities such as the Jacobian of the Lagrangian map on the Biot domain lack the regularity needed to rigorously interpret all terms in the formally derived weak formulation.

%We consider a \textit{new fluid-poroelastic structure interaction model} in which the fluid and the Biot poroelastic material are in \textit{direct contact with each other}. Namely, many moving boundary models of fluid-poroelastic structure interaction (involving geometric nonlinearities arising from the moving interface) consider models in which a thin reticular plate (of lower dimension) separates the fluid and the poroelastic material, and hence regularizes the interface by adding inertia to the moving interface.

In this manuscript, we introduce a new approach, called the \textit{regularized interface method}, which allows us to analyze a class of FPSI problems in which the fluid and poroelastic structure interact dynamically via two-way coupling along a moving interface, along which the fluid and poroelastic structure are in direct contact. 
While we focus on FPSI problems, we note that  the main ideas of this work can be extended to FSI problems involving purely elastic structures, and more generally,  to moving boundary problems in which the regularity of the moving domains depend on potentially lower regularity solutions.

The main idea behind the \textit{regularized interface method} is to introduce a regularization parameter $\delta$ and to regularize the vector-valued Biot structure displacement $\bd{\eta}: \Omega_{b} \to \R^{2}$ via spatial convolution against a kernel with support of size $\delta$.
Because the regularized displacement $\bd{\eta}^\delta$ is smooth, we can use $\bd{\eta}^\delta$ in place of $\bd{\eta}$ to define regularized moving fluid and Biot domains, as well as a regularized moving interface. In the weak formulation, any integrals over moving domains are then computed on these \textit{regularized} geometric domains.

{{However, this is not sufficient to define a {\emph{consistent regularized interface weak formulation}}, as a straightforward regularization of $\bd{\eta}$ does not preserve the energy structure of the original problem. To achieve an energy-consistent formulation, certain terms in the regularized weak formulation must be modified so that its solutions satisfy the same type of energy estimate as those of the original, non-regularized problem.}}
After taking all this into account, the resulting formulation is referred to as the \textit{regularized interface weak formulation} with regularization parameter~$\delta$.

The goal of this manuscript is to demonstrate that this formulation is robust, in the sense that it admits a weak solution. This solution can be viewed as an {\emph{approximation}} to the solution of the original, non-regularized problem. 

In the upcoming manuscript, we will show that the {\emph{regularized interface}} formulation introduced here is {\emph{consistent}} with the original problem--not only in terms of energy, but also in the sense that, as the approximation parameter $\delta \to 0$, the approximate solution constructed in this manuscript converges to a classical solution of the original, non-regularized problem, provided such a classical solution exists.

We emphasize that the spatial regularization introduced by the regularized interface method offers an additional advantage: it enables the treatment of \textit{vector-valued} elastic displacements. A common difficulty in fluid-structure interaction problems with vector-valued displacements is the potential for geometric degeneracy and loss of injectivity in the moving domains. For this reason, many FSI models restrict attention to scalar transverse or normal interface displacements. However, the additional regularity gained through spatial convolution of the Biot displacement yields uniform-in-time control over various spatial norms--such as Lipschitz norms--of the solution. This, in turn, ensures uniform-in-time control of the evolving geometry of the time-dependent domains, making the analysis of vector-valued displacements feasible.

To obtain existence of a weak solution to the regularized interface FPSI problem with geometric nonlinearities discussed above, we introduce a thin plate  of thickness $h>0$ with mass and elastic energy, positioned at the interface separating the regions of free fluid flow and Biot poroelastic medium.
We aim to show that for each fixed regularization parameter $\delta > 0$, we will recover the desired weak solution of the problem in which Biot poroelastic medium is in direct contact with free fluid flow, by taking the limit, as $h \to 0$.
The introduction of a plate with thickness $h > 0$ is advantageous, as it enables us to build on prior results from \cite{FPSIJMPA}, which established the existence of weak solutions to a geometrically nonlinear FPSI problem with a plate separating the Biot and fluid domains.
More specifically, to prove the existence of weak solutions for each fixed $h > 0$ (and a given fixed $\delta > 0$), we employ a constructive approach based on Lie operator splitting. We discretize the time interval $[0, T]$ into $N$ subintervals of size $\Delta t = T/N$, and then solve two subproblems on each time subinterval:
(1) a fluid/Biot subproblem, in which the fluid and Biot quantities are updated, and
(2) a plate subproblem, in which the interface dynamics are updated.

To pass to the limit as discretization parameter $\Delta t \to 0$, we use uniform bounds and compactness arguments--specifically, the Aubin--Lions compactness theorem for functions defined on moving domains. These tools allow us to extract subsequences of approximate solutions from the splitting scheme that converge as $\Delta t \to 0$. We then pass to the limit in the approximate weak formulations to obtain a regularized interface weak solution to the geometrically nonlinear FPSI problem with a reticular plate of thickness $h > 0$.
We remark that this existence analysis is particularly involved, due in part to the fact that the thin plate allows for displacements in all spatial direction--a feature that is new in the analytical study of FPSI problems. 

To establish existence of a regularized interface weak solution for the FPSI problem \textit{without a plate}, we consider approximate solutions with plate thickness $h > 0$ and take the singular limit as $h \to 0$. We obtain uniform-in-$h$ bounds and apply Aubin--Lions-type compactness arguments to extract strongly convergent subsequences. In the limit $h \to 0$, we recover a regularized interface weak solution satisfying the limiting weak formulation corresponding to the original problem with direct contact.

We remark that the {\emph{regularized interface construction is essential for this singular limit $h \to 0$ to work}}; without regularizing the interface, one does not obtain uniform geometric control on the dynamics of the interface separating the fluid and Biot material. This is because there are no sufficiently strong uniform bounds on the interface quantities in the limit as $h \to 0$ without additional regularization, since any spatial regularization of the interface displacement arising from the presence of the plate (giving rise to plate displacements in $H^{2}(\Gamma)$) is scaled by powers of $h$ due to the presence of the plate thickness in the plate elastodynamics equations, which causes us to lose uniform geometric control of the Biot-fluid interface in the limit as $h \to 0$, without interface regularization.

This manuscript is organized as follows. In Section \ref{problemdef}, we introduce the fluid-poroelastic structure problem under consideration and review prior work on FPSI models, contrasting our notion of \textit{regularized interface} solutions with earlier moving-boundary approaches. Section \ref{weaksection} defines weak solutions on both fixed and moving domains, and Section \ref{regintmethod} introduces the notion of a \textit{regularized interface} weak solution. In Section \ref{hproof}, we prove the main existence result for regularized interface weak solutions to the FPSI model with a moving boundary and a plate of thickness $h>0$. Section \ref{finallimit} then employs uniform bounds and compactness arguments to pass to the singular limit $h\to 0$, yielding a limiting \textit{regularized interface} weak solution for the case of direct fluid-Biot contact (no plate). Finally, Section \ref{conclusion} summarizes the main results of this manuscript, and discusses 
how the regularized interface method can be seen as a natural way for developing weak solutions to problems involving moving domains with limited geometric boundary regularity.

\if 1 = 0
The model in question has several novelties:
\begin{itemize}
\item \textbf{Nonlinear geometric coupling and bulk elasticity.} The Biot equations and the Navier-Stokes equations are posed on moving domains that depend on the displacement of the Biot material from its reference configurations. While there are past works that pose these equations on fixed domains in models with linear geometric coupling, nonlinear geometric coupling only been explored very recently in recent work. Moving domains pose a particularly difficult challenge here, since the structure and fluid are of the same dimension (bulk elasticity), and hence, geometric nonlinearities are particularly difficult to deal with, due to low spatial regularity of geometric terms for weak solutions in finite energy spaces.
\item \textbf{Direct Biot-fluid contact along the moving interface.} Most FPSI problems in consideration thus far involve elastic plates or membranes separating the fluid and Biot structure, and only recently, in the linear geometric coupled case has well-posedness of a FPSI model with direct Biot-fluid coupling been rigorously considered \cite{Gurvich}. This motivates us to naturally extend the study of Biot-fluid coupled systems with direct contact along the interface to moving boundary FPSI problems, which will be the primary goal of this manuscript.
\item \textbf{Vector-valued displacements.} We assume that the $2D$ bulk poroelastic BIot material has vector-valued displacements from its reference configuration. Thus, the $1D$ interface of direct contact between the fluid and the Biot material can potentially have significant vector-valued deformations. Much of the past literature on fluid-structure interaction (FSI) is limited to scalar transverse/normal displacements, and results on FSI with vector-valued elastic displacements are uncommon. In particular, handling injectivity and preventing geometric degeneracies in moving domains is often difficult in FSI problems.
\end{itemize}

We summarize key novelties of this work:

    Nonlinear geometric coupling and bulk elasticity. The Biot and Navier--Stokes equations are posed on moving domains determined by the Biot displacement. Unlike prior work based on fixed domains and linear coupling, this model incorporates nonlinear geometric coupling, which introduces significant analytical challenges�especially since the fluid and structure occupy domains of the same dimension. These challenges stem from the low spatial regularity of geometric terms in finite-energy weak solution spaces.

    Direct Biot�fluid contact along the moving interface. While most FPSI models include a separating plate or membrane between the fluid and Biot structure, this work considers direct contact along a moving interface. Well-posedness for such configurations has only recently been studied in the linearly coupled case \cite{Gurvich}. This manuscript aims to extend that analysis to geometrically nonlinear, moving-boundary FPSI problems with direct Biot�fluid interaction.

    Vector-valued displacements. The $2D$ Biot material undergoes vector-valued displacements, meaning the $1D$ interface in direct contact with the fluid can experience nontrivial vector deformations. Unlike much of the existing FSI literature, which often restricts to scalar transverse displacements, this model handles full vector-valued motion�requiring additional care to avoid issues like loss of injectivity or geometric degeneracy in the evolving domains.

We will consider a 2D-2D model of Biot-fluid FPSI with direct contact along a moving 1D interface. The main challenge here is that the interface will be defined via the trace of the Biot displacement along the interface. In the finite energy space, the displacement of the Biot material is in $H^{1}(\Omega_{b})$, so its trace will be in $H^{1/2}(\Gamma)$ which is not sufficient regularity to even define the moving interface as the graph of a continuous function in 1D. This poses a problem in properly defining a weak formulation to the problem at the level of finite-energy weak solutions, as moving domains are not well-defined, and geometric quantities (such as the Jacobian of the Lagrangian map) do not have sufficient regularity to allow all terms of the (formally defined) weak formulation to be rigorously interpreted. 

Thus, we will introduce a regularization method in this manuscript, which we refer to as the \textit{regularized interface method}, to handle the difficulties associated with FPSI problems posed on moving domains with potentially low regularity geometric domains. This will involve introducing a regularization parameter $\delta$, which will be fixed for the purpose of the existence proof. The regularized interface method involves taking the vector-valued Biot structure displacement $\bd{\eta}: \Omega_{b} \to \R^{2}$, which defines all of the moving domains in the problem and the moving interface, and using spatial convolution against a kernel of support $\delta$, to regularize it into a smooth regularized displacement $\bd{\eta}^{\delta}$. Because the regularized displacement $\bd{\eta}^{\delta}$ is now smooth, we can then use $\bd{\eta}^{\delta}$ instead of $\bd{\eta}$ to define regularized moving fluid and Biot domains, and a regularized moving interface. In the weak formulation, any integrals over moving domains will then be computed on the moving \textit{regularized} geometric domains and any terms arising from equations posed on moving domains are appropriately modified so as to preserve an energy estimate for the system. The resulting weak solution, which satisfies a regularized form of the weak formulation with respect to the regularization parameter $\delta > 0$, will be referred to as a \textit{regularized interface weak solution with regularization parameter $\delta$}. The goal of this manuscript is to show that this regularized interface weak solution is robust in the sense that we recover existence of weak solutions for the direct contact nonlinear geometrically coupled FPSI model, with this new definition of (regularized interface) weak solution.

The purpose of regularized interface solutions is to address deficiencies in spatial/geometric regularity for higher-dimensional problems through a sufficient, yet minimal, regularization of the weak formulation. Weak formulations to FPSI problems involving a moving time-dependent interface, due to geometric nonlinearities arising from the moving boundary, have certain geometric nonlinearities that cannot be appropriately defined in the finite-energy spaces, in the sense that certain terms in the weak formulation do not possess sufficient regularity to be interpreted rigorously. Hence, \textit{regularized interface weak solutions} are a way of addressing this deficiency in regularity; namely, regularized interface weak solutions satisfy a regularized form of the weak formulation, in which we use spatial convolution to regularize the least number of terms in the weak formulation, in such a way so that (1) the weak formulation is well-defined for finite energy solutions, (2) the energy structure of the problem is preserved (we still have a physically reasonable energy inequality), and (3) solutions as the regularization parameter $\delta \to 0$ agree with classical solutions to the original problem, whenever such classical solutions to the original problem exist. In addition, since the interface itself is regularized, regularized interface weak solutions must satisfy a form of the kinematic coupling condition (no-slip condition) that holds along the \textit{regularized} interface. 

We note that the additional spatial regularization arising from the regularized interface method, which spatially regularizes the Biot displacement via a spatial convolution, has the additional advantage of allowing us to handle \textit{vector-valued} elastic displacements. The usual challenge in considering fluid-structure interaction problems with vector-valued displacements is potential geometric degeneracy and loss of injectivity in moving domains, and hence, many fluid-structure interaction models consider scalar transverse/normal displacement of elastic structures. However, the additional regularity due to spatial convolution of the Biot displacement allows us to obtain uniform-in-time control of arbitrary spatial norms, such as Lipschitz norms, of the solution. This allows us to obtain uniform-in-time control of the moving geometry of time-dependent domains, which makes the consideration of vector-valued displacements tractable. 

To obtain existence of a regularized interface weak solution to the nonlinearly geometrically coupled FPSI problem with direct contact, we use an approximation parameter $h > 0$, where the approximate problem is the same FPSI model, but with a plate of thickness $h > 0$ along the moving interface separating the fluid and the Biot structure. The presence of the plate with thickness $h > 0$ is advantageous, as it allows us to build off of past results on the existence of weak solutions to a regularized formulation of a nonlinearly geometrically coupled FPSI problem with a plate separating the Biot material and the fluid, see \cite{FPSIJMPA}. Therefore, we will show existence of regularized interface weak solutions for the approximate problem with plate thickness $h > 0$, using a constructive existence proof via Lie operator splitting. This involves discretizing the time interval $[0, T]$ on which we are constructing the solution, into $N$ time intervals of size $\Delta t = T/N$ and then running a (1) fluid/Biot and (2) plate subproblem, in which we update the fluid/Biot quantities and then we update the interface dynamics, successively on each time subinterval. We then use uniform boundedness and compactness arguments (Aubin-Lions compactness arguments) for functions on moving domains, which will allow us to extract subsequences that converge as $\Delta t \to 0$ from the splitting scheme approximate solutions. This allows us to pass to the limit in approximate weak formulations in order to obtain a regularized interface weak solution to the nonlinearly geometrically coupled FPSI problem with a reticular plate of thickness $h > 0$. 

To then obtain the existence of a regularized interface weak solution for the nonlinearly geometrically coupled FPSI problem \textit{without a plate} (with direct contact), we use our approximate solutions for plate thickness parameters $h > 0$ and take a singular limit as $h \to 0$. More concretely, we then pass to the limit in the plate thickness $h \to 0$ via a singular limit by obtaining uniform bounds on the approximate solutions, which are uniform in $h$. This will involve using compactness arguments via compact embeddings of Aubin-Lions type in order to obtain strongly convergent subsequences of approximate solutions as $h \to 0$. We will then obtain in the limit as $h \to 0$, a regularized interface weak solution for the problem with direct contact (no plate), which satisfies the limiting weak formulation as $h \to 0$, which will exactly be the regularized interface weak formulation. We remark that the regularized interface construction is essential for this singular limit as $h \to 0$ to work; without regularizing the interface, one does not obtain uniform geometric control on the dynamics of the interface separating the fluid and Biot material. This is because there are no sufficiently strong uniform bounds on the interface quantities in the limit as $h \to 0$ without additional regularization, since any spatial regularization of the interface displacement arising from the presence of the plate (giving rise to plate displacements in $H^{2}(\Gamma)$) is scaled by powers of $h$ due to the presence of the plate thickness in the plate elastodynamics equations, which causes us to lose uniform geometric control of the Biot-fluid interface in the limit as $h \to 0$, without interface regularization.

This manuscript is organized as follows. In Section \ref{problemdef}, we define the fluid-poroelastic structure problem under consideration, and review past work for FPSI models, contrasting the notion of regularized interface solutions to past work on moving boundary FPSI. We then define the notion of a weak solution on both the fixed and moving domains in Section \ref{weaksection}, and we then define the notion of a \textit{regularized interface} weak solution in Section \ref{regintmethod}. In Section \ref{hproof}, we prove the main result on existence of regularized interface weak solutions to the FPSI model with moving boundary and plate with thickness $h > 0$. Finally, in Section \ref{finallimit}, we consider regularized interface weak solutions for the problem with plate of thickness $h > 0$ obtained in the previous section, and use uniform bounds and compactness arguments to take the singular limit as the plate thickness $h \to 0$, to obtain a limiting notion of regularized interface weak solution for the FPSI problem in which the fluid and poroelastic material are in direct contact, without a reticular plate separating them. In Section \ref{conclusion}, we summarize the resulting definition of a regularized interface weak solution for the FPSI problem with a moving fluid-Biot interface and direct contact, and discuss potential generalizations of the approach introduced in this manuscript, which we refer to as the \textit{regularized interface method}. We discuss in this final section how the regularized interface method can be seen as a natural way for developing weak solutions to problems involving moving domains with limited geometric boundary regularity.
\fi

\section{Problem description}\label{problemdef}

We consider the coupled dynamical interaction between a poroelastic medium described by the Biot equations and an incompressible viscous fluid described by the incompressible Navier-Stokes equations. We pose these equations on moving (time-dependent) domains that depend on the dynamical quantities of the problem itself (namely the displacement of the Biot material). The entire Biot-fluid system will be posed on a 2D domain, where the Biot material and the fluid occupy two-dimensional domains of the same spatial dimension. The poroelastic material is modeled by the (geometrically nonlinear) Biot equations posed on a time-dependent (moving) 2D domain $\Omega_{b}(t)$ occupied by the poroelastic material, and the incompressible fluid is modeled by the incompressible Navier-Stokes equations on a time-dependent (moving) 2D domain $\Omega_{f}(t)$ occupied by the fluid. The dynamical quantities that we will model are:
\begin{itemize}
\item \textbf{The Biot quantities.} These include the Biot vector-valued displacement $\bd{\eta}$ from its reference configuration (describing elastic displacements) and the scalar pore pressure $p$. 
\item \textbf{The fluid quantities.} These include the (divergence-free) vector-valued fluid velocity $\bd{u}$ and the scalar fluid pressure $\pi$.  
\end{itemize}

\subsection{The geometry of the problem}

The total domain occupied by both the Biot material and the fluid will be denoted by:
\begin{equation*}
\hat{\Omega} := \{\bd{x} = (x, y) \in \mathbb{R}^{2} : |\bd{x}| < 2\}.
\end{equation*}
See Fig.~\ref{geometryfig}. 
For all times $t > 0$, the time-dependent Biot and fluid domains will altogether occupy $\hat{\Omega}$, namely:
\begin{equation*}
\hat{\Omega} = \Omega_{b}(t) \cup \Omega_{f}(t) \cup \Gamma(t), \qquad \text{ for all } t \ge 0,
\end{equation*}
where $\Omega_{b}(t)$ is the time-dependent (two-dimensional) Biot domain, $\Omega_{f}(t)$ is the time-dependent (two-dimensional) fluid domain, and $\Gamma(t)$ is the (one-dimensional) moving interface along which the Biot poroelastic material and the fluid are in direct contact with each other.

\begin{figure}
\center
\includegraphics[scale=0.4]{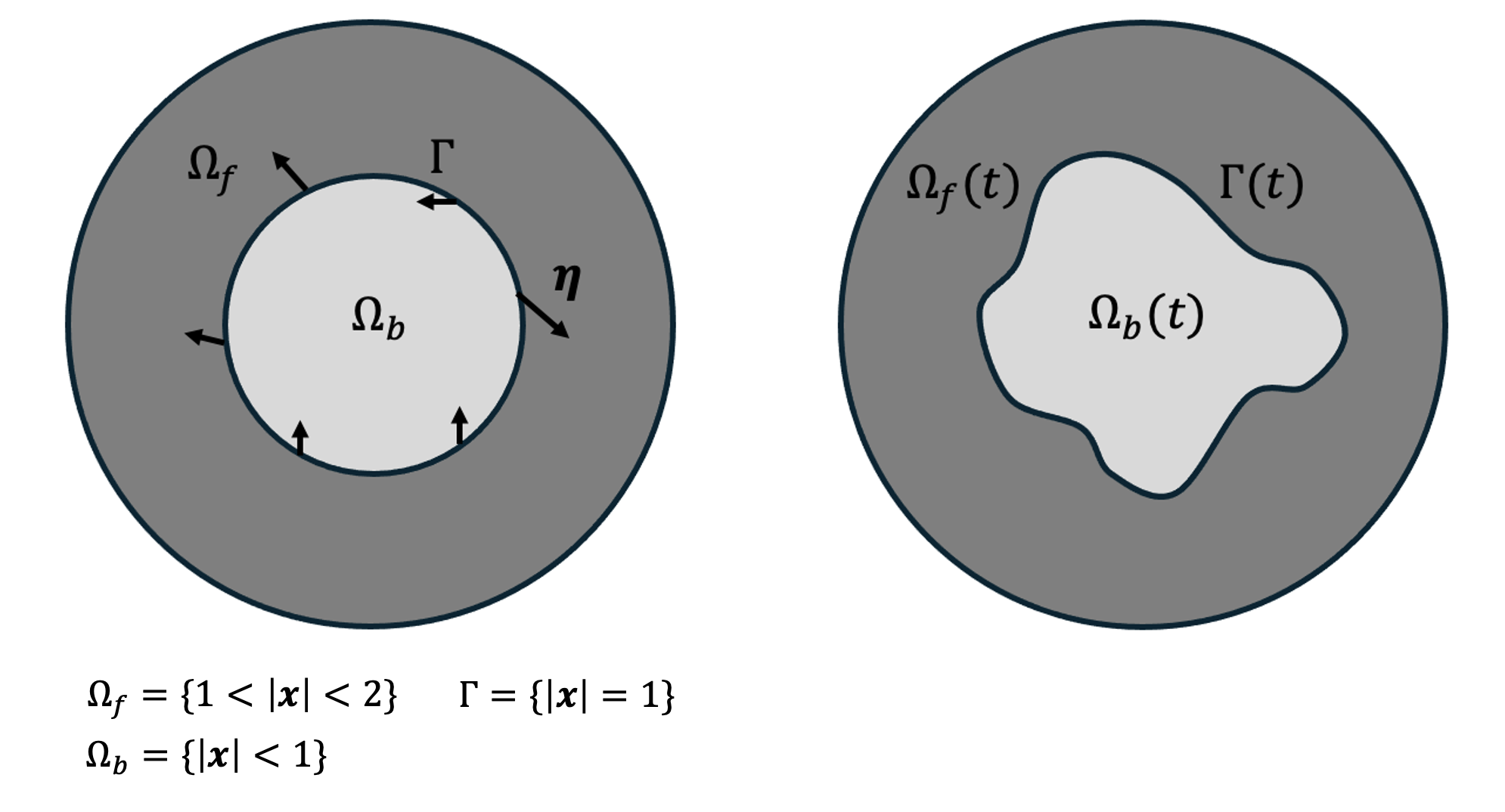}
\caption{A diagram of the fixed domain (left) and the moving domain (right) geometry of the direct contact FPSI problem.}
\label{geometryfig}
\end{figure}

The moving domains will be defined by transformations of reference (fixed) domains for both the Biot material and the fluid, where the Biot reference domain $\hat{\Omega}_{b}$ and the fluid reference domain $\hat{\Omega}_{f}$ are defined by
\begin{equation*}
\hat{\Omega}_{b} := \{\hat{\bd{x}} = (\hat{x}, \hat{y}) \in \R^{2} : |\hat{\bd{x}}| < 1\} , \qquad \hat{\Omega}_{f} := \{\hat{\bd{x}} = (\hat{x}, \hat{y}) \in \R^{2} : 1 < |\hat{\bd{x}}| < 2\},
\end{equation*}
and the reference configuration for the fluid-structure interface is
\begin{equation*}
\hat{\Gamma} := \{\hat{\bd{x}} = (\hat{x}, \hat{y}) \in \R^{2} : |\hat{\bd{x}}| = 1\}.
\end{equation*}
We will define the moving Biot domain by
\begin{equation*}
\Omega_{b}(t) := \{(\hat{x}, \hat{y}) + \hat{\bd{\eta}}(\hat{x}, \hat{y}) : (\hat{x}, \hat{y}) \in \hat{\Omega}_{b}\},
\end{equation*}
and we will define the moving interface by
\begin{equation}\label{gammat}
\Gamma(t) = \{(\hat{x}, \hat{y}) + \hat{\bd{\eta}}|_{\hat{\Gamma}}(\hat{x}) : (\hat{x}, \hat{y}) \in \hat{\Gamma}\},
\end{equation}
where the \textit{displacement of the moving interface from its reference configuration $\hat{\Gamma}$} is defined by $\hat{\bd{\eta}}|_{\hat{\Gamma}}$, which is to be interpreted as a trace along $\hat{\Gamma}$. We then define the moving fluid domain $\Omega_{f}(t)$ to be the remaining part of the entire domain:
\begin{equation}\label{omegaft}
\Omega_{f}(t) = \hat{\Omega} \setminus (\Omega_{b}(t) \cup \Gamma(t)).
\end{equation}
We therefore observe that the moving fluid/Biot domains and the moving interface depend on the a priori unknown Biot displacement $\hat{\bd{\eta}}$. See Figure \ref{geometryfig}.

\begin{remark}\label{notation}
We set a notation convention that will be used throughout the manuscript. We will use the ``hat" notation for geometric domains whenever these domains are fixed reference domains (such as in $\hat{\Omega}_{b}$, $\hat{\Omega}_{f}$, and $\hat{\Omega}$), and we use the ``hat" notation for physical quantities whenever they are defined on reference domains, such as the Biot displacement $\hat{\bd{\eta}}$ defined on $\hat{\Omega}_{b}$. Domains without hats or physical quantities without hats correspond to time-dependent moving domains (for example, as in the time-dependent fluid domain $\Omega_{f}(t)$). 
\end{remark}

\begin{remark}[A remark about the geometry of $\hat{\Gamma}$]
In the explicit analysis of the problem, it will be useful to parametrize the reference configuration $\hat{\Gamma}$ for the direct contact interface by 
\begin{equation*}
\hat{z} \in [0, 2\pi] \to (\cos(\hat{z}), \sin(\hat{z})) \in \hat{\Gamma},
\end{equation*}
so that we regard the unit circle $\hat{\Gamma}$ as a torus, or equivalently as the interval $\hat{z} \in [0, 2\pi]$ with periodic boundary conditions. It will be useful to regard functions on $\hat{\Gamma}$ as (periodic) functions of $\hat{z} \in [0, 2\pi]$. See the discussion in Section \ref{transformed}.
\end{remark}

\subsection{The Biot subproblem}

We first specify the dynamics of the Biot poroelastic material, and for this, we model the Biot displacement $\hat{\bd{\eta}}: \hat{\Omega} \to \R^{2}$ and the pore pressure $p: \hat{\Omega} \to \R$ using the nonlinear Biot equations. The Biot displacement is related to the elastic properties of the material, and it describes the displacement of the Biot material from its reference configuration $\hat{\Omega}_{b}$. Namely, we can define an associated \textbf{Lagrangian map} to the Biot displacement $\hat{\bd{\eta}}$, given by
\begin{equation}\label{lagrangianmap}
\hat{\bd{\Phi}}_{b}^{\eta}(\hat{x}, \hat{y}): \hat{\Omega}_{b} \to \Omega_{b}(t), \qquad \hat{\bd{\Phi}}^{\eta}_{b}(\hat{x}, \hat{y}) = (\hat{x}, \hat{y}) + \hat{\bd{\eta}}(\hat{x}, \hat{y}), \text{ for } (\hat{x}, \hat{y}) \in \hat{\Omega}_{b}.
\end{equation}
We can then describe the coupled dynamics of the elastic displacement of Biot material and the flow of fluid through its pores through the nonlinear Biot equations, which is the following system of coupled equations:
\begin{equation}\label{biot}
\begin{cases} 
\rho_{b}\partial_{tt}\hat{\bd{\eta}} = \hat{\nabla} \cdot \hat{S}_{b}(\hat{\nabla}\hat{\bd{\eta}}, \hat{p}) \qquad &\text{ in } \hat{\Omega}_{b}, \\
\displaystyle \frac{c_{0}}{[\text{det}(\hat{\nabla} \hat{\bd{\Phi}}^{\eta}_{b})] \circ (\bd{\Phi}^{\eta}_{b})^{-1}} \frac{D}{Dt} p + \alpha \nabla \cdot \frac{D}{Dt} \bd{\eta} = \nabla \cdot (\kappa \nabla p) \qquad &\text{ in } \Omega_{b}(t),
\end{cases}
\end{equation}
where the (visco)elastic Piola-Kirchhoff stress tensor is given by
\begin{equation}\label{kirchhoff}
\hat{S}_{b}(\nabla \bd{\eta}, p) = 2\mu_{e} \hat{\bd{D}}(\hat{\bd{\eta}}) + \lambda_{e} (\hat{\nabla} \cdot \hat{\bd{\eta}}) \bd{I} + 2\mu_{v} \hat{\bd{D}}(\hat{\bd{\eta}}_{t}) + \lambda_{v} (\hat{\nabla} \cdot \hat{\bd{\eta}}_{t}) \bd{I} - \alpha \text{det}(\hat{\nabla} \hat{\bd{\Phi}}^{\eta}_{b}) \hat{p} (\hat{\nabla} \hat{\bd{\Phi}}^{\eta}_{b})^{-t},
\end{equation}
and where $\hat{\Phi}^{\eta}_{b}$ denotes the Lagrangian map defined in \eqref{lagrangianmap}. Note that the elastodynamics equation, which is the first equation in \eqref{biot}, is defined in the Lagrangian formulation on the reference domain $\hat{\Omega}_{b}$, whereas the momentum equation for the pore pressure, which is the second equation in \eqref{biot}, is defined in the Eulerian formulation on the moving domain $\Omega_{b}(t)$. In the second equation of \eqref{biot}, we can interpret the term on the right-hand side in terms of the \textit{filtration velocity}, where the velocity $\bd{q}$ of the flow through the poroelastic material is given by \textbf{Darcy's law} via the relationship:
\begin{equation}\label{darcy}
\bd{q} := -\kappa \nabla p,
\end{equation}
for a permeability constant $\kappa > 0$. For the Biot material, we assume the elasticity coefficients in \eqref{kirchhoff} satisfy $\mu_{e}, \lambda_{e} > 0$, and $\mu_{v}, \lambda_{v} \ge 0$. We say that the Biot medium is:
\begin{equation}\label{poroviscoelastic}
\begin{cases}
\textbf{poroelastic} \text{ if } \mu_{v} = 0 \text{ and } \lambda_{v} = 0 \\
\textbf{poroviscoelastic} \text{ if } \mu_{v} > 0 \text{ and } \lambda_{v} > 0.
\end{cases}
\end{equation}
We emphasize that our analysis in terms of existence of regularized interface weak solutions holds for both cases, and we carry out the explicit analysis in the more challenging case of a purely poroelastic Biot medium.

We emphasize the difference in notation for the displacement and the pore pressure between these two equations. While $\hat{\bd{\eta}}: \hat{\Omega}_{b} \to \R^{2}$ and $\hat{p}: \hat{\Omega}_{b} \to \R$ are defined on the reference Biot domain, $\bd{\eta}$ and $p$ are defined on the moving Biot domain via the Lagrangian map:
\begin{equation*}
\bd{\eta}(x, y) = \hat{\bd{\eta}}\Big((\hat{\bd{\Phi}}^{\eta}_{b})^{-1}(x, y)\Big), \qquad p(x, y) = \hat{p}\Big((\hat{\bd{\Phi}}^{\eta}_{b})^{-1}(x, y)\Big), \qquad \text{ for } (x, y) \in \Omega_{b}(t). 
\end{equation*}
It will also be convenient to define the Lagrangian Biot velocity $\hat{\bd{\xi}}$ and the Eulerian Biot velocity $\bd{\xi}$:
\begin{equation}\label{eulerianvelocity}
\hat{\bd{\xi}}(\hat{x}, \hat{y}) = \partial_{t}\hat{\bd{\eta}}(\hat{x}, \hat{y}) \ \ \text{ for } (\hat{x}, \hat{y}) \in \hat{\Omega}_{b}, \qquad \bd{\xi}(x, y) = \hat{\bd{\xi}}\Big((\hat{\bd{\Phi}}^{\eta}_{b})^{-1}(x, y)\Big) \ \  \text{ for } (x, y) \in \Omega_{b}(t).
\end{equation}
We note that these definitions are in agreement with our notational convention in Remark \ref{notation}.

Similarly, it will at times be convenient to consider the (visco)elastic Piola-Kirchhoff stress tensor defined in \eqref{kirchhoff} on the fixed domain $\hat{\Omega}_{b}$, on the moving domain $\Omega_{b}(t)$ instead. We can do this by transforming $\hat{S}_{b}(\nabla \bd{\eta}, p)$ via the Piola transform (see \cite{Ciarlet}), to obtain the resulting Piola-Kirchhoff stress tensor defined in Eulerian (moving domain) coordinates:
\begin{multline}\label{kirchlagrangian}
S_{b}(\nabla \bd{\eta}, p) = [\text{det}(\hat{\nabla} \hat{\bd{\Phi}}^{\eta}_{b})^{-1} \hat{S}_{b}(\hat{\nabla} \hat{\bd{\eta}}, \hat{p}) (\hat{\nabla} \hat{\bd{\Phi}}^{\eta}_{b})^{t}] \circ (\bd{\Phi}^{\eta}_{b})^{-1} \\
= \left(\frac{1}{\text{det}(\hat{\nabla}\hat{\bd{\Phi}}^{\eta}_{b})} \Big[2\mu_{e} \hat{D}(\hat{\bd{\eta}}) + \lambda_{e}(\hat{\nabla} \cdot \hat{\bd{\eta}}) + 2\mu_{v} \hat{\bd{D}}(\hat{\bd{\eta}}_{t}) + \lambda_{v} (\hat{\nabla} \cdot \hat{\bd{\eta}}_{t})\Big] (\hat{\nabla} \hat{\bd{\Phi}}^{\eta}_{b})^{t}\right) \circ (\bd{\Phi}^{\eta}_{b})^{-1} - \alpha p \bd{I} \\
\text{ on } \Omega_{b}(t).
\end{multline}

\subsection{The fluid subproblem}

We model the incompressible, viscous fluid using the incompressible Navier-Stokes equations, which describe the evolution of the fluid velocity $\bd{u}: \Omega_{f}(t) \to \R^{2}$ and the fluid pressure $\pi: \Omega_{f}(t) \to \R$ via equations for balance of momentum and conservation of mass (incompressibility) respectively:
\begin{equation}\label{NS}
\begin{cases}
\partial_{t} \bd{u} + (\bd{u} \cdot \nabla) \bd{u} = \nabla \cdot \bd{\sigma}_{f}(\nabla \bd{u}, \pi), \\
\nabla \cdot \bd{u} = 0, \\
\end{cases}
\text{ on } \Omega_{f}(t),
\end{equation}
where the Cauchy stress tensor $\bd{\sigma}_{f}(\nabla \bd{u}, \pi)$ for the (Newtonian) fluid is given by
\begin{equation*}
\bd{\sigma}_{f}(\nabla \bd{u}, \pi) = 2\mu \bd{D}(\bd{u}) - \pi \bd{I}, \qquad \bd{D}(\bd{u}) = \frac{1}{2}\Big((\nabla \bd{u}) + (\nabla \bd{u})^{t}\Big).
\end{equation*}
These equations are posed on the moving fluid domain $\Omega_{f}(t)$, which is determined by the Biot displacement $\hat{\bd{\eta}}$ via \eqref{omegaft}. 

\subsection{The coupling conditions}

We couple the Biot and fluid subproblems via two sets of coupling conditions: a kinematic coupling condition and a dynamic coupling condition. These coupling conditions are evaluated along the moving (a priori unknown) interface $\Gamma(t)$, which introduces geometric nonlinearities into the problem, and furthermore, these coupling conditions are evaluated along an interface $\Gamma(t)$ along which the fluid and Biot material \textit{are in direct contact with each other}. We recall the definition of the Eulerian Biot velocity $\bd{\xi}$ from \eqref{eulerianvelocity}. In addition, we define $\bd{n}(t)$ and $\bd{\tau}(t)$ to be the outward pointing unit normal vector and the clockwise pointing tangential vector respectively to $\Gamma(t)$. 

\medskip

\noindent \underline{\textbf{I. Kinematic coupling conditions.}} We consider two kinematic coupling conditions evaluated along the moving interface $\Gamma(t)$, one for the \textbf{continuity of the normal velocities}:
\begin{equation}\label{kin1}
\bd{u} \cdot \bd{n}(t) = (\bd{q} + \bd{\xi}) \cdot \bd{n}(t), \qquad \text{ on } \Gamma(t) \text{ for all } t,
\end{equation}
and also, the \textbf{Beavers-Joseph-Saffman condition} for the fluid tangential slip along the moving interface:
\begin{equation}\label{kin2}
\beta (\bd{\xi} - \bd{u}) \cdot \bd{\tau}(t) = \bd{\sigma}_{f}(\nabla \bd{u}, \pi) \bd{n}(t) \cdot \bd{\tau}(t), \qquad \text{ on } \Gamma(t) \text{ for all } t.
\end{equation}

\bigskip

\noindent \underline{\textbf{II. Dynamic coupling conditions.}} We consider two dynamic coupling conditions, one for \textbf{balance of stress}:
\begin{equation}\label{normalstress}
\bd{\sigma}_{f}(\nabla \bd{u}, \pi) \bd{n}(t)  = S_{b}(\nabla \bd{\eta}, p) \bd{n}(t), \qquad \text{ on } \Gamma(t) \text{ for all } t,
\end{equation}
where $S_{b}(\nabla \bd{\eta}, p)$ is the Piola-Kirchhoff stress tensor in Lagrangian coordinates, see \eqref{kirchlagrangian}, and also \textbf{balance of pressure} along the moving interface:
\begin{equation}\label{dyn2}
-\bd{\sigma}_{f}(\nabla \bd{u}, \pi) \bd{n}(t) \cdot \bd{n}(t) + \frac{1}{2}|\bd{u}|^{2} = p, \qquad \text{ on } \Gamma(t) \text{ for all } t.
\end{equation}

\subsection{The boundary and initial conditions}

We assume that the total domain $\hat{\Omega} := \{\hat{\bd{x}} \in \R^{2} : |\hat{\bd{x}}| < 2\}$ has an outer rigid wall. Therefore, we prescribe Dirichlet boundary conditions for the fluid velocity along the rigid wall:
\begin{equation}\label{boundaryconditions}
\bd{u} = \bd{0} \qquad \text{ along } \partial\hat{\Omega}.
\end{equation}
We also prescribe finite-energy initial data for the problem:
\begin{equation*}
\bd{u}(0, \cdot) = \bd{u}_{0} \in L^{2}(\Omega_{f}(0)), \quad \hat{\bd{\eta}}(0, \cdot) = \hat{\bd{\eta}}_{0} \in H^{1}(\hat{\Omega}_{b}), \quad \partial_t\hat{\bd{\eta}}(0, \cdot)=\hat{\bd{\xi}}_{0} \in L^2(\hat{\Omega}_{b}),  \quad \hat{p}(0, \cdot) = \hat{p}_0 \in L^{2}(\hat{\Omega}_{b}),
\end{equation*}
with 
\begin{equation}\label{U0CC}
\nabla \cdot \bd{u}_{0} = 0\; {\rm on}\; \Omega_{f}(0),
\quad \bd{u}_0\cdot\bd{n}(0) = 0\; {\rm on}\; \partial\hat{\Omega}.
\end{equation}

\subsection{A discussion of past work in FPSI}

The study of fluid-poroelastic structure interaction (FPSI) is a mathematical area of recent interest, featuring applications to physical and biological sciences, which has been studied through a variety of mathematical approaches: analytic, numerical, and modeling. FPSI has grown out of the broader and classical study of the dynamical interaction between incompressible fluids and deformable elastic structures, known as fluid-structure interaction (FSI). 

The study of FSI began with fluid-structure models with \textbf{linear geometric coupling}, in which the fluid and structure domains (although moving in practice) are fixed for the purposes of the model, and the fluid/structure equations are hence posed on fixed domains \cite{barGruLasTuff2, BarGruLasTuff, KukavicaTuffahaZiane}, as a model for small displacements. The study of FSI expanded to more complex models of fluid-structure dynamics with \textbf{nonlinear geometric coupling}, involving a priori unknown and time-dependent moving fluid and structure domains, which introduce additional geometric nonlinearities and significant mathematical difficulties into the mathematical analysis. These types of moving boundary models with nonlinear geometric coupling have been analyzed through a variety of approaches, including strong solution approaches \cite{BdV1, Lequeurre, Grandmont16, FSIforBIO_Lukacova} and weak solution approaches \cite{CDEM, CG, MuhaCanic13, LengererRuzicka, FSIforBIO_Lukacova} to models involving incompressible fluids occupying domains with elastic boundaries of lower dimensionality. In addition, models involving immersed elastic solids and deformable solids of the same dimension as the fluid domain have been analyzed in \cite{BKS24, BrKS24, CSS1, CSS2, ChengShkollerCoutand, ChenShkoller, Kuk, IgnatovaKukavica, ignatova2014well, Raymond}.

Poroelasticity has been of interest since the mid-1900s, when Biot proposed the \textbf{Biot equations} of poroelasticity for modeling poroelastic materials, in the context of geosciences and soil consolidation \cite{Biot1, Biot2}. Since then, the use of the Biot equations of poroelasticity have expanded into a much broader range of real-life applications, including the modeling of fractures and groundwater flows \cite{GWG15, LDQ11} and the modeling of biological tissues by poroelasticity in biomedical applications \cite{FluidsCanic,YifanDES,CGH14, YRC14}. The practical importance of poroelasticity and the abundance of poroelastic materials in science have motivated the rigorous mathematical analysis of the Biot equations, and by now, there has been extensive work on studying well-posedness for the Biot equations via rigorous mathematical PDE analysis \cite{BGM23, Biotwell1, Biotwell2, Biotwell3, Biotwell4, Biotwell5, Biotwell6, Biotwell7, Biotwell8,BiotWell22,BiotWell23}.

Motivated in particular by applications to problems in which poroelastic media interacts with an elastic or deformable solid, the study of fluid-poroelastic structure interaction has been a recent field of mathematical research, which inherits many of the mathematical difficulties (including complex coupling along interfaces, moving domains, geometric nonlinearities) of FSI problems at large and the mathematical challenges associated with the analysis of the Biot equations. FPSI models with linear geometric coupling (with equations posed on fixed interfaces) were studied in works such as \cite{AEN19, Ces17, Sho05}, and a FPSI model with linear geometric coupling and a poroelastic plate at the fluid-poroelastic material interface was considered in \cite{BCMW21}, arising for a need for a rigorous mathematical model for bioartificial organs in biomedical applications \cite{FluidsCanic}. 

In \cite{BCMW21}, the poroelastic plate exists at the interface, separating the fluid domain and the Biot domain containing the poroelastic material. The use of plates at the fluid-structure interface is physically motivated, since many biological tissues are multilayered with both thick elastic layers and thin elastic layers, which has already been observed for example in FSI applications \cite{BorSunMultiLayered,YifanCMAME,FSIStent,BCMW21}. However, in addition as has been proposed in \cite{BorSunMultiLayered} in the broader context of FSI, the thin layer at the interface between the thick elastic structure and fluid can serve as a regularizing mechanism for the interface dynamics, as the trace of the elastic structure displacement along the interface at the level of weak solutions does not have sufficient regularity to interpret the interface as a continuous well-defined interface, without additional regularity in the elastodynamics of the structure. 

The idea of using a thin elastic plate to regularize interface dynamics in moving boundary FSI problems was applied to FPSI in \cite{NonlinearFPSI_CRM23, FPSIJMPA}, where existence of a weak solution to a regularized weak formulation was achieved for a moving boundary FPSI model with nonlinear geometric coupling and a reticular plate at the Biot-fluid interface. The regularized weak formulation in this case involved regularizing terms in the weak formulation related to the geometry of the Lagrangian map, and it is shown in \cite{FPSIJMPA} that this regularization is consistent with classical solutions to the original non-regularized problem when such classical solutions exist, in the sense that these regularized solutions converge as the regularization parameter goes to zero, which is called \textit{weak-classical consistency}. This motivates the approach of a regularized interface solution used in the current manuscript, though it is crucial to note that the interface itself is not regularized in \cite{FPSIJMPA} because of the elastodynamics of the reticular plate, which already gives the interface displacement a regularity of $H^{2}_{0}(\Gamma)$. However, in the current model with direct contact, we use a regularized interface condition, where the interface displacement is $\bd{\eta}^{\delta}|_{\Gamma}$, given by a spatial convoluted version of $\bd{\eta}$, rather than $\bd{\eta}|_{\Gamma}$, since there is no plate at the interface to serve as a regularizing mechanism for the interface dynamics.

Such a direct contact problem between a Biot medium and a fluid is classical and of both practical and mathematical importance, as many mathematical well-posedness results for the Biot-fluid problem are direct contact problems. In the literature, we mention in particular the recent works \cite{Gurvich,AW25}, which study well-posedness for a direct contact Biot-fluid problem with a Beavers-Joseph-Saffman condition imposed along the (linearized) direct contact interface and \textit{linear geometric coupling}. In particular, this work establishes existence of strong solutions satisfying an energy identity and weak solutions satisfying an energy inequality for appropriate initial data and source terms, via a semigroup approach. \textit{The work in this manuscript extends work on direct contact Biot-fluid FPSI problems by providing, to the best of our knowledge, the first rigorous well-posedness result for a direct contact Biot/fluid FPSI problem with nonlinear geometric coupling/a priori unknown moving domains, hence representing a significant development in the field of FPSI.} We remark that such direct contact FPSI models are essential in applications, as shown by the wealth of mathematical literature on the numerical simulation of such models for example in the work \cite{QuainiQuarteroniPoroelastic}, by mixed formulation approaches \cite{MixedBiot, MixedBiot2, MixedBiot3, MixedBiot4} and loosely coupled partitioned schemes \cite{BukacBiot, OyekoleBukacFPSI}. As such, the current manuscript addresses a relevant problem of particular importance to the broader FSI/FPSI mathematical and engineering communities.

Our approach relies fundamentally on a constructive existence scheme given by Lie operator splitting, first proposed in \cite{MuhaCanic13}, in which the full multiphysical FSI/FPSI problem is split into separate subproblems in a modular approach. This is a robust approach to analyzing FSI and FPSI problems both analytically in terms of well-posedness of FSI models via a constructive existence approach \cite{MuhaCanic13, BorSun3d, BorSunMultiLayered, BorSunNonLinearKoiter, BorSunSlip, BCMW21, TW20, MMNT22} and numerically especially for FPSI models in \cite{BYZFPSI, BukacBiot, OyekoleBukacFPSI}. We use this operator splitting approach to show constructive existence for an approximate problem that approximates the direct contact problem we study, which splits the interface displacement from the Biot/fluid dynamics.

Second, our approach builds on the regularized weak formulation framework developed in \cite{NonlinearFPSI_CRM23, FPSIJMPA} for problems involving bulk (poro)elasticity and a thin plate at the interface, transparent to fluid flow, and extends it by introducing a novel \textit{regularized interface method} that enables the analysis of {\emph{direct contact problems}} with nonlinear geometric coupling and bulk elasticity. We believe the regularized interface method proposed here is broadly applicable, including to FSI models with bulk elasticity and to nonlinear elastic materials in three spatial dimensions. In particular, this approach is well-suited for studying problems with \textbf{vector-valued elastic displacements in bulk elastic materials} within the weak solutions framework—a significant advance, as most existing works on weak solutions for FSI models focus on scalar displacements. Exceptions include \cite{Tawri3D, Grandmont2D, KSS2D, BorSunSlip, FSIStent, BKS24}, which consider vector-valued displacements for elastic plates or shells (rather than bulk materials), with \cite{Grandmont2D, KSS2D, BorSunSlip} addressing 2D fluids coupled to 1D structures, penalization of domain degeneracy via compression in \cite{KSS2D}, and additional a priori assumptions on Lipschitz domains in \cite{FSIStent}. Moreover, in \cite{BKS24, BrKS24}, weak solutions to FSI problems with bulk $3D$ viscoelasticity were obtained, but only for higher-order viscoelasticity models. While our exposition focuses on a 2D/2D bulk elasticity problem, we expect the regularized interface method to be robust for FSI problems with vector-valued displacements in higher dimensions, and we introduce new techniques for controlling geometric degeneracies in bulk elastic materials with vector-valued displacements.

\section{Definition of the weak formulation}\label{weaksection}

\subsection{Weak formulation on the physical domain}\label{derivation}

We first define an appropriate weak formulation for our nonlinearly geometrically coupled FPSI problem with direct contact on the physical domain by testing against sufficiently regular test functions. We derive the weak formulation for the full problem as follows. 

\medskip

\noindent \textbf{Fluid momentum equation.} We test the momentum equation in the fluid equation \eqref{NS} by a fluid test function $\bd{v}$ by multiplying by $\bd{v}$ and then integrating over $\Omega_{f}(t)$. We obtain by using the Reynold's transport theorem and integration by parts:
\begin{multline}\label{weakderive1}
\int_{\Omega_{f}(t)} (\partial_{t} \bd{u} + (\bd{u} \cdot \nabla) \bd{u}) - (\nabla \cdot \bd{\sigma}_{f}(\nabla \bd{u}, \pi))) \cdot \bd{v} = \frac{d}{dt}\int_{\Omega_{f}(t)} \bd{u} \cdot \bd{v} - \int_{\Omega_{f}(t)} \bd{u} \cdot \partial_{t} \bd{v} + 2\nu \int_{\Omega_{f}(t)} \bd{D}(\bd{u}) : \bd{D}(\bd{v}) \\
+ \frac{1}{2} \int_{\Omega_{f}(t)} [((\bd{u} \cdot \nabla) \bd{u}) \cdot \bd{v} - ((\bd{u} \cdot \nabla) \bd{v}) \cdot \bd{u}] + \frac{1}{2} \int_{\Gamma(t)} (\bd{u} \cdot \bd{n} - 2\bd{\xi} \cdot \bd{n}) \bd{u} \cdot \bd{v} - \int_{\Gamma(t)} \bd{\sigma}_{f}(\nabla \bd{u}, \pi) \bd{n} \cdot \bd{v},
\end{multline}
where $\bd{\xi}$ is the Eulerian structure velocity defined in \eqref{eulerianvelocity}.

\medskip

\noindent \textbf{Biot elasticity equation.} We test the first elastodynamics equation in the Biot equations \eqref{biot} by a test function $\hat{\bd{\psi}}$ and obtain by integration by parts, by using properties of the Piola transform, and by using the definition of the viscoelastic Piola-Kirchhoff stress tensor in \eqref{kirchhoff}:
\begin{multline}\label{weakderive2}
\int_{\hat{\Omega}_{b}} (\rho_{b} \partial_{tt}\hat{\bd{\eta}} - \hat{\nabla} \cdot \hat{S}_{b} (\hat{\nabla} \hat{\bd{\eta}}, \hat{p}) \cdot \hat{\bd{\psi}} = \rho_{b} \frac{d}{dt} \int_{\hat{\Omega}_{b}} \partial_{t}\hat{\bd{\eta}} \cdot \hat{\bd{\psi}} - \rho_{b} \int_{\hat{\Omega}_{b}} \partial_{t}\hat{\bd{\eta}} \cdot \partial_{t}\hat{\bd{\psi}} - \alpha \int_{\Omega_{b}(t)} p(\nabla \cdot \bd{\psi}) \\
+ \int_{\hat{\Omega}_{b}} (2\mu_{e} \hat{D}(\hat{\bd{\eta}}) : \hat{\bd{D}}(\hat{\bd{\psi}}) + \lambda_{e} (\hat{\nabla} \cdot \hat{\bd{\eta}}) (\hat{\nabla} \cdot \hat{\bd{\psi}}) + 2\mu_{v} \hat{\bd{D}}(\partial_{t}\hat{\bd{\eta}}) : \hat{\bd{D}}(\hat{\bd{\psi}}) + \lambda_{v} (\hat{\nabla} \cdot \partial_{t}\hat{\bd{\eta}}) (\nabla \cdot \hat{\bd{\psi}}) + \int_{\hat{\Gamma}} \hat{S}_{b}(\hat{\nabla} \hat{\bd{\eta}}, \hat{p}) \bd{e}_{r} \cdot \hat{\bd{\psi}},
\end{multline}
where $p = \hat{p} \circ (\bd{\Phi}^{\eta}_{b})^{-1}$ and $\bd{\psi} = \hat{\bd{\psi}} \circ (\bd{\Phi}^{\eta}_{b})^{-1}$.

\medskip

\noindent \textbf{Biot pore pressure equation.} We multiply the second pore pressure equation in the Biot equations \eqref{biot} by a pore pressure test function $r$ and integrate by parts and use a change of variables via the Lagrangian map $\hat{\bd{\Phi}}^{\eta}_{b}$, recalling that the Darcy velocity $\bd{q}$ is defined as $\bd{q} := -\kappa \nabla p$ on $\Omega_{b}(t)$:
\begin{multline}\label{weakderive3}
\int_{\Omega_{b}(t)} \left(\frac{c_{0}}{[\text{det}(\hat{\nabla}\hat{\bd{\Phi}}^{\eta}_{b})] \circ (\bd{\Phi}^{\eta}_{b})^{-1}} \frac{D}{Dt} p + \alpha \nabla \cdot \frac{D}{Dt} \bd{\eta} - \nabla \cdot (\kappa \nabla p) \right) r \\
= c_{0} \frac{d}{dt} \int_{\hat{\Omega}_{b}} \hat{p} \cdot \hat{r} - c_{0} \int_{\hat{\Omega}_{b}} \hat{p} \cdot \partial_{t}\hat{r} - \alpha \int_{\Omega_{b}(t)} \frac{D}{Dt} \bd{\eta} \cdot \nabla r + \kappa \int_{\Omega_{b}(t)} \nabla p \cdot \nabla r - \alpha \int_{\Gamma(t)} (\bd{\xi} \cdot \bd{n}) r - \int_{\Gamma(t)} (\bd{q} \cdot \bd{n}) r. 
\end{multline}

\medskip

\noindent \textbf{Conclusion of the derivation.} To obtain the final weak formulation, we add together the results of \eqref{weakderive1}, \eqref{weakderive2}, and \eqref{weakderive3}. Since the filtration velocity $\bd{q}$ is not explicitly solved for in a weak solution, we use the continuity of velocity condition in \eqref{kin1} to rewrite the term $\displaystyle \int_{\Gamma(t)} (\bd{q} \cdot \bd{n}) r$ in \eqref{weakderive3} as
\begin{equation*}
\int_{\Gamma(t)} (\bd{q} \cdot \bd{n}) r = \int_{\Gamma(t)} ((\bd{u} - \bd{\xi}) \cdot \bd{n}) r.
\end{equation*}
In addition, we can combine the boundary terms arising from the fluid and Biot stress tensors as follows. By the continuity of normal stresses in \eqref{normalstress}, we obtain that
\begin{align*}
-\int_{\Gamma(t)} \bd{\sigma}_{f}(\nabla \bd{u}, \pi) \bd{n} \cdot \bd{v} &+ \int_{\hat{\Gamma}} \hat{S}_{b}(\hat{\nabla} \hat{\bd{\eta}}, \hat{p}) \bd{e}_{r} \cdot \hat{\bd{\psi}} = \int_{\Gamma(t)} \bd{\sigma}_{f}(\nabla \bd{u}, \pi) \bd{n} \cdot (\bd{\psi} - \bd{v}) \\
&= \int_{\Gamma(t)} \bd{\sigma}_{f}(\nabla \bd{u}, \pi) \bd{n} \cdot \bd{n} [(\bd{\psi} - \bd{v}) \cdot \bd{n}] + \int_{\Gamma(t)} \bd{\sigma}_{f}(\nabla \bd{u}, \pi) \bd{n} \cdot \bd{\tau} [(\bd{\psi} - \bd{v}) \cdot \bd{\tau}] \\
&= \int_{\Gamma(t)} \left(\frac{1}{2} |\bd{u}|^{2} - p\right) (\bd{\psi} - \bd{v}) \cdot \bd{n} + \int_{\Gamma(t)} \beta (\bd{\xi} - \bd{u}) \cdot \bd{\tau} (\bd{\psi} - \bd{v}) \cdot \bd{\tau},
\end{align*}
where we used the coupling conditions \eqref{kin2} and \eqref{dyn2} in the last step. We therefore, after integrating from $t = 0$ to $t = T$ in the time variable, have obtained the following weak formulation for the direct contact FPSI problem with nonlinear geometric coupling.

\begin{definition}[Weak solution, moving domain formulation]\label{weakmoving}
We say that $(\bd{u}, \bd{\eta}, p)$ is a weak solution to the FPSI problem if for all $C_{c}^{1}([0, T))$ (continuously differentiable) test functions $(\bd{v}, \bd{\psi}, r)$ taking values in an appropriate test space with $\bd{v}$ defined on $\Omega_{f}(t)$, and both $\bd{\psi}$ and $r$ defined on $\Omega_{b}(t)$, the following weak formulation holds:
\begin{multline*}
-\int_{0}^{T} \int_{\Omega_{f}(t)} \bd{u} \cdot \partial_{t} \bd{v} + \frac{1}{2} \int_{0}^{T} \int_{\Omega_{f}(t)} [((\bd{u} \cdot \nabla) \bd{u}) \cdot \bd{v} - ((\bd{u} \cdot \nabla) \bd{v}) \cdot \bd{u}] + \frac{1}{2} \int_{0}^{T} \int_{\Gamma(t)} (\bd{u} \cdot \bd{n} - 2\bd{\xi} \cdot \bd{n}) \bd{u} \cdot \bd{v} \\
+ 2\mu \int_{0}^{T} \int_{\Omega_{f}(t)} \bd{D}(\bd{u}) : \bd{D}(\bd{v}) + \int_{0}^{T} \int_{\Gamma(t)} \left(\frac{1}{2} |\bd{u}|^{2} - p\right) (\bd{\psi} - \bd{v}) \cdot \bd{n}(t) + \beta \int_{0}^{T} \int_{\Gamma} (\bd{\xi} - \bd{u}) \cdot \bd{\tau} (\bd{\psi} - \bd{v}) \cdot \bd{\tau} \\
- \rho_{b} \int_{0}^{T} \int_{\hat{\Omega}_{b}} \partial_{t} \hat{\bd{\eta}} \cdot \partial_{t} \hat{\bd{\psi}} + 2\mu_{e} \int_{0}^{T} \int_{\hat{\Omega}_{b}} \hat{\bd{D}}(\hat{\bd{\eta}}) : \hat{\bd{D}}(\hat{\bd{\psi}}) + \lambda_{e} \int_{0}^{T} \int_{\hat{\Omega}_{b}} (\hat{\nabla} \cdot \hat{\bd{\eta}}) (\hat{\nabla} \cdot \hat{\bd{\psi}}) + 2\mu_{v} \int_{0}^{T} \int_{\hat{\Omega}_{b}} \hat{\bd{D}}(\partial_{t}\hat{\bd{\eta}}) : \hat{\bd{D}}(\hat{\bd{\psi}}) \\
+ \lambda_{v} \int_{0}^{T} \int_{\hat{\Omega}_{b}} (\hat{\nabla} \cdot \partial_{t}\hat{\bd{\eta}})(\hat{\nabla} \cdot \hat{\bd{\psi}}) - \alpha \int_{0}^{T} \int_{\Omega_{b}(t)} p \nabla \cdot \bd{\psi} - c_{0} \int_{0}^{T} \int_{\hat{\Omega}_{b}} \hat{p} \cdot \partial_{t}\hat{r} \\
- \alpha \int_{0}^{T} \int_{\Omega_{b}(t)} \frac{D}{Dt} \bd{\eta} \cdot \nabla r - \alpha \int_{0}^{T} \int_{\Gamma(t)} (\bd{\xi} \cdot \bd{n})r + \kappa \int_{0}^{T} \int_{\Omega_{b}(t)} \nabla p \cdot \nabla r - \int_{0}^{T} \int_{\Gamma(t)} ((\bd{u} - \bd{\xi}) \cdot \bd{n}) r \\
= \int_{\Omega_{f}(0)} \bd{u}_0 \cdot \bd{v}(0) + \rho_{b} \int_{\hat{\Omega}_{b}}\hat{\bd{\xi}}_{0}  \cdot \hat{\bd{\psi}}(0) + c_{0} \int_{\hat{\Omega}_{b}} \hat{p}_0 \cdot \hat{r}(0). 
\end{multline*}
\end{definition}

\subsection{Transformations between fixed and moving domains}\label{transformed}

Our next goal is to define an appropriate weak formulation on the fixed domain, to give a weak formulation in the fixed domain (Lagrangian) formulation. This will involve examining maps between fixed domains and moving (physical) domains, and examining how differential operators and functions transform between the fixed and moving domains. 

\medskip

\noindent \textbf{Transformation between the fixed and moving Biot domains.} We have the Lagrangian map $\hat{\bd{\Phi}}^{\eta}_{b}: \hat{\Omega}_{b} \to \Omega_{b}(t)$, defined in \eqref{lagrangianmap}. Given a scalar function $g: \Omega_{b}(t) \to \R$ or a vector-valued function $\bd{g}: \Omega_{b}(t) \to \R^{2}$, we can define corresponding transformed functions on the fixed domain via:
\begin{equation}\label{hatbiot}
\hat{g}(\hat{x}, \hat{y}) := g(\hat{\bd{\Phi}}^{\eta}_{b}(\hat{x}, \hat{y})), \qquad \hat{\bd{g}}(\hat{x}, \hat{y}) := \bd{g}(\hat{\bd{\Phi}}^{\eta}_{b}(\hat{x}, \hat{y})), \qquad \text{ for } (\hat{x}, \hat{y}) \in \hat{\Omega}_{b}.
\end{equation}
Furthermore, by the chain rule, we have the following formula for the transformation of spatial gradients via the Lagrangian map:
\begin{equation}\label{etanabla}
\hat{\nabla}^{\eta}_{b}\hat{g} = (\nabla g) \circ \hat{\bd{\Phi}}^{\eta}_{b}, \qquad \text{ where } \hat{\nabla}^{\eta}_{b} \hat{g} = \left(\frac{\partial \hat{g}}{\partial \hat{x}}, \frac{\partial \hat{g}}{\partial \hat{y}}\right) \cdot (\bd{I} + \hat{\nabla} \hat{\bd{\eta}})^{-1}.
\end{equation}
To clarify notation, $\displaystyle \hat{\nabla} \hat{g} := \left(\frac{\partial \hat{g}}{\partial \hat{x}}, \frac{\partial \hat{g}}{\partial \hat{y}}\right)$ is the (standard) gradient on the fixed domain $\hat{\Omega}_{b}$, $\displaystyle \nabla g := \left(\frac{\partial g}{\partial x}, \frac{\partial g}{\partial y}\right)$ is the (standard) gradient on the moving domain $\Omega_{b}(t)$, and $\hat{\nabla}^{\eta}_{b}$ is the transformed gradient on the reference domain $\hat{\Omega}_{b}$. The Jacobian of the Lagrangian map will be denoted by
\begin{equation}\label{jeta}
\hat{\mathcal{J}}^{\eta}_{b} := \text{det}(\bd{I} + \hat{\nabla} \hat{\bd{\eta}}).
\end{equation}

\medskip

\noindent \textbf{Transformation between the fixed and moving fluid domains.} Given the vector-valued interface displacement $\hat{\bd{\eta}}|_{\hat{\Gamma}}$, we will define an associated \textbf{Arbitrary Lagrangian-Eulerian (ALE) map} $\hat{\bd{\Phi}}^{\eta}_{f}: \hat{\Omega}_{f} \to \Omega_{f}(t)$. Because $\hat{\bd{\eta}}|_{\hat{\Gamma}}$ is a two-dimensional vector field along $\hat{\Gamma}$, we use the approach of \textit{harmonic extension} to define the ALE map. Namely, we define $\hat{\bd{\Phi}}^{\eta}_{f}: \overline{\hat{\Omega}_{f}} \to \R^{2}$ to be the unique solution to the following elliptic problem:
\begin{equation}\label{aleeta1}
\Delta \hat{\bd{\Phi}}^{\eta}_{f} = \bd{0} \qquad \text{ on } \hat{\Omega}_{f},
\end{equation}
with boundary conditions
\begin{equation}\label{aleeta2}
\hat{\bd{\Phi}}^{\eta}_{f}(\hat{x}, \hat{y}) = (\hat{x}, \hat{y}) + \hat{\bd{\eta}}|_{\hat{\Gamma}}(\hat{x}, \hat{y}) \text{ on } \hat{\Gamma}, \qquad \hat{\bd{\Phi}}^{\eta}_{f}(\hat{x}, \hat{y}) = (\hat{x}, \hat{y}) \text{ on } \partial \hat{\Omega}_{f} \setminus \hat{\Gamma}.
\end{equation}
By elliptic regularity theorems and Sobolev embedding (see \cite{Grisvard} and the discussion preceding Proposition 2.1 in \cite{Tawri3D}), we have that for all $p \ge 2$, 
\begin{equation*}
\|\hat{\bd{\Phi}}^{\eta}_{f}\|_{W^{2, p}(\hat{\Omega}_{f})} \le C_{p}\|(\hat{\bd{\eta}}|_{\hat{\Gamma}})\|_{W^{2 - \frac{1}{p}, p}(\hat{\Gamma})} \le C_{p}\|(\hat{\bd{\eta}}|_{\hat{\Gamma}})\|_{H^{\frac{5}{2} - \frac{2}{p}}(\hat{\Gamma})}. 
\end{equation*}
We then define the Jacobian of the ALE mapping $\hat{\bd{\Phi}}^{\eta}_{f}: \hat{\Omega}_{f} \to \Omega_{f}(t)$ by
\begin{equation*}
\mathcal{J}^{\eta}_{f} := \text{det}(\hat{\bd{\Phi}}^{\eta}_{f})
\end{equation*}
and we define the associated velocity of the ALE map defined on the reference domain $\hat{\Omega}_{f}$:
\begin{equation}\label{ALEvel}
\hat{\bd{w}}^{\eta}(\hat{x}, \hat{y}) = \partial_{t}\hat{\bd{\Phi}}^{\eta}_{f}(\hat{x}, \hat{y}), \qquad \text{ for } (\hat{x}, \hat{y}) \in \hat{\Omega}_{f}.
\end{equation}
We can use the ALE map to transform functions from the moving (physical) domain to the reference domain: given a function $h$ defined on $\Omega_{f}(t)$, we can define
\begin{equation}\label{hatfluid}
\hat{h}(\hat{x}, \hat{y}) := h(\hat{\bd{\Phi}}^{\eta}_{f}(\hat{x}, \hat{y})), \qquad \text{ for } (\hat{x}, \hat{y}) \in \hat{\Omega}_{f},
\end{equation}
and by the chain rule, we have the following transformation rule for spatial derivatives:
\begin{equation}\label{gradtransform}
\hat{\nabla}^{\eta}_{f} \hat{h} = (\nabla h) \circ \hat{\bd{\Phi}}^{\eta}_{f} \quad \text{ where } \hat{\nabla}^{\eta}_{f} \hat{h} := \left(\frac{\partial \hat{h}}{\partial \hat{x}}, \frac{\partial \hat{h}}{\partial \hat{y}}\right) \cdot \Big(\hat{\nabla} \hat{\bd{\Phi}}^{\eta}_{f}\Big)^{-1},
\end{equation}
analogously to \eqref{etanabla}, and by taking traces, we can define the transformed divergence:
\begin{equation}\label{transformdiv}
\hat{\nabla}^{\eta}_{f} \cdot \hat{h} = (\nabla \cdot h) \circ \hat{\bd{\Phi}}^{\omega}_{f} \quad \text{ where } \hat{\nabla}^{\eta}_{f} \cdot \hat{h} := \text{tr}\left[\left(\frac{\partial \hat{h}}{\partial \hat{x}}, \frac{\partial \hat{h}}{\partial \hat{y}}\right) \cdot \Big(\hat{\nabla} \hat{\bd{\Phi}}^{\eta}_{f}\Big)^{-1}\right].
\end{equation}
We also have the following transformation rule for time derivatives;
\begin{equation}\label{aletime}
\partial_{t} h(\bd{\Phi}^{\eta}_{f}(\hat{x}, \hat{y})) = \partial_{t}\hat{h}(\hat{x}, \hat{y}) - (\hat{\bd{w}}^{\eta} \cdot \hat{\nabla}^{\eta}_{f}) \hat{h}(\hat{x}, \hat{y}).
\end{equation}

More generally, given a (vector-valued) function $\hat{\bd{\omega}}: \hat{\Gamma} \to \R^{2}$, we define the associated fluid ALE map $\hat{\bd{\Phi}}^{\omega}_{f}: \hat{\Omega}_{f} \to \Omega_{f}(t)$ as the solution to 
\begin{equation}\label{alef1}
\Delta \hat{\bd{\Phi}}^{\omega}_{f} = \bd{0} \qquad \text{ on } \hat{\Omega}_{f},
\end{equation}
with boundary conditions
\begin{equation}\label{alef2}
\hat{\bd{\Phi}}^{\omega}_{f}(\hat{x}, \hat{y}) = (\hat{x}, \hat{y}) + \hat{\bd{\omega}}(\hat{x}, \hat{y}) \text{ on } \hat{\Gamma}, \qquad \hat{\bd{\Phi}}^{\omega}_{f}(\hat{x}, \hat{y}) = (\hat{x}, \hat{y}) \text{ on } \partial \hat{\Omega}_{f} \setminus \hat{\Gamma}.
\end{equation}
In this context, we have the analogous definitions of
\begin{equation*}
\hat{h}(\hat{x}, \hat{y}) := h(\hat{\bd{\Phi}}^{\omega}_{f}(\hat{x}, \hat{y})), \qquad \hat{\nabla}^{\omega}_{f}\hat{h} = (\nabla h) \circ \hat{\bd{\Phi}}^{\omega}_{f} \quad \text{ where } \hat{\nabla}^{\omega}_{f}\hat{h} := \left(\frac{\partial \hat{h}}{\partial \hat{x}}, \frac{\partial \hat{h}}{\partial \hat{y}}\right) \cdot \left(\hat{\nabla} \hat{\bd{\Phi}}^{\omega}_{f}\right)^{-1},
\end{equation*}
\begin{equation}\label{omegadiff}
\hat{\nabla}^{\omega}_{f} \cdot \hat{h} = (\nabla \cdot h) \circ \hat{\bd{\Phi}}^{\omega}_{f} \quad \text{ where } \hat{\nabla}^{\omega}_{f} \cdot \hat{h} := \text{tr}\left[\left(\frac{\partial \hat{h}}{\partial \hat{x}}, \frac{\partial \hat{h}}{\partial \hat{y}}\right) \cdot \left(\hat{\nabla} \hat{\bd{\Phi}}^{\omega}_{f}\right)^{-1}\right].
\end{equation}
We can also define the ALE velocity $\hat{\bd{w}}^{\omega} = \partial_{t}\hat{\bd{\Phi}}^{\omega}_{f}$ and we obtain a similar transformation rule for time derivatives as in \eqref{aletime}.

\begin{remark}[Notation remarks]
Using the more general notational convention for the ALE map in \eqref{alef1} and \eqref{alef2} given a vector-valued interface displacement $\hat{\bd{\omega}}$, the definition of the ALE map defined the Biot displacement $\hat{\bd{\Phi}}^{\eta}_{f}$ given in \eqref{aleeta1} and \eqref{aleeta2} is more accurately described by $\hat{\bd{\Phi}}^{\eta|_{\Gamma}}_{f}$, since the interface displacement is given by $\hat{\bd{\eta}}|_{\hat{\Gamma}}$ in the direct contact fluid-Biot problem. However, we choose to use the notation $\hat{\bd{\Phi}}^{\eta}_{f}$ for notational simplicity. 

As a notational remark, note that we use the same notation for functions pulled back to the reference domain, for functions defined on the Biot domain via the Biot Lagrangian map $\hat{\bd{\Phi}}^{\eta}_{b}$ in \eqref{hatbiot} and for functions defined on the fluid domain via the ALE fluid map $\hat{\bd{\Phi}}^{\eta}_{f}$ in \eqref{hatfluid}. While we do not explicitly disambiguate the notation as both are denoted using the same ``hat" notation $\hat{h}$ for example, we rely on context to disambiguate these two notions, as it will be clear from the context whether the associated function is defined on the Biot domain or the fluid domain.
\end{remark}

\medskip

\noindent \textbf{Transformation between the fixed and moving interface.} Next, we describe the transformation from the fixed reference interface $\hat{\Gamma}$ and the moving interface $\Gamma(t)$. To do this, we note that it will be convenient to use a standard parametrization to parametrize the unit circle $\hat{\Gamma}$:
\begin{equation}\label{circle}
\hat{z} \in [0, 2\pi] \to (\cos(\hat{z}), \sin(\hat{z})) \in \hat{\Gamma}.
\end{equation}
Hence, given a (scalar-valued or vector-valued) function $\hat{f}(\hat{x}, \hat{y})$ for $(\hat{x}, \hat{y}) \in \hat{\Gamma}$, we can regard the function as a periodic function on $[0, 2\pi]$ as:
\begin{equation*}
\hat{f}(\hat{z}) := \hat{f}(\cos(\hat{z}), \sin(\hat{z})), \qquad \text{ for } \hat{z} \in [0, 2\pi].
\end{equation*}
To simplify notation, we will implicitly switch between the notations $\hat{f}(\hat{z})$ and $\hat{f}(\hat{x}, \hat{y})$ for functions $\hat{f}$ defined on $\hat{\Gamma}$, without explicitly notating the composition of the function with the standard parametrization \eqref{circle}.  

Next, we define the transformation between $\hat{\Gamma}$ and $\Gamma(t)$, described by the vector-valued interface displacement $\hat{\bd{\eta}}|_{\hat{\Gamma}}: \hat{\Gamma} \to \R^{2}$, denoted component-wise as $\hat{\bd{\eta}}|_{\hat{\Gamma}}(\hat{z}) := (\hat{\eta}_{1}|_{\hat{\Gamma}}(\hat{z}), \hat{\eta}_{2}|_{\hat{\Gamma}}(\hat{z}))$, $\hat{z} \in [0, 2\pi]$, for the (vector-valued) Biot displacement $\hat{\bd{\eta}} = (\hat{\eta}_{1}, \hat{\eta}_{2})$. We define the map $\bd{\Phi}^{\eta}_{\Gamma}: \hat{\Gamma} \to \Gamma(t)$ defined by
\begin{equation}\label{inteta}
\hat{\bd{\Phi}}^{\eta}_{\Gamma}(\hat{z}) := (\cos(\hat{z}), \sin(\hat{z})) + \hat{\bd{\eta}}|_{\hat{\Gamma}}(\hat{z}) = (\cos(\hat{z}) + \hat{\eta}_{1}|_{\hat{\Gamma}}(\hat{z}), \sin(\hat{z}) + \hat{\eta}_{2}|_{\hat{\Gamma}}(\hat{z})). 
\end{equation}
In addition, we define some relevant quantities related to the geometry of the moving interface $\Gamma(t)$. We first define the arc length element, given by
\begin{equation}\label{arclengthdef}
\hat{\mathcal{S}}^{\eta}_{\Gamma} := \sqrt{(-\sin(\hat{z}) + \partial_{\hat{z}}\hat{\eta}_{1}|_{\hat{\Gamma}})^{2} + (\cos(\hat{z}) + \partial_{\hat{z}}\hat{\eta}_{2}|_{\hat{\Gamma}})^{2}},
\end{equation}
and we define the normal vector $\bd{n}$ and the tangent vector $\bd{\tau}$ to the moving interface $\Gamma(t)$ (with a standardized orientation) by
\begin{equation}\label{ntdef}
\bd{n} := \frac{1}{\hat{S}^{\eta}_{\Gamma}} (\cos(\hat{z}) + \partial_{\hat{z}}\hat{\eta}_{2}|_{\hat{\Gamma}}, \sin(\hat{z}) - \partial_{\hat{z}}\hat{\eta}_{1}|_{\hat{\Gamma}}), \qquad \bd{\tau} := \frac{1}{\hat{S}^{\eta}_{\Gamma}} (-\sin(\hat{z}) + \partial_{\hat{z}}\hat{\eta}_{1}|_{\hat{\Gamma}}, \cos(\hat{z}) + \partial_{\hat{z}}\hat{\eta}_{2}|_{\hat{\Gamma}}).
\end{equation}
As for the fluid ALE map in \eqref{alef1} and \eqref{alef2}, it will at times be more convenient to instead emphasize the dependence of the Lagrangian map for the interface by its vector-valued displacement $\hat{\bd{\omega}}$ in which case, the notation is analogous:
\begin{equation}\label{intomega}
\hat{\bd{\Phi}}^{\omega}_{\Gamma}(\hat{z}) := (\cos(\hat{z}), \sin(\hat{z})) + \hat{\bd{\omega}}(\hat{z}) = (\cos(\hat{z}) + \hat{\omega}_{1}(\hat{z}), \sin(\hat{z}) + \hat{\omega}_{2}(\hat{z})) \quad \text{ for } \hat{\bd{\Phi}}^{\omega}_{\Gamma}: \hat{\Gamma} \to \Gamma(t),
\end{equation}
\begin{equation*}
\hat{\mathcal{S}}^{\omega}_{\Gamma} := \sqrt{(-\sin(\hat{z}) + \partial_{\hat{z}}\hat{\omega}_{1})^{2} + (\cos(\hat{z}) + \partial_{\hat{z}} \hat{\omega}_{2})^{2}},
\end{equation*}
\begin{equation}\label{ntdef2}
\bd{n} := \frac{1}{\hat{S}^{\omega}_{\Gamma}} (\cos(\hat{z}) + \partial_{\hat{z}}\hat{\omega}_{2}, \sin(\hat{z}) - \partial_{\hat{z}}\hat{\omega}_{1}), \quad \bd{\tau} := \frac{1}{\hat{\mathcal{S}}^{\omega}_{\Gamma}}(-\sin(\hat{z}) + \partial_{\hat{z}}\hat{\omega}_{1}, \cos(\hat{z}) + \partial_{\hat{z}}\hat{\omega}_{2}).
\end{equation}

\medskip

\noindent \textbf{Some results about injectivity.} We observe that these transformations between fixed and moving domains that we have defined for the fluid, the Biot medium, and the fluid-Biot interface, are only well-defined when the maps are injective, and therefore, it will be useful to have ways of characterizing injectivity. We start with the following result about injectivity of the map for the moving interface.

\begin{proposition}\label{gammainjective}
Let $\hat{\bd{\omega}}^{*}: [0, 2\pi] \to \R^{2}$ be a continuously differentiable $C^{1}(\hat{\Gamma})$ function with $\hat{\bd{\omega}}^{*}(0) = \hat{\bd{\omega}}^{*}(2\pi)$, which satisfies the following non-degeneracy condition for some positive constant $\alpha > 0$:
\begin{equation*}
|(-\sin(\hat{z}), \cos(\hat{z})) + \partial_{\hat{z}}\hat{\bd{\omega}}^{*}(\hat{z})| \ge \alpha \text{ for all } \hat{z} \in [0, 2\pi],
\end{equation*}
and $\hat{\bd{\Phi}}^{\omega^{*}}_{\Gamma}(\hat{z}) := (\cos(\hat{z}), \sin(\hat{z})) + \hat{\bd{\omega}}^{*}(\hat{z})$ is an injective map on $\hat{\Gamma}$ such that $\bd{\Phi}^{\omega^{*}}_{\Gamma}(\hat{\Gamma}) \cap \partial \hat{\Omega} = \varnothing$ for $\hat{\Omega} := \{\hat{\bd{x}} \in \R^{2} : |\hat{\bd{x}}| < 2\}$. Then, there exists $\delta > 0$ sufficiently small such that for all $C^{1}(\hat{\Gamma})$ functions $\hat{\bd{\omega}}: [0, 2\pi] \to \R^{2}$ such that 
\begin{equation*}
\hat{\bd{\omega}}(0) = \hat{\bd{\omega}}(2\pi) \quad \text{ and } \quad \|\hat{\bd{\omega}} - \hat{\bd{\omega}}^{*}\|_{C^{1}(\hat{\Gamma})} \le \delta,
\end{equation*}
the map $\hat{\bd{\Phi}}^{\omega}_{\Gamma}(\hat{z}) := (\cos(\hat{z}), \sin(\hat{z})) + \hat{\bd{\omega}}(\hat{z})$ is injective and $\hat{\bd{\Phi}}^{\omega}_{\Gamma}(\hat{\Gamma}) \cap \partial \hat{\Omega} = \varnothing$, with 
\begin{equation}\label{nondegeneracy}
\|(-\sin(\hat{z}), \cos(\hat{z})) + \partial_{\hat{z}} \hat{\bd{\omega}}(\hat{z})\| \ge \frac{\alpha}{2} \text{ for all }\hat{z} \in [0, 2\pi].
\end{equation}

\end{proposition}

\begin{proof}
Consider the following (vector-valued) function $f: [0, 2\pi] \times [0, 2\pi] \to \R^{2}$, defined by
\begin{equation*}
f^{\omega}(z_{1}, z_{2}) = \begin{cases}
\displaystyle \frac{\hat{\bd{\Phi}}^{\omega}_{\Gamma}(\hat{z}_{1}) - \hat{\bd{\Phi}}^{\omega}_{\Gamma}(\hat{z}_{2})}{|z_{1} - z_{2}|}, \quad &\text{ if } \hat{z}_{1} \ne \hat{z}_{2}, \\
(-\sin(\hat{z}), \cos(\hat{z})) + \partial_{\hat{z}} \hat{\bd{\omega}}(\hat{z}), \quad &\text{ if } \hat{z}_{1} = \hat{z}_{2} = \hat{z}. \\
\end{cases}
\end{equation*}
Note that by the given assumptions on $\hat{\bd{\omega}}^{*}$, we have that $|f^{\omega^{*}}(\hat{z}_{1}, \hat{z}_{2})| > 0$ for all $\hat{z}_{1}, \hat{z}_{2} \in [0, 2\pi]$, and since $\hat{\bd{\omega}}^{*}$ is continuously differentiable, $f^{\omega^{*}}(\hat{z}_{1}, \hat{z}_{2})$ is continuous on $[0, 2\pi] \times [0, 2\pi]$. So there exists some positive constant $c$ such that $f^{\omega^{*}}(\hat{z}_{1}, \hat{z}_{2}) \ge c > 0$ for $\hat{z}_{1}, \hat{z}_{2} \in [0, 2\pi]$. For $\hat{z}_{1} \ne \hat{z}_{2}$, 
\begin{align*}
|f^{\omega}(\hat{z}_{1}, \hat{z}_{2}) - f^{\omega^{*}}(\hat{z}_{1}, \hat{z}_{2})| &= \left|\frac{\hat{\bd{\Phi}}^{\omega}_{\Gamma}(\hat{z}_{1}) - \hat{\bd{\Phi}}^{\omega}_{\Gamma}(\hat{z}_{2})}{|\hat{z}_{1} - \hat{z}_{2}|} - \frac{\hat{\bd{\Phi}}^{\omega^{*}}_{\Gamma}(\hat{z}_{1}) - \hat{\bd{\Phi}}^{\omega^{*}}_{\Gamma}(\hat{z}_{2})}{|\hat{z}_{1} - \hat{z}_{2}|}\right| \\
&= \frac{|(\hat{\bd{\omega}} - \hat{\bd{\omega}}^{*})(\hat{z}_{1}) - (\hat{\bd{\omega}} - \hat{\bd{\omega}}^{*})(\hat{z}_{2})|}{|\hat{z}_{1} - \hat{z}_{2}|} \le \|\hat{\bd{\omega}} - \hat{\bd{\omega}}^{*}\|_{C^{1}(\hat{\Gamma})}.
\end{align*}
Therefore, there exists $\delta > 0$ sufficiently small such that for all $\hat{\bd{\omega}}$ with $\hat{\bd{\omega}}(0) = \hat{\bd{\omega}}(2\pi)$ and $\|\hat{\bd{\omega}} - \hat{\bd{\omega}}^{*}\|_{C^{1}(\hat{\Gamma})} \le \delta$, we have that $f^{\omega}(\hat{z}_{1}, \hat{z}_{2}) \ge c/2 > 0$ for all $\hat{z}_{1}, \hat{z}_{2} \in [0, L]$, which implies that $\hat{\bd{\Phi}}^{\omega}_{\Gamma}$ is injective. Finally, to show that $\hat{\bd{\Phi}}^{\omega}_{\Gamma}(\hat{\Gamma}) \cap \partial \hat{\Omega} = \varnothing$ for $\delta$ sufficiently small, note that $\text{dist}\Big(\hat{\bd{\Phi}}^{\omega^{*}}_{\Gamma}(\hat{\Gamma}), \partial \hat{\Omega}\Big) \ge c > 0$ for some positive constant $c$, so for $\|\hat{\bd{\omega}} - \hat{\bd{\omega}}^{*}\|_{C^{1}(\hat{\Gamma})} \le \delta$ for $\delta < c/2$, we have that $\text{dist}\Big(\hat{\bd{\Phi}}^{\omega}_{\Gamma}(\hat{\Gamma}), \partial \hat{\Omega}\Big) \ge c/2$. The non-degeneracy claim in \eqref{nondegeneracy} follows similarly.

\end{proof}

\begin{proposition}\label{mapinjective}
Consider a $C^{1}(\hat{\Gamma})$ function $\hat{\bd{\omega}}: [0, 2\pi] \to \R^{2}$ such that $\hat{\bd{\omega}}(0) = \hat{\bd{\omega}}(2\pi)$, $\hat{\bd{\Phi}}^{\omega}_{f}$ is injective on $[0, L]$, $\hat{\bd{\Phi}}^{\omega}_{f}(\hat{\Gamma}) \cap \partial \hat{\Omega} = \varnothing$, and \eqref{nondegeneracy} is satisfied for some positive constant $\alpha > 0$. 
\begin{itemize}
\item If $\text{det}(\bd{I} + \nabla \hat{\bd{\eta}}) > 0$ on $\hat{\Omega}_{b}$ for some $C^{1}$ function $\hat{\bd{\eta}}$ on $\hat{\overline{\Omega}}_{b}$, then $\hat{\bd{\Phi}}^{\eta}_{b}$ is a bijective map from $\hat{\Omega}_{b}$ to $\Omega_{b}(t)$. 
\item If $\text{det}(\hat{\bd{\Phi}}^{\omega}_{f}) > 0$ on $\hat{\Omega}_{f}$ where $\hat{\bd{\Phi}}^{\omega}_{f}$ is $C^{1}$ on $\hat{\overline{\Omega}}_{f}$, then $\hat{\bd{\Phi}}^{\omega}_{f}$ is a bijective map from $\hat{\Omega}_{f}$ to $\Omega_{f}(t)$. 
\end{itemize}
\end{proposition}

\begin{proof}
This is a direct consequence of Theorem 5-5-2 in \cite{Ciarlet} and the previous proposition.
\end{proof}

\subsection{Weak formulation on the fixed domain}

Using the function transformations defined in the previous section, we can consider the weak formulation to the direct contact FPSI problem on the fixed domain. For this purpose, we introduce the rescaled normal and tangent vectors, compare with the unit normal and tangent vectors defined in \eqref{ntdef}, as follows:
\begin{equation}\label{rescalednt}
\hat{\bd{n}}^{\eta} := (\cos(\hat{z}) + \partial_{\hat{z}}\hat{\eta}_{2}|_{\hat{\Gamma}}, \sin(\hat{z}) -  \partial_{\hat{z}}\hat{\eta}_{1}|_{\hat{\Gamma}}), \qquad \hat{\bd{\tau}}^{\eta} = (-\sin(\hat{z}) + \partial_{\hat{z}}\hat{\eta}_{1}|_{\hat{\Gamma}}, \cos(\hat{z}) + \partial_{\hat{z}}\hat{\eta}_{2}|_{\hat{\Gamma}}),
\end{equation}
and we also recall the definition of the arc length element from \eqref{arclengthdef}. Using the maps from the reference to moving domains for the Biot medium, the fluid domain, and the plate, we thus have the following definition of a weak solution to the FPSI problem on the reference domain.

\begin{definition}[Weak formulation, fixed domain formulation]\label{woregularitydef}
We say that $(\hat{\bd{u}}, \hat{\bd{\eta}}, \hat{p})$ is a weak solution to the FPSI problem (on the fixed domain) if for all $C_{c}^{1}([0, T))$ (continuously differentiable) test functions $(\hat{\bd{v}}, \hat{\bd{\psi}}, \hat{r})$ for the fluid velocity, Biot displacement, and pore pressure respectively, taking values in an appropriate test space, the following weak formulation holds:
\begin{multline*}
-\int_{0}^{T} \int_{\hat{\Omega}_{f}} \hat{\mathcal{J}}^{\eta}_{f} \hat{\bd{u}} \cdot \partial_{t} \hat{\bd{v}} + \frac{1}{2} \int_{0}^{T} \int_{\hat{\Omega}_{f}} \hat{\mathcal{J}}^{\eta}_{f} \Big[((\hat{\bd{u}} - \hat{\bd{w}}^{\eta}) \cdot \hat{\nabla}^{\eta}_{f} \hat{\bd{u}}) \cdot \hat{\bd{v}} - ((\hat{\bd{u}} - \hat{\bd{w}}^{\eta}) \cdot \hat{\nabla}^{\eta}_{f} \hat{\bd{v}}) \cdot \hat{\bd{u}}\Big] \\
- \frac{1}{2} \int_{0}^{T} \int_{\hat{\Omega}_{f}} (\partial_{t}\hat{\mathcal{J}}^{\eta}_{f}) \hat{\bd{u}} \cdot \hat{\bd{v}} + \frac{1}{2} \int_{0}^{T} \int_{\hat{\Gamma}} (\hat{\bd{u}} - \hat{\bd{\xi}}) \cdot \hat{\bd{n}}^{\eta} (\hat{\bd{u}} \cdot \hat{\bd{v}}) + 2\nu \int_{0}^{T} \int_{\hat{\Omega}_{f}} \hat{\mathcal{J}}^{\eta}_{f} \hat{\bd{D}}^{\eta}_{f}(\hat{\bd{u}}) : \hat{\bd{D}}^{\eta}_{f}(\hat{\bd{v}}) \\
+ \int_{0}^{T} \int_{\hat{\Gamma}} \left(\frac{1}{2} |\hat{\bd{u}}|^{2} - \hat{p}\right) (\hat{\bd{\psi}} - \hat{\bd{v}}) \cdot \hat{\bd{n}}^{\eta} + \beta \int_{0}^{T} \int_{\hat{\Gamma}} \Big(\hat{\mathcal{S}}^{\eta}_{\Gamma}\Big)^{-1} (\hat{\bd{\xi}} - \hat{\bd{u}}) \cdot \hat{\bd{\tau}}^{\eta} (\hat{\bd{\psi}} - \hat{\bd{v}}) \cdot \hat{\bd{\tau}}^{\eta} \\
- \rho_{b} \int_{0}^{T} \int_{\hat{\Omega}_{b}} \partial_{t} \hat{\bd{\eta}} \cdot \partial_{t} \hat{\bd{\psi}} + 2\mu_{e} \int_{0}^{T} \int_{\hat{\Omega}_{b}} \hat{\bd{D}}(\hat{\bd{\eta}}) : \hat{\bd{D}}(\hat{\bd{\psi}}) + \lambda_{e} \int_{0}^{T} \int_{\hat{\Omega}_{b}} (\hat{\nabla} \cdot \hat{\bd{\eta}}) (\hat{\nabla} \cdot \hat{\bd{\psi}}) + 2\mu_{v} \int_{0}^{T} \int_{\hat{\Omega}_{b}} \hat{\bd{D}}(\partial_{t}\hat{\bd{\eta}}) : \hat{\bd{D}}(\hat{\bd{\psi}}) \\
+ \lambda_{v} \int_{0}^{T} \int_{\hat{\Omega}_{b}} (\hat{\nabla} \cdot \partial_{t} \hat{\bd{\eta}}) (\hat{\nabla} \cdot \hat{\bd{\psi}}) - \alpha \int_{0}^{T} \int_{\hat{\Omega}_{b}} \hat{\mathcal{J}}^{\eta}_{b} \hat{p} \hat{\nabla}^{\eta}_{b} \cdot \bd{\psi} - c_{0} \int_{0}^{T} \int_{\hat{\Omega}_{b}} \hat{p} \cdot \partial_{t} \hat{r} - \alpha \int_{0}^{T} \int_{\hat{\Omega}_{b}} \hat{\mathcal{J}}^{\eta}_{b} \partial_{t} \hat{\bd{\eta}} \cdot \hat{\nabla}^{\eta}_{b} \hat{r} - \alpha \int_{0}^{T} \int_{\hat{\Gamma}} (\hat{\bd{\xi}} \cdot \hat{\bd{n}}^{\eta}) \hat{r} \\
+ \kappa \int_{0}^{T} \int_{\hat{\Omega}_{b}} \hat{\mathcal{J}}^{\eta}_{b} \hat{\nabla}^{\eta}_{b} \hat{p} \cdot \hat{\nabla}^{\eta}_{b} \hat{r} - \int_{0}^{T} \int_{\hat{\Gamma}} ((\hat{\bd{u}} - \hat{\bd{\xi}}) \cdot \hat{\bd{n}}^{\eta}) \hat{r} = \int_{\Omega_{f}(0)} \bd{u}_0 \cdot \bd{v}(0) + \rho_{b} \int_{\hat{\Omega}_{b}} \hat{\bd{\xi}}_{0} \cdot \hat{\bd{\psi}}(0) + c_{0} \int_{\hat{\Omega}_{b}} \hat{p}_0 \cdot \hat{r}(0).
\end{multline*}
\end{definition}

This weak formulation follows similarly as in the moving domain weak formulation derivation in Section \ref{derivation}, with the exception of the integral $\displaystyle \int_{\Omega_{f}(t)} \Big(\partial_{t}\bd{u} + (\bd{u} \cdot \nabla) \bd{u}\Big) \cdot \bd{v}$ being handled using \eqref{aletime}. This can be handled for general vector-valued displacements using harmonic extension ALE maps as in \cite{BorSunSlip}, and we include the calculation here for completeness. We have the following computation using \eqref{gradtransform} and \eqref{aletime}, where $\partial_{t}\hat{\bd{u}}|_{\Omega_{f}(t)}(x, y) = \partial_{t}\hat{\bd{u}}\Big((\bd{\Phi}^{\eta}_{f})^{-1}(x, y)\Big)$ and $\bd{w}^{\eta}(x, y) = \hat{\bd{w}}\Big((\bd{\Phi}^{\eta}_{f})^{-1}(x, y)\Big)$ are defined on $\Omega_{f}(t)$, and where we integrate by parts and use the divergence free condition on $\bd{u}$:
\begin{align*}
\int_{0}^{T} \int_{\Omega_{f}(t)} &\Big(\partial_{t}\bd{u} + (\bd{u} \cdot \nabla) \bd{u}\Big) \cdot \bd{v} = \int_{0}^{T} \int_{\Omega_{f}(t)} \Big(\partial_{t}\hat{\bd{u}}|_{\Omega_{f}(t)} + (\bd{u} - \bd{w}^{\eta}) \cdot \nabla \bd{u}\Big) \cdot \bd{v} \\
& = \int_{0}^{T} \int_{\hat{\Omega}_{f}} \hat{\mathcal{J}}^{\eta}_{f} \partial_{t}\hat{\bd{u}} \cdot \bd{v} + \frac{1}{2} \int_{0}^{T} \int_{\Omega_{f}(t)} \Big[((\bd{u} - \bd{w}^{\eta}) \cdot \nabla \bd{u}) \cdot \bd{v} - ((\bd{u} - \bd{w}^{\eta}) \cdot \nabla \bd{v}) \cdot \bd{u}\Big] \\
&+ \frac{1}{2} \int_{0}^{T} \int_{\Omega_{f}(t)} (\nabla \cdot \bd{w}^{\eta}) \bd{u} \cdot \bd{v} + \frac{1}{2} \int_{0}^{T} \int_{\Gamma(t)} (\bd{u} - \bd{w}^{\eta}) \cdot \bd{n} (\bd{u} \cdot \bd{v}).
\end{align*}
By using the definition of the ALE velocity in \eqref{ALEvel} and the boundary conditions for the ALE map in \eqref{aleeta2}, we deduce that $\hat{\bd{w}}^{\eta}|_{\Gamma(t)} = \bd{\xi}|_{\Gamma(t)}$. We also integrate by parts in time, change variables to the reference fluid domain $\hat{\Omega}_{f}$, use \eqref{gradtransform} and \eqref{rescalednt}, and use the identity:
\begin{equation*}
\partial_{t}\hat{\mathcal{J}}^{\eta}_{f} = \hat{\mathcal{J}}^{\eta}_{f} (\nabla^{\omega}_{f} \cdot \hat{\bd{w}}^{\eta})
\end{equation*}
to deduce that
\begin{multline*}
\int_{0}^{T} \int_{\Omega_{f}(t)} \Big(\partial_{t}\bd{u} + (\bd{u} \cdot \nabla) \bd{u}\Big) \cdot \bd{v} = -\int_{0}^{T} \int_{\hat{\Omega}_{f}} \hat{\mathcal{J}}^{\eta}_{f} \hat{\bd{u}} \cdot \partial_{t}\hat{\bd{v}} - \int_{\Omega_{f}(0)} \bd{u}_{0} \cdot \bd{v}(0) \\
+ \frac{1}{2} \int_{0}^{T} \int_{\hat{\Omega}_{f}} \hat{\mathcal{J}}^{\eta}_{f} \Big[((\hat{\bd{u}} - \hat{\bd{w}}^{\eta}) \cdot \hat{\nabla}^{\eta}_{f}\hat{\bd{u}}) \cdot \hat{\bd{v}} - ((\hat{\bd{u}} - \hat{\bd{w}}^{\eta}) \cdot \hat{\nabla}^{\eta}_{f}\hat{\bd{v}}) \cdot \hat{\bd{u}}\Big] - \frac{1}{2} \int_{0}^{T} \int_{\hat{\Omega}_{f}} (\partial_{t}\hat{\mathcal{J}}^{\eta}_{f}) \hat{\bd{u}} \cdot \hat{\bd{v}} + \frac{1}{2} \int_{0}^{T} \int_{\hat{\Gamma}} (\hat{\bd{u}} - \hat{\bd{\xi}}) \cdot \hat{\bd{n}}^{\eta} (\hat{\bd{u}} \cdot \hat{\bd{v}}). 
\end{multline*}
This gives the weak formulation in Definition \ref{woregularitydef} in the fixed domain formulation.

We note that there are various terms in the weak formulation in Definition \ref{woregularitydef} for which weak solutions at the level of finite energy spaces do not suffice. 

\medskip

\noindent (1) \textbf{First}, we note that the Biot displacement $\hat{\bd{\eta}}$, which in analogy to traditional bulk elasticity would have finite energy regularity of $H^{1}(\hat{\Omega}_{b})$, does not have enough regularity for a term like 
\begin{equation*}
-\alpha \int_{0}^{T} \int_{\hat{\Omega}_{b}} \hat{\mathcal{J}}^{\eta}_{b} \hat{p} \hat{\nabla}^{\eta}_{b} \cdot \bd{\psi},
\end{equation*}
to make sense. This is because $\hat{p} \in H^{1}(\hat{\Omega}_{b})$ in the finite energy space and $\hat{\mathcal{J}}^{\eta}_{b} := \text{det}(\bd{I} + \nabla \hat{\bd{\eta}})$ in two spatial dimensions is only in $L^{1}(\hat{\Omega}_{b})$. In particular, the low regularity of geometric terms arising from the Biot displacement, such as the Jacobian $\hat{\mathcal{J}}^{\eta}_{b}$ and the nonlinear operator $\nabla^{\eta}_{b}$, prevent us from properly interpreting the weak formulation in Definition \ref{woregularitydef} at the level of finite energy weak solutions.

\medskip

\noindent (2) \textbf{Second}, the moving domains and moving interface are not well-defined. Since $\hat{\bd{\eta}} \in H^{1}(\hat{\Omega}_{b})$, $\hat{\bd{\eta}}$ is not even a continuous function on $\hat{\Omega}_{b}$ and its trace along the reference configuration $\hat{\Gamma}$ of the interface is only in $H^{1/2}(\hat{\Gamma})$, which is also not continuous. Therefore, it is unclear if the Lagrangian map $\hat{\bd{\Phi}}^{\eta}_{b}: \hat{\Omega}_{b} \to \Omega_{b}(t)$ defined in \eqref{lagrangianmap} is injective or even well-defined. Similarly, the interface displacement $\hat{\bd{\eta}}|_{\hat{\Gamma}}$ is not even a continuous function along $\hat{\Gamma}$, so the moving interface $\Gamma(t)$ is not well-defined, and hence the fluid domain $\Omega_{f}(t)$ in \eqref{omegaft} is not well-defined
-- this is in addition to geometric quantities involving the ALE map which depend on $\hat{\bd{\eta}}|_{\hat{\Gamma}}$.

\medskip

Therefore, in order to analyze this FPSI problem, at least with the current tools we have, we need to have additional regularity. Both of the problems above arise from the low regularity of the Biot displacement $\hat{\bd{\eta}}$, which prevents us from analyzing terms in the FPSI system arising from the moving Biot domain $\Omega_{b}(t)$ and which prevents us from having a well-defined interface $\Gamma(t)$. While in past work, such as in \cite{NonlinearFPSI_CRM23, FPSIJMPA}, the interface can be defined by having a reticular plate separating the Biot and fluid material (so that the interface displacement becomes the plate displacement which is in $H^{2}(\hat{\Gamma})$), in this problem with direct contact between the Biot material and the fluid, we must obtain the interface displacement directly from the Biot displacement. This motivates the use of a regularization method to analyze this FPSI problem with direct contact, which we refer to as the \emph{regularized interface method}, which we describe in the next section.

\section{The regularized interface method}\label{regintmethod}

\subsection{Overview of method and the regularized Biot displacement}

As mentioned in the previous section, the main challenges with regularity arise from the low regularity of the Biot displacement $\hat{\bd{\eta}}$, which prevent us from properly interpreting the weak formulation of equations posed on the moving Biot domain and which prevents us from having a well-defined moving interface. \textit{The main idea of the \textbf{regularized interface method} is to spatially regularize the Biot displacement $\hat{\bd{\eta}}$ to obtain a smooth regularized Biot displacement $\hat{\bd{\eta}}^{\delta}$ for a fixed but arbitrary regularization parameter $\delta > 0$, which we can then use to define all geometric quantities in the problem and in the weak formulation.} Assuming the initial data is sufficiently well-behaved, this guarantees existence of a weak solution locally in time, for a time of existence potentially depending on the regularization parameter $\delta > 0$.

Given a Biot displacement $\hat{\bd{\eta}}: \hat{\Omega}_{b} \to \R^{2}$ and a regularization parameter $\delta > 0$, we define a regularized Biot displacement $\hat{\bd{\eta}}$ by convolving spatially with a convolution kernel with support of size $\delta$. Since we are working on a bounded domain $\hat{\Omega}_{b}$, we must use an extension to first extend $\hat{\bd{\eta}}$ to $\R^{2}$ before convolving. To do this, we recall the definition of a strong $1$-extension from Section 5.17 in \cite{Adams} --  such an extension exists for our current geometry in $\hat{\Omega}_{b}$, as discussed in Chapter 5 of \cite{Adams}.

\begin{definition}[Strong $1$-extension from $\hat{\Omega}_{b}$ to $\R^{2}$]\label{extension}
For the domain $\hat{\Omega}_{b}$, a \textbf{strong $1$-extension} is a linear operator $E$ that acts on functions in $W^{k, p}(\hat{\Omega}_{b})$, for all $1 \le p < \infty$ and $k = 0, 1$, such that for each $1 \le p < \infty$ and $k = 0, 1$, $E: W^{k, p}(\hat{\Omega}_{b}) \to W^{k, p}(\R^{2})$ is a linear map such that $Ef|_{\hat{\Omega}_{b}} = f$ for all $f \in W^{k, p}(\hat{\Omega}_{b})$ and there exists a uniform constant $C_{k, p}$ (depending only on $k$ and $p$) such that $\|Ef\|_{W^{k, p}(\R^{2})} \le C\|f\|_{W^{k, p}(\hat{\Omega}_{b})}$. 
\end{definition}

We use this strong $1$-extension to define a regularized Biot displacement. Consider a compactly supported, nonnegative, radially symmetric, smooth function $\varphi \in C_{c}^{\infty}(\R^{2})$ with support being the closed ball of radius one in $\R^{2}$ and define the standard convolution kernel $\varphi_{\delta}(z) = \delta^{-2}\varphi(\delta^{-1}z)$ for $z \in \R^{2}$.

\begin{definition}[Regularized Biot displacement and regularized domains]\label{regulardefs}
Given a Biot displacement $\hat{\bd{\eta}}: \hat{\Omega}_{b} \to \R^{2}$ and a regularity parameter $\delta > 0$, we define the \textbf{regularized Biot displacement} $\hat{\bd{\eta}}^{\delta}: \hat{\Omega}_{b} \to \R^{2}$ by
\begin{equation*}
\hat{\bd{\eta}}^{\delta} := (E\hat{\bd{\eta}} * \varphi_{\delta})|_{\hat{\Omega}_{b}} \text{ on } \hat{\Omega}_{b},
\end{equation*}
where $E$ is a simple $H^{1}(\hat{\Omega}_{b})$ extension, see Definition \ref{extension}. Then, we define the \textbf{regularized Lagrangian map} by
\begin{equation*}
\hat{\bd{\Phi}}^{\eta^{\delta}}_{b}(t, \cdot): \hat{\Omega}_{b} \to \R^{2}, \qquad \hat{\bd{\Phi}}^{\eta^{\delta}}_{b}(t, \hat{x}, \hat{y}) := (\hat{x}, \hat{y}) + \hat{\bd{\eta}}^{\delta}(t, \hat{x}, \hat{y}) \quad \text{ for } (\hat{x}, \hat{y}) \in \hat{\Omega}_{b}.
\end{equation*}
We define the \textbf{regularized Biot domain} $\Omega^{\delta}_{b}(t)$ by
\begin{equation*}
\Omega^{\delta}_{b}(t) := \hat{\bd{\Phi}}^{\eta^{\delta}}_{b}(t, \hat{\Omega}_{b}).
\end{equation*}
The trace $\hat{\bd{\eta}}^{\delta}|_{\hat{\Gamma}}$ of the regularized Biot displacement defines the \textit{regularized interface displacement}, and hence, we define the \textbf{moving interface} $\Gamma(t)$ by
\begin{equation*}
\Gamma(t) := \{(\hat{x}, \hat{y}) + \hat{\bd{\eta}}^{\delta}|_{\hat{\Gamma}}(t, \hat{x}, \hat{y}) : (\hat{x}, \hat{y}) \in \hat{\Gamma}\},
\end{equation*}
and we define the \textbf{moving fluid domain} $\Omega_{f}(t)$ by
\begin{equation*}
\Omega_{f}(t) = \hat{\Omega} \setminus (\Omega_{b}(t) \cup \Gamma(t)),
\end{equation*}
where we recall that $\hat{\Omega} := \{\hat{\bd{x}} \in \R^{2} : |\hat{\bd{x}}| < 2\}$ is the total domain.
\end{definition}

The definition above motivates the characterization of this regularization scheme as the \textit{regularized interface method}. Namely, by regularizing the Biot displacement, we obtain a spatially smooth regularized Biot displacement $\hat{\bd{\eta}}^{\delta}$, which is locally in time injective given that the initial data is sufficiently well-behaved. Then, we can use this smooth regularized Biot displacement to define all geometric quantities in the problem, which can be defined in a well-defined manner as long as the regularized Lagrangian map is bijective, since now the Lagrangian map is a smooth map. The idea will then be to use all of these regularized geometric quantities in the (regularized) weak formulation, instead of the previous geometric quantities defined using the finite energy Biot displacement $\hat{\bd{\eta}}$, which does not possess sufficiently regularity to produce sufficiently spatially regular geometric terms in the weak formulation. However, the Biot displacement $\hat{\bd{\eta}}$ (before regularization) is what we will keep track of in the weak formulation and in most terms; we only use the regularized Biot displacement $\hat{\bd{\eta}}^{\delta}$ when dealing with any geometric quantities. We will hence now define a \textit{regularized interface weak solution} by stating a new \textit{regularized interface weak formulation}.

\subsection{Definition of a regularized interface weak solution}

We first define the solution and test spaces. We have the following solution spaces for the \textbf{fluid velocity on the moving domain}:
\begin{equation*}
\begin{cases}
V_{f}(t) := \{\bd{u} \in H^{1}(\Omega_{f}(t)) : \nabla \cdot \bd{u} = 0 \text{ on } \Omega_{f}(t) \text{ and } \bd{u} = \bd{0} \text{ on } \partial \Omega_{f}(t) \setminus \Gamma(t)\}, \\
\mathcal{V}_{f} := L^{\infty}(0, T; L^{2}(\Omega_{f}(t))) \cap L^{2}(0, T; V_{f}(t)).
\end{cases}
\end{equation*}
We can transfer this to the fixed domain by noting the transformation law for the divergence in \eqref{transformdiv}, and we obtain the solution spaces for the \textbf{fluid velocity on the fixed domain}, given the interface displacement $\hat{\bd{\eta}}^{\delta}|_{\hat{\Gamma}}$:
\begin{equation}\label{Vetaf}
\begin{cases}
V^{\eta^{\delta}}_{f}(t) := \{\hat{\bd{u}} \in H^{1}(\hat{\Omega}_{f}) : \nabla^{\eta^{\delta}(t)}_{f} \cdot \hat{\bd{u}} = 0 \text{ on } \hat{\Omega}_{f} \text{ and } \hat{\bd{u}} = \bd{0} \text{ on } \partial \hat{\Omega}_{f} \setminus \hat{\Gamma}\}, \\
\mathcal{V}^{\eta^{\delta}}_{f} := L^{\infty}(0, T; L^{2}(\hat{\Omega}_{f})) \cap L^{2}(0, T; V^{\eta^{\delta}}_{f}(t))\}.
\end{cases}
\end{equation}
Next, we define the following solution space for the \textbf{Biot displacement} $\hat{\bd{\eta}}$:
\begin{equation}\label{Vd}
\begin{cases}
V_{d} := H^{1}(\hat{\Omega}_{b}), \\
\mathcal{V}_{d} := W^{1, \infty}(0, T; L^{2}(\hat{\Omega}_{b})) \cap L^{\infty}(0, T; V_{d}).
\end{cases}
\end{equation}
and we define the following solution space for the \textbf{Biot pore pressure} $\hat{p}$:
\begin{equation}\label{Vp}
\begin{cases}
V_{p} := H^{1}(\hat{\Omega}_{b}), \\
\mathcal{V}_{p} := L^{\infty}(0, T; L^{2}(\hat{\Omega}_{b})) \cap L^{2}(0, T; V_{p}).
\end{cases}
\end{equation}
Finally, we define the fully coupled solution spaces in the moving domain and fixed domain formulations respectively:
\begin{equation}\label{directsolution}
\begin{cases}
\mathcal{V} := \mathcal{V}_{f} \times \mathcal{V}_{d} \times \mathcal{V}_{p}, \\
\mathcal{V}^{\eta^{\delta}} := \mathcal{V}_{f}^{\eta^{\delta}} \times \mathcal{V}_{d} \times \mathcal{V}_{p},
\end{cases}
\end{equation}
and the fully coupled test spaces in the moving domain and fixed domain formulations respectively:
\begin{equation}\label{directtest}
\begin{cases}
\mathcal{Q} := \{(\bd{v}, \hat{\bd{\psi}}, \hat{r}) \in C_{c}^{1}([0, T); V_{f}(t) \times V_{d} \times V_{p})\}, \\
\mathcal{Q}^{\eta^{\delta}} := \{(\hat{\bd{v}}, \hat{\bd{\psi}}, \hat{r}) \in C_{c}^{1}([0, T); V^{\eta^{\delta}}_{f}(t) \times V_{d} \times V_{p})\}.
\end{cases}
\end{equation}

\begin{remark}
We note that in the case of a Biot \textit{poroviscoelastic} material with $\lambda_{v} > 0$ and $\mu_{v} > 0$, see \eqref{poroviscoelastic}, the solution space $\mathcal{V}_{d}$ for the Biot displacement $\hat{\bd{\eta}}$ is more specifically:
\begin{equation*}
\mathcal{V}_{d} := W^{1, \infty}(0, T; L^{2}(\hat{\Omega}_{b})) \cap L^{\infty}(0, T; V_{d}) \cap H^{1}(0, T; V_{d}).
\end{equation*}
However, since the case of a purely poroelastic Biot medium is more challenging (due to less regularization), we have defined the solution space for the Biot displacement above for this case, and \textit{we adopt the convention throughout the paper of writing the existence proof for the more challenging case of a purely poroelastic Biot medium.}
\end{remark}

With the solution and test spaces defined, we can state the following regularized interface weak formulation for the FPSI problem with direct contact and nonlinear geometric coupling.

\begin{definition}[Regularized interface weak solution: fixed domain formulation]\label{regularweak} 
An ordered triple $(\hat{\bd{u}}, \hat{\bd{\eta}}, \hat{p}) \in \mathcal{V}^{\eta^{\delta}}$ is a \textbf{regularized interface weak solution with regularization parameter $\delta > 0$} to the FPSI problem if the following \textit{regularized interface weak formulation} holds for all test functions $(\hat{\bd{v}}, \hat{\bd{\psi}}, \hat{r}) \in \mathcal{Q}^{\eta^{\delta}}$, where the regularized Biot displacement, moving domains, and interface are defined in Definition \ref{regulardefs}:
\begin{multline*}
-\int_{0}^{T} \int_{\hat{\Omega}_{f}} \hat{\mathcal{J}}^{\eta^{\delta}}_{f} \hat{\bd{u}} \cdot \partial_{t} \hat{\bd{v}} + \frac{1}{2} \int_{0}^{T} \int_{\hat{\Omega}_{f}} \hat{\mathcal{J}}^{\eta^{\delta}}_{f} \Big[((\hat{\bd{u}} - \hat{\bd{w}}^{\eta}) \cdot \hat{\nabla}^{\eta^{\delta}}_{f} \hat{\bd{u}}) \cdot \hat{\bd{v}} - ((\hat{\bd{u}} - \hat{\bd{w}}^{\eta}) \cdot \hat{\nabla}^{\eta^{\delta}}_{f} \hat{\bd{v}}) \cdot \hat{\bd{u}}\Big] \\
- \frac{1}{2} \int_{0}^{T} \int_{\hat{\Omega}_{f}} (\partial_{t}\hat{\mathcal{J}}^{\eta^{\delta}}_{f}) \hat{\bd{u}} \cdot \hat{\bd{v}} + \frac{1}{2} \int_{0}^{T} \int_{\hat{\Gamma}} (\hat{\bd{u}} - \hat{\bd{\xi}}^{\delta}) \cdot \hat{\bd{n}}^{\eta^{\delta}} (\hat{\bd{u}} \cdot \hat{\bd{v}}) + 2\nu \int_{0}^{T} \int_{\hat{\Omega}_{f}} \hat{\mathcal{J}}^{\eta^{\delta}}_{f} \hat{\bd{D}}^{\eta^{\delta}}_{f}(\hat{\bd{u}}) : \hat{\bd{D}}^{\eta^{\delta}}_{f}(\hat{\bd{v}}) \\
+ \int_{0}^{T} \int_{\hat{\Gamma}} \left(\frac{1}{2} |\hat{\bd{u}}|^{2} - \hat{p}\right) (\hat{\bd{\psi}}^{\delta} - \hat{\bd{v}}) \cdot \hat{\bd{n}}^{\omega} + \beta \int_{0}^{T} \int_{\hat{\Gamma}} \Big(\hat{\mathcal{S}}^{\eta^{\delta}}_{\Gamma}\Big)^{-1} (\hat{\bd{\xi}}^{\delta} - \hat{\bd{u}}) \cdot \hat{\bd{\tau}}^{\omega} (\hat{\bd{\psi}}^{\delta} - \hat{\bd{v}}) \cdot \hat{\bd{\tau}}^{\omega} \\
- \rho_{b} \int_{0}^{T} \int_{\hat{\Omega}_{b}} \partial_{t} \hat{\bd{\eta}} \cdot \partial_{t} \hat{\bd{\psi}} + 2\mu_{e} \int_{0}^{T} \int_{\hat{\Omega}_{b}} \hat{\bd{D}}(\hat{\bd{\eta}}) : \hat{\bd{D}}(\hat{\bd{\psi}}) + \lambda_{e} \int_{0}^{T} \int_{\hat{\Omega}_{b}} (\hat{\nabla} \cdot \hat{\bd{\eta}}) (\hat{\nabla} \cdot \hat{\bd{\psi}}) + 2\mu_{v} \int_{0}^{T} \int_{\hat{\Omega}_{b}} \hat{\bd{D}}(\partial_{t}\hat{\bd{\eta}}) : \hat{\bd{D}}(\hat{\bd{\psi}}) \\
+ \lambda_{v} \int_{0}^{T} \int_{\hat{\Omega}_{b}} (\hat{\nabla} \cdot \partial_{t} \hat{\bd{\eta}}) (\hat{\nabla} \cdot \hat{\bd{\psi}}) - \alpha \int_{0}^{T} \int_{\hat{\Omega}_{b}} \hat{\mathcal{J}}^{\eta^{\delta}}_{b} \hat{p} \hat{\nabla}^{\eta^{\delta}}_{b} \cdot \hat{\bd{\psi}}^{\delta} - c_{0} \int_{0}^{T} \int_{\hat{\Omega}_{b}} \hat{p} \cdot \partial_{t} \hat{r} - \alpha \int_{0}^{T} \int_{\hat{\Omega}_{b}} \hat{\mathcal{J}}^{\eta^{\delta}}_{b} \partial_{t} \hat{\bd{\eta}}^{\delta} \cdot \hat{\nabla}^{\eta^{\delta}}_{b} \hat{r} - \alpha \int_{0}^{T} \int_{\hat{\Gamma}} (\hat{\bd{\xi}}^{\delta} \cdot \hat{\bd{n}}^{\eta^{\delta}}) \hat{r} \\
+ \kappa \int_{0}^{T} \int_{\hat{\Omega}_{b}} \hat{\mathcal{J}}^{\eta^{\delta}}_{b} \hat{\nabla}^{\eta^{\delta}}_{b} \hat{p} \cdot \hat{\nabla}^{\eta^{\delta}}_{b} \hat{r} - \int_{0}^{T} \int_{\hat{\Gamma}} ((\hat{\bd{u}} - \hat{\bd{\xi}}^{\delta}) \cdot \hat{\bd{n}}^{\eta^{\delta}}) \hat{r} = \int_{\Omega_{f}(0)} \bd{u}_0 \cdot \bd{v}(0) + \rho_{b} \int_{\hat{\Omega}_{b}} \hat{\bd{\xi}}_{0} \cdot \hat{\bd{\psi}}(0) + c_{0} \int_{\hat{\Omega}_{b}} \hat{p}_0 \cdot \hat{r}(0),
\end{multline*}
where $\hat{\bd{\psi}}^{\delta}$ for the test function $\hat{\bd{\psi}}$ is defined analogously via simple $H^{1}(\hat{\Omega}_{b})$ extension and convolution as in Definition \ref{regulardefs}, and the geometric quantities and operators are defined in Section \ref{transformed}.
\end{definition}

\subsection{Discussion of regularized interface weak solution and a priori estimate}

If one compares the weak formulation for the original FPSI problem (Definition \ref{woregularitydef}) and the regularized interface weak formulation for the regularized interface FPSI problem (Definition \ref{regularweak}), we can make several observations. First, we see that the moving domain geometry is regularized via the regularized Biot displacement $\hat{\bd{\eta}}^{\delta}$ and that any equations posed on the moving geometry, namely the pore pressure equation in the Biot equations \eqref{biot}, involve integrals over the \textit{regularized moving Biot domain}. For example, the terms 
\begin{equation*}
-\alpha \int_{0}^{T} \int_{\hat{\Omega}_{b}} \hat{\mathcal{J}}^{\eta^{\delta}}_{b} \hat{p} \hat{\nabla}^{\eta^{\delta}}_{b} \cdot \hat{\bd{\psi}}^{\delta} \qquad \text{ and } \qquad -\alpha \int_{0}^{T} \int_{\hat{\Omega}_{b}} \hat{\mathcal{J}}^{\eta^{\delta}}_{b} \partial_{t} \hat{\bd{\eta}}^{\delta} \cdot \hat{\nabla}^{\eta^{\delta}}_{b} \hat{r}
\end{equation*}
can be rewritten by a change of variables via the Lagrangian map as integrals over the \textit{regularized Biot domain} $\Omega^{\delta}_{b}(t)$, defined in Definition \ref{regulardefs}. These integrals are now well-defined, since $\hat{\bd{\eta}}^{\delta}$ is smooth, in contrast to their counterparts in the original weak formulation (Definition \ref{woregularitydef})
\begin{equation*}
-\alpha \int_{0}^{T} \int_{\hat{\Omega}_{b}} \hat{\mathcal{J}}^{\eta}_{b} \hat{p} \hat{\nabla}^{\eta}_{b} \cdot \hat{\bd{\psi}} \qquad \text{ and } \qquad -\alpha \int_{0}^{T} \int_{\hat{\Omega}_{b}} \hat{\mathcal{J}}^{\eta}_{b} \partial_{t} \hat{\bd{\eta}} \cdot \hat{\nabla}^{\eta}_{b} \hat{r},
\end{equation*}
which do not have enough spatial regularity (due to the limited $H^{1}(\hat{\Omega}_{b})$ regularity of $\hat{\bd{\eta}}$ leading to low regularity for the geometric terms) to be properly interpreted.

When comparing the weak formulations for the original and the regularized interface FPSI problem, we see another set of differences arising from the \textit{regularized interface condition}. While initially, for the direct problem, we obtained the movement of the fluid-Biot interface along which there is direct contact via a trace $\hat{\bd{\eta}}|_{\hat{\Gamma}}$ along $\hat{\Gamma}$, in order to have sufficient regularity for the interface, we define the interface displacement in the regularized interface method via the trace of the \textit{regularized Biot displacement} $\hat{\bd{\eta}}^{\delta}|_{\hat{\Gamma}}$. Therefore, some additional regularizations must be done in the weak formulation in order to take into account this regularized interface condition, in order to preserve the a priori energy estimate. For example, if we examine the following terms in the weak formulation
\begin{equation*}
-\alpha \int_{0}^{T} \int_{\hat{\Omega}_{b}} \hat{\mathcal{J}}^{\eta^{\delta}}_{b} \hat{p} \hat{\nabla}^{\eta^{\delta}}_{b} \cdot \hat{\bd{\psi}}^{\delta} - \alpha \int_{0}^{T} \int_{\hat{\Omega}_{b}} \hat{\mathcal{J}}^{\eta^{\delta}}_{b} \partial_{t} \hat{\bd{\eta}}^{\delta} \cdot \hat{\nabla}^{\eta^{\delta}}_{b} \hat{r} - \alpha \int_{0}^{T} \int_{\hat{\Gamma}} (\hat{\bd{\xi}}^{\delta} \cdot \hat{\bd{n}}^{\eta^{\delta}}) \hat{r},
\end{equation*}
we see the reason why the first term in the immediately preceding expression has a regularized test function $\hat{\bd{\psi}}^{\delta}$ instead of just $\hat{\bd{\psi}}$ and the second term has $\partial_{t}\hat{\bd{\eta}}^{\delta}$ instead of $\partial_{t} \hat{\bd{\eta}}$. Namely, if we derive an energy inequality by substituting $\hat{\bd{\psi}} = \partial_{t} \hat{\bd{\eta}}$ for the Biot displacement test function and $\hat{r} = \hat{p}$ for the Biot pore pressure test function, these terms are consistent and cancel out, since $\partial_{t}\hat{\bd{\eta}}^{\delta} := \hat{\bd{\xi}}^{\delta}$, which preserves the desired a priori energy estimate:
\begin{align*}
&-\alpha \int_{0}^{T} \int_{\hat{\Omega}_{b}} \hat{\mathcal{J}}^{\eta^{\delta}}_{b} \hat{p} \hat{\nabla}^{\eta^{\delta}}_{b} \cdot \hat{\bd{\eta}}^{\delta} - \alpha \int_{0}^{T} \int_{\hat{\Omega}_{b}} \hat{\mathcal{J}}^{\eta^{\delta}}_{b} \partial_{t} \hat{\bd{\eta}}^{\delta} \cdot \hat{\nabla}^{\eta^{\delta}}_{b} \hat{p} - \alpha \int_{0}^{T} \int_{\hat{\Gamma}} (\hat{\bd{\xi}}^{\delta} \cdot \hat{\bd{n}}^{\eta^{\delta}}) \hat{p} \\
&= -\alpha \int_{0}^{T} \int_{\Omega^{\delta}_{b}(t)} p (\nabla \cdot \bd{\eta}^{\delta}) - \alpha \int_{0}^{T} \int_{\Omega^{\delta}_{b}(t)} \partial_{t} \bd{\eta}^{\delta} \cdot \nabla p - \alpha \int_{0}^{T} \int_{\Gamma(t)} (\bd{\xi}^{\delta} \cdot \bd{n}) p = 0, \\
&\qquad \qquad \qquad \qquad \qquad \text{ where } \bd{\eta}^{\delta} := \hat{\bd{\eta}}^{\delta} \circ (\bd{\Phi}^{\eta^{\delta}}_{b})^{-1} \text{ and } \bd{\xi}^{\delta} = \hat{\bd{\xi}}^{\delta} \circ (\bd{\Phi}^{\eta^{\delta}}_{b})^{-1}.
\end{align*}

More generally, if we substitute $(\hat{\bd{v}}, \hat{\bd{\psi}}, \hat{r}) = (\hat{\bd{u}}, \partial_{t} \hat{\bd{\eta}}, \hat{p})$ for the test functions into the regularized interface weak formulation (Definition \ref{woregularitydef}), we obtain the following energy estimate:
\begin{equation}\label{energyest}
E_{f}(t) + E_{b}(t) + D(t) = E_{f}(0) + E_{b}(0),
\end{equation}
where 
\begin{align*}
E_{f}(t) &= \frac{1}{2} \int_{\Omega_{f}(t)} |\bd{u}|^{2} = \frac{1}{2} \int_{\hat{\Omega}_{f}} \hat{\mathcal{J}}^{\eta^{\delta}}_{f} |\bd{u}|^{2}, \quad (\text{fluid kinetic energy}), \\
E_{b}(t) &= \frac{1}{2} \int_{\hat{\Omega}_{b}} |\partial_{t}\hat{\bd{\eta}}|^{2} + \mu_{e} \int_{\hat{\Omega}_{b}} |\hat{\bd{D}}(\hat{\bd{\eta}})|^{2} + \frac{1}{2}\lambda_{e} \int_{\hat{\Omega}_{b}} |\hat{\nabla} \cdot \hat{\bd{\eta}}|^{2} + \frac{1}{2} c_{0} \int_{\hat{\Omega}_{b}} |\hat{p}|^{2}, \quad (\text{Biot kinetic energy}), \\
D(t) &= 2\nu \int_{0}^{t} \int_{\hat{\Omega}_{f}} \hat{\mathcal{J}}^{\eta^{\delta}}_{f} |\hat{\bd{D}}^{\eta^{\delta}}_{f}(\hat{\bd{u}})|^{2} + \beta \int_{0}^{T} \int_{\hat{\Gamma}} \Big(\hat{S}^{\eta^{\delta}}_{\Gamma}\Big)^{-1} |(\hat{\bd{\xi}}^{\delta} - \hat{\bd{u}}) \cdot \bd{\tau}^{\eta^{\delta}}|^{2} + 2\mu_{v} \int_{0}^{t} \int_{\hat{\Omega}_{b}} |\hat{\bd{D}}(\partial_{t} \hat{\bd{\eta}})|^{2} \\
&\ \ \ + \lambda_{v} \int_{0}^{t} \int_{\hat{\Omega}_{b}} |\hat{\nabla} \cdot \partial_{t}\hat{\bd{\eta}}|^{2} + \kappa \int_{0}^{t} \int_{\hat{\Omega}_{b}} \hat{\mathcal{J}}^{\eta^{\delta}}_{b} |\hat{\nabla}^{\eta^{\delta}}_{b} \hat{p}|^{2}, \quad (\text{dissipation}).
\end{align*}

In this sense, the regularization of the weak formulation found in the regularized interface weak formulation (Definition \ref{regularweak}) is the minimal regularization of the original weak formulation (Definition \ref{woregularitydef}) such that:
\begin{enumerate}
\item All terms in the weak formulation are well-defined at the level of finite energy solutions.
\item The regularization does not alter the a priori energy estimate.
\end{enumerate}

\subsection{Main results and proof strategy}

The main result of this manuscript is the existence of a regularized interface weak solution to the FPSI problem with direct contact and nonlinear geometric coupling, for each fixed but arbitrary regularization parameter $\delta > 0$.

\begin{theorem}\label{mainthm}
Consider a fixed regularization parameter $\delta > 0$ and either a purely poroelastic or poroviscoelastic Biot medium, see \eqref{poroviscoelastic}. Suppose that we are given a divergence-free initial fluid velocity $\bd{u}_{0} \in L^{2}(\Omega_{f}(0))$, an initial Biot displacement $\hat{\bd{\eta}}_{0} \in H^{1}(\hat{\Omega}_{b})$, an initial Biot velocity $\hat{\bd{\xi}}_{0} \in L^{2}(\hat{\Omega}_{b})$, and an initial pore pressure $\hat{p}_{0} \in L^{2}(\hat{\Omega}_{b})$. Assume that the following geometric nondegeneracy conditions on the initial data hold:
\begin{itemize}
\item The initial Lagrangian map $\hat{\bd{\Phi}}^{\eta^{\delta}_{0}}_{b}: \hat{\Omega}_{b} \to \Omega^{\delta}_{b}(0)$ is bijective and satisfies $\text{det}(\bd{I} + \hat{\nabla} \hat{\bd{\eta}}_{0}^{\delta}) \ge c > 0$ and $|\bd{I} + \hat{\nabla} \hat{\bd{\eta}}^{\delta}_{0}| \ge c > 0$ on $\overline{\Omega_{b}}$ for some positive constant $c > 0$.
\item The initial fluid ALE map $\hat{\bd{\Phi}}^{\eta^{\delta}_{0}}_{f}: \hat{\Omega}_{f} \to \Omega_{f}(0)$ is bijective and satisfies $\text{det}(\hat{\nabla} \hat{\bd{\Phi}}^{\eta^{\delta}_{0}}_{f}) \ge c > 0$ and $|\hat{\nabla} \hat{\bd{\Phi}}^{\eta^{\delta}_{0}}_{f}| \ge c > 0$ on $\overline{\Omega_{f}}$ for some positive constant $c > 0$. 
\item The parametrization of the initial moving regularized direct contact interface:
\begin{equation*}
\bd{r}(\hat{z}) = \Big(\cos(\hat{z}) + \hat{\bd{\eta}}^{\delta}_{0}|_{\hat{\Gamma}}(\hat{z}), \sin(\hat{z}) + \hat{\bd{\eta}}^{\delta}_{0}|_{\hat{\Gamma}}(\hat{z})\Big), \quad \hat{z} \in [0, 2\pi],
\end{equation*}
has $\|\bd{r}'(\hat{z})\| \ge \alpha > 0$ for all $\hat{z} \in [0, 2\pi]$ for some positive constant $\alpha > 0$. 
\end{itemize}
In addition, assume that the initial fluid velocity satisfies the compatibility condition \eqref{U0CC}. Then, for some time $T > 0$ (depending potentially on $\delta > 0$), there exists a regularized interface weak solution $(\hat{\bd{u}}, \hat{\bd{\eta}}, \hat{p})$ to the direct contact FPSI problem with nonlinear geometric coupling in the sense of Definition \ref{regularweak} on the time interval $[0, T]$.
\end{theorem}

This existence result is novel since there is direct contact between the Biot medium and the fluid. The only past existence result on Biot-fluid FPSI with nonlinear geometric coupling is in \cite{NonlinearFPSI_CRM23, FPSIJMPA}, with a plate at the Biot-fluid interface, whose elastodynamics regularizes the interface, for which there is existence to a slightly different regularized weak formulation (without a regularized interface condition). 

Therefore, the proof strategy for establishing Theorem \ref{mainthm} will be to use an approximation parameter $h > 0$, where the approximate problem will be the problem with a plate of thickness $h > 0$ at the fluid-structure interface $\hat{\Gamma}$. Namely, the approximate problem will have a plate of thickness $h$, and the interface displacement $\hat{\bd{\omega}}$ will be the (vector-valued) displacement of the plate from its reference configuration $\hat{\Gamma}$. To impose the condition in the regularized interface method that the interface displacement is determined by the trace $\hat{\bd{\eta}}^{\delta}|_{\hat{\Gamma}}$, we require the following regularized interface kinematic coupling condition:
\begin{equation*}
\hat{\bd{\omega}} = \hat{\bd{\eta}}^{\delta}|_{\hat{\Gamma}}. 
\end{equation*}
Furthermore, the plate of thickness $h > 0$ at the interface will have inertia and will hence have its own elastodynamics that affects the entire coupled system. Specifically, the interface displacement (which is now the plate displacement) will be modeled by a viscoelastic plate equation given by: 
\begin{equation}\label{plate}
h \partial_{t}^{2} \hat{\bd{\omega}} + h \hat{\Delta}^{2} \partial_{t} \hat{\bd{\omega}} + h \hat{\Delta}^{2} \hat{\bd{\omega}} = \hat{\bd{F}}_{p} \text{ on } \hat{\Gamma}, \qquad \text{ for } \hat{\Delta} := \partial^{2}_{\hat{z}},
\end{equation}
where we consider this equation for $\hat{z} \in [0, 2\pi]$ with periodic boundary conditions using the natural parametrization of $\hat{\Gamma}$ in \eqref{circle}. In addition, instead of direct balance of normal stress (in the case of direct contact in \eqref{normalstress}), we have a \textbf{dynamic coupling condition} in the approximate problem with the plate, specifying the load on the elastodynamics of the plate $\hat{\bd{F}}_{p}$ as the difference in normal stress from the Biot medium and fluid:
\begin{equation}\label{traction}
\hat{\bd{F}}_{p} = -\text{det}(\nabla \hat{\bd{\Phi}}^{\omega}_{f}) [\bd{\sigma}_{f}(\nabla \bd{u}, \pi) \circ \hat{\bd{\Phi}}^{\omega}_{f}] (\nabla \hat{\bd{\Phi}}^{\omega}_{f})^{-t} \hat{\bd{n}} + \hat{S}_{b}(\hat{\nabla} \hat{\bd{\eta}}, \hat{p})\hat{\bd{n}}|_{\hat{\Gamma}},
\end{equation}
where the ALE map $\hat{\bd{\Phi}}^{\omega}_{f}$ is defined via \eqref{alef1} and \eqref{alef2}, using the plate displacement $\hat{\bd{\omega}}$. 

\medskip

We then show existence of a regularized interface weak solution (satisfying a suitably regularized weak formulation) to the approximate problem with a plate of thickness $h > 0$ at the interface for each parameter $h > 0$. This will be done via a constructive existence scheme using Lie operator splitting, in which we discretize in time into $N$ total subintervals of length $\Delta t$. Upon obtaining approximate solutions for each $N$ via Lie operator splitting, we then obtain uniform bounds in $N$ and use compactness arguments (Aubin-Lions) to obtain strongly convergent subsequences and pass to the limit as $N \to \infty$ in the semidiscrete weak formulations to obtain a regularized weak solution to the problem with a plate of thickness $h > 0$. At this stage, we have a regularized interface weak solution for each $h > 0$, and from our previous analysis of constructive existence, we can show that \textbf{the interval of existence is independent of the parameter $h$, so that all regularized weak interface solutions can be constructed on a common interval $[0, T]$ independently of $h$, but which potentially depends on $\delta > 0$ which is fixed.} We then take the singular limit in the resulting regularized interface weak solutions (to the regularized problem with plate with thickness parameter $h > 0$) as $h \to 0$ to obtain a regularized weak interface solution which satisfies the limiting weak formulation with direct contact in Definition \ref{regularweak}. This will involve obtaining uniform bounds independently of $h$, and then passing to the limit using compactness arguments. This gives us a limiting regularized interface solution as $h \to 0$ satisfying the regularized interface weak formulation (with $h = 0$) in Definition \ref{regularweak}, which will complete the proof. 

Before we show existence of a regularized interface weak solution to the approximate problem with a plate of thickness $h > 0$ in the sense of Definition \ref{plateweak}, we make some comments about the choice of the plate equation \eqref{plate} approximate problem.
\begin{enumerate}
\item \textbf{Presence of the plate.} Having a plate at the boundary is advantageous as it simplifies constructive existence of solutions to the nonlinearly coupled FPSI problem. The main advantage is that in the constructive existence proof which uses a splitting scheme, the plate displacement (which updates the moving interface configuration and hence determines the geometry of the problem) can be updated independently (in a separate subproblem) from the Biot/fluid quantities, which are posed on domains that rely on the interface location. This enables us to separate the time-discrete moving interface update and the dynamical update of the Biot/fluid dynamical quantities, which allows us to construct a stable scheme. In particular, we can pose the Biot/fluid subproblem on a fixed geometry that is determined already by a prior plate subproblem update in the prior step, which at the approximate level allows us to solve the Biot/fluid subproblem on a fixed domain at each time-discrete step.
\item \textbf{Second order viscoelasticity and regularization of plate.} Note in addition that in the equation \eqref{plate}, there is second order viscoelasticity and regularization, so both the plate displacement $\hat{\bd{\eta}}$ and the plate velocity $\partial_{t}\hat{\bd{\eta}}$ are in $H^{2}(\hat{\Gamma})$. The plate displacement $\hat{\bd{\eta}}$ being in $H^{2}(\hat{\Gamma})$ is needed since we have a \textbf{vector-valued} (rather than purely transverse) \textbf{plate displacement}, and the stability results we use for preventing geometric degeneracies rely on having $C^{1}$ interfaces, see Proposition \ref{gammainjective}. We also need $H^{2}(\hat{\Gamma})$ estimates on the plate velocity $\partial_{t}\hat{\bd{\eta}}$, because we need to propagate uniform geometric estimates on $\hat{\mathcal{J}}^{\omega}_{f}$ and $\nabla \hat{\bd{\Phi}}^{\omega}_{f}$ defined in \eqref{omegadiff} involving $\nabla \hat{\bd{\eta}}$ in time. This requires uniform pointwise estimates on the time derivative of $\nabla \hat{\bd{\eta}}$, for which we need the plate velocity to be in $H^{2}(\hat{\Gamma})$ since the plate is one-dimensional, see Proposition \ref{geometricN}. \textbf{We emphasize that the second-order viscoelasticity is needed only for the construction via splitting scheme, which staggers the interface (geometric) update and the fluid/Biot update, and this requirement vanishes in the limit as $h \to 0$ once we have a limiting solution that is no longer time-discrete.}
\end{enumerate}

\section{The approximate problem with reticular plate of thickness $h > 0$}\label{hproof}

We consider the approximate problem with a plate of thickness $h > 0$ at the interface, described by the elastodynamics equation \eqref{plate} and the dynamic coupling condition specifying the load on the plate in \eqref{traction}. In this case, the movement of the interface is determined by the (vector-valued) displacement $\hat{\bd{\omega}}$ of the reticular plate from its reference configuration $\hat{\Gamma}$. 

We define a concept of a regularized interface weak solution for the problem with plate thickness (approximation) parameter $h > 0$. Throughout the analysis, we will consider the regularization parameter $\delta > 0$ to be a fixed but arbitrary parameter (even in the singular limit as $h \to 0$ which will be used to obtain a regularized interface weak solution for the problem with direct contact, we will keep $\delta > 0$ fixed). Since the plate equation has finite energy solutions in $H^{2}(\hat{\Gamma})$, we have that the solution space and test space for the approximate problem are as follows, in analogy to the direct contact solution/test spaces in \eqref{directsolution} and \eqref{directtest}. In contrast to the direct contact solution space (which does not include the interface displacement $\hat{\bd{\omega}}$ since this is just the trace of the Biot displacement), the solution space for the problem with parameter $h > 0$ has the plate displacement included. We define the following (fixed domain) solution space and test space:

\begin{equation*}
\begin{cases}
\mathcal{V}_h^{\omega} := H^1(0,T;H^2(\hat{\Gamma})),\\
\mathcal{V}_{h}^{\omega} := \{(\hat{\bd{u}}, \hat{\bd{\omega}}, \hat{\bd{\eta}}, \hat{p})\in\mathcal{V}_{f}^{\omega} \times \mathcal{V}_h^{\omega} \times \mathcal{V}_{d} \times \mathcal{V}_{p},\; \hat{\bd{\omega}} := \hat{\bd{\eta}}^{\delta}|_{\hat{\Gamma}}\}, \\
\mathcal{Q}_{h}^{\omega} := \{(\hat{\bd{v}}, \hat{\bd{\varphi}}, \hat{\bd{\psi}}, \hat{r}) \in C_{c}^{1}([0, T); V^{\omega}_{f}(t) \times H^{2}(\hat{\Gamma}) \times V_{d} \times V_{p}) : \hat{\bd{\varphi}} := \hat{\bd{\psi}}^{\delta}|_{\hat{\Gamma}}\},
\end{cases}
\end{equation*}
where $\mathcal{V}_{d}$ and $\mathcal{V}_{p}$ are defined in \eqref{Vd} and \eqref{Vp}, and the fluid space is now indexed by the plate displacement $\hat{\bd{\omega}}$. In analogy to \eqref{Vetaf}, we also define the moving domain solution spaces:
\begin{equation}\label{Vomegaf}
\begin{cases}
V^{\omega}_{f} := \{\bd{u} \in H^{1}(\hat{\Omega}_{f}) : \nabla^{\omega}_{f} \cdot \hat{\bd{u}} = 0 \text{ on } \hat{\Omega}_{f} \text{ and } \hat{\bd{u}} = \bd{0} \text{ on } \partial \hat{\Omega}_{f} \setminus \hat{\Gamma}\}, \\
\mathcal{V}^{\omega}_{f} := L^{\infty}(0, T; L^{2}(\hat{\Omega}_{f})) \cap L^{2}(0, T; V^{\omega}_{f}(t)),
\end{cases}
\end{equation}
where we recall the definition of the nonlinear differential operators indexed by $\hat{\bd{\omega}}$ in \eqref{omegadiff}. We also define the following rescaled normal and tangent vectors in analogy to \eqref{rescalednt} given the vector-valued plate displacement $\hat{\bd{\omega}} := (\hat{\omega}_{1}, \hat{\omega}_{2})$:
\begin{equation}\label{ntomega}
\bd{n}^{\omega} := (\cos(\hat{z}) + \partial_{\hat{z}}\hat{\omega}_{2}, \sin(\hat{z}) - \partial_{\hat{z}}\hat{\omega}_{1}), \qquad \bd{\tau}^{\omega} := (-\sin(\hat{z}) + \partial_{\hat{z}}\hat{\omega}_{1}, \cos(\hat{z}) + \partial_{\hat{z}}\hat{\omega}_{2}).
\end{equation}
We then have the following \textit{regularized interface weak formulation} for the regularized interface weak solutions to the problem with plate of thickness $h > 0$. 

\begin{definition}[Regularized interface weak solution: problem with plate of thickness $h$]\label{plateweak}
An ordered collection $(\hat{\bd{u}}, \hat{\bd{\omega}}, \hat{\bd{\eta}}, \hat{p}) \in \mathcal{V}^{\omega}_{h}$ is a regularized interface weak solution to the problem with regularization parameter $\delta > 0$ and a plate of thickness $h > 0$ if $\hat{\bd{\omega}} := \hat{\bd{\eta}}^{\delta}|_{\hat{\Gamma}}$ and the following \textit{regularized interface weak formulation} holds for all test functions $(\hat{\bd{v}}, \hat{\bd{\varphi}}, \hat{\bd{\psi}}, \hat{r}) \in \mathcal{Q}^{\omega}_{h}$:
\begin{multline*}
-\int_{0}^{T} \int_{\hat{\Omega}_{f}} \hat{\mathcal{J}}^{\omega}_{f} \hat{\bd{u}} \cdot \partial_{t} \hat{\bd{v}} + \frac{1}{2} \int_{0}^{T} \int_{\hat{\Omega}_{f}} \hat{\mathcal{J}}^{\omega}_{f} \Big[((\hat{\bd{u}} - \hat{\bd{w}}^{\eta}) \cdot \hat{\nabla}^{\omega}_{f} \hat{\bd{u}}) \cdot \hat{\bd{v}} - ((\hat{\bd{u}} - \hat{\bd{w}}^{\eta}) \cdot \hat{\nabla}^{\omega}_{f} \hat{\bd{v}}) \cdot \hat{\bd{u}}\Big] \\
- \frac{1}{2} \int_{0}^{T} \int_{\hat{\Omega}_{f}} (\partial_{t}\hat{\mathcal{J}}^{\omega}_{f}) \hat{\bd{u}} \cdot \hat{\bd{v}} + \frac{1}{2} \int_{0}^{T} \int_{\hat{\Gamma}} (\hat{\bd{u}} - \hat{\bd{\xi}}^{\delta}) \cdot \hat{\bd{n}}^{\omega} (\hat{\bd{u}} \cdot \hat{\bd{v}}) + 2\nu \int_{0}^{T} \int_{\hat{\Omega}_{f}} \hat{\mathcal{J}}^{\omega}_{f} \hat{\bd{D}}^{\omega}_{f}(\hat{\bd{u}}) : \hat{\bd{D}}^{\omega}_{f}(\hat{\bd{v}}) \\
+ \int_{0}^{T} \int_{\hat{\Gamma}} \left(\frac{1}{2} |\hat{\bd{u}}|^{2} - \hat{p}\right) (\hat{\bd{\psi}}^{\delta} - \hat{\bd{v}}) \cdot \hat{\bd{n}}^{\omega} + \beta \int_{0}^{T} \int_{\hat{\Gamma}} \Big(\hat{\mathcal{S}}^{\omega}_{\Gamma}\Big)^{-1} (\hat{\bd{\xi}}^{\delta} - \hat{\bd{u}}) \cdot \hat{\bd{\tau}}^{\omega} (\hat{\bd{\psi}}^{\delta} - \hat{\bd{v}}) \cdot \hat{\bd{\tau}}^{\omega} \\
- h \int_{0}^{T} \int_{\hat{\Gamma}} \partial_{t}\hat{\bd{\omega}} \cdot \partial_{t} \hat{\bd{\varphi}} + h \int_{0}^{T} \int_{\hat{\Gamma}} \Delta \partial_{t} \hat{\bd{\omega}} \cdot \Delta \hat{\bd{\varphi}} + h \int_{0}^{T} \int_{\hat{\Gamma}} \hat{\Delta} \hat{\bd{\omega}} \cdot \hat{\Delta} \hat{\bd{\varphi}} - \rho_{b} \int_{0}^{T} \int_{\hat{\Omega}_{b}} \partial_{t} \hat{\bd{\eta}} \cdot \partial_{t} \hat{\bd{\psi}} + 2\mu_{e} \int_{0}^{T} \int_{\hat{\Omega}_{b}} \hat{\bd{D}}(\hat{\bd{\eta}}) : \hat{\bd{D}}(\hat{\bd{\psi}}) \\
+ \lambda_{e} \int_{0}^{T} \int_{\hat{\Omega}_{b}} (\hat{\nabla} \cdot \hat{\bd{\eta}}) (\hat{\nabla} \cdot \hat{\bd{\psi}}) + 2\mu_{v} \int_{0}^{T} \int_{\hat{\Omega}_{b}} \hat{\bd{D}}(\partial_{t}\hat{\bd{\eta}}) : \hat{\bd{D}}(\hat{\bd{\psi}}) + \lambda_{v} \int_{0}^{T} \int_{\hat{\Omega}_{b}} (\hat{\nabla} \cdot \partial_{t} \hat{\bd{\eta}}) (\hat{\nabla} \cdot \hat{\bd{\psi}}) \\
- \alpha \int_{0}^{T} \int_{\hat{\Omega}_{b}} \hat{\mathcal{J}}^{\eta^{\delta}}_{b} \hat{p} \hat{\nabla}^{\eta^{\delta}}_{b} \cdot \hat{\bd{\psi}}^{\delta} - c_{0} \int_{0}^{T} \int_{\hat{\Omega}_{b}} \hat{p} \cdot \partial_{t} \hat{r} - \alpha \int_{0}^{T} \int_{\hat{\Omega}_{b}} \hat{\mathcal{J}}^{\eta^{\delta}}_{b} \partial_{t} \hat{\bd{\eta}}^{\delta} \cdot \hat{\nabla}^{\eta^{\delta}}_{b} \hat{r} - \alpha \int_{0}^{T} \int_{\hat{\Gamma}} (\bd{\zeta} \cdot \hat{\bd{n}}^{\omega}) \hat{r} \\
+ \kappa \int_{0}^{T} \int_{\hat{\Omega}_{b}} \hat{\mathcal{J}}^{\eta^{\delta}}_{b} \hat{\nabla}^{\eta^{\delta}}_{b} \hat{p} \cdot \hat{\nabla}^{\eta^{\delta}}_{b} \hat{r} - \int_{0}^{T} \int_{\hat{\Gamma}} ((\hat{\bd{u}} - \hat{\bd{\zeta}}) \cdot \hat{\bd{n}}^{\omega}) \hat{r} = \int_{\Omega_{f}(0)} \bd{u}_0 \cdot \bd{v}(0) + \rho_{b} \int_{\hat{\Omega}_{b}} \hat{\bd{\xi}}_{0} \cdot \hat{\bd{\psi}}(0) + c_{0} \int_{\hat{\Omega}_{b}} \hat{p}_0 \cdot \hat{r}(0),
\end{multline*}
\end{definition}

\begin{remark}[Notational convention in the approximate problem with a plate]
In the approximate problem, we have a viscoelastic plate of thickness $h > 0$ at the interface that has its own vector-valued displacement $\hat{\bd{\omega}}$, which satisfies continuity of displacements at the interface, $\hat{\bd{\omega}} := \hat{\bd{\eta}}^{\delta}|_{\hat{\Gamma}}$. Therefore, to emphasize the presence of the plate, whose displacement directly specifies the time-dependent interface $\Gamma(t)$, we index all geometric quantities in the fluid and plate by $\hat{\bd{\omega}}$ instead of $\hat{\bd{\eta}}^{\delta}$, as in $\hat{\mathcal{J}}^{\omega}_{f}$, $\hat{\nabla}^{\omega}_{f}$, $\hat{\bd{n}}^{\omega}$, and $\hat{\bd{\tau}}^{\omega}$, see the discussion around \eqref{alef1} and \eqref{alef2}, and the definitions in \eqref{ntomega}. While for example by \eqref{aleeta1} and \eqref{aleeta2}, we can equivalently denote the map $\hat{\bd{\Phi}}^{\omega}_{f}$ by $\hat{\bd{\Phi}}^{\eta^{\delta}}_{f}$, we choose to use $\omega$ in Definition \ref{plateweak} to emphasize the presence of the plate, whereas in contrast, we use $\eta^{\delta}$ to index geometric quantities in the regularized interface direct contact weak formulation in Definition \ref{regularweak} to emphasize the fact that there is direct contact (no plate).
\end{remark}

The goal is then (1) to establish existence of regularized weak solutions for the regularized problem of the form above (with the additional regularized kinematic coupling condition) and then (2) derive uniform estimates, which are independent of the plate thickness parameter $h > 0$, which will allow us to take the singular limit of solutions as $h \to 0$ for fixed regularization parameter $\delta > 0$ to obtain existence of weak solutions for a limiting problem without a plate separating the poroelastic medium and the fluid. \textit{Note that an important part of this analysis involves showing that the weak solutions to the regularized problem with regularization parameter $\delta > 0$ and plate thickness $h > 0$ exist on a uniform time interval of existence that is independent of $h > 0$.} In particular, the goal of this section is to prove the following result:

\begin{theorem}\label{hthm}
Consider a fixed regularization parameter $\delta > 0$ and consider an initial plate displacement and plate velocity $\hat{\bd{\omega}}_{0}, \bd{\zeta}_{0} \in H^{2}(\hat{\Gamma})$ and in addition: an initial fluid velocity $\bd{u}_{0}$, initial Biot displacement $\hat{\bd{\eta}}_{0}$, an initial Biot velocity $\bd{\xi}_{0}$, and an initial pore pressure $\hat{p}_{0}$ satisfying the conditions listed in Theorem \ref{mainthm}, along with the regularized kinematic coupling condition: $\bd{\omega}_{0} = \bd{\eta}_{0}^{\delta}|_{\Gamma}$ and $\bd{\zeta}_{0} = \bd{\xi}_{0}^{\delta}|_{\Gamma}$. Then, for some time $T > 0$ (depending only on $\delta > 0$) and \textit{independent of the plate thickness $h > 0$}, there exists a regularized interface weak solution $(\hat{\bd{u}}, \hat{\bd{\omega}}, \hat{\bd{\eta}}, \hat{p})$ to the FPSI problem with nonlinear geometric coupling and a reticular plate of thickness $h > 0$ in the sense of Definition \ref{plateweak}.
\end{theorem}

Note that by formally substituting for the test function $(\hat{\bd{v}}, \hat{\bd{\varphi}}, \hat{\bd{\psi}}, \hat{r}) = (\hat{\bd{u}}, \hat{\bd{\omega}}, \partial_{t} \hat{\bd{\eta}}, \hat{p})$, we obtain the following a priori energy estimate, which is the analogue of the energy estimate in \eqref{energyest}, where the terms $E_{f}(t)$, $E_{b}(0)$, and $D(t)$ are defined, but with the addition of energy arising from the plate (which depends on the plate thickness $h > 0$):
\begin{align*}
E_{f}&(t) + E_{b}(t) + E_{h}(t) + D_{h}(t) = E_{f}(0) + E_{b}(0) + E_{h}(0), \\
&\text{ where } E_{h}(t) := \frac{1}{2} h \int_{\hat{\Gamma}} |\partial_{t} \hat{\bd{\omega}}|^{2} + \frac{1}{2} h \int_{\hat{\Gamma}} |\hat{\Delta} \hat{\bd{\omega}}|^{2} \text{ and } D_{h}(t) := D(t) + h \int_{0}^{t} \int_{\hat{\Gamma}} |\hat{\Delta} \partial_{t}\hat{\bd{\omega}}|^{2}.
\end{align*}
We note that $E_{h}(t)$ is the contribution arising to the energy arising from the plate (scaled by the plate thickness $h > 0$), and $D_{h}(t)$ is the usual dissipation $D(t)$ from \eqref{energyest} with the additional dissipation arising from the viscoelasticity of the plate.

\medskip

\noindent \textbf{A remark on a notational convention.} To simplify the notation in the remainder of the manuscript, from this point onwards, we will omit the ``hat" notation when referring to reference domains and functions/differential operators defined on reference domains. We will instead rely on the context to disambiguate functions/differential operators defined on reference and moving (time-dependent) domains. For example, integrands can be disambiguated by observing whether the domain of integration is a reference domain (such as $\Omega_{f}$, $\Omega_{b}$, or $\Gamma$) or a moving domain (such as $\Omega_{f}(t)$, $\Omega_{b}(t)$, or $\Gamma(t)$). 

\if 1 = 0
\section{Remarks about the singular limit and the additional kinematic coupling regularization}

Note that the weak formulation presented in \eqref{regularweakref} looks the same as the one used for the original problem, presented in the past paper in JMPA, with regularization, and in fact the weak formulation is exactly the same. However, while the paper in JMPA works with a structure Biot displacement $\bd{\eta}$ and a plate displacement $\omega$ such that
\begin{equation*}
\bd{\eta}|_{\Gamma} = \omega \bd{e}_{y},
\end{equation*}
this is not sufficient to pass to the limit in the plate thickness as $h \to 0$. In particular, we would have the following energy estimate:
\begin{multline}\label{energy}
\frac{1}{2} \int_{\Omega_{f}(T)} |\boldsymbol{u}(T)|^{2} + 2\nu \int_{0}^{T} \int_{\Omega_{f}(t)} |\boldsymbol{D}(\boldsymbol{u})|^{2} + \beta \int_{0}^{T} \int_{\Gamma(t)} |(\boldsymbol{\xi} - \boldsymbol{u}) \cdot \boldsymbol{\tau}|^{2} + \frac{1}{2}\rho_{p} h \int_{\hat{\Gamma}} |\partial_{t} \hat{\omega}(T)|^{2} + C_{0}h^{3} \int_{\hat{\Gamma}} |\hat{\Delta} \hat{\omega}(T)|^{2} \\
+ \frac{1}{2} \rho_{b} \int_{\hat{\Omega}_{b}} |\partial_{t}\hat{\boldsymbol{\eta}}(T)|^{2} + 2\mu_{e} \int_{\hat{\Omega}_{b}} |\hat{\bd{D}}(\hat{\bd{\eta}})(T)|^{2} + 2\lambda_{e} \int_{\hat{\Omega}_{b}} |\hat{\nabla} \cdot \hat{\bd{\eta}}(T)|^{2} + 2\mu_{v} \int_{0}^{T} \int_{\hat{\Omega}_{b}} |\hat{\boldsymbol{D}}(\partial_{t} \hat{\bd{\eta}})|^{2} \\
+ \lambda_{v} \int_{0}^{T} \int_{\hat{\Omega}_{b}} |\hat{\nabla} \cdot \partial_{t}\hat{\boldsymbol{\eta}}|^{2}
+ \frac{1}{2} c_{0} \int_{\hat{\Omega}_{b}} |\hat{p}(T)|^{2} + \kappa \int_{0}^{T} \int_{\Omega_{b}^{\delta}(t)} |\nabla p|^{2} = \frac{1}{2} \int_{\Omega_{f}(0)} |\boldsymbol{u}(0)|^{2} + \frac{1}{2} \rho_{b} \int_{\hat{\Omega}_{b}} |\hat{\bd{\xi}}_{0}|^{2} \\
+ \frac{1}{2} \rho_{p} h \int_{\hat{\Gamma}} |\partial_{t} \hat{\omega}(0)|^{2} + C_{0}h^{3} \int_{\hat{\Gamma}} |\hat{\Delta} \hat{\omega}(0)|^{2} + 2\mu_{e} \int_{\hat{\Omega}_{b}} |\hat{\bd{D}}(\hat{\bd{\eta}})(0)|^{2} + 2\lambda_{e} \int_{\hat{\Omega}_{b}} |\hat{\nabla} \cdot \hat{\bd{\eta}}(0)|^{2} + \frac{1}{2} c_{0} \int_{\hat{\Omega}_{b}} |\hat{p}_0|^{2}.
\end{multline}
Notationally, we have omitted the dependence of the solutions on $\delta, h$ for ease of notation. Note that we therefore, directly from the energy estimate have the following uniform bounds (where we have notated the dependence on solutions on $\delta, h$ once again):
\begin{itemize}
\item $\{\bd{u}_{\delta, h}\}_{h > 0}$ is uniformly bounded in $L^{\infty}(0, T; L^{2}(\Omega_{f}(t))) \cap L^{2}(0, T; H^{1}(\Omega_{f}(t)))$.
\item $\{h^{3/2} \cdot \omega_{\delta, h}\}_{h > 0}$ is uniformly bounded in $L^{\infty}(0, T; H^{2}(\Gamma))$.
\item $\{h^{1/2} \cdot \partial_{t}\omega_{\delta, h}\}_{h > 0}$ is uniformly bounded in $L^{\infty}(0, T; L^{2}(\Gamma))$. 
\item $\{\bd{\eta}_{\delta, h}\}_{h > 0}$ is uniformly bounded in $W^{1, \infty}(0, T; L^{2}(\Omega_{b})) \cap L^{\infty}(0, T; H^{1}(\Omega_{b}))$.
\item $\{\bd{\xi}_{\delta, h}\}_{h > 0}$ is uniformly bounded in $L^{2}(0, T; H^{1}(\Omega_{b})) \cap L^{\infty}(0, T; L^{2}(\Omega_{b}))$. 
\item $\{p_{\delta, h}\}_{h > 0}$ is uniformly bounded in $L^{\infty}(0, T; L^{2}(\Omega_{b})) \cap L^{2}(0, T; H^{1}(\Omega_{b}))$.
\end{itemize}
But note that for the plate displacement $\omega_{\delta, h}$, all of these uniform estimates that can be obtained directly from the energy estimate involve the multiplicative factor of $h$ to some power, and hence, we do not have uniform control yet of $\omega_{\delta, h}$ itself in any space. And having uniform control of $\omega_{\delta, h}$ independently of $h$ is important for passing to the limit, since one needs to obtain convergence of $\omega_{\delta, h}$ strongly as $h \to 0$ in order to pass to the limit in geometric nonlinearities.

Thus, we would need to appeal to the kinematic coupling condition to obtain any information about the plate displacement $\omega_{\delta, h}$, in the hopes of transferring information about uniform bounds on the Biot displacement $\bd{\eta}_{\delta, h}$ to information about the plate displacement $\omega_{\delta, h}$ (which is the trace of $\bd{\eta}_{\delta, h}$ along $\Gamma$). We can now see why having the ordinary kinematic coupling condition (as in the JMPA paper) would not be enough to pass to the singular limit as $h \to 0$. To demonstrate this, consider at this point what would happen if we had the ordinary kinematic coupling condition (as in the JMPA paper):
\begin{equation*}
\bd{\eta}_{\delta, h}|_{\Gamma} = \omega_{\delta, h} \bd{e}_{y}
\end{equation*}
However, from this, we would only have that $\{\omega_{\delta, h}\}_{h \ge 0}$ is bounded uniformly in $L^{2}(0, T; H^{1/2}(\Gamma))$ as the trace along $\Gamma$ of the Biot displacements $\{\bd{\eta}_{\delta, h}\}_{h \ge 0}$ which are uniformly bounded in $L^{2}(0, T; H^{1}(\Omega_{b}))$. \textit{Since $H^{1/2}(\Gamma)$ does not even embed into the space of continuous functions along $\Gamma$ in even the simplest case of $\Gamma$ being one-dimensional, this does not give us enough uniform control of the plate displacements in the limit as $h \to 0$ to define a limiting fluid-structure interface, that is for example at least a continuous interface.}

Hence, the strategy is to add an additional layer of regularization in order to pass to the singular limit as $h \to 0$. In particular, we will add an additional regularization to the \textit{kinematic coupling condition}, so that we instead require:
\begin{equation*}
\bd{\eta}^{\delta}_{\delta, h}|_{\Gamma} = \omega_{\delta, h} \bd{e}_{y} 
\end{equation*}
so that the plate displacement is the trace of the \textit{regularized Biot displacement} $\bd{\eta}^{\delta}_{\delta, h}$ rather than the Biot displacement $\bd{\eta}_{\delta, h}$. The only change this would make to the weak formulation \eqref{regularweakref} are:
\begin{itemize}
\item The plate displacement $\omega_{\delta, h}$ and Biot displacement $\bd{\eta}_{\delta, h}$ would be required to instead satisfy a regularized and $\delta$-dependent kinematic coupling condition of $\bd{\eta}_{\delta, h}^{\delta}|_{\Gamma} = \omega_{\delta, h} \bd{e}_{y}$.
\item This same regularized kinematic coupling condition would be required to be imposed on the test functions also so that $\bd{\psi}^{\delta}|_{\Gamma} = \varphi \bd{e}_{y}$. This would justify testing with $\bd{\psi} = \bd{\eta}_{\delta, h}$ and $\varphi = \omega_{\delta, h}$ to obtain an a priori estimate.
\end{itemize}
However, we emphasize that the integral formulation \eqref{regularweakref} is unchanged (even though the test functions and kinematic coupling condition itself on the solution are different). This is important because this means that the energy estimate \eqref{energy} \textit{still holds even with this additional kinematic coupling regularization}, so we have the same uniform bounds that are outlined in the list provided after \eqref{energy}, but \textit{in fact we have even more due to the regularization in the kinematic coupling condition.} In particular, for fixed $\delta$, we can use the fact that $\{\bd{\eta}_{\delta, h}\}_{h > 0}$ is uniformly bounded in $L^{2}(0, T; H^{1}(\Omega_{b}))$ and the regularization properties of spatial convolution combined with the trace inequality to conclude that since $\bd{\eta}^{\delta}_{\delta, h}|_{\Gamma} = \omega_{\delta, h} \bd{e}_{y}$, we in fact have that
\begin{equation*}
\{\omega_{\delta, h}\}_{h > 0} \text{ is uniformly bounded in $L^{2}(0, T; H^{s}(\Gamma))$}, \text{ for all $s > 0$}.
\end{equation*}
This gives us sufficiently strong regularity to pass to the limit. 

Therefore, the main novelty here would be constructing a weak solution for the problem with the additional regularized kinematic coupling condition. However, it is expected that since the integral equation in the weak formulation itself is unchanged that this would proceed in an analogous manner to the problem studied in JMPA, even though it is slightly different. A similar splitting could be used that preserves the energy structure of the problem, as long as one uses the regularized kinematic coupling condition $\bd{\eta}^{\delta}|_{\Gamma} = \omega \bd{e}_{y}$ instead of the usual kinematic coupling condition $\bd{\eta}|_{\Gamma} = \omega \bd{e}_{y}$. However, an interesting feature of this problem is that the test functions now are required to satisfy a regularized kinematic coupling condition also.

\section{Existence of weak solutions to the regularized problem}

The existence of weak solutions to the same problem, without the \textit{regularized kinematic coupling condition}, has been established previously in TODO via a constructive splitting scheme approach, compactness results on moving domains, and construction of test functions on moving domains. The same methodology can be used to show existence of weak solutions to the regularized weak formulation with additional kinematic coupling regularization, so we simply outline the proof without providing full details and refer the reader to full details in TODO. We will, in particular, only provide explicit details for those steps which have notable differences arising due to the regularized kinematic coupling condition imposed on both the actual solution itself and also, the test functions.

\subsection{A priori energy estimate}

We first establish an a priori energy estimate for the regularized FPSI problem with a regularized kinematic coupling condition $\bd{\eta}^{\delta}|_{\Gamma} = \varphi \bd{e}_{y}$, and regularized test functions $\bd{\psi}^{\delta}|_{\Gamma} = \varphi \bd{e}_{y}$. Because the regularization of the kinematic coupling condition is reflected explicitly in the solution itself and in the test functions, we can formally substitute $(\bd{q}, \bd{\psi}, r, \varphi) = (\bd{u}, \partial_{t}\bd{\eta}, p, \partial_{t}\omega)$ into the weak formulation. We have that the following cancellations in the weak formulation will occur:
\begin{equation*}
\frac{1}{2} \int_{0}^{T} \int_{\hat{\Gamma}} (\hat{\bd{u}} \cdot \hat{\bd{n}}^{\omega} - \hat{\zeta} \bd{e}_{y} \cdot \hat{\bd{n}}^{\omega}) |\hat{\bd{u}}|^{2} + \int_{0}^{T} \int_{\hat{\Gamma}} \left(\frac{1}{2}|\hat{\bd{u}}|^{2} - \hat{p}\right) (\hat{\zeta}\bd{e}_{y} - \hat{\bd{u}}) \cdot \hat{\bd{n}}^{\omega} - \int_{0}^{T} \int_{\hat{\Gamma}} ((\hat{\bd{u}} - \hat{\zeta} \bd{e}_{y}) \cdot \hat{\bd{n}}^{\omega}) \hat{p} = 0,
\end{equation*}
where we used the fact that $\partial_{t}\bd{\eta}^{\delta} = \hat{\zeta} \bd{e}_{y}$. In addition, by transferring back to the physical domain and using integration by parts on $\Omega_{b}(t)$, we obtain that 
\begin{equation*}
-\alpha \int_{0}^{T} \int_{\hat{\Omega}_{b}} \hat{\mathcal{J}}^{\eta^{\delta}}_{b} \hat{p} \hat{\nabla}^{\eta^{\delta}}_{b} \cdot \partial_{t}\hat{\bd{\eta}}^{\delta} - \alpha \int_{0}^{T} \int_{\hat{\Omega}_{b}} \hat{\mathcal{J}}^{\eta^{\delta}}_{b} \partial_{t} \hat{\bd{\eta}}^{\delta} \cdot \hat{\nabla}^{\eta^{\delta}}_{b} \hat{p} - \alpha \int_{0}^{T} \int_{\hat{\Gamma}} (\hat{\zeta} \bd{e}_{y} \cdot \hat{\bd{n}}^{\omega}) \hat{p} = 0.
\end{equation*}
We therefore obtain the following a priori energy estimate:
\begin{multline*}
\frac{1}{2} \int_{\Omega_{f}(T)} |\hat{\bd{u}}(T)|^{2} + 2\nu \int_{0}^{T} \int_{\hat{\Omega}_{f}(t)} |\hat{\bd{D}}(\hat{\bd{u}})|^{2} + \frac{\beta}{\hat{\mathcal{J}}^{\omega}_{\Gamma}} \int_{0}^{T} \int_{\hat{\Gamma}} |(\hat{\zeta} \bd{e}_{y} - \hat{\bd{u}}) \cdot \hat{\bd{\tau}}^{\omega}|^{2} + \frac{1}{2}\rho_{p} h \int_{0}^{T} \int_{\hat{\Gamma}} |\partial_{t}\omega|^{2} \\
+ \frac{1}{2}C_{0}h^{3} \int_{0}^{T} \int_{\hat{\Gamma}} |\hat{\Delta} \hat{\omega}|^{2} + \frac{1}{2}\rho_{b} \int_{\hat{\Omega}_{b}} |\partial_{t}\hat{\bd{\eta}}(T)|^{2} + 2\mu_{e} \int_{0}^{T} \int_{\hat{\Omega}_{b}} |\hat{\bd{D}}(\hat{\bd{\eta}})|^{2} + \lambda_{e} \int_{0}^{T} \int_{\hat{\Omega}_{b}} |\hat{\nabla} \cdot \hat{\bd{\eta}}|^{2} \\
+ 2\mu_{v} \int_{0}^{T} \int_{\hat{\Omega}_{b}} |\hat{\bd{D}}(\partial_{t}\hat{\bd{\eta}})|^{2} + \lambda_{v} \int_{0}^{T} \int_{\hat{\Omega}_{b}} |\hat{\nabla} \cdot \partial_{t}\hat{\bd{\eta}}|^{2} + \frac{1}{2}c_{0} \int_{\hat{\Omega}_{b}} |\hat{p}|^{2} + \kappa \int_{0}^{T} \int_{\hat{\Omega}_{b}} \hat{\mathcal{J}}^{\eta^{\delta}}_{b} |\hat{\nabla}^{\eta^{\delta}}_{b} \hat{p}|^{2} \\
= \frac{1}{2} \int_{\Omega_{f}(0)} |\hat{\bd{u}}(0)|^{2} + \frac{1}{2} \rho_{p} \int_{\hat{\Gamma}} |\partial_{t}\hat{\omega}(0)|^{2} + \frac{1}{2} \rho_{b} \int_{\hat{\Omega}_{b}} |\hat{\bd{\xi}}_{0}|^{2} + \frac{1}{2} c_{0} \int_{\hat{\Omega}_{b}} |\hat{p}_0|^{2}. 
\end{multline*}
\fi

\subsection{A constructive splitting scheme}

We use the following splitting scheme to semidiscretize the problem in time and construct approximate solutions. We split the entire time interval $[0, T]$ into $N$ subintervals $[t_{n}, t_{n + 1}]$ where $t_{n} = n\Delta t$ for $n = 0, 1, ..., N - 1$ and solve the following two subproblems iteratively to construct approximate solutions. We will denote the (spatially non-discrete and time-discrete) approximate solution on $[t_{n}, t_{n + 1}]$ by
\begin{equation*}
\left(\bd{u}^{n + \frac{i}{2}}_{N}, \bd{\omega}^{n + \frac{i}{2}}_{N}, \bd{\zeta}^{n + \frac{i}{2}}_{N}, \bd{\eta}^{n + \frac{i}{2}}_{N}, p^{n + \frac{i}{2}}_{N}\right), \qquad \text{ for } i = 1, 2
\end{equation*}
where $\bd{\zeta}^{n + \frac{i}{2}}_{N}$ is the approximate (vector-valued) plate velocity. The approximate vector for each $n$ at $i = 1$ represents the result after running the plate subproblem and the approximate vector for each $n$ at $i = 2$ represents the result after running the Biot/fluid subproblem. We use the initial data to initialize the splitting scheme:
\begin{equation*}
\Big(\bd{u}^{0}_{N}, \bd{\omega}^{N}, \bd{\zeta}^{0}_{N}, \bd{\eta}^{0}_{N}, p^{0}_{N}\Big) := \Big(\bd{u}_{0}, \bd{\eta}_{0}^{\delta}|_{\Gamma}, \bd{\xi}_{0}^{\delta}|_{\Gamma}, \bd{\eta}_{0}, p_{0}\Big)
\end{equation*}
Note that in accordance with the notational convention just introduced, we have omitted the hat notation on these approximate semidiscrete solutions.

This constructive splitting scheme shows why the plate approximation is helpful for obtaining existence, as to obtain a stable scheme, it is useful to have a separate subproblem for the interface update, which is accomplished in this case by giving the interface mass/inertia. Thus, the plate (interface) displacement in the splitting scheme evolves dynamically via a \textit{separate elastodynamics equation} for the movement of the interface.

\medskip

\noindent \textbf{Plate subproblem.} In the plate subproblem, we solve the viscoelastic plate equation with zero external forcing, in order to update the interface displacement (which will hence be used in the following Biot/fluid subproblem to define the semidiscrete moving domains). We keep $\bd{u}^{n + \frac{1}{2}}_{N} = \bd{u}^{n}_{N}$, $\bd{\eta}^{n + \frac{1}{2}}_{N} = \bd{\eta}^{n}_{N}$, and $p^{n + \frac{1}{2}}_{N} = p^{n}_{N}$ unchanged. The \textbf{plate subproblem} is to find $\bd{\omega}^{n + \frac{1}{2}}_{N} \in H^{2}(\Gamma)$ and $\bd{\zeta}^{n + \frac{1}{2}}_{N} \in H^{2}(\Gamma)$ such that: 
\begin{equation}\label{plateweak1}
\int_{\Gamma} \left(\frac{\bd{\omega}^{n + \frac{1}{2}}_{N} - \bd{\omega}^{n - \frac{1}{2}}_{N}}{\Delta t}\right) \cdot \phi = \int_{\Gamma} \bd{\zeta}^{n + \frac{1}{2}}_{N} \cdot \phi, \qquad \text{ for all } \phi \in L^{2}(\Gamma),
\end{equation}
\begin{equation}\label{plateweak2}
h \int_{\Gamma} \left(\frac{\bd{\zeta}^{n + \frac{1}{2}}_{N} - \bd{\zeta}^{n}_{N}}{\Delta t}\right) \cdot \varphi + h \int_{\Gamma} \Delta \bd{\zeta}^{n + \frac{1}{2}}_{N} \cdot \Delta \bd{\varphi} + h \int_{\Gamma} \Delta \bd{\omega}^{n + \frac{1}{2}}_{N} \cdot \Delta \bd{\varphi} = 0, \qquad \text{ for all } \varphi \in H^{2}(\Gamma),
\end{equation}
where we are solving the plate elastodynamics equation in weak formulation with zero external loading. Existence and uniqueness for the plate subproblem is classically established using the Lax-Milgram lemma, \cite{Evans}. By substituting $\bd{\varphi} = \bd{\zeta}^{n + \frac{1}{2}}_{N}$ and using the facts that $\displaystyle \bd{\zeta}^{n + \frac{1}{2}}_{N} = \frac{\bd{\omega}^{n + \frac{1}{2}}_{N} - \bd{\omega}^{n - \frac{1}{2}}_{N}}{\Delta t}$ and $\bd{\omega}^{n}_{N} = \bd{\omega}^{n - \frac{1}{2}}_{N}$, we obtain the following energy estimate:
\begin{equation}\label{energyplate}
\frac{1}{2} h \int_{\Gamma} |\bd{\zeta}^{n + \frac{1}{2}}_{N}|^{2} + \frac{1}{2} h \int_{\Gamma} |\Delta \bd{\omega}^{n + \frac{1}{2}}_{N}|^{2} + (\Delta t) h \int_{\Gamma} |\Delta \bd{\zeta}^{n + \frac{1}{2}}_{N}|^{2} = \frac{1}{2} h \int_{\Gamma} |\bd{\zeta}^{n}_{N}|^{2} + \frac{1}{2} h \int_{\Gamma} |\Delta \bd{\omega}^{n}_{N}|^{2}.
\end{equation}

\medskip

\noindent \textbf{Biot-fluid subproblem.} We keep the plate displacement unchanged $\omega^{n + 1}_{N} = \omega^{n + \frac{1}{2}}_{N}$. We recall the function spaces $V^{\omega}_{f}$, $V_{d}$, and $V_{p}$ from \eqref{Vomegaf}, \eqref{Vd}, and \eqref{Vp}. The \textbf{Biot-fluid subproblem} is to find $(\bd{u}^{n + 1}_{N}, \bd{\zeta}^{n + 1}_{N}, \bd{\eta}^{n + 1}_{N}, p^{n + 1}_{N}) \in V^{\omega^{n}_{N}}_{f} \times H^{2}(\Gamma) \times V_{d} \times V_{p}$ such that the following weak formulation is satisfied for all test functions $(\bd{v}, \bd{\varphi}, \bd{\psi}, r) \in V^{\omega^{n}_{N}}_{f} \times H^{2}(\Gamma) \times V_{d} \times V_{p}$ satisfying the regularized kinematic coupling condition $\bd{\psi}^{\delta}|_{\Gamma} = \bd{\varphi}$:
\begin{equation}\label{weakBiot1}
\begin{aligned}
&\int_{\Omega_{f}} \mathcal{J}^{\omega^{n}_{N}}_{f} \dot{\bd{u}}^{n + 1}_{N} \cdot \boldsymbol{v} 
+ 2\nu \int_{\Omega_{f}} \mathcal{J}^{\omega^{n}_{N}}_{f} \bd{D}^{\omega^{n}_{N}}_{f}(\boldsymbol{u}^{n + 1}_{N}) : \bd{D}^{\omega^{n}_{N}}_{f}(\boldsymbol{v}) + \int_{\Gamma} \left(\frac{1}{2}\boldsymbol{u}^{n + 1}_{N} \cdot \boldsymbol{u}^{n}_{N} - p^{n + 1}_{N}\right)(\boldsymbol{\psi}^{\delta} - \boldsymbol{v})\cdot \boldsymbol{n}^{\omega^{n}_{N}} 
\\
&+ \frac{1}{2} \int_{\Omega_{f}} \mathcal{J}^{\omega^{n}_{N}}_{f} \left[\left(\left(\boldsymbol{u}^{n}_{N} - \bd{w}^{n + 1}_{N}\right) \cdot \nabla^{\omega^{n}_{N}}_{f} \boldsymbol{u}^{n + 1}_{N}\right) \cdot \boldsymbol{v} - \left(\left(\boldsymbol{u}^{n}_{N} - \bd{w}^{n + 1}_{N}\right) \cdot \nabla^{\omega^{n}_{N}}_{f} \boldsymbol{v}\right) \cdot \boldsymbol{u}^{n + 1}_{N}\right] 
\\
&+ \frac{1}{2} \int_{\Omega_{f}} \left(\frac{\mathcal{J}^{\omega^{n + 1}_{N}}_{f} - \mathcal{J}^{\omega^{n}_{N}}_{f}}{\Delta t}\right) \bd{u}^{n + 1}_{N} \cdot \boldsymbol{v} + \frac{1}{2} \int_{\hat{\Gamma}} \left(\boldsymbol{u}^{n + 1}_{N} - \left(\dot{\bd{\eta}}^{n + 1}_{N}\right)^{\delta}\right) \cdot \boldsymbol{n}^{\omega^{n}_{N}} (\boldsymbol{u}^{n}_{N} \cdot \boldsymbol{v}) 
\\
&+ \beta \int_{\Gamma} \Big(\mathcal{S}^{\omega}_{\Gamma}\Big)^{-1} \left(\left(\boldsymbol{\dot{\eta}}^{n + 1}_{N}\right)^{\delta} - \boldsymbol{u}^{n + 1}_{N}\right) \cdot \boldsymbol{\tau}^{\omega^{n}_{N}} (\boldsymbol{\psi}^{\delta} - \boldsymbol{v}) \cdot \boldsymbol{\tau}^{\omega^{n}_{N}} + \rho_{b} \int_{\Omega_{b}} \left(\frac{\boldsymbol{\dot{\eta}}^{n + 1}_{N} - \boldsymbol{\dot{\eta}}^{n}_{N}}{\Delta t}\right) \cdot \boldsymbol{\psi} 
\\
&+ h \int_{\Gamma} \left(\frac{\bd{\zeta}^{n + 1}_{N} - \bd{\zeta}^{n + \frac{1}{2}}_{N}}{\Delta t}\right) \cdot \bd{\varphi} + 2\mu_{e} \int_{\Omega_{b}} \boldsymbol{D}(\boldsymbol{\eta}^{n + 1}_{N}) : \boldsymbol{D}(\boldsymbol{\psi}) + \lambda_{e} \int_{\Omega_{b}} (\nabla \cdot \boldsymbol{\eta}^{n + 1}_{N})(\nabla \cdot \boldsymbol{\psi}) 
\\
&+ 2\mu_{v} \int_{\Omega_{b}} \bd{D}(\dot{\bd{\eta}}^{n + 1}_{N}) : \bd{D}(\bd{\psi}) + \lambda_{v} \int_{\Omega_{b}} (\nabla \cdot \dot{\bd{\eta}}^{n + 1}_{N}) (\nabla \cdot \bd{\psi}) - \alpha \int_{\Omega_{b}} \mathcal{J}^{(\eta^{n}_{N})^{\delta}}_{b} p^{n + 1}_{N} \nabla^{(\eta^{n}_{N})^{\delta}}_{b} \cdot \boldsymbol{\psi}^{\delta} 
\\
&+ c_{0}  \int_{\Omega_{b}} \frac{p^{n + 1}_{N} - p^{n}_{N}}{\Delta t} r - \alpha \int_{\Omega_{b}} \mathcal{J}^{(\eta^{n}_{N})^{\delta}}_{b} \left(\boldsymbol{\dot{\eta}}^{n + 1}_{N}\right)^{\delta} \cdot \nabla^{(\eta^{n}_{N})^{\delta}}_{b} r - \alpha \int_{\Gamma} \left(\left(\boldsymbol{\dot{\eta}}^{n + 1}_{N}\right)^{\delta} \cdot \boldsymbol{n}^{\omega^{n}_{N}}\right) r 
\\
&+ \kappa \int_{\Omega_{b}} \mathcal{J}^{(\eta^{n}_{N})^{\delta}}_{b} \nabla^{(\eta^{n}_{N})^{\delta}}_{b} p^{n + 1}_{N} \cdot \nabla^{(\eta^{n}_{N})^{\delta}}_{b} r - \int_{\Gamma} \left[\left(\boldsymbol{u}^{n + 1}_{N} - \left(\boldsymbol{\dot{\eta}}^{n + 1}_{N}\right)^{\delta}\right) \cdot \boldsymbol{n}^{\omega^{n}_{N}}\right]r = 0,
\end{aligned}
\end{equation}
and
\begin{equation}\label{weakBiot2}
\int_{\Gamma} \left(\frac{\left(\boldsymbol{\eta}^{n + 1}_{N}\right)^{\delta} - \left(\boldsymbol{\eta}^{n}_{N}\right)^{\delta}}{\Delta t} \right) \cdot \boldsymbol{\phi} = \int_{\Gamma} \bd{\zeta}^{n + 1}_{N} \cdot \boldsymbol{\phi}, \qquad \text{ for all } \boldsymbol{\phi} \in L^{2}(\Gamma),
\end{equation}
where the discrete ALE velocity is defined by:
\begin{equation}\label{aleveldisc}
\bd{w}^{n + 1}_{N} = \frac{1}{\Delta t}\Big(\bd{\Phi}^{\omega^{n + 1}_{N}}_{f} - \bd{\Phi}^{\omega^{n}_{N}}_{f}\Big),
\end{equation}
and where we use the shorthand ``dot" notation for discrete time derivatives:
\begin{equation*}
\dot{f}^{n}_{N} = \frac{f^{n}_{N} - f^{n - 1}_{N}}{\Delta t}.
\end{equation*}
We emphasize that to derive this weak formulation, we must use the fact that $\displaystyle \bd{\zeta}^{n + 1}_{N}|_{\Gamma} = (\bd{\dot{\eta}}^{n + 1}_{N})^{\delta}$. 

\medskip

\noindent \underline{\textit{Existence and uniqueness for the Biot-fluid subproblem.}} Due to the additional regularization in the kinematic coupling condition, the well-posedness for the Biot-fluid subproblem is different from that found in other works on nonlinearly coupled FPSI (such as in Section 6.2 of \cite{FPSIJMPA}), so we present the full argument and emphasize the parts that change due to the additional kinematic regularization. 

The goal will be to eliminate the variable $\bd{\zeta}^{n + 1}_{N}$ from the weak formulation so that we are solving for $(\bd{u}^{n + 1}_{N}, \bd{\eta}^{n + 1}_{N}, p^{n + 1}_{N})$. By the regularized kinematic coupling condition which is deduced from \eqref{weakBiot2}:
\begin{equation*}
\bd{\zeta}^{n + 1}_{N} = \left(\frac{\left(\bd{\eta}^{n + 1}_{N}\right)^{\delta} - \left(\bd{\eta}^{n}_{N}\right)^{\delta}}{\Delta t}\right).
\end{equation*}
We will also use the rescaling of test functions $v \to (\Delta t) v$ and $r \to (\Delta t) r$, and we hence obtain the following weak formulation of the problem: find $(\bd{u}^{n + 1}_{N}, \bd{\eta}^{n + 1}_{N}, p^{n + 1}_{N}) \in V^{\omega^{n}_{N}}_{f} \times V_{d} \times V_{p}$ such that for all test functions $(\bd{v}, \bd{\psi}, r) \in V^{\omega^{n}_{N}}_{f} \times V_{d} \times V_{p}$:
\begin{small}
\begin{multline}\label{BiotFluidLaxMilgram}
B[\bd{u}^{n + 1}_{N}, \bd{v}, \bd{\eta}^{n + 1}_{N}, \bd{\psi}, p^{n + 1}_{N}, r] = (\Delta t)^{2} \int_{\Omega_{f}} \mathcal{J}^{\omega^{n}_{N}}_{f} \bd{u}^{n}_{N} \cdot \boldsymbol{v} - \frac{1}{2} (\Delta t)^{2} \int_{\Gamma} \left(\boldsymbol{\eta}^{n}_{N}\right)^{\delta} \cdot \boldsymbol{n}^{\omega^{n}_{N}} (\boldsymbol{u}^{n}_{N} \cdot \boldsymbol{v}) \\
+ \beta (\Delta t) \int_{\Gamma} \Big(\mathcal{S}^{\omega}_{\Gamma}\Big)^{-1} \left(\boldsymbol{\eta}^{n}_{N}\right)^{\delta} \cdot \boldsymbol{\tau}^{\omega^{n}_{N}} (\boldsymbol{\psi}^{\delta} - (\Delta t) \boldsymbol{v}) \cdot \boldsymbol{\tau}^{\omega^{n}_{N}} + \rho_{b} \int_{\Omega_{b}} (2\boldsymbol{\eta}^{n}_{N} - \boldsymbol{\eta}^{n - 1}_{N}) \cdot \boldsymbol{\psi} \\
+ h \int_{\Gamma} \left(\left(\boldsymbol{\eta}^{n}_{N}\right)^{\delta} + (\Delta t) \bd{\zeta}^{n + \frac{1}{2}}_{N}\right) \cdot \boldsymbol{\psi}^{\delta} + 2\mu_{v}(\Delta t) \int_{\Omega_{b}} \bd{D}(\bd{\eta}^{n}_{N}) : \bd{D}(\bd{\psi}) + \lambda_{v} (\Delta t) \int_{\Omega_{b}} (\nabla \cdot \bd{\eta}^{n}_{N}) (\nabla \cdot \bd{\psi}) \\
+ c_{0} (\Delta t)^{2} \int_{\Omega_{b}} p^{n}_{N} r - \alpha (\Delta t)^{2} \int_{\Omega_{b}} \mathcal{J}^{(\eta^{n}_{N})^{\delta}}_{b} \left(\boldsymbol{\eta}^{n}_{N}\right)^{\delta} \cdot \nabla^{(\eta^{n}_{N})^{\delta}}_{b} r - \alpha (\Delta t)^{2} \int_{\Gamma} \left(\left(\boldsymbol{\eta}^{n}_{N}\right)^{\delta} \cdot \boldsymbol{n}^{\omega^{n}_{N}}\right) r + (\Delta t)^{2} \int_{\Gamma} \left(\left(\boldsymbol{\eta}^{n}_{N}\right)^{\delta} \cdot \boldsymbol{n}^{\omega^{n}_{N}}\right) r,
\end{multline}
\end{small}
for the bilinear form $B: V^{\omega^{n}_{N}}_{f} \times V_{d} \times V_{p} \times V^{\omega^{n}_{N}}_{f} \times V_{d} \times V_{p}$ defined by:
\begin{small}
\begin{align}\label{bilinear}
&B[\bd{u}, \bd{v}, \bd{\eta}, \bd{\psi}, p, r] := (\Delta t)^{2} \int_{\Omega_{f}} \mathcal{J}^{\omega^{n}_{N}}_{f} \boldsymbol{u} \cdot \boldsymbol{v} \nonumber \\
&+ \frac{1}{2} (\Delta t)^{3} \int_{\Omega_{f}} \mathcal{J}^{\omega^{n}_{N}}_{f} \Big[\Big((\boldsymbol{u}^{n}_{N} - \bd{w}^{n + 1}_{N}) \cdot \nabla^{\omega^{n}_{N}} \boldsymbol{u}\Big) \cdot \boldsymbol{v} - \Big((\boldsymbol{u}^{n}_{N} - \bd{w}^{n + 1}_{N}) \cdot \nabla^{\omega^{n}_{N}} \boldsymbol{v}\Big) \cdot \boldsymbol{u}\Big] \nonumber \\
&+ \frac{(\Delta t)^{2}}{2} \int_{\Omega_{f}} \Big(\mathcal{J}^{\omega^{n + 1}_{N}}_{f} - \mathcal{J}^{\omega^{n}_{N}}_{f}\Big) \boldsymbol{u} \cdot \boldsymbol{v} + \frac{1}{2} (\Delta t)^{3} \int_{\Gamma} (\boldsymbol{u} - (\Delta t)^{-1} \boldsymbol{\eta}^{\delta}) \cdot \boldsymbol{n}^{\omega^{n}_{N}} (\boldsymbol{u}^{n}_{N} \cdot \boldsymbol{v}) \nonumber \\
&+ 2\nu (\Delta t)^{3} \int_{\Omega_{f}} \mathcal{J}^{\omega^{n}_{N}}_{f} \boldsymbol{D}^{\omega^{n}_{N}}_{f}(\boldsymbol{u}) : \boldsymbol{D}^{\omega^{n}_{N}}_{f}(\boldsymbol{v}) + (\Delta t)^{2} \int_{\Gamma} \left(\frac{1}{2}\boldsymbol{u} \cdot \boldsymbol{u}^{n}_{N} - p\right)(\boldsymbol{\psi}^{\delta} - (\Delta t) \boldsymbol{v})\cdot \boldsymbol{n}^{\omega^{n}_{N}} \nonumber \\
&+ \beta (\Delta t)^{2} \int_{\Gamma} \Big(\mathcal{S}^{\omega}_{\Gamma}\Big)^{-1} \left((\Delta t)^{-1} \boldsymbol{\eta}^{\delta} - \boldsymbol{u}\right) \cdot \boldsymbol{\tau}^{\omega^{n}_{N}} (\boldsymbol{\psi}^{\delta} - (\Delta t) \boldsymbol{v}) \cdot \boldsymbol{\tau}^{\omega^{n}_{N}} + \rho_{b} \int_{\Omega_{b}} \boldsymbol{\eta} \cdot \boldsymbol{\psi} + h \int_{\Gamma} \boldsymbol{\eta}^{\delta} \cdot \boldsymbol{\psi}^{\delta} \nonumber \\
&+ (2\mu_{e} (\Delta t)^{2} + 2\mu_{v}(\Delta t)) \int_{\Omega_{b}} \boldsymbol{D}(\boldsymbol{\eta}) : \boldsymbol{D}(\boldsymbol{\psi}) + (\lambda_{e} (\Delta t)^{2} + \lambda_{v}(\Delta t)) \int_{\Omega_{b}} (\nabla \cdot \boldsymbol{\eta})(\nabla \cdot \boldsymbol{\psi}) \nonumber \\
&- \alpha (\Delta t)^{2} \int_{\Omega_{b}} \mathcal{J}^{(\eta^{n}_{N})^{\delta}}_{b} p \nabla^{(\eta^{n}_{N})^{\delta}}_{b} \cdot \boldsymbol{\psi}^{\delta} + c_{0} (\Delta t)^{2} \int_{\Omega_{b}} p r - \alpha (\Delta t)^{2} \int_{\Omega_{b}} \mathcal{J}^{(\eta^{n}_{N})^{\delta}}_{b} \boldsymbol{\eta}^{\delta} \cdot \nabla^{(\eta^{n}_{N})^{\delta}}_{b} r \nonumber \\
&- \alpha (\Delta t)^{2} \int_{\Gamma} (\boldsymbol{\eta}^{\delta} \cdot \boldsymbol{n}^{\omega^{n}_{N}}) r + \kappa (\Delta t)^{3} \int_{\Omega_{b}} \mathcal{J}^{(\eta^{n}_{N})^{\delta}}_{b} \nabla^{(\eta^{n}_{N})^{\delta}}_{b} p \cdot \nabla^{(\eta^{n}_{N})^{\delta}}_{b} r - (\Delta t)^{3} \int_{\Gamma} [(\boldsymbol{u} - (\Delta t)^{-1} \boldsymbol{\eta}^{\delta})\cdot \boldsymbol{n}^{\omega^{n}_{N}}]r.
\end{align}
\end{small}
Note that the influence of the kinematic coupling regularization appears in several terms in the weak formulation, through additional convolution terms involving the convolution parameter $\delta$. However, the terms in which these additional convolutions appear are all consistent in the sense that the resulting bilinear form for the Biot-fluid subproblem is \textbf{still coercive}, which reflects the fact that the influence of the kinematic regularization appears consistently throughout the weak formulation in such a way that there is still a well-defined energy estimate. In particular,
\begin{multline}\label{coerciveB}
B[(\bd{u}, \bd{\eta}, p), (\bd{u}, \bd{\eta}, p)] = \frac{1}{2}(\Delta t)^{2} \int_{\Omega_{f}} \left(\mathcal{J}^{\omega^{n}_{N}}_{f} + \mathcal{J}^{\omega^{n + 1}_{N}}_{f}\right) |\bd{u}|^{2} + 2\nu (\Delta t)^{3} \int_{\Omega_{f}} \mathcal{J}^{\omega^{n}_{N}}_{f} |\bd{D}^{\omega^{n}_{N}}_{f}(\bd{u})|^{2} \\
+ \beta (\Delta t) \int_{\Gamma} \Big(\mathcal{S}^{\omega}_{\Gamma}\Big)^{-1} |(\bd{\eta}^{\delta} - (\Delta t) \bd{u}) \cdot \bd{\tau}^{\omega^{n}_{N}}|^{2} + \rho_{b} \int_{\Omega_{b}} |\bd{\eta}|^{2} + h \int_{\Gamma} |\bd{\eta}^{\delta}|^{2} + (2\mu_{e}(\Delta t)^{2} + 2\mu_{v}(\Delta t)) \int_{\Omega_{b}} |\bd{D}(\bd{\eta})|^{2} \\
+ (\lambda_{e}(\Delta t)^{2} + \lambda_{v}(\Delta t)) \int_{\Omega_{b}} |\nabla \cdot \bd{\eta}|^{2} + c_{0}(\Delta t)^{2} \int_{\Omega_{b}} p^{2} + \kappa(\Delta t)^{3} \int_{\Omega_{b}} \mathcal{J}^{(\eta^{n}_{N})^{\delta}} |\nabla^{(\eta^{n}_{N})^{\delta}}_{b} p|^{2}.
\end{multline}
Therefore, the existence of a unique solution $(\bd{u}^{n + 1}_{N}, \bd{\eta}^{n + 1}_{N}, p^{n + 1}_{N}) \in V^{\omega^{n}_{N}}_{f} \times V_{d} \times V_{p}$ satisfying the weak formulation is guaranteed by the Lax-Milgram lemma and Korn's inequality (see Proposition 6.1 in \cite{FPSIJMPA}), which allows us to find a solution to the Biot-fluid subproblem, as long as the bilinear form defined above is coercive. This will be true under the following non-degeneracy assumptions.

\begin{proposition}\label{coercivegeometry}
Suppose that the following geometric non-degeneracy assumptions are satisfied for a fixed but arbitrary $N$ and $0 \le n \le N - 1$:
\begin{enumerate}
\item \textbf{Non-degeneracy of the moving interface.} The map $\hat{\bd{\Phi}}^{\omega^{n}_{N}}_{\Gamma}: \hat{\Gamma} \to \Gamma(t)$ is injective, and $\hat{\bd{\Phi}}_{f}^{\omega^{n}_{N}}(\hat{\Gamma}) \cap \partial \hat{\Omega} = \varnothing$. Furthermore, there exists a positive constant $\alpha > 0$ (independent of $n$) such that
\begin{equation*}
\|(-\sin(z), \cos(z)) + \partial_{z} \bd{\omega}^{n}_{N}(z)\| \ge \alpha, \qquad \text{ for all } z \in [0, 2\pi].
\end{equation*}
\item \textbf{Invertibility of the regularized Lagrangian map.} There exists positive constants $c_{1}, c_{2} > 0$ such that
\begin{equation*}
\text{det}\Big(\bd{I} + \nabla(\bd{\eta}^{n}_{N})^{\delta}\Big) \ge c_{1} > 0 \ \ \text{ and } \ \ |\bd{I} + \nabla (\bd{\eta}^{n}_{N})^{\delta}| \le c_{2} \quad \text{ on } \overline{\Omega_{b}}.
\end{equation*}
\item \textbf{Non-degeneracy of the (fluid domain) ALE map.} There exists a positive constant $c_{3} > 0$ such that
\begin{equation*}
\mathcal{J}^{\omega^{n + 1}_{N}}_{f} \ge c_{3} > 0, \quad \mathcal{J}^{\omega^{n}_{N}}_{f} \ge c_{3} > 0, \quad \text{ on } \overline{\Omega_{b}}.
\end{equation*}
\end{enumerate}
Then, there exists a unique solution $(\bd{u}^{n + 1}_{N}, \bd{\zeta}^{n + 1}_{N}, \bd{\eta}^{n + 1}_{N}, p^{n + 1}_{N})$ to the weak formulation of the fluid/Biot subproblem in \eqref{weakBiot1} and \eqref{weakBiot2}.
\end{proposition}

\begin{proof}
Under the assumptions, the existence of a solution to the fluid-Biot problem follows from a usual application of the Lax-Milgram theorem to the bilinear form $B$ defined in \eqref{bilinear}, which as a result of the geometric assumptions, is coercive due to the calculation in \eqref{coerciveB} and the observation using \eqref{etanabla} and \eqref{jeta} that
\begin{equation}\label{kappabound}
\kappa (\Delta t)^{3} \int_{\Omega_{b}} \mathcal{J}^{(\eta^{n}_{N})^{\delta}} |\nabla_{b}^{(\eta^{n}_{N})^{\delta}}p|^{2} \ge \kappa (\Delta t)^{3} c_{1} \int_{\Omega_{b}} |\bd{I} + \nabla(\bd{\eta}^{n}_{N})^{\delta}|^{-2} |\nabla p|^{2} \ge \kappa (\Delta t)^{3} c_{1} c_{2}^{-2} \int_{\Omega_{b}} |\nabla p|^{2}.
\end{equation}
Furthermore, the discrete Lagrangian map $\hat{\bd{\Phi}}^{(\eta^{n}_{N})^{\delta}}_{b}: \hat{\Omega}_{b} \to (\Omega_{b})^{n, \delta}_{N}$ is bijective by Proposition \ref{mapinjective} and the given assumption on the non-degeneracy of the interface and the regularized Lagrangian map. From this point, the rest of the proof of existence/uniqueness for the fluid-Biot subproblem proceeds as in the proof of Lemma 6.2 in \cite{FPSIJMPA}.
\end{proof}

We further obtain the following energy estimate for the fluid/Biot subproblem by substituting $(\bd{v}, \bd{\varphi}, \bd{\psi}, r) = (\bd{u}^{n + 1}_{N}, \bd{\zeta}^{n + 1}_{N}, \bd{\eta}^{n + 1}_{N}, p^{n + 1}_{N})$ in the weak formulation for the Biot-fluid subproblem:
\begin{multline}\label{energyBiot}
\frac{1}{2} \int_{\Omega_{f}} \mathcal{J}^{\omega^{n + 1}_{N}}_{f} |\bd{u}^{n + 1}_{N}|^{2} + 2\nu \int_{\Omega_{f}} \mathcal{J}^{\omega^{n}_{N}}_{f} |\bd{D}^{\omega^{n}_{N}}_{f}(\bd{u}^{n + 1}_{N})|^{2} + \beta \int_{\Gamma} \Big(\mathcal{S}^{\omega}_{\Gamma}\Big)^{-1} \left|\left(\left(\dot{\bd{\eta}}^{n + 1}_{N}\right)^{\delta} - \bd{u}^{n + 1}_{N}\right) \cdot \bd{\tau}^{\omega^{n}_{N}}\right|^{2} \\
+ \frac{1}{2}\rho_{b} \int_{\Omega_{b}} |\dot{\bd{\eta}}^{n + 1}_{N}|^{2} + \frac{1}{2} h \int_{\Gamma} |\bd{\zeta}^{n + 1}_{N}|^{2} + \mu_{e} \int_{\Omega_{b}} |\bd{D}(\bd{\eta}^{n + 1}_{N})|^{2} + \frac{1}{2} \lambda_{e} \int_{\Omega_{b}} |\nabla \cdot \bd{\eta}^{n + 1}_{N}|^{2} \\
+ 2\mu_{v} \int_{\Omega_{b}} |\bd{D}(\dot{\bd{\eta}}^{n + 1}_{N})|^{2} + \lambda_{v} \int_{\Omega_{b}} |\nabla \cdot \dot{\bd{\eta}}^{n + 1}_{N}|^{2} + \frac{1}{2}c_{0}\int_{\Omega_{b}} |p^{n + 1}_{N}|^{2} + \kappa \int_{\Omega_{b}} \mathcal{J}^{(\eta^{n}_{N})^{\delta}}_{b} \left|\nabla^{(\eta^{n}_{N})^{\delta}}_{b} p^{n + 1}_{N}\right|^{2} \\
= \frac{1}{2} \int_{\Omega_{f}} \mathcal{J}^{\omega^{n}_{N}}_{f} |\bd{u}^{n}_{N}|^{2} + \frac{1}{2}\rho_{b} \int_{\Omega_{b}} |\dot{\bd{\eta}}^{n}_{N}|^{2} + \mu_{e} \int_{\Omega_{b}} |\bd{D}(\bd{\eta}^{n}_{N})|^{2} + \frac{1}{2}\lambda_{e} \int_{\Omega_{b}} |\nabla \cdot \bd{\eta}^{n}_{N}|^{2} + \frac{1}{2} \int_{\Gamma} |\bd{\zeta}^{n}_{N}|^{2} + \frac{1}{2} c_{0} \int_{\Omega_{b}} |p^{n}_{N}|^{2}.
\end{multline}

\medskip

\subsection{Uniform estimates and approximate solutions}

Adding the weak formulations for the plate and the Biot-fluid subproblems, we obtain the following semidiscrete formulation, which is satisfied for all test functions $(\bd{v}, \varphi, \bd{\psi}, r) \in V^{\omega^{n}_{N}}_{f} \times H^{2}(\Gamma) \times V_{s} \times V_{p}$ such that $\bd{\psi}^{\delta}|_{\Gamma} = \bd{\varphi}$:

\begin{small}
\begin{equation}\label{semi1}
\begin{aligned}
&\int_{\Omega_{f}} \mathcal{J}^{\omega^{n}_{N}}_{f} \boldsymbol{\dot{u}}^{n + 1}_{N} \cdot \boldsymbol{v} 
+ 2\nu \int_{\Omega_{f}} \mathcal{J}^{\omega^{n}_{N}}_{f} \boldsymbol{D}^{\omega^{n}_{N}}_{f}(\boldsymbol{u}^{n + 1}_{N}) : \boldsymbol{D}^{\omega^{n}_{N}}_{f}(\boldsymbol{v}) + \int_{\Gamma} \left(\frac{1}{2}\boldsymbol{u}^{n + 1}_{N} \cdot \boldsymbol{u}^{n}_{N} - p^{n + 1}_{N}\right)(\boldsymbol{\psi}^{\delta} - \boldsymbol{v})\cdot \boldsymbol{n}^{\omega^{n}_{N}} 
\\
&+ \frac{1}{2} \int_{\Omega_{f}} \mathcal{J}^{\omega^{n}_{N}}_{f} \left[\left(\left(\boldsymbol{u}^{n}_{N} - \bd{w}^{n + 1}_{N}\right) \cdot \nabla^{\omega^{n}_{N}}_{f} \boldsymbol{u}^{n + 1}_{N}\right) \cdot \boldsymbol{v} - \left(\left(\boldsymbol{u}^{n}_{N} - \bd{w}^{n + 1}_{N}\right) \cdot \nabla^{\omega^{n}_{N}}_{f} \boldsymbol{v}\right) \cdot \boldsymbol{u}^{n + 1}_{N}\right] 
\\
&+ \frac{1}{2} \int_{\Omega_{f}} \left(\frac{\mathcal{J}^{\omega^{n + 1}_{N}}_{f} - \mathcal{J}^{\omega^{n}_{N}}_{f}}{\Delta t}\right) \boldsymbol{u}^{n + 1}_{N} \cdot \boldsymbol{v} + \frac{1}{2} \int_{\Gamma} \left(\boldsymbol{u}^{n + 1}_{N} - \left(\boldsymbol{\dot{\eta}}^{n + 1}_{N}\right)^{\delta}\right) \cdot \boldsymbol{n}^{\omega^{n}_{N}} (\boldsymbol{u}^{n}_{N} \cdot \boldsymbol{v}) 
\\
&+ \beta \int_{\Gamma} \Big(\mathcal{S}^{\omega}_{\Gamma}\Big)^{-1} \left(\left(\boldsymbol{\dot{\eta}}^{n + 1}_{N}\right)^{\delta} - \boldsymbol{u}^{n + 1}_{N}\right) \cdot \boldsymbol{\tau}^{\omega^{n}_{N}} (\boldsymbol{\psi}^{\delta} - \boldsymbol{v}) \cdot \boldsymbol{\tau}^{\omega^{n}_{N}} + \rho_{b} \int_{\Omega_{b}} \left(\frac{\boldsymbol{\dot{\eta}}^{n + 1}_{N} - \boldsymbol{\dot{\eta}}^{n}_{N}}{\Delta t}\right) \cdot \boldsymbol{\psi} 
\\
&+ h \int_{\Gamma} \left(\frac{\bd{\zeta}^{n + 1}_{N} - \bd{\zeta}^{n}_{N}}{\Delta t}\right) \cdot \bd{\varphi} + 2\mu_{e} \int_{\Omega_{b}} \boldsymbol{D}(\boldsymbol{\eta}^{n + 1}_{N}) : \boldsymbol{D}(\boldsymbol{\psi}) + \lambda_{e} \int_{\Omega_{b}} (\nabla \cdot \boldsymbol{\eta}^{n + 1}_{N})(\nabla \cdot \boldsymbol{\psi}) 
\\
&+ 2\mu_{v} \int_{\Omega_{b}} \bd{D}(\dot{\bd{\eta}}^{n + 1}_{N}) : \bd{D}(\bd{\psi}) + \lambda_{v} \int_{\Omega_{b}} (\nabla \cdot \dot{\bd{\eta}}^{n + 1}_{N}) (\nabla \cdot \bd{\psi}) - \alpha \int_{\Omega_{b}} \mathcal{J}^{(\eta^{n}_{N})^{\delta}}_{b} p^{n + 1}_{N} \nabla^{(\eta^{n}_{N})^{\delta}}_{b} \cdot \boldsymbol{\psi}^{\delta} + c_{0}  \int_{\Omega_{b}} \frac{p^{n + 1}_{N} - p^{n}_{N}}{\Delta t} r \\
&- \alpha \int_{\Omega_{b}} \mathcal{J}^{(\eta^{n}_{N})^{\delta}}_{b} \left(\boldsymbol{\dot{\eta}}^{n + 1}_{N}\right)^{\delta} \cdot \nabla^{(\eta^{n}_{N})^{\delta}}_{b} r - \alpha \int_{\Gamma} \left(\left(\boldsymbol{\dot{\eta}}^{n + 1}_{N}\right)^{\delta} \cdot \boldsymbol{n}^{\omega^{n}_{N}}\right) r + \kappa \int_{\Omega_{b}} \mathcal{J}^{(\eta^{n}_{N})^{\delta}}_{b} \nabla^{(\eta^{n}_{N})^{\delta}}_{b} p^{n + 1}_{N} \cdot \nabla^{(\eta^{n}_{N})^{\delta}}_{b} r \\
&- \int_{\Gamma} \left[\left(\boldsymbol{u}^{n + 1}_{N} - \left(\boldsymbol{\dot{\eta}}^{n + 1}_{N}\right)^{\delta}\right) \cdot \boldsymbol{n}^{\omega^{n}_{N}}\right]r + h \int_{\Gamma} \Delta \bd{\zeta}^{n + \frac{1}{2}}_{N} \cdot \Delta \bd{\varphi} + h \int_{\Gamma} \Delta \bd{\omega}^{n + \frac{1}{2}}_{N} \cdot \Delta \bd{\varphi} = 0,
\end{aligned}
\end{equation}
\end{small}

\begin{multline*}
\int_{\Gamma} \left(\frac{\bd{\omega}^{n + \frac{1}{2}}_{N} - \bd{\omega}^{n - \frac{1}{2}}_{N}}{\Delta t}\right) \cdot \bd{\phi} = \int_{\Gamma} \bd{\zeta}^{n + \frac{1}{2}}_{N} \cdot \bd{\phi}, \qquad \int_{\Gamma} \left(\frac{(\bd{\eta}^{n + 1}_{N})^{\delta} - (\bd{\eta}^{n}_{N})^{\delta}}{\Delta t}\right) \cdot \bd{\phi} = \int_{\Gamma} \bd{\zeta}^{n + 1}_{N} \cdot \bd{\phi}, \\
\qquad \qquad \qquad \qquad \text{ for all } \bd{\phi} \in L^{2}(\Gamma).
\end{multline*}

In addition, we can obtain an energy estimate for the approximate solutions. To do this, we define the discrete energy $E^{n + \frac{i}{2}}_{N}$ and the dissipation $D^{n}_{N}$ for $n = 0, 1, ..., N$ and $i = 0, 1$ as follows:
\begin{multline*}
E^{n + \frac{i}{2}}_{N} = \frac{1}{2} \int_{\Omega_{f}} \mathcal{J}^{\omega^{n}_{N}}_{f} |\bd{u}^{n + \frac{i}{2}}_{N}|^{2} + \frac{1}{2}\rho_{b} \int_{\Omega_{b}} |\dot{\bd{\eta}}^{n + \frac{i}{2}}_{N}|^{2} + \frac{1}{2}  h \int_{\Gamma} |\bd{\zeta}^{n + \frac{i}{2}}_{N}|^{2} + \frac{1}{2} h \int_{\Gamma} |\Delta \bd{\omega}^{n + \frac{i}{2}}_{N}|^{2} \\
+ \mu_{e} \int_{\Omega_{b}} |\bd{D}(\bd{\eta}^{n + \frac{i}{2}}_{N})|^{2} + \frac{1}{2}\lambda_{e} \int_{\Omega_{b}} |\nabla \cdot \bd{\eta}^{n + \frac{i}{2}}_{N}|^{2} + \frac{1}{2} c_{0} \int_{\Omega_{b}} |p^{n + \frac{i}{2}}_{N}|^{2},
\end{multline*}
\begin{multline*}
D^{n}_{N} = 2\nu \int_{\Omega_{f}} \mathcal{J}^{\omega^{n}_{N}}_{f} |\bd{D}^{\omega^{n}_{N}}_{f}(\bd{u}^{n + 1}_{N})|^{2} + \beta \int_{\Gamma} \Big(\mathcal{S}^{\omega}_{\Gamma}\Big)^{-1} \left|\left(\left(\dot{\bd{\eta}}^{n + 1}_{N}\right)^{\delta} - \bd{u}^{n + 1}_{N}\right) \cdot \bd{\tau}^{\omega^{n}_{N}}\right|^{2} \\
+ 2\mu_{v} \int_{\Omega_{b}} |\bd{D}(\dot{\bd{\eta}}^{n + 1}_{N})|^{2} + \lambda_{v} \int_{\Omega_{b}} |\nabla \cdot \dot{\bd{\eta}}^{n + 1}_{N}|^{2} + \kappa \int_{\Omega_{b}} \mathcal{J}^{(\eta^{n}_{N})^{\delta}}_{b} \left|\nabla^{(\eta^{n}_{N})^{\delta}}_{b} p^{n + 1}_{N}\right|^{2} + h\int_{\Gamma} |\Delta \bd{\zeta}^{n + \frac{1}{2}}_{N}|^{2}.
\end{multline*}

We then obtain the following energy estimates, which show that the energy of the approximate solutions is uniformly bounded independently of the time discretization parameter $N$. In addition, we obtain uniform bounds on geometric quantities involved in the problem up until some sufficiently small time $T > 0$ which is independent of $N$ and which depends only on the initial data of the problem. This will be useful for showing non-degeneracy of the regularized Lagrangian map for the Biot domain and of the ALE map for the fluid domain.

\begin{proposition}\label{geometricN}
There exists a sufficiently small time $T > 0$, \textit{which is independent of the parameters $N$}, such that for all positive integers $N$ and $n = 0, 1, ..., N$, there exist positive constants $c_{0}$, $c_{1}$, $c_{2}$, and $c_{3}$ such that:
\begin{equation}\label{uniformgeom1}
\text{det}(\bd{I} + \nabla (\bd{\eta}^{n}_{N})^{\delta}) \ge c_{0} > 0, \quad |(\bd{I} + \nabla(\bd{\eta}^{n}_{N})^{\delta})^{-1}| \ge c_{1} > 0, \quad |\bd{I} + \nabla (\bd{\eta}^{n}_{N})^{\delta}| \le c_{2}, \quad \text{ on } \overline{\Omega_{b}},
\end{equation}
\begin{equation}\label{uniformgeom2}
0 < c_{3}^{-1} \le \mathcal{J}^{\omega^{n}_{N}}_{f} \le c_{3}, \quad |\nabla \bd{\Phi}^{\omega^{n}_{N}}_{f}| \le c_{3}, \quad \text{ on } \overline{\Omega_{f}},
\end{equation}
and furthermore, the map $\bd{\Phi}^{\omega^{n}_{N}}_{\Gamma}: \Gamma \to \Gamma(t)$ is injective, with $\bd{\Phi}^{\omega^{n}_{N}}_{\Gamma}(\Gamma) \cap \partial \Omega = \varnothing$, and 
\begin{equation}\label{tnorm}
\|(-\sin(z), \cos(z)) + \partial_{z} \bd{\omega}^{n}_{N}(z)\| \ge \alpha > 0, \quad \text{ for all } z \in \Gamma,
\end{equation}
for all $n = 0, 1, ..., N$ and for all positive integers $N$, where the constant $\alpha > 0$ is independent of $n$ and $N$. In addition, there exists a uniform constant $C$ that is independent of $N$ which depends only on the initial data in the problem, such that
\begin{equation}\label{energyN}
\max_{i = 0, 1, 2}\left(\max_{n = 0, 1, ..., N - 1} E^{n + \frac{i}{2}}\right) \le C, \qquad \sum_{n = 0}^{N - 1} D^{n}_{N} \le C,
\end{equation}
and such that we have the following uniform numerical dissipation estimate:
\begin{multline}\label{numerdiss}
\sum_{n = 0}^{N - 1} \Bigg(\int_{\Omega_{f}} \mathcal{J}^{\omega^{n}_{N}}_{f} |\bd{u}^{n + 1}_{N} - \bd{u}^{n}_{N}|^{2} + \rho_{b} \int_{\Omega_{b}} |\dot{\bd{\eta}}^{n + 1}_{N} - \dot{\bd{\eta}}^{n}_{N}|^{2} + h\int_{\Gamma} |\bd{\zeta}^{n + 1}_{N} - \bd{\zeta}^{n + \frac{1}{2}}_{N}| + h\int_{\Gamma} |\bd{\zeta}^{n + \frac{1}{2}}_{N} - \bd{\zeta}^{n}_{N}|^{2} \\
+ h \int_{\Gamma} |\Delta \bd{\omega}^{n + \frac{1}{2}}_{N} - \Delta \bd{\omega}^{n}_{N}|^{2} + \mu_{e} \int_{\Omega_{b}} |\bd{D}(\bd{\eta}^{n + 1}_{N}) - \bd{D}(\bd{\eta}^{n}_{N})|^{2} + \lambda_{e} \int_{\Omega_{b}} |\nabla \cdot (\bd{\eta}^{n + 1}_{N} - \bd{\eta}^{n}_{N})|^{2} + c_{0} \int_{\Omega_{b}} |p^{n + 1}_{N} - p^{n}_{N}|^{2}\Bigg) \le C
\end{multline}
with the following explicit energy equalities:
\begin{equation*}
E^{n + \frac{1}{2}}_{N} + \frac{1}{2} h \int_{\Gamma} |\bd{\zeta}^{n + \frac{1}{2}}_{N} - \bd{\zeta}^{n}_{N}|^{2} + \frac{1}{2} h \int_{\Gamma} |\Delta \bd{\omega}^{n + \frac{1}{2}}_{N} - \Delta \bd{\omega}^{n}_{N}|^{2} = E^{n}_{N},
\end{equation*}
\begin{multline*}
E^{n + 1}_{N} + D^{n}_{N} + \frac{1}{2} \int_{\Omega_{f}} \mathcal{J}^{\omega^{n}_{N}}_{f} |\bd{u}^{n + 1}_{N} - \bd{u}^{n}_{N}|^{2} + \frac{1}{2} \rho_{b} \int_{\Omega_{b}} \left|\dot{\bd{\eta}}^{n + 1}_{N} - \dot{\bd{\eta}}^{n}_{N}\right|^{2} + \frac{1}{2} h \int_{\Gamma} |\bd{\zeta}^{n + 1}_{N} - \bd{\zeta}^{n + \frac{1}{2}}_{N}|^{2} \\
+ \mu_{e} \int_{\Omega_{b}} |\bd{D}(\bd{\eta}^{n + 1}_{N} - \bd{\eta}^{n}_{N})|^{2} + \frac{1}{2} \lambda_{e} \int_{\Omega_{b}} |\nabla \cdot (\bd{\eta}^{n + 1}_{N} - \bd{\eta}^{n}_{N})|^{2} + \frac{1}{2} c_{0} \int_{\Omega_{b}} |p^{n + 1}_{N} - p^{n}_{N}|^{2} = E^{n + \frac{1}{2}}_{N}.
\end{multline*}
\end{proposition}

\begin{proof}
The semidiscrete energy estimates follow by combining the semidiscrete energy estimates in the plate subproblem \eqref{energyplate} and the Biot/fluid subproblem \eqref{energyBiot}. Therefore, it remains to prove the uniform geometric estimates, which are uniform in $N$. By the uniform energy estimates, note that there exists some uniform constant $C$, independent of $N$, such that
\begin{equation}\label{esth}
(\Delta t) h \sum_{i = 0}^{N - 1} \|\bd{\zeta}^{i + \frac{1}{2}}_{N}\|^{2}_{H^{2}(\Gamma)} \le C.
\end{equation}
Therefore, using the fact from \eqref{plateweak1} in the plate subproblem that $\bd{\omega}^{n + \frac{1}{2}}_{N} - \bd{\omega}^{n - \frac{1}{2}}_{N} = (\Delta t) \bd{\zeta}^{n + \frac{1}{2}}_{N}$ and the fact from the Biot/fluid subproblem that $\bd{\omega}^{n + 1}_{N} = \bd{\omega}^{n + \frac{1}{2}}_{N}$, we obtain:
\begin{align}\label{H02esth}
\|\bd{\omega}^{n}_{N} - \bd{\omega}_{0}\|_{C^{1}(\Gamma)} \le C\|\bd{\omega}^{n}_{N} - \bd{\omega}_{0}\|_{H^{2}(\Gamma)} &\le C\sum_{i = 0}^{N - 1} \|\bd{\omega}^{i + 1}_{N} - \bd{\omega}^{i}_{N}\|_{H^{2}(\Gamma)} \le C(\Delta t) \sum_{i = 0}^{N - 1} \|\bd{\zeta}^{i + \frac{1}{2}}_{N}\|_{H^{2}(\Gamma)} \nonumber \\
&\le C(\Delta t) N^{1/2} \left(\sum_{i = 0}^{N - 1} \|\bd{\zeta}^{i + \frac{1}{2}}_{N}\|_{H^{2}(\Gamma)}^{2}\right)^{1/2} \le CT^{1/2},
\end{align}
since $N^{1/2}(\Delta t)^{1/2} = T^{1/2}$. Note that the constant in the final estimate \eqref{H02esth} potentially depends on $h$. Then, by taking $T$ sufficiently small, using assumptions on the initial data, and applying Proposition \ref{gammainjective}, we obtain the desired statements about the injectivity and non-degeneracy of the map $\bd{\Phi}^{\omega^{n}_{N}}_{\Gamma}$. 

The statements about uniform geometric non-degeneracy of the Lagrangian map for the Biot medium follow similarly as in the proof of Lemma 7.2 in \cite{FPSIJMPA}, so it remains to prove the corresponding statement for the ALE maps for the fluid domain. For this purpose, we observe from our previous estimate that 
\begin{equation*}
\|\bd{\omega}^{n}_{N} - \bd{\omega}_{0}\|_{H^{2}(\Gamma)} \le CT^{1/2},
\end{equation*}
for all $n = 0, 1, ..., N$ and for all $N$, for a constant $C$ that is independent of $N$ (which may potentially depend on $h$, which is fixed for now). By the definition of the ALE map in \eqref{alef1} and \eqref{alef2}, and regularity properties of elliptic equations (see Section 5 in \cite{Grisvard}), we obtain that
\begin{equation*}
\|\bd{\Phi}^{\omega^{n}_{N}}_{f} - \bd{\Phi}^{\omega_{0}}_{f}\|_{C^{1}(\Omega_{f})} \le C\|\bd{\Phi}^{\omega^{n}_{N}}_{f} - \bd{\Phi}^{\omega_{0}}_{f}\|_{H^{5/2}(\Omega_{f})} \le C\|\bd{\omega}^{n}_{N} - \bd{\omega}_{0}\|_{H^{2}(\Gamma)} \le CT^{1/2},
\end{equation*}
so since the initial data satisfies $\mathcal{J}^{\omega_{0}}_{f} := \text{det}(\nabla \bd{\Phi}^{\omega_{0}}_{f}) > 0$ and $|\nabla \bd{\Phi}^{\omega_{0}}_{f}| \le C$ for some positive constant $C$ on $\overline{\Omega}_{f}$, the desired assertion on the fluid domain geometric quantities in \eqref{uniformgeom2} follows by continuity.  
\end{proof}

\begin{remark}[A remark about dependence on the plate thickness $h > 0$]
Due to the estimate in \eqref{esth}, the constants $C$ in the preceding proof all implicitly depend on the plate thickness $h > 0$. At first, this seems to imply that the time of existence of a solution depends on $h$, even though it is explicitly stated in Theorem \ref{mainthm} that the time of existence is independent of $h$. However, we will see that this dependence on $h$ is a technicality in the scheme, that vanishes in the limit as $N \to \infty$ when we obtain the continuous-in-time problem. In particular, the difference of the discrete plate displacement $\bd{\omega}^{n}_{N}$ is determined by the plate velocities $\bd{\zeta}^{n + \frac{1}{2}}_{N}$ obtained in the plate subproblem, which may be different from the plate velocities $\bd{\zeta}^{n + 1}_{N}$ obtained in the Biot/fluid subproblem. However, in the limit as $N \to \infty$, these plate velocities coincide, and therefore, we can obtain estimates on the plate velocity $\bd{\zeta}$ directly from the trace of the regularized Biot velocity $\bd{\xi}^{\delta}$, and the estimates on the Biot velocity are \textit{independent of $h$}, since they arise directly from the energy of the Biot medium, not from the energy of the plate (which depends on $h$). Thus, we can obtain estimates in the spirit of Proposition \ref{gammainjective} which are independent of the plate thickness $h > 0$ in the continuous limit as $N \to \infty$, which gives a (uniform) time of existence that is independent of $h$. 
\end{remark}

\medskip

\noindent \underline{\textit{Approximate solutions.}} We have approximate solutions $\left(\bd{u}^{n + \frac{i}{2}}_{N}, p^{n + \frac{i}{2}}_{N}, \bd{\omega}^{n + \frac{i}{2}}_{N}, \bd{\zeta}^{n + \frac{i}{2}}_{N}, \bd{\eta}^{n + \frac{i}{2}}_{N}, \bd{\xi}^{n + \frac{i}{2}}_{N}\right)$ on each time step $[n\Delta t, (n + 1)\Delta t]$, and we will now use piecewise constant extensions and (continuous-in-time) linear interpolations to combine these approximate solutions on each individual time step into an approximate solution defined on the full time interval $[0, T]$, for each time discretization parameter $N$. We define an approximate fluid velocity $\bd{u}_{N}$, Biot displacement $\bd{\eta}_{N}$, Biot pore pressure $p_{N}$, plate displacement $\bd{\omega}_{N}$, and plate velocity $\bd{\zeta}_{N}$ by taking the value at the right endpoint of each subinterval $(n\Delta t, (n + 1)\Delta t]$, namely:
\begin{equation}\label{approxdef}
(\bd{u}_{N}(t), \bd{\eta}_{N}(t), p_{N}(t), \bd{\omega}_{N}(t), \bd{\zeta}_{N}(t)) = (\bd{u}^{n + 1}_{N}, \bd{\eta}^{n + 1}_{N}, p^{n + 1}_{N}, \bd{\omega}^{n + 1}_{N}, \bd{\zeta}^{n + 1}_{N}), \quad \text{ for $n \Delta t < t \le (n + 1)\Delta t$}.
\end{equation}
Because the plate velocity is updated after both the plate subproblem and the Biot/fluid subproblem, giving rise to both $\bd{\zeta}^{n}_{N}$ and $\bd{\zeta}^{n + \frac{1}{2}}_{N}$ at each time step, we define two versions of the approximate plate velocity since both appear in the semidiscrete formulation: (1) the piecewise constant extension $\bd{\zeta}_{N}$ from before defined by
\begin{equation}\label{zetaN}
\bd{\zeta}_{N}(t) = \bd{\zeta}^{n + 1}_{N}, \quad \text{ for } n\Delta t < t \le (n + 1)\Delta t,
\end{equation}
and (2) a piecewise constant extension $\bd{\zeta}^{*}_{N}$ using the value of the plate velocity at the intermediate ``half" steps:
\begin{equation}\label{zetastar}
\bd{\zeta}_{N}^{*}(t) = \bd{\zeta}^{n + \frac{1}{2}}_{N}, \quad \text{ for } n\Delta t < t \le (n + 1)\Delta t.
\end{equation}
Note that these piecewise constant extensions are potentially discontinuous at each $t_{n} = n\Delta t$. Hence, in cases where we need the approximate solutions to be continuous or where we want the approximate solutions to have a weak time derivative, we use the following linear interpolations as approximate solutions:
\begin{equation*}
\overline{\bd{u}}_{N}(t) = \left(\frac{t - n\Delta t}{\Delta t}\right) \bd{u}^{n}_{N} + \left(\frac{(n + 1)\Delta t - t}{\Delta t}\right) \bd{u}^{n + 1}_{N}, \quad \text{ for } n\Delta t \le t \le (n + 1)\Delta t,
\end{equation*}
\begin{equation*}
\bd{\overline{\zeta}}_{N}(t) = \left(\frac{t - n\Delta t}{\Delta t}\right) \bd{\zeta}^{n}_{N} + \left(\frac{(n + 1)\Delta t - t}{\Delta t}\right) \bd{\zeta}^{n + 1}_{N}, \quad \text{ for } n\Delta t \le t \le (n + 1)\Delta t,
\end{equation*}
\begin{equation*}
\overline{p}_{N}(t) = \left(\frac{t - n\Delta t}{\Delta t}\right) p^{n}_{N} + \left(\frac{(n + 1)\Delta t - t}{\Delta t}\right) p^{n + 1}_{N}, \quad \text{ for } n\Delta t \le t \le (n + 1)\Delta t,
\end{equation*}
\begin{equation*}
\overline{\bd{\omega}}_{N}(t) = \left(\frac{t - n\Delta t}{\Delta t}\right) \bd{\omega}^{n}_{N} + \left(\frac{(n + 1)\Delta t - t}{\Delta t}\right) \bd{\omega}^{n + 1}_{N}, \quad \text{ for } n\Delta t \le t \le (n + 1)\Delta t,
\end{equation*}
\begin{equation*}
\overline{\bd{\eta}}_{N}(t) = \left(\frac{t - n\Delta t}{\Delta t}\right) \bd{\eta}^{n}_{N} + \left(\frac{(n + 1)\Delta t - t}{\Delta t}\right) \bd{\eta}^{n + 1}_{N}, \quad \text{ for } n\Delta t \le t \le (n + 1)\Delta t.
\end{equation*}
From the uniform energy estimate in Proposition \ref{geometricN}, we therefore have the following uniform boundedness results for the approximate solutions in the parameter $N$, which also implies weak and weak-star subsequential compactness results for the approximate solutions.

\begin{proposition}\label{uniformenergy}
For each fixed $h > 0$ and $\delta > 0$, we have the following uniform bounds in the time discretization parameter $N$ for the approximate solutions:
\begin{itemize}
\item $\{\bd{u}_{N}\}_{N = 1}^{\infty}$ is uniformly bounded in $L^{\infty}(0, T; L^{2}(\Omega_{f}(t)))$ and $L^{2}(0, T; H^{1}(\Omega_{f}(t)))$. 
\item $\{\bd{\eta}_{N}\}_{N = 1}^{\infty}$ is uniformly bounded in $L^{\infty}(0, T; H^{1}(\Omega_{b}))$ and $\{\overline{\bd{\eta}}_{N}\}_{N = 1}^{\infty}$ is uniformly bounded in $L^{\infty}(0, T; H^{1}(\Omega_{b}))$ and $W^{1, \infty}(0, T; L^{2}(\Omega_{b}))$. 
\item $\{p_{N}\}_{N = 1}^{\infty}$ is uniformly bounded in $L^{\infty}(0, T; L^{2}(\Omega_{b}))$ and $L^{2}(0, T; H^{1}(\Omega_{b}))$. 
\item $\{\bd{\omega}_{N}\}_{N = 1}^{\infty}$ is uniformly bounded in $L^{\infty}(0, T; H^{2}(\Gamma))$ and $\{\overline{\bd{\omega}}_{N}\}_{N = 1}^{\infty}$ is uniformly bounded in $L^{\infty}(0, T; H^{2}(\Gamma))$, $W^{1, \infty}(0, T; L^{2}(\Gamma))$, and $H^{1}(0, T; H^{2}(\Gamma))$. 
\end{itemize}
Hence, along a subsequence in $N$ (which is still denoted by $N$ for notational simplicity), we have the following convergences to limiting functions $\bd{\eta}$, $p$, and $\bd{\omega}$:
\begin{align*}
\bd{\eta}_{N} \rightharpoonup \bd{\eta}, \qquad &\text{ weakly-star in $L^{\infty}(0, T; L^{2}(\Omega_{b}))$}, \\
p_{N} \rightharpoonup p, \qquad &\text{ weakly-star in $L^{\infty}(0, T; L^{2}(\Omega_{b}))$ and weakly in $L^{2}(0, T; H^{1}(\Omega_{b}))$}, \\
\bd{\omega}_{N} \rightharpoonup \bd{\omega}, \qquad &\text{ weakly-star in $L^{\infty}(0, T; H^{2}(\Gamma))$}. \\
\end{align*}
\end{proposition}

\begin{proof}
These are all immediate from \eqref{energyN}, and we hence comment only on the uniform bound of $p_{N}$ in $L^{2}(0, T; H^{1}(\Omega_{b}))$. From the uniform energy estimates, we have that for a constant $C$ independent of $N$,
\begin{equation*}
\kappa \sum_{n = 0}^{N - 1} \int_{\Omega_{b}} \mathcal{J}^{(\eta^{n}_{N})^{\delta}}_{b} \Big|\nabla^{(\eta^{n}_{N})^{\delta}}_{b} p^{n + 1}_{N}\Big|^{2} \le C.
\end{equation*}
We can then deduce a bound for $\displaystyle \sum_{n = 0}^{N - 1} \int_{\Omega_{b}} |\nabla p|^{2}$ using a similar calculation, as in \eqref{kappabound}, along with the uniform geometric bounds in \eqref{uniformgeom1}.
\end{proof}

In addition, we can obtain the following uniform bound on the trace of the fluid velocities, which will be needed for estimating certain terms in the semidiscrete weak formulation. We note that this is not an immediate consequence of a trace inequality, since the fluid velocities $\bd{u}^{n + 1}_{N}$ are defined on varying spatial domains $\Omega^{n}_{f, N} := \bd{\Phi}^{\omega^{n}_{N}}_{f}(\Omega_{f})$. 

\begin{proposition}\label{traceuniform}
The traces of the approximate fluid velocities $\bd{u}_{N}|_{\Gamma}$ along $\Gamma$ are uniformly bounded in $L^{2}(0, T; L^{2}(\Gamma))$. 
\end{proposition}

\begin{proof}
We will establish this uniform bound by using the ALE map to map the fluid velocities $\bd{u}^{n + 1}_{N}$ defined on the moving approximate fluid domains $\Omega^{n}_{f, N} := \bd{\Phi}^{\omega^{n}_{N}}_{f}(\Omega_{f})$ back to the reference fluid domain $\Omega_{f}$. We define the following transformed fluid velocities on the reference fluid domain $\Omega_{f}$, which map the fluid velocities $\bd{u}^{n + 1}_{N}$ back to the reference domain via the ALE map:
\begin{equation*}
\hat{\bd{u}}^{n + 1}_{N}(t, \hat{x}, \hat{y}) = \bd{u}^{n + 1}_{N}\Big(t, \bd{\Phi}^{\omega^{n}_{N}}_{f}(\hat{x}, \hat{y})\Big).
\end{equation*}
We claim that the function 
\begin{equation}\label{hatu}
\hat{\bd{u}}_{N} = \sum_{n = 0}^{N - 1} \hat{\bd{u}}^{n + 1}_{N} \cdot 1_{[n\Delta t, (n + 1)\Delta t)}(t)
\end{equation}
is uniformly bounded in $L^{2}(0, T; H^{1}(\Omega_{f}))$, where we emphasize that $\Omega_{f}$ is now the fixed reference fluid domain, as the desired result on the traces would follow immediately from the trace inequality applied to the reference domain. By the uniform bounds on the Jacobian of the fluid domain ALE map in \eqref{uniformgeom2}, it is clear that $\hat{\bd{u}}_{N}$ is uniformly bounded in $L^{2}(0, T; L^{2}(\Omega_{f}))$ given the uniform bound of $\bd{u}_{N}$ in $L^{2}(0, T; L^{2}(\Omega_{f}(t)))$. By \eqref{omegadiff}, we can further estimate that
\begin{equation*}
\int_{0}^{T} \int_{\Omega_{f}} |\nabla \hat{\bd{u}}_{N}|^{2} \le C \int_{0}^{T} \int_{\Omega_{f}(t)} |\nabla \bd{u}_{N}|^{2} |\nabla \bd{\Phi}^{\omega}_{f}|^{2} \le C \int_{0}^{T} \int_{\Omega_{f}(t)} |\nabla \bd{u}_{N}|^{2},
\end{equation*}
for a constant $C$ that is independent of $N$, due to the uniform fluid domain bounds in \eqref{uniformgeom2}.
\end{proof}

Because we only obtain weak and weak-star convergences from these uniform boundedness results, we need compactness arguments to obtain strong convergence of approximate solutions in appropriate function spaces, in order to pass to the limit in the nonlinear terms in the approximate weak formulations. This will be the goal of the next subsections.

\subsection{Compactness arguments as $N \to \infty$}\label{compactN}

\noindent \textbf{Compactness for the approximate Biot/plate displacements.} Using the classical Arzela-Ascoli compactness results and the uniform bounds of $\overline{\bd{\eta}}_{N}$ in the function spaces $L^{\infty}(0, T; H^{1}(\Omega_{b}))$ and $W^{1, \infty}(0, T; L^{2}(\Omega_{b}))$ and of $\bd{\overline{\omega}}_{N}$ in $L^{\infty}(0, T; H^{2}(\Gamma))$ and $W^{1, \infty}(0, T; L^{2}(\Gamma))$ from Proposition \ref{uniformenergy}, we have the following strong convergence:

\begin{proposition}\label{omegaconv}
Along a subsequence, $\overline{\bd{\eta}}_{N} \to \bd{\eta}$ in $C(0, T; L^{2}(\Omega_{b}))$ and $\overline{\bd{\omega}}_{N} \to \bd{\omega}$ in $C(0, T; H^{s}(\Gamma))$ for any $0 < s < 2$.
\end{proposition}

\begin{proof}
The convergence of $\bd{\eta}_{N}$ in $C(0, T; L^{2}(\Omega_{b}))$ and $\overline{\bd{\omega}}_{N}$ in $C(0, T; H^{s}(\Gamma))$ follows from Arzela-Ascoli, once we use interpolation to deduce a uniform bound of $\overline{\bd{\omega}}_{N}$ in $C^{\alpha}(0, T; H^{s}(\Gamma))$ for each $0 < s < 2$, for some $\alpha \in (0, 1)$ depending on $s \in (0, 2)$. From the uniform dissipation estimates in \eqref{numerdiss} and the uniform bounds in \eqref{energyN},
\begin{align*}
(\Delta t) \sum_{n = 0}^{N - 1} &\Big(\|\bd{\eta}^{n + 1}_{N} - \bd{\eta}^{n}_{N}\|_{L^{2}(\Omega_{b})}^{2} + \|\nabla (\bd{\eta}^{n + 1}_{N} - \bd{\eta}^{n}_{N})\|_{L^{2}(\Omega_{b})}^{2}\Big) \\
&\le (\Delta t)^{3} \sum_{n = 0}^{N - 1} \|\bd{\xi}^{n + 1}_{N}\|_{L^{2}(\Omega_{b})}^{2} + (\Delta t) \sum_{n = 0}^{N - 1} \|\nabla (\bd{\eta}^{n + 1}_{N} - \bd{\eta}^{n}_{N})\|_{L^{2}(\Omega_{b})}^{2} \le C\Delta t(1 + (\Delta t)).
\end{align*}
Hence, $\|\overline{\bd{\eta}}_{N} - \bd{\eta}_{N}\|_{L^{2}(0, T; H^{1}(\Omega_{b}))} \to 0$ (and by a similar argument, $\|\overline{\bd{\omega}}_{N} - \bd{\omega}_{N}\|_{L^{2}(0, T; H^{2}(\Gamma))} \to 0$), as $N \to \infty$, so the strong limits of $\overline{\bd{\eta}}_{N}$ and $\overline{\bd{\omega}}_{N}$ in this proposition coincide with the weak and weak-star limits of $\bd{\eta}_{N}$ and $\bd{\omega}_{N}$ in Proposition \ref{uniformenergy}.
\end{proof}

\medskip

\noindent \textbf{Compactness for the Biot pore pressures.} Next, we pass to the limit in the approximate pore pressure $p_{N}$ using a compactness criterion for piecewise constant functions, given by Dreher and Jungel (see Theorem 1 in \cite{DreherJungel}). In particular, we have the following compactness result:

\begin{proposition}
Along a subsequence, $p_{N} \to p$ in $L^{2}(0, T; L^{2}(\Omega_{b}))$. 
\end{proposition}

\begin{proof}
Using the chain of embeddings:
\begin{equation*}
H^{1}(\Omega_{b}) \subset \subset L^{2}(\Omega_{b}) \subset V_{p}'
\end{equation*}
where $V_{p}'$ is the dual of the pore pressure function space $V_{p}$ defined in \eqref{Vp}. Since we already have the uniform bound that $\|p_{N}\|_{L^{2}(0, T; H^{1}(\Omega_{b}))} \le C$ independently of $N$, to establish the desired compactness result via the Dreher-Jungel compactness criterion, it suffices to show that
\begin{equation*}
\left\|\frac{\tau_{\Delta t} p_{N} - p_{N}}{\Delta t}\right\|_{L^{1}(\Delta t, T; V_{p}')} \le C
\end{equation*}
uniformly in the time discretization parameter $N$, where $\tau_{\Delta t}f(t, \cdot) = f(t - \Delta t, \cdot)$ for $t \in [\Delta t, T]$ is a time shift by $\Delta t$. In the semidiscrete formulation \eqref{weakBiot1}, we can take the test functions $\bd{\psi}$, $\bd{v}$, and $\varphi$ to all be zero, which leaves the following semidiscrete weak formulation for the approximate Biot pore pressure functions, which holds for all pore pressure test functions $r \in V_{p}$:
\begin{multline*}
c_{0} \int_{\Omega_{b}} \frac{p^{n + 1}_{N} - p^{n}_{N}}{\Delta t} r - \alpha \int_{\Omega_{b}} \mathcal{J}^{(\eta^{n}_{N})^{\delta}}_{b} (\dot{\bd{\eta}}^{n + 1}_{N})^{\delta} \cdot \nabla^{(\eta^{n}_{N})^{\delta}}_{b} r - \alpha \int_{\Gamma} \left(\left(\dot{\bd{\eta}}^{n + 1}_{N}\right)^{\delta} \cdot \bd{n}^{\omega^{n}_{N}}\right) r \\
+ \kappa \int_{\Omega_{b}} \mathcal{J}^{(\eta^{n}_{N})^{\delta}}_{b} \nabla^{(\eta^{n}_{N})^{\delta}}_{b} p^{n + 1}_{N} \cdot \nabla^{(\eta^{n}_{N})^{\delta}}_{b} r - \int_{\Gamma} \left[\left(\bd{u}^{n + 1}_{N} - (\dot{\bd{\eta}}^{n + 1}_{N})^{\delta}\right) \cdot \bd{n}^{\omega^{n}_{N}}\right] r = 0.
\end{multline*}
We will hence estimate $\displaystyle \left\|\frac{p^{n + 1}_{N} - p^{n}_{N}}{\Delta t}\right\|_{V_{p}'}$ using the definition of the dual norm in $V_{p}'$ and the uniform bounds on the approximate solutions. In particular, we have the following estimates for any $r \in V_{p}$ with $\|r\|_{H^{1}(\Omega_{b})} \le 1$ on the terms in the semidiscrete weak formulation for the pore pressure:
\begin{itemize}
\item \textbf{Term 1.} First, we recall the definitions of $\mathcal{J}^{(\eta^{n}_{N})^{\delta}}_{b}$ and $\nabla^{(\eta^{n}_{N})^{\delta}}_{b} r$ from \eqref{etanabla} and \eqref{jeta}, and the uniform geometric bounds in \eqref{uniformgeom1}. Since $\|\dot{\bd{\eta}}^{n + 1}_{N}\|_{L^{2}(\Omega_{b})}$ is uniformly bounded independently of $n$ and $N$, by the test function estimate $\|r\|_{H^{1}(\Omega_{b})} \le 1$ and properties of spatial convolution,
\begin{equation*}
\sup_{r \in V_{p}, \|r\|_{V_{p}} \le 1} \left|\alpha \int_{\Omega_{b}} \mathcal{J}^{(\eta^{n}_{N})^{\delta}}_{b} (\dot{\bd{\eta}}^{n + 1}_{N})^{\delta} \cdot \nabla^{(\eta^{n}_{N})^{\delta}}_{b} r \right| \le C, \qquad \text{ independently of $n$ and $N$}.
\end{equation*}
\item \textbf{Term 2.} Next, we recall the definition of the rescaled normal vector from \eqref{ntomega}. Note that $\|\bd{n}^{\omega^{n}_{N}}\| \le C$ for a constant $C$ independent of $n$ and $N$, by the pointwise bound of $\partial_{z}\omega^{n}_{N}$ in $C(0, T; C(\Gamma))$, arising from the uniform bound of $\bd{\omega}_{N}$ in $L^{\infty}(0, T; H^{2}(\Gamma))$ and Sobolev embedding. Since $\|r\|_{H^{1}(\Omega_{b})} \le 1$, and $\|\dot{\bd{\eta}}^{n + 1}_{N}\|_{L^{2}(\Omega_{b})} \le C$ independently of $n$ and $N$, we have the uniform bound (independent of $n$ and $N$):
\begin{equation*}
\sup_{r \in V_{p}, \|r\|_{V_{p}} \le 1} \left|\alpha \int_{\Gamma} \left(\left(\dot{\bd{\eta}}^{n + 1}_{N}\right)^{\delta} \cdot \bd{n}^{\omega^{n}_{N}}\right) r\right| \le C,
\end{equation*}
which we note depends potentially on $h$ but is independent of $N$ due to the uniform bound of $\bd{\omega}_{N}$ in $L^{\infty}(0, T; H^{2}(\Gamma))$ (uniform in $N$, depending on $h$).  
\item \textbf{Term 3.} Similarly to Term 1, we obtain the uniform estimate using \eqref{uniformgeom2}:
\begin{equation*}
\sup_{r \in V_{p}, \|r\|_{V_{p}} \le 1} \left|\kappa \int_{\Omega_{b}} \mathcal{J}^{(\eta^{n}_{N})^{\delta}}_{b} \nabla^{(\eta^{n}_{N})^{\delta}}_{b} p^{n + 1}_{N} \cdot \nabla^{(\eta^{n}_{N})^{\delta}}_{b} r\right| \le C\|p^{n + 1}_{N}\|_{H^{1}(\Omega_{b})}, \ \ \text{independently of $n$ and $N$}.
\end{equation*}
\item \textbf{Term 4.} Finally, for the last term, we recall from Proposition \ref{traceuniform} that we have a uniform estimate (independent of $N$) on $\|\bd{u}_{N}\|_{L^{2}(0, T; L^{2}(\Gamma))}$. So similarly to Term 2, we estimate (independently of $n$ and $N$) that:
\begin{equation*}
\sup_{r \in V_{p}, \|r\|_{V_{p}} \le 1} \left|\int_{\Gamma} \left[\left(\bd{u}^{n + 1}_{N} - (\dot{\bd{\eta}}^{n + 1}_{N})^{\delta}\right) \cdot \bd{n}^{\omega^{n}_{N}}\right] r\right| \le C\left(1 + \|\bd{u}^{n + 1}_{N}\|_{L^{2}(\Gamma)}\right),
\end{equation*}
where we recall the definition of the rescaled normal vector \eqref{ntomega}. 
\end{itemize}
Hence, combining these estimates, we have that
\begin{align*}
\left\|\frac{\tau_{\Delta t} p_{N} - p_{N}}{\Delta t}\right\|_{L^{1}(\Delta t, T; V_{p}')} &= (\Delta t) \sum_{n = 0}^{N - 1} \left\|\frac{p^{n + 1}_{N} - p^{n}_{N}}{\Delta t}\right\|_{V_{p}'} \le C(\Delta t) \sum_{n = 0}^{N - 1} \left(1 + \|\bd{u}^{n + 1}_{N}\|_{L^{2}(\Gamma)} + \|p^{n + 1}_{N}\|_{H^{1}(\Omega_{b})} \right) \\
&\le 2C(\Delta t) \sum_{n = 0}^{N - 1} \left(1 + \|\bd{u}^{n + 1}_{N}\|^{2}_{L^{2}(\Gamma)} + \|p^{n + 1}_{N}\|^{2}_{H^{1}(\Omega_{b})}\right) \\
&\le 2C\left(1 + \|\bd{u}_{N}\|^{2}_{L^{2}(0, T; L^{2}(\Gamma))} + \|p_{N}\|_{L^{2}(0, T; H^{1}(\Omega_{b})}^{2}\right) \le C,
\end{align*}
by the definition of the approximate solutions $\bd{u}_{N}$ and $p_{N}$, the uniform bounds of $\bd{u}_{N}$ in $L^{2}(0, T; L^{2}(\Gamma))$ and $p_{N}$ in $L^{2}(0, T; H^{1}(\Omega_{b}))$, independently of $N$, see Proposition \ref{uniformenergy} and Proposition \ref{traceuniform}.
\end{proof}

\medskip

\noindent \textbf{Compactness for the Biot/plate velocities.} Next, we use the Dreher-Jungel compactness criterion to show a similar convergence result for the approximate Biot and plate velocities. Since the Biot and plate velocities are coupled by the regularized kinematic coupling condition, namely $\bd{\xi}_{N}|_{\Gamma} = \bd{\zeta}_{N}$ by \eqref{weakBiot2} and the definitions of the approximate solutions in \eqref{approxdef} and \eqref{zetaN}, we must consider them as a pair for the following compactness arguments.

\begin{proposition}
Along a subsequence in $N$, and for $0 < s < 1$, 
\begin{equation*}
(\bd{\xi}_{N}, \bd{\zeta}_{N}) \to (\partial_{t}\bd{\eta}, \partial_{t}\bd{\omega}), \qquad \text{ in } L^{2}(0, T; H^{-s}(\Omega_{b}) \times H^{-s}(\Gamma)).
\end{equation*}
\end{proposition}

\begin{proof}
We recall the definition of the space $V_{d}$ from \eqref{Vd}, and we define the following Biot and plate velocity test space:
\begin{equation}\label{Qv}
\mathcal{Q}_{v} = \{(\bd{\psi}, \bd{\varphi}) \in V_{d} \times H^{2}(\Gamma): \bd{\psi}^{\delta}|_{\Gamma} = \bd{\varphi}\},
\end{equation}
endowed with the $H^{1}(\Omega_{b}) \times H^{2}(\Gamma)$ norm, where we emphasize that the regularized kinematic coupling condition is incorporated into the definition of the test space $\mathcal{Q}_{v}$. We can then observe the following chain of embeddings:
\begin{equation*}
L^{2}(\Omega_{b}) \times L^{2}(\Gamma) \subset \subset H^{-s}(\Omega_{b}) \times H^{-s}(\Gamma) \subset \mathcal{Q}_{v}'.
\end{equation*}
To apply the Dreher-Jungel compactness criterion (Theorem 1 in \cite{DreherJungel}) and conclude the desired convergence result, it suffices to establish the following estimate (independently of $N$):
\begin{equation*}
\left\|\frac{\tau_{\Delta t}(\bd{\xi}_{N}, \bd{\zeta}_{N}) - (\bd{\xi}_{N}, \bd{\zeta}_{N})}{\Delta t}\right\|_{L^{1}(\Delta t, T; \mathcal{Q}_{v}')} \le C,
\end{equation*}
since we already have a uniform bound on $\|(\bd{\xi}_{N}, \bd{\zeta}_{N})\|_{L^{2}(\Omega_{b}) \times L^{2}(\Gamma)}$ independently of $N$, as a result of the uniform energy estimates, see Proposition \ref{uniformenergy}.

So it suffices to estimate the following quantity for test functions $(\bd{\psi}, \bd{\varphi}) \in \mathcal{Q}_{v}$ such that $\|\bd{\psi}\|_{H^{1}(\Omega_{b})} \le 1$ and $\|\bd{\varphi}\|_{H^{2}(\Gamma)} \le 1$:
\begin{equation*}
\left|\int_{\Omega_{b}} \frac{\bd{\xi}^{n + 1}_{N} - \bd{\xi}^{n}_{N}}{\Delta t} \cdot \bd{\psi} + \int_{\Gamma} \left(\frac{\bd{\zeta}^{n + 1}_{N} - \bd{\zeta}^{n}_{N}}{\Delta t}\right) \cdot \bd{\varphi}\right|
\end{equation*}
To do this, we set the test functions $\bd{v}$ and $r$ to be zero in the semidiscrete weak formulation in \eqref{weakBiot1}. Hence, we have that the approximate Biot and plate velocities satisfy the following identity for all test functions $(\bd{\psi}, \varphi) \in \mathcal{Q}_{v}$:
\begin{small}
\begin{multline*}
\rho_{b} \int_{\Omega_{b}} \frac{\bd{\xi}^{n + 1}_{N} - \bd{\xi}^{n}_{N}}{\Delta t} \cdot \bd{\psi} + \int_{\Gamma} \left(\frac{1}{2}\bd{u}^{n + 1}_{N} \cdot \bd{u}^{n}_{N} - p^{n + 1}_{N}\right) \bd{\psi}^{\delta} \cdot \bd{n}^{\omega^{n}_{N}} + \beta \int_{\Gamma} \Big(\mathcal{S}^{\omega^{n}_{N}}_{\Gamma}\Big)^{-1} \left(\left(\dot{\bd{\eta}}^{n + 1}_{N}\right)^{\delta} - \bd{u}^{n + 1}_{N}\right) \cdot \bd{\tau}^{\omega^{n}_{N}} \left(\bd{\psi}^{\delta} \cdot \bd{\tau}^{\omega^{n}_{N}}\right) \\
+ h \int_{\Gamma} \left(\frac{\bd{\zeta}^{n + 1}_{N} - \bd{\zeta}^{n}_{N}}{\Delta t}\right) \cdot \bd{\varphi} + 2\mu_{e} \int_{\Omega_{b}} \bd{D}(\bd{\eta}^{n + 1}_{N}) : \bd{D}(\bd{\psi}) + \lambda_{e} \int_{\Omega_{b}} (\nabla \cdot \bd{\eta}^{n + 1}_{N}) (\nabla \cdot \bd{\psi}) + 2\mu_{v} \int_{\Omega_{b}} \bd{D}(\bd{\xi}^{n + 1}_{N}) : \bd{D}(\bd{\psi}) \\
+ \lambda_{v} \int_{\Omega_{b}} (\nabla \cdot \bd{\xi}^{n + 1}_{N}) (\nabla \cdot \bd{\psi}) - \alpha \int_{\Omega_{b}} \mathcal{J}^{(\eta^{n}_{N})^{\delta}}_{b} p^{n + 1}_{N} \nabla^{(\eta^{n}_{N})^{\delta}}_{b} \cdot \bd{\psi}^{\delta} + h \int_{\Gamma} \Delta \bd{\zeta}^{n + \frac{1}{2}}_{N} \cdot \Delta \bd{\varphi} + h \int_{\Gamma} \Delta \bd{\omega}^{n + \frac{1}{2}}_{N} \cdot \Delta \bd{\varphi} = 0.
\end{multline*}
\end{small}
We estimate the terms for test functions $\|\bd{\psi}\|_{H^{1}(\Omega_{b})} \le 1$ and $\|\bd{\varphi}\|_{H^{2}(\Gamma)} \le 1$ as follows:
\begin{itemize}
\item Recall the definition of the rescaled normal vector in \eqref{ntomega}. Since $\|\bd{\omega}^{n}_{N}\|_{H^{2}(\Gamma)}$ is uniformly bounded in $n$ and $N$, we have that $\|\bd{n}^{\omega^{n}}\|_{L^{\infty}(\Gamma)}$ is uniformly bounded in $n$ and $N$. Hence,
\begin{equation*}
\left|\int_{\Gamma} \left(\frac{1}{2} \bd{u}^{n + 1}_{N} \cdot \bd{u}^{n}_{N} - p^{n + 1}_{N}\right) \bd{\psi}^{\delta} \cdot \bd{n}^{\omega^{n}_{N}}\right| \le C\left(\|\bd{u}^{n}_{N}\|^{2}_{L^{2}(\Gamma)} + \|\bd{u}^{n + 1}_{N}\|_{L^{2}(\Gamma)}^{2} + \|p^{n + 1}_{N}\|_{L^{2}(\Gamma)}\right).
\end{equation*}
\item Recalling the definition in \eqref{ntomega}, we have that $\|\bd{\tau}^{\omega^{n}_{N}}\|_{L^{\infty}(\Gamma)} \le C$ independently of $n$ and $N$. We have the estimate $(\mathcal{S}^{\omega^{n}_{N}}_{\Gamma})^{-1} \le C$ independently of $n$ and $N$, for the arc length element defined in \eqref{ntdef2} as a direct consequence of the uniform geometric estimate \eqref{tnorm}. By the regularized kinematic coupling condition which is enforced in the scheme in \eqref{weakBiot2}, $(\dot{\bd{\eta}}^{n + 1}_{N})^{\delta}|_{\Gamma} = \bd{\zeta}^{n + 1}_{N}$ where we have a uniform bound on $\bd{\zeta}^{n + 1}_{N}$ in $L^{2}(\Gamma)$. Therefore, we estimate that
\begin{equation*}
\left|\beta \int_{\Gamma} \Big(\mathcal{S}^{\omega^{n}_{N}}_{\Gamma}\Big)^{-1} \left(\left(\dot{\bd{\eta}}^{n + 1}_{N}\right)^{\delta} - \bd{u}^{n + 1}_{N}\right) \cdot \bd{\tau}^{\omega^{n}_{N}} \left(\bd{\psi}^{\delta} \cdot \bd{\tau}^{\omega^{n}_{N}}\right)\right| \le C\left(1 + \|\bd{u}^{n + 1}_{N}\|_{L^{2}(\Gamma)}\right).
\end{equation*}
\item From \eqref{etanabla} and \eqref{jeta}, $\mathcal{J}^{(\eta^{n}_{N})^{\delta}}_{b} = \text{det}\left(\bd{I} + \nabla (\bd{\eta}^{n}_{N})^{\delta}\right)$ and
\begin{equation*}
\nabla^{(\eta^{n}_{N})^{\delta}}_{b} \cdot \bd{\psi}^{\delta} = \text{tr}\left(\nabla \bd{\psi}^{\delta} \cdot \left(\bd{I} + \nabla (\bd{\eta}^{n}_{N})^{\delta}\right)^{-1}\right).
\end{equation*}
Because the spatial regularization parameter $\delta$ is fixed, by the regularizing properties of convolution, $\|(\bd{\eta}^{n}_{N})^{\delta}\|_{H^{k}(\Omega_{b})}$ is uniformly bounded independently of $n$ and $N$ for any positive integer $k$. Since $\|p^{n}_{N}\|_{L^{2}(\Omega_{b})}$ is uniformly bounded independently of $n$ and $N$, we obtain using the uniform bounds in \eqref{uniformgeom1}:
\begin{equation*}
\left|\alpha \int_{\Omega_{b}} \mathcal{J}^{(\eta^{n}_{N})^{\delta}}_{b} p^{n + 1}_{N} \nabla^{(\eta^{n}_{N})^{\delta}}_{b} \cdot \bd{\psi}^{\delta}\right| \le C, \qquad \text{ independently of $n$ and $N$}.
\end{equation*}
\item Finally, by the uniform bounds of $\bd{\eta}^{n + 1}_{N}$ in $H^{1}(\Omega_{b})$ and $\bd{\omega}^{n + \frac{1}{2}}_{N}$ in $H^{2}(\Gamma)$, we obtain the following estimate for the remaining terms:
\begin{small}
\begin{multline*}
\left|2\mu_{e} \int_{\Omega_{b}} \bd{D}(\bd{\eta}^{n + 1}_{N}) : \bd{D}(\bd{\psi}) + \lambda_{e} \int_{\Omega_{b}} (\nabla \cdot \bd{\eta}^{n + 1}_{N}) (\nabla \cdot \bd{\psi}) + 2\mu_{v} \int_{\Omega_{b}} \bd{D}(\bd{\xi}^{n + 1}_{N}) : \bd{D}(\bd{\psi}) \right. \\
\left. + \lambda_{v} \int_{\Omega_{b}} (\nabla \cdot \bd{\xi}^{n + 1}_{N}) (\nabla \cdot \bd{\psi}) + h \int_{\Gamma} \Delta \bd{\zeta}^{n + \frac{1}{2}}_{N} \cdot \Delta \bd{\varphi} + h \int_{\Gamma} \Delta \bd{\omega}^{n + \frac{1}{2}}_{N} \cdot \Delta \bd{\varphi}\right| \\
\begin{cases}
\le C\left(1 + \|\bd{\zeta}^{n + \frac{1}{2}}_{N}\|_{H^{2}_{0}(\Gamma)}\right), \ \ \text{ for $\mu_{v}, \lambda_{v} = 0$ (purely poroelastic Biot medium)}, \\
\le C\left(1 + \|\bd{\xi}^{n + 1}_{N}\|_{H^{1}(\Omega_{b})} + \|\bd{\zeta}^{n + \frac{1}{2}}_{N}\|_{H^{2}_{0}(\Gamma)}\right), \ \ \text{ for $\mu_{v}, \lambda_{v} > 0$ (poroviscoelastic Biot medium)}.
\end{cases}
\end{multline*}
\end{small}
\end{itemize}
Hence, we obtain the final estimate that:
\begin{small}
\begin{multline*}
\left\|\frac{(\bd{\xi}^{n + 1}_{N}, \bd{\zeta}^{n + 1}_{N}) - (\bd{\xi}^{n}_{N}, \bd{\zeta}^{n}_{N})}{\Delta t}\right\|_{\mathcal{Q}_{v}'} \\
\begin{cases}
\le C\left(1 + \|\bd{u}^{n}_{N}\|^{2}_{L^{2}(\Gamma)} + \|\bd{u}^{n + 1}_{N}\|_{L^{2}(\Gamma)}^{2} + \|p^{n + 1}_{N}\|_{L^{2}(\Gamma)}^{2} + \|\bd{\zeta}^{n + \frac{1}{2}}_{N}\|_{H^{2}(\Gamma)}^{2}\right), \ \ \text{ for $\mu_{v}, \lambda_{v} = 0$}, \\
\le C\left(1 + \|\bd{u}^{n}_{N}\|^{2}_{L^{2}(\Gamma)} + \|\bd{u}^{n + 1}_{N}\|_{L^{2}(\Gamma)}^{2} + \|p^{n + 1}_{N}\|_{L^{2}(\Gamma)}^{2} + \|\bd{\xi}^{n + 1}_{N}\|^{2}_{H^{1}(\Omega_{b})} + \|\bd{\zeta}^{n + \frac{1}{2}}_{N}\|_{H^{2}(\Gamma)}^{2}\right), \ \ \text{ for $\mu_{v}, \lambda_{v} > 0$}.
\end{cases}
\end{multline*}
\end{small}
Therefore, summing over $n = 0, 1, ..., N - 1$, we obtain the desired estimate:
\begin{small}
\begin{multline*}
\left\|\frac{\tau_{\Delta t}(\bd{\xi}_{N}, \bd{\zeta}_{N}) - (\bd{\xi}_{N}, \bd{\zeta}_{N})}{\Delta t}\right\|_{L^{1}(\Delta t, T; \mathcal{Q}_{v}')} \\
\begin{cases}
\le 2C\left(T + \|\bd{u}_{0}\|_{H^{1}(\Omega)}^{2} + \|\bd{u}_{N}\|_{L^{2}(0, T; H^{1}(\Omega_{f}))}^{2} + \|p_{N}\|^{2}_{L^{2}(0, T; H^{1}(\Omega_{b}))} + \|\bd{\zeta}^{*}_{N}\|^{2}_{L^{2}(0, T; H^{2}_{0}(\Gamma))}\right) \le C, \ \ \text{ for $\mu_{v}, \lambda_{v} = 0$}, \\
\le 2C\left(T + \|\bd{u}_{0}\|_{H^{1}(\Omega)}^{2} + \|\bd{u}_{N}\|_{L^{2}(0, T; H^{1}(\Omega_{f}))}^{2} + \|p_{N}\|^{2}_{L^{2}(0, T; H^{1}(\Omega_{b}))} + \|\bd{\xi}_{N}\|^{2}_{L^{2}(0, T; H^{1}(\Omega_{b}))} + \|\bd{\zeta}_{N}^{*}\|_{L^{2}(0, T; H^{2}(\Gamma))}^{2}\right) \le C, \\
\qquad \qquad \qquad \qquad \qquad \qquad \qquad \qquad \qquad \qquad \qquad \qquad \qquad \qquad \qquad \qquad \qquad \qquad \qquad \text{ for } \mu_{v}, \lambda_{v} > 0,
\end{cases}
\end{multline*}
\end{small}
where we used the uniform trace bound in Proposition \ref{traceuniform} and the uniform dissipation estimates on the approximate solutions in Proposition \ref{uniformenergy}.
\end{proof}

\medskip

\noindent \textbf{Compactness arguments for the approximate fluid velocities.} Recall that the approximate (time-discrete) fluid velocity $\bd{u}^{n + 1}_{N}$ from \eqref{weakBiot1} is defined on a fluid domain $\Omega^{n}_{f, N} := \bd{\Phi}^{\omega^{n}_{N}}_{f}(\Omega_{f})$ determined by the plate displacement $\bd{\omega}^{n}_{N}$ from the previous time step, so that the fluid velocities are defined on fluid domains that are changing with respect to $n$ and $N$. To handle this, we will consider the fluid velocities to be defined on the maximal domain $\Omega := \{\bd{x} \in \R^{2} : |\bd{x}| < 2\}$. Note that by the injectivity of the Lagrangian map $\hat{\bd{\Phi}}^{\omega_{N}}_{\Gamma}$ and the fact that $\hat{\bd{\Phi}}^{\omega_{N}}_{\Gamma}(\hat{\Gamma}) \cap \partial \Omega = \varnothing$ in Proposition \ref{geometricN}, we have that $\Omega$ contains all of the moving time-discrete fluid domains $\Omega^{n}_{f, N}$. Then, we can define the approximate fluid velocities $\bd{u}_{N}$ on the maximal domain $\Omega$ by using an extension by zero, and by the properties of extension by zero (see Theorem 8.2 in \cite{KuanTawri}), we have the following result:

\begin{proposition}\label{fluidcompact}
The fluid velocities $\bd{u}_{N}$ defined on the maximal domain $\Omega$ are uniformly bounded in $L^{\infty}(0, T; L^{2}(\Omega))$ and $L^{2}(0, T; H^{s}(\Omega))$ for $0 < s < 1/2$. 
\end{proposition}

The goal is then to use the uniform bounds and the semidiscrete weak formulation to show that the fluid velocities $\bd{u}_{N}$ are precompact in $L^{2}(0, T; L^{2}(\Omega))$, which is the content of the next proposition. To prove this result, we use an extension of the Aubin-Lions compactness lemma to sequences of functions defined on moving domains, see \cite{AubinLions}. For this purpose, it will be easier to work with fluid velocities defined on the physical domains, and we will rewrite the semidiscrete formulation for the fluid velocity functions defined on the physical domain. To do this, we recall that $\bd{u}^{n + 1}_{N}$ is a function in the following time-discrete solution space:
\begin{equation*}
V^{n}_{N} = \{\bd{u} \in H^{1}(\Omega^{n}_{f, N}) : \nabla \cdot \bd{u} = 0 \text{ on } \Omega^{n}_{f, N} \text{ and } \bd{u} = 0 \text{ on } \partial \Omega^{n}_{f, N} \setminus \Gamma^{n}_{N}\},
\end{equation*}
where 
\begin{equation*}
\Omega^{n}_{N} := \bd{\Phi}^{\omega^{n}_{N}}_{f}(\Omega_{f}), \qquad \Gamma^{n}_{N} = \bd{\Phi}^{\omega^{n}_{N}}_{\Gamma}(\Gamma),
\end{equation*}
and we define the corresponding test space $Q^{n}_{N}$ to be a more regular subspace of the solution space:
\begin{equation}\label{Qmoreregular}
Q^{n}_{N} = V^{n}_{N} \cap H^{3}(\Omega^{n}_{N}).
\end{equation}
We will also define the following function spaces on the entire maximal domain:
\begin{equation*}
V = H^{s}(\Omega) \text{ for some } 0 < s < 1/2, \qquad H = L^{2}(\Omega),
\end{equation*}
where by properties of extension by zero, the extension by zero of any function in $V^{n}_{N}$ is in the common solution space $V$ defined on the fixed maximal domain $\Omega$. We recall that because the Biot-fluid subproblem in the splitting scheme uses the previous plate displacement $\bd{\omega}^{n}_{N}$ to compute $\bd{u}^{n + 1}_{N}$, we have that $\bd{u}^{n + 1}_{N}$ is defined on the (previous) fluid domain $\Omega^{n}_{f, N}$. To write the semidiscrete formulation for the approximate fluid velocity, it is hence helpful to introduce the notation:
\begin{equation}\label{tildeuN}
\tilde{\bd{u}}^{n}_{N} = \bd{u}^{n}_{N} \circ \bd{\Phi}^{\omega^{n - 1}_{N}}_{f} \circ (\bd{\Phi}^{\omega^{n}_{N}}_{f})^{-1},
\end{equation}
which transforms the fluid velocity $\bd{u}^{n}_{N}$ on $\Omega^{n - 1}_{f, N}$ to a fluid velocity $\tilde{\bd{u}}^{n}_{N}$ defined on the domain $\Omega^{n}_{f, N}$. We then obtain the following semidiscrete formulation for the approximate fluid velocities defined on the physical fluid domain, which holds for all fluid velocity test functions $\bd{q} \in V^{n}_{N}$:
\begin{multline}\label{fluidnN}
\int_{\Omega^{n}_{f, N}} \frac{\bd{u}^{n + 1}_{N} - \tilde{\bd{u}}^{n}_{N}}{\Delta t} \cdot \bd{v} + 2\nu \int_{\Omega^{n}_{f, N}} \bd{D}(\bd{u}^{n + 1}_{N}) : \bd{D}(\bd{v}) -  \int_{\Gamma^{n}_{N}} \left(\frac{1}{2} \bd{u}^{n + 1}_{N} \cdot \tilde{\bd{u}}^{n}_{N} - p^{n + 1}_{N}\right) (\bd{v} \cdot \bd{n}) \\
+ \frac{1}{2} \int_{\Omega^{n}_{f, N}} \left[\left(\left(\tilde{\bd{u}}^{n}_{N} - \tilde{\bd{w}}^{n + 1}_{N}\right) \cdot \nabla \bd{u}^{n + 1}_{N}\right) \cdot \bd{v} - \left(\left(\tilde{\bd{u}}^{n}_{N} - \tilde{\bd{w}}^{n + 1}_{N}\right) \cdot \nabla \bd{v}\right) \cdot \bd{u}^{n + 1}_{N}\right] \\
+ \frac{1}{2} \int_{\Omega^{n}_{f, N}} \frac{1}{\Delta t} \left(\frac{\mathcal{J}^{\omega^{n + 1}_{N}}_{f}}{\mathcal{J}^{\omega^{n}_{N}}_{f}} - 1\right) \bd{u}^{n + 1}_{N} \cdot \bd{v} + \frac{1}{2} \int_{\Gamma^{n}_{N}} \left(\bd{u}^{n + 1}_{N} - \bd{\zeta}^{n + 1}_{N}\right) \cdot \bd{n} (\tilde{\bd{u}}^{n}_{N} \cdot \bd{v}) - \beta \int_{\Gamma^{n}_{N}} \left(\bd{\zeta}^{n + 1}_{N} - \bd{u}^{n + 1}_{N}\right) \cdot \bd{\tau} (\bd{v} \cdot \bd{\tau}) = 0,
\end{multline}
where $\bd{w}^{n + 1}_{N}$ is the discretized (fluid domain) ALE velocity defined on the fixed fluid domain $\Omega_{f}$:
\begin{equation}\label{wale}
\bd{w}^{n + 1}_{N}(x, y) = \frac{1}{\Delta t}\Big(\bd{\Phi}^{\omega^{n + 1}_{N}}_{f} - \bd{\Phi}^{\omega^{n}_{N}}_{f}\Big)(x, y) \quad \text{ for } (x, y) \in \Omega_{f},
\end{equation}
and on the physical fluid domain:
\begin{equation}\label{tildew}
\tilde{\bd{w}}^{n + 1}_{N} = \bd{w}^{n + 1}_{N} \circ \Big(\bd{\Phi}^{\omega^{n}_{N}}_{f}\Big)^{-1}.
\end{equation}
We can use this semidiscrete weak formulation and the Aubin-Lions compactness criterion for functions defined on moving domains \cite{AubinLions} to obtain the following strong convergence result for the approximate fluid velocities $\bd{u}_{N}$ extended by zero to the maximal domain $\Omega$. 

\begin{proposition}\label{compactfluid}
Along a subsequence in $N$, $\bd{u}_{N} \to \bd{u}$ in $L^{2}(0, T; L^{2}(\Omega))$. 
\end{proposition}

This proposition will follow from an Aubin-Lions compactness result \cite{AubinLions} for functions defined on moving domains as in the proof of Proposition 8.3 in \cite{FPSIJMPA}, once we verify the following essential condition, referred to in the proof of Proposition 8.3 in \cite{FPSIJMPA} as \textbf{Generalized Property B}.

\begin{lemma}\label{generalizedB}
Let $P^{n}_{N}: H \to \overline{Q^{n}_{N}}^{H}$ denote the orthogonal projection operator, where $\overline{Q^{n}_{N}}^{H}$ denotes the topological closure of $Q^{n}_{N} \subset H$ with respect to the topology of strong convergence in $H$. Then, the following estimate holds:
\begin{equation*}
\left\|P^{n}_{N} \frac{\bd{u}^{n + 1}_{N} - \bd{u}^{n}_{N}}{\Delta t} \right\|_{(Q^{n}_{N})'} \le C\left(a^{n}_{N} + \|\bd{u}^{n}_{N}\|_{V^{n}_{N}} + \|\bd{u}^{n + 1}_{N}\|_{V^{n + 1}_{N}}\right)^{3/2},
\end{equation*}
for nonnegative constants $\{a^{n}_{N}\}$ satisfying $(\Delta t) \displaystyle \sum_{n = 0}^{N - 1} |a^{n}_{N}|^{2} \le c$, where the constants $c$ and $C$ are independent of $n$ and $N$. 
\end{lemma}

The strategy to establish Generalized Property B in Lemma \ref{generalizedB} is to use $\tilde{\bd{u}}_{N}$ defined in \eqref{tildeuN} as an intermediate step, by estimating the following quantity for test functions $\|\bd{q}\|_{Q^{n}_{N}} \le 1$:
\begin{equation}\label{triangleineq}
\left|\int_{\Omega} \left(\frac{\bd{u}^{n + 1}_{N} - \bd{u}^{n}_{N}}{\Delta t}\right) \cdot \bd{q}\right| \le \left|\int_{\Omega} \left(\frac{\bd{u}^{n + 1}_{N} - \tilde{\bd{u}}^{n}_{N}}{\Delta t}\right) \cdot \bd{q}\right| + \left|\int_{\Omega} \left(\frac{\tilde{\bd{u}}^{n}_{N} - \bd{u}^{n}_{N}}{\Delta t}\right) \cdot \bd{q}\right|.
\end{equation}
The first term can be estimated using the semidiscrete formulation \eqref{weakBiot1}, which will be the content of the proof of Lemma \ref{generalizedB} below. We will estimate the second term using the following lemma.

\begin{lemma}\label{tildecomp}
There exists a constant $C$ that is independent of $N$ such that for the test space $Q^{n}_{N}$ defined in \eqref{Qmoreregular}:
\begin{equation*}
\sup_{\|\bd{q}\|_{Q^{n}_{N}} \le 1} \left|\int_{\Omega^{n}_{f, N}} \left(\frac{\tilde{\bd{u}}^{n}_{N} - \bd{u}^{n}_{N}}{\Delta t}\right) \cdot \bd{q}\right| \le C\Big(a^{n}_{N} + \|\bd{u}^{n}_{N}\|_{V^{n}_{N}}\Big)^{3/2},
\end{equation*}
for nonnegative constants $\{a^{n}_{N}\}$ satisfying $\displaystyle (\Delta t) \sum_{n = 0}^{N - 1} |a^{n}_{N}|^{2} \le c$ for constants $c$ and $C$ that are independent of $n$ and $N$. 
\end{lemma}

\begin{proof}
Let us define $\hat{\bd{u}}^{n}_{N} := \bd{u}^{n}_{N} \circ \bd{\Phi}^{\omega^{n}_{N}}_{f}$. Then, consider an arbitrary $\bd{q}$ such that $\|\bd{q}\|_{Q^{n}_{N}} \le 1$ with $Q^{n}_{N}$ defined in \eqref{Qmoreregular}, so that $\|\bd{q}\|_{L^{\infty}(\Omega^{n}_{f, N})} \le C$. Recalling the definition of $\tilde{\bd{u}}^{n}_{N}$ from \eqref{tildeuN}, we estimate:
\begin{align*}
\left|\int_{\Omega^{n}_{f, N}} \Big(\tilde{\bd{u}}^{n}_{N} - \bd{u}^{n}_{N}\Big) \cdot \bd{q}\right| &\le C \int_{\Omega_{f}} \mathcal{J}^{\omega^{n}_{N}}_{f} \Big|\hat{\bd{u}}^{n}_{N} - \hat{\bd{u}}^{n}_{N} \circ \Big((\bd{\Phi}^{\omega^{n}_{N}}_{f})^{-1} \circ \bd{\Phi}^{\omega^{n - 1}_{N}}_{f}\Big)\Big| \\
&\le C \left (\int_{A_{1}} + \int_{A_{2}}\right ) \mathcal{J}^{\omega^{n}_{N}}_{f} \Big|\hat{\bd{u}}^{n}_{N} - \hat{\bd{u}}^{n}_{N} \circ \Big((\bd{\Phi}^{\omega^{n}_{N}}_{f})^{-1} \circ \bd{\Phi}^{\omega^{n - 1}_{N}}_{f}\Big)\Big|, \\
&\ \ \text{ for } A_{1} := (\bd{\Phi}^{\omega^{n}_{N}}_{f})^{-1} (\Omega^{n - 1}_{f, N} \cap \Omega^{n}_{f, N}) \text{ and } A_{2} := \Omega_{f} \setminus A_{1}.
\end{align*}

\medskip

\noindent \textbf{Estimate of integral over $A_{1}$.} We estimate the integral over $A_{1}$, where we recall that $\bd{u}^{n}_{N}$ is defined on $\Omega^{n - 1}_{N, f}$ (and extended by zero outside). By the definition of $A_{1}$, both $(x, y)$ and $\Big((\bd{\Phi}^{\omega^{n}_{N}}_{f})^{-1} \circ \bd{\Phi}^{\omega^{n - 1}_{N}}_{f}\Big)(x, y)$ are in $\Omega_{f}$ if $(x, y) \in A_{1}$. So let us define the polar coordinates of $(x, y)$ and $\Big((\bd{\Phi}^{\omega^{n}_{N}}_{f})^{-1} \circ \bd{\Phi}^{\omega^{n - 1}_{N}}_{f}\Big)(x, y)$ to be the $(r_{1}, \theta_{1})$ and $(r_{2}, \theta_{2})$ for some $\theta_{1}, \theta_{2} \in (-2\pi, 2\pi)$, see the Appendix and the discussion preceding Proposition \ref{gammaprop}. We then define as in the Appendix, the curves in polar coordinates for $0 \le s \le 1$:
\begin{equation*}
\gamma_{1}: s \to ((1 - s)r_{1} + sr_{2}, \theta_{1}), \qquad \gamma_{2}: s \to (r_{2}, (1 - s)\theta_{1} + s\theta_{2}).
\end{equation*}
See Figure \ref{gammafig} in the Appdendix. Thus, for $(x, y) \in A_{1} \subset \Omega_{f} := \{\bd{x} \in \R^{2} : 1 < |\bd{x}| < 2\}$, we can use the fundamental theorem of line integrals to estimate:
\begin{multline}\label{A1bound1}
\Big|\hat{\bd{u}}^{n}_{N} - \hat{\bd{u}}^{n}_{N} \circ \Big((\bd{\Phi}^{\omega^{n}_{N}}_{f})^{-1} \circ \bd{\Phi}^{\omega^{n - 1}_{N}}_{f}\Big)\Big| \le \int_{\gamma_{1}(x, y)} \nabla \hat{\bd{u}}^{n}_{N} \cdot d\bd{r} + \int_{\gamma_{2}(x, y)} \nabla \hat{\bd{u}}^{n}_{N} \cdot d\bd{r} \\
\le \Big|(x, y) - \Big((\bd{\Phi}^{\omega^{n}_{N}}_{f})^{-1} \circ \bd{\Phi}^{\omega^{n - 1}_{N}}_{f}\Big)(x, y)\Big| \left(\int_{1}^{2} |\nabla \hat{\bd{u}}^{n}_{N}(r, \theta_{1})| dr + \int_{|\bd{x}| = r_{2}} |\nabla \hat{\bd{u}}^{n}_{N}| ds \right).
\end{multline}
Note that
\begin{equation*}
|\bd{\Phi}^{\omega^{n}_{N}}_{f}(x, y) - \bd{\Phi}^{\omega^{n - 1}_{N}}_{f}(x, y)| \le \|\bd{\Phi}^{\omega^{n}_{N}}_{f}(\cdot) - \bd{\Phi}^{\omega^{n - 1}_{N}}_{f}(\cdot)\|_{H^{3/2}(\Omega_{f})} \le C(\Delta t) \|\bd{\zeta}^{n - \frac{1}{2}}_{N}\|_{H^{1}(\Gamma)},
\end{equation*}
and hence, by Proposition \ref{lengthtransform} and the uniform bounds \eqref{uniformgeom2}:
\begin{equation}\label{A1bound2}
\Big|(x, y) - \Big((\bd{\Phi}^{\omega^{n}_{N}}_{f})^{-1} \circ \bd{\Phi}^{\omega^{n - 1}_{N}}_{f}\Big)(x, y)\Big| \le \|(\nabla \bd{\Phi}^{\omega^{n}_{N}}_{f})^{-1}\| _{L^{\infty}(\Omega_{f})} \cdot |\bd{\Phi}^{\omega^{n}_{N}}_{f}(x, y) - \bd{\Phi}^{\omega^{n - 1}_{N}}_{f}(x, y)| \le C(\Delta t) \|\bd{\zeta}^{n - \frac{1}{2}}_{N}\|_{H^{1}(\Gamma)}.
\end{equation}
We also compute by a change of variables and the uniform geometric bounds in \eqref{uniformgeom2}:
\begin{align}\label{a1hardintegral}
\int_{A_{1}} \int_{|\bd{x}| = r_{2}} &|\nabla \hat{\bd{u}}^{n}_{N}| ds dx dy \nonumber \\
&= \int_{A_{1}} \int_{0}^{2\pi} \Big|\Big((\bd{\Phi}^{\omega^{n}_{N}}_{f})^{-1} \circ \bd{\Phi}^{\omega^{n - 1}_{N}}_{f}\Big)(x, y)\Big| \cdot \Bigg| \nabla \hat{\bd{u}}^{n}_{N}\Big(\Big|\Big((\bd{\Phi}^{\omega^{n}_{N}}_{f})^{-1} \circ \bd{\Phi}^{\omega^{n - 1}_{N}}_{f}\Big)(x, y)\Big|, \theta\Big)\Bigg| d\theta dx dy \nonumber\\
&\le \text{det}\Big(\nabla \Big(\bd{\Phi}^{\omega^{n}_{N}}_{f} \circ (\bd{\Phi}^{\omega^{n - 1}_{N}}_{f})^{-1}\Big)\Big) \int_{\Omega_{f}} \int_{0}^{2\pi} |(x, y)| \Big|\nabla \hat{\bd{u}}^{n}_{N}(|(x, y)|, \theta)\Big| d\theta dx dy \nonumber \\
&\le C \int_{\Omega_{f}} \int_{|x| = (x, y)} |\nabla \hat{\bd{u}}^{n}_{N}| ds dx dy,
\end{align}
where by the cofactor formula for the matrix inverse, the uniform lower bound on $\mathcal{J}^{\omega^{n}_{N}}_{f}$ and the upper bound on $|\nabla \bd{\Phi}^{\omega^{n}_{N}}_{f}|$ in \eqref{uniformgeom2} give a uniform upper bound on $|(\nabla \bd{\Phi}^{\omega^{n}_{N}}_{f})^{-1}|$. Hence, by integrating \eqref{A1bound1} over $A_{1}$ (which is bounded by its integral over the entire $\Omega_{f}$), by applying \eqref{A1bound2}, and by using the coarea formula (see \cite{Evans}) $\displaystyle \int_{\Omega_{f}} h = \int_{1}^{2} \int_{|\bd{x}| = r} h ds dr$, along with the uniform geometric bounds in Proposition \ref{geometricN} and \eqref{a1hardintegral}:
\begin{align}\label{unNtilde1}
\int_{A_{1}} &\Big|\hat{\bd{u}}^{n}_{N} - \hat{\bd{u}}^{n}_{N} \circ \Big((\bd{\Phi}^{\omega^{n}_{N}}_{f})^{-1} \circ \bd{\Phi}^{\omega^{n - 1}_{N}}_{f}\Big)\Big| \nonumber \\
&\le C(\Delta t) \|\bd{\zeta}^{n - \frac{1}{2}}_{N}\|_{H^{1}(\Gamma)} \left(\int_{1}^{2} \int_{0}^{2\pi} \int_{1}^{2} |\nabla \hat{\bd{u}}^{n}_{N}(\tilde{r}, \theta)| r d\tilde{r} d\theta dr + \int_{1}^{2} \int_{0}^{2\pi} \int_{0}^{2\pi} |\nabla \hat{\bd{u}}^{n}_{N}(r, \tilde{\theta})| rd\tilde{\theta} d\theta dr\right) \nonumber \\
&\le C(\Delta t) \|\bd{\zeta}^{n - \frac{1}{2}}_{N}\|_{H^{1}(\Gamma)} \|\hat{\bd{u}}^{n}_{N}\|_{H^{1}(\Omega_{f})} \le C(\Delta t) \|\bd{\zeta}^{n - \frac{1}{2}}_{N}\|_{H^{1}(\Gamma)} \|\bd{u}^{n}_{N}\|_{V^{n}_{N}} \nonumber \\
&\le C(\Delta t) \Big(\|\bd{\zeta}^{n - \frac{1}{2}}_{N}\|_{H^{1}(\Gamma)}^{2} + \|\bd{u}^{n}_{N}\|_{V^{n}_{N}}\Big)^{3/2}.
\end{align}

\medskip

\noindent \textbf{Estimate of integral over $A_{2}$.} By the definition of $A_{1}$ and $A_{2}$, we note that for $\bd{u}^{n}_{N}$ defined on $\Omega^{n - 1}_{f, N}$ and extended by zero outside, we have that $\hat{\bd{u}}^{n}_{N}|_{A_{2}} = \bd{0}$, so by a change of variables:
\begin{align}\label{A2boundinterm}
\int_{A_{2}} \mathcal{J}^{\omega^{n}_{N}}_{f} \Big|\hat{\bd{u}}^{n}_{N} - \hat{\bd{u}}^{n}_{N} \circ \Big((\bd{\Phi}^{\omega^{n}_{N}}_{f})^{-1} \circ \bd{\Phi}^{\omega^{n - 1}_{N}}_{f}\Big)\Big| &= \int_{A_{2}} \mathcal{J}^{\omega^{n}_{N}}_{f} \Big|\hat{\bd{u}}^{n}_{N} \circ \Big((\bd{\Phi}^{\omega^{n}_{N}}_{f})^{-1} \circ \bd{\Phi}^{\omega^{n - 1}_{N}}_{f}\Big)\Big| \nonumber \\
&= \int_{\left(\bd{\Phi}^{\omega^{n}_{N}}_{f}\right)^{-1}(\Omega^{n}_{f, N} \setminus \Omega^{n - 1}_{f, N})} \mathcal{J}^{\omega^{n}_{N}}_{f} \Big|\bd{u}^{n}_{N} \circ \bd{\Phi}^{\omega^{n - 1}_{N}}_{f}\Big|.
\end{align}
Let $\gamma(z; s)$ for $z \in \Gamma$ be the line connecting $\bd{\omega}^{n - 1}_{N}(z)$ and $\bd{\omega}^{n}_{N}(z)$, parametrized for fixed $z \in \Gamma$ by:
\begin{equation*}
\gamma(x; s) = s \bd{\Phi}^{\omega^{n - 1}_{N}}_{\Gamma}(z) + (1 - s) \bd{\Phi}^{\omega^{n}_{N}}_{\Gamma}(z), \qquad 0 \le s \le 1,
\end{equation*}
which has a length as a curve contained in the maximal domain $\Omega$ of
\begin{equation}\label{lengthcalc1}
\text{Length}(\gamma(z; \cdot)) \le \bd{\omega}^{n}_{N}(z) - \bd{\omega}^{n - 1}_{N}(z) = (\Delta t) \bd{\zeta}^{n - \frac{1}{2}}_{N}(z) \le c_{0}(\Delta t)\|\bd{\zeta}^{n - \frac{1}{2}}_{N}\|_{H^{1}(\Gamma)}.
\end{equation}
Therefore, every point in $\Omega^{n}_{f, N} \setminus \Omega^{n - 1}_{f, N}$ is a distance of less than or equal to $c_{0}(\Delta t) \|\bd{\zeta}^{n - \frac{1}{2}}_{N}\|_{H^{1}(\Gamma)}$ away from $\Gamma^{n}_{f, N}$. Furthermore, every point in $\Omega^{n}_{f, N} \setminus \Omega^{n - 1}_{f, N}$ can be connected to a point on $\Gamma^{n}_{N}$, via a curve lying entirely in $\overline{\Omega^{n}_{f, N}}$ of length less than or equal to $c_{0}(\Delta t) \|\bd{\zeta}^{n - \frac{1}{2}}_{N}\|_{H^{1}(\Gamma)}$. So by the transformation theorem for lengths in Proposition \ref{lengthtransform}, we conclude that the region $(\bd{\Phi}^{\omega^{n}_{N}}_{f})^{-1}(\Omega^{n - 1}_{f, N} \cap \Omega^{n}_{f, N})$ is contained within the thin annulus:
\begin{equation*}
D_{\epsilon} := \{\bd{x} \in \R^{2} : 1 < |\bd{x}| \le 1 + \epsilon \}, \qquad \text{ for } \epsilon = c_{0} (\Delta t) \|\bd{\zeta}^{n - \frac{1}{2}}_{N}\|_{H^{1}(\Gamma)} \cdot \|\nabla (\bd{\Phi}^{\omega^{n}_{N}}_{f})^{-1}\|_{L^{\infty}(\Omega^{n}_{f, N})}.
\end{equation*}
See Figure \ref{intersectfig} for a visual representation.

\begin{figure}
\center
\includegraphics[scale=0.4]{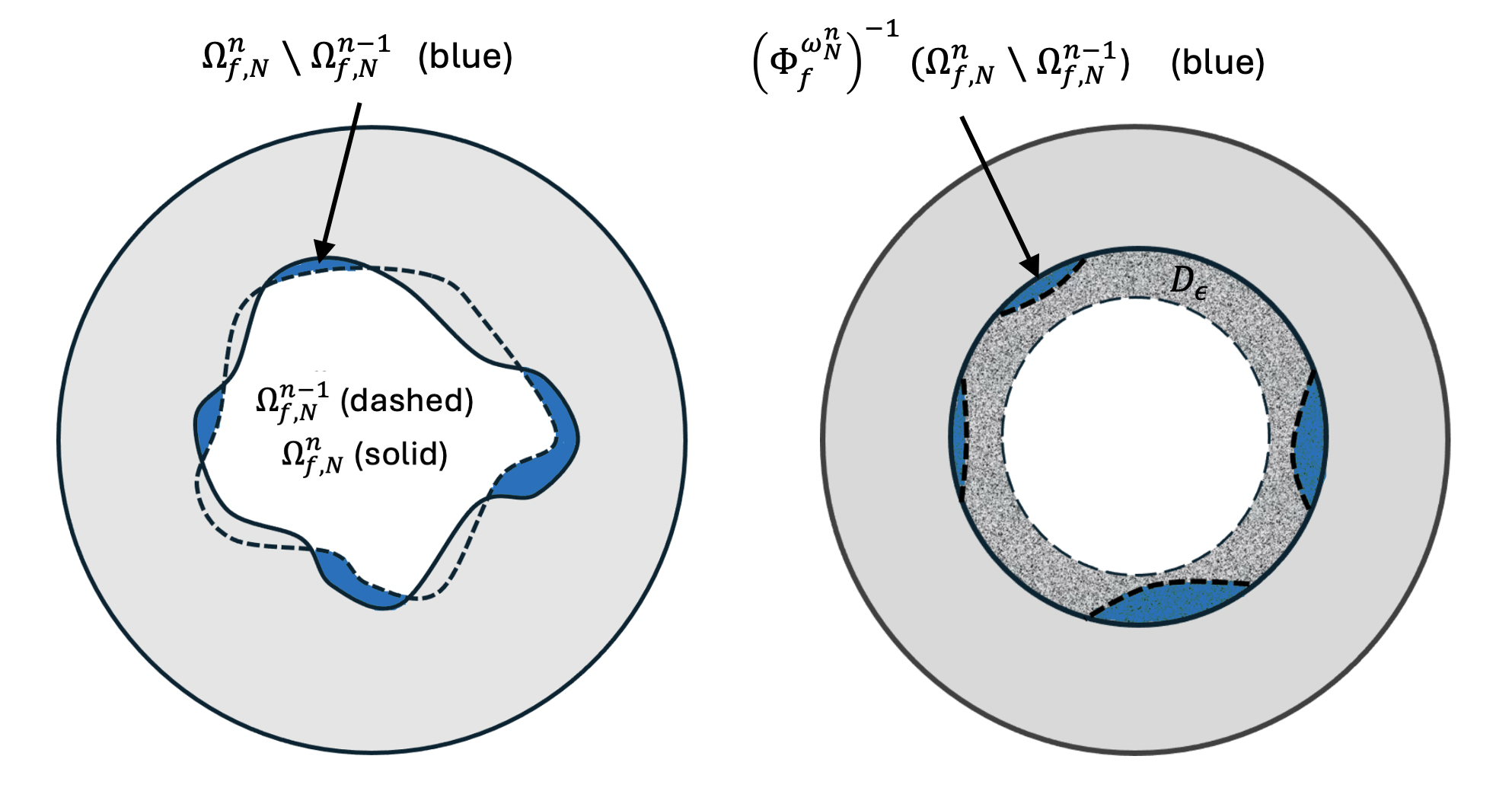}
\caption{An illustration of the integral over $A_{2}$ and the associated domains in both the moving domain (left) and the reference domain (right) configuration, transformed by the map $\bd{\Phi}^{\omega^{n}_{N}}_{f}$.}
\label{intersectfig}
\end{figure}

Thus, from \eqref{A2boundinterm}, Proposition \ref{geometricN} and the bounds \eqref{uniformgeom2}, and the Dirichlet boundary condition that $\hat{\bd{u}}^{n}_{N} = 0$ on $|\bd{x}| = 2$:
\begin{align}\label{unNtilde2}
\int_{A_{2}} \mathcal{J}^{\omega^{n}_{N}}_{f} \Big|\hat{\bd{u}}^{n}_{N} - \hat{\bd{u}}^{n}_{N} \circ \Big((\bd{\Phi}^{\omega^{n}_{N}}_{f})^{-1} \circ \bd{\Phi}^{\omega^{n - 1}_{N}}_{f}\Big)\Big| &\le C \int_{1}^{1 + \epsilon} \int_{0}^{2\pi} |\bd{u}^{n}_{N} \circ \bd{\Phi}^{\omega^{n - 1}_{N}}_{f}(r, \theta)| r dr d\theta \nonumber \\
&\le C \int_{1}^{1 + \epsilon} \int_{0}^{2\pi} \left|\int_{r}^{2} \partial_{r}\Big(\bd{u}^{n}_{N} \circ \bd{\Phi}^{\omega^{n - 1}_{N}}_{f}\Big)(\tilde{r}, \theta) d\tilde{r}\right| d\theta dr \nonumber \\
&\le C \int_{1}^{1 + \epsilon} \int_{0}^{2\pi} \int_{1}^{2} \Big|\nabla \Big(\bd{u}^{n}_{N} \circ \bd{\Phi}^{\omega^{n - 1}_{N}}_{f}\Big)(\tilde{r}, \theta)\Big| d\tilde{r} d\theta dr \nonumber  \\
&\le C\int_{1}^{1 + \epsilon} \Big\|\nabla\Big(\bd{u}^{n}_{N} \circ \bd{\Phi}^{\omega^{n - 1}_{N}}_{f}\Big)\Big\|_{L^{2}(\Omega_{f})} dr \nonumber \\
&\le C \epsilon \cdot \Big\|\bd{u}^{n}_{N} \circ (\bd{\Phi}^{\omega^{n - 1}_{N}}_{f})^{-1}\Big\|_{H^{1}(\Omega_{f})} \nonumber \\
&\le C(\Delta t) \|\bd{\zeta}^{n - \frac{1}{2}}_{N}\|_{H^{1}(\Gamma)} \cdot \|\bd{u}^{n}_{N}\|_{V^{n}_{N}} \nonumber \\
&\le C(\Delta t) \Big(\|\bd{\zeta}^{n - \frac{1}{2}}_{N}\|_{H^{1}(\Gamma)}^{2} + \|\bd{u}^{n}_{N}\|_{V^{n}_{N}}\Big)^{3/2}.
\end{align}

\medskip

\noindent \textbf{Conclusion of proof.} Combining \eqref{unNtilde1} and \eqref{unNtilde2}, we obtain
\begin{equation*}
\sup_{\|\bd{q}\|_{Q^{n}_{N}} \le 1} \left|\int_{\Omega^{n}_{f, N}} \frac{\tilde{\bd{u}}^{n}_{N} - \bd{u}^{n}_{N}}{\Delta t} \cdot \bd{q}\right| \le C\Big(1 + \|\bd{\zeta}^{n - \frac{1}{2}}_{N}\|^{2}_{H^{1}(\Gamma)} + \|\bd{u}^{n}_{N}\|_{V^{n}_{N}}\Big)^{3/2},
\end{equation*}
where we can set $a^{n}_{N} := 1 + \|\bd{\zeta}^{n - \frac{1}{2}}_{N}\|_{H^{1}(\Gamma)}^{2}$ in the statement of the lemma, since $\displaystyle (\Delta t) \sum_{n = 0}^{N - 1} |a^{n}_{N}|^{2} \le c$ as a result of $\bd{\zeta}^{*}_{N} \in L^{\infty}(0, T; L^{2}(\Gamma)) \cap L^{2}(0, T; H^{2}(\Gamma))$ and hence $\bd{\zeta}^{*}_{N} \in L^{4}(0, T; H^{1}(\Gamma))$. 
\end{proof}

\begin{proof}[Proof of Generalized Property B, Lemma \ref{generalizedB}]
We consider a fluid test function $\bd{v} \in Q^{n}_{N}$ such that $\|\bd{v}\|_{Q^{n}_{N}} \le 1$. Then, by the triangle inequality estimate in \eqref{triangleineq} and the estimate in Lemma \ref{tildecomp}, it suffices to estimate the following quantity
\begin{equation*}
\sup_{\|\bd{q}\|_{Q^{n}_{N}} \le 1} \left|\int_{\Omega^{n}_{f, N}} \frac{\bd{u}^{n + 1}_{N} - \tilde{\bd{u}}^{n}_{N}}{\Delta t} \cdot \bd{q}\right|.
\end{equation*}
We can estimate this, using the semidiscrete formulation for the fluid in \eqref{fluidnN}, which gives rise to the following terms:

\medskip

\noindent \textbf{Term 1.} We estimate, by using $\|\bd{q}\|_{H^{3}(\Omega)} \le 1$ along with the Cauchy-Schwarz inequality, that $\displaystyle \sup_{\|\bd{q}\|_{Q^{n}_{N}} \le 1} \left|\int_{\Omega^{n}_{f, N}} \bd{D}(\bd{u}^{n + 1}_{N}) : \bd{D}(\bd{q})\right| \le C\|\bd{u}^{n}_{N}\|_{V^{n}_{N}}$. Here, we recall the definition of the more regular test space $Q^{n}_{N}$ in \eqref{Qmoreregular}, which gives us an $H^{3}(\Omega)$ estimate on the fluid test function $\bd{q}$. 

\medskip

\noindent \textbf{Term 2.} Since $\|\bd{q}\|_{Q^{n}_{N}} \le 1$, by the trace inequality, $\|\bd{q}\|_{L^{2}(\Gamma)} \le C$ for a constant $C$ independent of $n$ and $N$, by uniform bounds on the ALE fluid domain map as in Proposition \ref{traceuniform}. Hence, we can estimate using Sobolev embedding that
\begin{align*}
\sup_{\|q\|_{Q^{n}_{N}} \le 1} \left|\int_{\Gamma^{n}_{N}} \left(\frac{1}{2} \bd{u}^{n + 1}_{N} \cdot \tilde{\bd{u}}^{n}_{N} - p^{n + 1}_{N}\right) (\bd{q} \cdot \bd{n})\right| &\le C\Big(\|\bd{u}^{n + 1}_{N}\|_{L^{4}(\Gamma)} \|\bd{u}^{n}_{N}\|_{L^{4}(\Gamma)} + \|p^{n + 1}_{N}\|_{L^{2}(\Gamma)}\Big) \\
&\le C\Big(\|\bd{u}^{n + 1}_{N}\|_{H^{3/4}(\Omega_{f})} \|\bd{u}^{n}_{N}\|_{H^{3/4}(\Omega_{f})} + \|p^{n + 1}_{N}\|_{H^{1}(\Omega_{b})}\Big).
\end{align*}
By using the uniform bound on $\bd{u}_{N}$ in $L^{\infty}(0, T; L^{2}(\Omega_{f}))$ and interpolation, we continue the chain of inequalities as:
\begin{equation*}
\le C\Big(\|\bd{u}^{n + 1}_{N}\|^{3/4}_{H^{1}(\Omega_{f})} \|\bd{u}^{n}_{N}\|^{3/4}_{H^{1}(\Omega_{f})} + \|p^{n + 1}_{N}\|_{H^{1}(\Omega_{b})}\Big) \le C\Big(1 + \|p^{n + 1}_{N}\|_{H^{1}(\Omega_{b})} + \|\bd{u}^{n}_{N}\|_{V^{n}_{N}} + \|\bd{u}^{n + 1}_{N}\|_{V^{n + 1}_{N}}\Big)^{3/2}.
\end{equation*}

\medskip

\noindent \textbf{Term 3.} Before estimating Term 3, we derive an estimate for the discretized (fluid domain) ALE velocity $\tilde{\bd{w}}^{n + 1}_{N}$ in \eqref{tildew}. We note that on the fixed domain, $\bd{w}^{n + 1}_{N} := (\Delta t)^{-1}\Big(\bd{\Phi}^{\omega^{n + 1}_{N}}_{f} - \bd{\Phi}^{\omega^{n}_{N}}_{f}\Big)$ defined in \eqref{wale} satisfies the harmonic equation $\Delta \bd{w}^{n + 1}_{N} = 0$ on $\Omega_{f}$
with boundary conditions
\begin{equation*}
\bd{w}^{n + 1}_{N} = \bd{0} \text{ on } \partial \Omega_{f} \setminus \Gamma, \quad \bd{w}^{n + 1}_{N} = (\Delta t)^{-1}(\bd{\omega}^{n + 1}_{N} - \bd{\omega}^{n}_{N}) = \bd{\zeta}^{n + 1/2}_{N} \text{ on } \Gamma.
\end{equation*}
Note that $\bd{\zeta}^{*}_{N}$ defined in \eqref{zetastar} is uniformly bounded in $L^{\infty}(0, T; L^{2}(\Gamma)) \cap L^{2}(0, T; H^{2}(\Gamma))$ and is hence in $L^{4}(0, T; H^{1}(\Gamma))$ by uniform estimates, so therefore we can estimate,
\begin{equation*}
\|\bd{w}^{n + 1}_{N}\|_{L^{2}(\Omega_{f})} \le C\|\bd{w}^{n + 1}_{N}\|_{H^{1}(\Omega_{f})} \le C\|\bd{\zeta}^{n + \frac{1}{2}}_{N}\|_{H^{1}(\Gamma)}.
\end{equation*}
Thus, since $\|\bd{q}\|_{Q^{n}_{N}} \le 1$, we can use the uniform bounds on the Jacobian of the fluid ALE map in \eqref{uniformgeom2} to deduce that
\begin{align*}
\sup_{\|q\|_{Q^{n}_{N}} \le 1} \Bigg|\frac{1}{2} \int_{\Omega^{n}_{f, N}} &[((\tilde{\bd{u}}^{n}_{N} - \tilde{\bd{w}}^{n + 1}_{N}) \cdot \nabla \bd{u}^{n + 1}_{N}) \cdot \bd{q} - ((\tilde{\bd{u}}^{n}_{N} - \tilde{\bd{w}}^{n + 1}_{N}) \cdot \nabla \bd{q}) \cdot \bd{u}^{n + 1}_{N}]\Bigg| \\
&\le C\Big(\|\bd{u}^{n}_{N}\|_{L^{2}(\Omega_{f})} \|\bd{u}^{n + 1}_{N}\|_{H^{1}(\Omega_{f})} + \|\bd{w}^{n + 1}_{N}\|_{L^{2}(\Omega_{f})} \|\bd{u}^{n + 1}_{N}\|_{H^{1}(\Omega_{f})}\Big) \\
&\le C\Big(1 + \|\bd{w}^{n + 1}_{N}\|_{L^{2}(\Omega_{f})}^{3} + \|\bd{u}^{n + 1}_{N}\|_{V^{n + 1}_{N}}^{3/2}\Big) \le C\Big(1 + \|\bd{w}^{n + 1}_{N}\|_{L^{2}(\Omega_{f})}^{2} + \|\bd{u}^{n + 1}_{N}\|_{V^{n + 1}_{N}}\Big)^{3/2} \\
&\le C\Big(1 + \|\bd{\zeta}^{n + \frac{1}{2}}_{N}\|_{H^{1}(\Gamma)}^{2} + \|\bd{u}^{n + 1}_{N}\|_{V^{n + 1}_{N}}\Big)^{3/2}.
\end{align*}

\medskip

\noindent \textbf{Term 4.} We estimate the following term:
\begin{equation*}
\sup_{\|\bd{q}\|_{Q^{n}_{N}} \le 1} \left|\int_{\Omega^{n}_{f, N}} \frac{1}{\Delta t} \left(\frac{\mathcal{J}^{\omega^{n + 1}_{N}}_{f}}{\mathcal{J}^{\omega^{n}_{N}}_{f}} - 1\right) \bd{u}^{n + 1}_{N} \cdot \bd{q}\right|
\end{equation*}
By the uniform bounds on the ALE map for the fluid domain in \eqref{uniformgeom2}, we have that $\mathcal{J}^{\omega^{n}_{N}}_{f} \ge c > 0$ for some positive constant $c$. We compute that
\begin{align*}
\mathcal{J}^{\omega^{n + 1}_{N}}_{f} - \mathcal{J}^{\omega^{n}_{N}}_{f} &= \Big(\partial_{x}\Phi^{\omega^{n + 1}_{N}}_{f, x} \partial_{y} \Phi^{\omega^{n + 1}_{N}}_{f, y} - \partial_{x} \Phi^{\omega^{n}_{N}}_{f, x} \partial_{y} \Phi^{\omega^{n}_{N}}_{f, y}\Big) - \Big(\partial_{x}\Phi^{\omega^{n + 1}_{N}}_{f, y}\partial_{y}\Phi^{\omega^{n + 1}_{N}}_{f, x} - \partial_{x}\Phi^{\omega^{n}_{N}}_{f, y} \partial_{y} \Phi^{\omega^{n}_{N}}_{f, x}\Big) \\
&= \partial_{x}\Big(\Phi^{\omega^{n + 1}_{N}}_{f, x} - \Phi^{\omega^{n}_{N}}_{f, x}\Big) \partial_{y}\Phi^{\omega^{n + 1}_{N}}_{f, y} \\
&+ \partial_{x}\Phi^{\omega^{n}_{N}}_{f, x} \partial_{y}\Big(\Phi^{\omega^{n + 1}_{N}}_{f, y} - \Phi^{\omega^{n}_{N}}_{f, y}\Big) - \partial_{x}\Big(\Phi^{\omega^{n + 1}_{N}}_{f, y} - \Phi^{\omega^{n}_{N}}_{f, y}\Big) \partial_{y}\Phi^{\omega^{n + 1}_{N}}_{f, x} - \partial_{x}\Phi^{\omega^{n}_{N}}_{f, y} \partial_{y} \Big(\Phi^{\omega^{n + 1}_{N}}_{f, x} - \Phi^{\omega^{n}_{N}}_{f, x}\Big).
\end{align*}
We recall the uniform bound on $\bd{\omega}_{N}$ in $L^{\infty}(0, T; H^{2}(\Gamma))$ from Proposition \ref{uniformenergy}, which implies that 
\begin{equation*}
\|\bd{\Phi}^{\omega^{n}_{N}}_{f}\|_{H^{5/2}(\Omega_{f})} \le C\|\nabla \bd{\Phi}^{\omega^{n}_{N}}_{f}\|_{L^{\infty}(\Omega_{f})} \le C
\end{equation*}
for a constant $C$ that is independent of $n$ and $N$, and furthermore
\begin{equation*}
\left\|(\Delta t)^{-1} \nabla \Big(\bd{\Phi}^{\omega^{n + 1}_{N}}_{f} - \bd{\Phi}^{\omega^{n}_{N}}_{f}\Big)\right\|_{L^{2}(\Omega_{f})} \le \left\|(\Delta t)^{-1} (\bd{\Phi}^{\omega^{n + 1}_{N}}_{f} - \bd{\Phi}^{\omega^{n}_{N}}_{f})\right\|_{H^{1}(\Omega_{f})} \le C\|\bd{\zeta}^{n + \frac{1}{2}}_{N}\|_{H^{1}(\Gamma)}.
\end{equation*}
Therefore, we estimate that
\begin{multline*}
\left\|(\Delta t)^{-1}\Big(\mathcal{J}^{\omega^{n + 1}_{N}}_{f} - \mathcal{J}^{\omega^{n}_{N}}_{f}\Big)\right\|_{L^{2}(\Omega_{f})} \\
\le C\Big(\|\nabla \bd{\Phi}^{\omega^{n}_{N}}_{f}\|_{L^{\infty}(\Omega_{f})} + \|\nabla \bd{\Phi}^{\omega^{n + 1}_{N}}_{f}\|_{L^{\infty}(\Omega_{f})}\Big)\left\|(\Delta t)^{-1} \nabla \Big(\bd{\Phi}^{\omega^{n + 1}_{N}}_{f} - \bd{\Phi}^{\omega^{n}_{N}}_{f}\Big)\right\|_{L^{2}(\Omega_{f})} \le C\|\bd{\zeta}^{n + \frac{1}{2}}_{N}\|_{H^{1}(\Gamma)},
\end{multline*}
and hence, 
\begin{equation*}
\sup_{\|\bd{q}\|_{Q^{n}_{N}} \le 1} \left|\int_{\Omega^{n}_{f, N}} \frac{1}{\Delta t} \left(\frac{\mathcal{J}^{\omega^{n + 1}_{N}}_{f}}{\mathcal{J}^{\omega^{n}_{N}}_{f}} - 1\right) \bd{u}^{n + 1}_{N} \cdot \bd{q}\right| \le C\|\bd{\zeta}^{n + \frac{1}{2}}_{N}\|_{H^{1}(\Gamma)} \|\bd{u}^{n + 1}_{N}\|_{L^{2}(\Omega_{f})} \le C\Big(1 + \|\bd{\zeta}^{n + \frac{1}{2}}_{N}\|_{H^{1}(\Gamma)}\Big)^{3/2}.
\end{equation*}

\medskip

\noindent \textbf{Term 5.} Next, we estimate the term, using $\|\bd{q}\|_{H^{3}(\Omega_{f})} \le 1$:
\begin{align*}
\sup_{\|\bd{q}\|_{Q^{n}_{N}} \le 1} \left|\int_{\Gamma^{n}_{N}} (\bd{u}^{n + 1}_{N} - \bd{\zeta}^{n + 1}_{N}) \cdot \bd{n} (\tilde{\bd{u}}^{n}_{N} \cdot \bd{q})\right| &\le C\Big(\|\bd{u}^{n + 1}_{N}\|_{L^{4}(\Gamma)} \|\bd{u}^{n}_{N}\|_{L^{4}(\Gamma)} + \|\bd{\zeta}^{n + 1}_{N}\|_{L^{2}(\Gamma)}\|\bd{u}^{n}_{N}\|_{L^{2}(\Gamma)}\Big) \\
&\le C\Big(\|\bd{u}^{n + 1}_{N}\|_{L^{4}(\Gamma)} \|\bd{u}^{n}_{N}\|_{L^{4}(\Gamma)} + \|\bd{u}^{n}_{N}\|_{L^{2}(\Gamma)}\Big) \\
&\le C\Big(1 + \|\bd{u}^{n}_{N}\|_{V^{n}_{N}} + \|\bd{u}^{n + 1}_{N}\|_{V^{n + 1}_{N}}\Big)^{3/2},
\end{align*}
using the same interpolation estimates as in Term 2.

\medskip

\noindent \textbf{Term 6.} Finally, we estimate the slip term, by using the fact that $\|\bd{q}\|_{H^{3}(\Omega_{f})} \le 1$, Sobolev embedding, and the trace theorem:
\begin{equation*}
\sup_{\|\bd{q}\|_{Q^{n}_{N}} \le 1} \left|\int_{\Gamma^{n}_{N}} (\bd{\zeta}^{n + 1}_{N} - \bd{u}^{n + 1}_{N}) \cdot \bd{\tau} (\bd{q} \cdot \bd{\tau})\right| \le C\Big(\|\bd{\zeta}^{n + 1}_{N}\|_{L^{2}(\Gamma)} + \|\bd{u}^{n + 1}_{N}\|_{L^{2}(\Gamma)}\Big) \le C\Big(1 + \|\bd{u}^{n + 1}_{N}\|_{V^{n + 1}_{N}}\Big)^{3/2}.
\end{equation*}

\medskip

\noindent \textbf{Conclusion.} Combining the preceding estimates, we obtain that
\begin{equation*}
\sup_{\|\bd{q}\|_{Q^{n}_{N}} \le 1} \left|\int_{\Omega^{n}_{f, N}} \frac{\bd{u}^{n + 1}_{N} - \tilde{\bd{u}}^{n}_{N}}{\Delta t} \cdot \bd{q}\right| \le C\Big(1 + \|p^{n + 1}_{N}\|_{H^{1}(\Omega_{b})} + \|\bd{\zeta}^{n + \frac{1}{2}}_{N}\|^{2}_{H^{1}(\Gamma)} + \|\bd{u}^{n}_{N}\|_{V^{n}_{N}} + \|\bd{u}^{n + 1}_{N}\|_{V^{n + 1}_{N}}\Big)^{3/2},
\end{equation*}
which establishes Generalized Property B, since $\displaystyle (\Delta t) \sum_{n = 1}^{N} a^{n}_{N} \le C$ uniformly in $N$ for
\begin{equation*}
a^{n}_{N} := 1 + \|p^{n}_{N}\|_{H^{1}(\Omega_{b})} + \|\bd{\zeta}^{n + \frac{1}{2}}_{N}\|_{H^{1/2}(\Gamma)}^{2}
\end{equation*}
by uniform bounds in $N$ of $p_{N} \in L^{2}(0, T; H^{1}(\Omega_{b}))$ and $\zeta^{*}_{N} \in L^{4}(0, T; H^{1}(\Gamma)) \cap L^{2}(0, T; H^{2}(\Gamma))$.
\end{proof}

We hence have convergence of the physical fluid velocities on the maximal fluid domain in $L^{2}(0, T; L^{2}(\Omega))$. Because the semidiscrete formulation involves the fluid velocity transferred to the fixed domain, we want to obtain a corresponding convergence result for the fluid velocities $\hat{\bd{u}}_{N}$ on the fixed domain $\Omega_{f}$. In addition, the semidiscrete weak formulation also involves the trace of the fluid velocity along the moving interface, so we also establish the following convergence result for the traces of the fluid velocities, which follows from the strong convergence $\bd{u}_{N} \to \bd{u}$ in $L^{2}(0, T; L^{2}(\Omega))$.

\begin{proposition}\label{traceconvprop}
We also have the following strong convergences for the fluid velocities $\hat{\bd{u}}_{N} := \bd{u}_{N} \circ \Phi^{\omega_{N}}_{f}$ defined on the fixed fluid domain $\Omega_{f}$:
\begin{equation*}
\hat{\bd{u}}_{N} \to \hat{\bd{u}}, \quad \text{ strongly in $L^{2}(0, T; L^{2}(\Omega_{f}))$},
\end{equation*}
\begin{equation*}
\hat{\bd{u}}_{N}|_{\Gamma} \to \hat{\bd{u}}|_{\Gamma}, \quad \text{ strongly in $L^{2}(0, T; L^{2}(\Gamma))$}.
\end{equation*}
\end{proposition}

\begin{proof}
By the uniform geometric bounds in Proposition \ref{geometricN}, we have that $\{\hat{\bd{u}}_{N}\}_{N = 1}^{\infty}$ and $\hat{\bd{u}}$ are all uniformly bounded in $L^{2}(0, T; H^{1}(\Omega_{f}))$ independently of $N$, so it suffices to show the convergence $\hat{\bd{u}}_{N} \to \hat{\bd{u}}$ in $L^{2}(0, T; L^{2}(\Omega_{f}))$ as this implies the corresponding convergence of the traces by Lemma 9.1 in \cite{FPSIJMPA}. We define
\begin{equation*}
\tilde{\bd{u}}_{N} := \bd{u} \circ \bd{\Phi}^{\omega_{N}}_{f}
\end{equation*}
and estimate
\begin{equation*}
\|\tilde{\bd{u}}_{N} - \hat{\bd{u}}_{N}\|_{L^{2}(0, T; L^{2}(\Omega_{f}))}^{2} = \int_{0}^{T} \int_{\Omega_{f}} |\tilde{\bd{u}}_{N} - \hat{\bd{u}}_{N}|^{2} = \int_{0}^{T} \int_{\Omega_{f, N}(t)} \mathcal{J}^{\omega_{N}(t)}_{f} |\bd{u}_{N} - \bd{u}|^{2} \le C\int_{0}^{T} \int_{\Omega} |\bd{u}_{N} - \bd{u}|^{2} \to 0,
\end{equation*}
where we used the fact that $\Omega^{n}_{f, N} \subset \Omega$ for all $n$ and $N$, the uniform geometric bounds in Proposition \ref{geometricN}, and the fact that $\bd{u}_{N} \to \bd{u}$ in $L^{2}(0, T; L^{2}(\Omega))$ on the maximal domain as $N \to \infty$. 

So it suffices to show that $\|\tilde{\bd{u}}_{N} - \hat{\bd{u}}\|_{L^{2}(0, T; L^{2}(\Omega_{f}))}^{2} \to 0$, which is a similar computation to the computation in Lemma \ref{tildecomp}. Namely, we estimate
\begin{equation*}
\int_{0}^{T} \int_{\Omega_{f}} |\tilde{\bd{u}}_{N} - \hat{\bd{u}}|^{2} = \int_{0}^{T} \int_{A_{1}(t)} \cdot + \int_{0}^{T} \int_{A_{2}(t)} \Big|\hat{\bd{u}} - \hat{\bd{u}} \circ \Big((\bd{\Phi}^{\omega_{N}}_{f})^{-1} \circ \bd{\Phi}^{\omega}_{f}\Big)\Big|^{2},
\end{equation*}
for
\begin{equation*}
A_{1}(t) := (\bd{\Phi}^{\omega_{N}(t)}_{f})^{-1} (\Omega_{f, N}(t) \cap \Omega_{f}(t)) \text{ and } A_{2}(t) := \Omega_{f} \setminus A_{1}(t),
\end{equation*}
where we are using $\Omega_{f}(t)$ to denote the moving fluid domain associated with the limiting structure displacement $\bd{\omega}(t)$. To estimate the integral over $A_{1}(t)$, note that
\begin{equation*}
|\bd{\Phi}^{\omega_{N}}_{f}(x, y) - \bd{\Phi}^{\omega}_{f}(x, y)| \le \|\bd{\Phi}^{\omega_{N}}(\cdot) - \bd{\Phi}^{\omega}_{N}(\cdot)\|_{H^{3/2}(\Omega_{f})} \le C\|\bd{\omega}_{N} - \bd{\omega}\|_{H^{1}(\Gamma)}.
\end{equation*}
Then, combining this with the same calculation as in \eqref{A1bound1}, \eqref{A1bound2}, and \eqref{unNtilde1}, we estimate that
\begin{align}\label{A1final}
\int_{0}^{T} \int_{A_{1}(t)} \Big|\hat{\bd{u}} - \hat{\bd{u}} \circ \Big((\bd{\Phi}^{\omega_{N}}_{f})^{-1} \circ \bd{\Phi}^{\omega}_{f}\Big)\Big|^{2} &\le C\int_{0}^{T} \|\bd{\omega} - \bd{\omega}_{N}\|_{H^{1}(\Gamma)}^{2} \|\bd{u}\|_{H^{1}(\Omega_{f}(t))}^{2} \nonumber \\
&\le C\|\bd{\omega} - \bd{\omega}_{N}\|_{C(0, T; H^{1}(\Gamma)}^{2} \|\bd{u}\|^{2}_{L^{2}(0, T; H^{1}(\Omega_{f}(t))} \to 0,
\end{align}
by the strong convergence of the plate displacements in Proposition \ref{omegaconv}. To estimate the integral over $A_{2}(t)$, we use the calculation in \eqref{lengthcalc1} to conclude that the maximum radial ``width" of the region $A_{2}(t)$ at a coordinate of $z \in \Gamma$ is less than or equal to $$r_{\Delta t}(t, z) := C(\bd{\omega}(t, z) - \bd{\omega}_{N}(t, z))$$ so that by the same calculation in \eqref{unNtilde2},
\begin{align}\label{A2final}
\int_{0}^{T} \int_{A_{2}(t)} \Big|\hat{\bd{u}} - \hat{\bd{u}} \circ \Big((\bd{\Phi}^{\omega_{N}}_{f})^{-1} \circ \bd{\Phi}^{\omega}_{f}\Big)\Big|^{2} %&\le C \int_{0}^{T} \int_{0}^{L} r_{\Delta t}(t, x)^{2} \|\partial_{y}(\bd{u} \circ \bd{\Phi}^{\omega_{N}}_{f})(x, \cdot)\|_{L^{2}([0, R])}^{2} \nonumber \\
&\le C \int_{0}^{T} \|\bd{\omega} - \bd{\omega}_{N}\|_{C(\Gamma)}^{2} \|\bd{u}\|_{H^{1}(\Omega_{f}(t))}^{2} \nonumber \\
&\le C \|\bd{\omega} - \bd{\omega}_{N}\|_{C(0, T; H^{3/2}(\Gamma))}^{2} \|\bd{u}\|^{2}_{L^{2}(0, T; H^{1}(\Omega_{f}(t)))} \to 0,
\end{align}
by the strong convergence in Proposition \ref{omegaconv}. This completes the proof, as we can combine \eqref{A1final} and \eqref{A2final} to conclude that $\hat{\bd{u}}_{N} \to \hat{\bd{u}}$ in $L^{2}(0, T; L^{2}(\Omega_{f}))$. 
\end{proof}

\medskip

\subsection{Convergence of test functions and passage to the limit}\label{testfunctions}

The final step of the existence proof is to pass to the limit as $N \to \infty$ in the semidiscrete formulation \eqref{weakBiot1} to obtain the limiting weak formulation given in Definition \ref{plateweak}. Let us summarize the strong convergences that we have so far, which will be useful for this limit passage.

\begin{proposition}[Strong convergences]\label{strongsummary}
We have the following \textit{strong} convergences of the approximate solutions:
\begin{itemize}
\item $\bd{\omega}_{N} \to \bd{\omega}$ strongly in $C(0, T; C(\Gamma))$.
\item $\bd{\eta}_{N} \to \bd{\eta}$ strongly in $C(0, T; L^{2}(\Gamma))$.
\item $p_{N} \to p$ strongly in $L^{2}(0, T; L^{2}(\Omega_{b}))$.
\item $(\bd{\xi}_{N}, \bd{\zeta}_{N}) \to (\partial_{t}\bd{\eta}, \partial_{t}\bd{\omega})$ strongly in $L^{2}(0, T; H^{-s}(\Omega_{b}) \times H^{-s}(\Gamma))$ for $0 < s < 1$.
\item $\hat{\bd{u}}_{N} \to \hat{\bd{u}}$ strongly in $L^{2}(0, T; L^{2}(\Omega_{f}))$.
\item $\hat{\bd{u}}_{N}|_{\Gamma} \to \hat{\bd{u}}|_{\Gamma}$ strongly in $L^{2}(0, T; L^{2}(\Gamma))$. 
\end{itemize}
\end{proposition}

Because the moving domain appears in the semidiscrete weak formulations in terms of a nonlinear geometric condition on the test functions in the test space $\mathcal{V}^{\omega_{N}}_{f}$, namely:
\begin{equation*}
\nabla^{\tau_{\Delta t}\omega_{N}}_{f} \cdot \bd{v} = 0 \text{ on } \Omega_{f}
\end{equation*}
where we recall the definition of $\nabla^{\omega_{N}}_{f}$ from \eqref{omegadiff}, we must also carefully pass to the limit in the fluid test functions. Recall from the Biot/fluid subproblem in \eqref{weakBiot1} that the fluid velocity $\bd{u}^{n + 1}_{N}$ satisfies $\nabla^{\omega^{n}_{N}} \cdot \bd{u}^{n + 1}_{N} = 0$ on $\Omega_{f}$, in terms of the \textit{previous} structure displacement, and hence, the time-shifted approximate structure displacement $\tau_{\Delta t}\bd{\omega}_{N}$ appears in the divergence-free condition:
\begin{equation*}
\tau_{\Delta t} \bd{\omega}_{N}(t, \cdot) = \bd{\omega}_{N}(t - \Delta t, \cdot), \qquad \text{ for } \Delta t \le t \le T.
\end{equation*}
In contrast to other FSI problems which involve kinematic coupling between the fluid velocity and the plate velocity at the interface in terms of $\bd{u}|_{\Gamma} = \bd{\omega}$ (see for example \cite{MuhaCanic13}), the fluid test functions in this FPSI problem are fully decoupled from the other test functions, which somewhat simplifies this test function construction. This was already observed for moving domain (nonlinearly geometrically coupled) FPSI problems in \cite{FPSIJMPA}, for example.

The goal will be to consider a target test function for the limit and then approximate the target test function by test functions for the semidiscrete formulation \eqref{weakBiot1} which converge strongly to the target test function as $N \to \infty$. To define the \textit{target test space}, we define the function space $\mathcal{X}$ consisting of all continuously differentiable functions $\bd{v} \in C^{1}_{c}([0, T); H^{1}(\Omega))$ defined on the maximal domain $\Omega$, such that the following additional conditions are satisfied:
\begin{itemize}
\item \textbf{Spatial smoothness for all times.} $\bd{v}(t) \in C^{\infty}(\overline{\Omega})$ for all $t \in [0, T]$.
\item \textbf{Divergence-free condition.} $\nabla \cdot \bd{v} = 0$ for all $(t, x) \in [0, T] \times \Omega$.
\item \textbf{Dirichlet boundary conditions.} $\bd{v}|_{\partial \Omega_{f} \setminus \Gamma} = \bd{0}$. 
\end{itemize}
We will then use the ALE map to pull back the spatially smooth test functions in $\mathcal{X}$ back to the fixed fluid domain $\Omega_{f}$:
\begin{equation*}
\hat{\bd{v}} := \bd{v} \circ \bd{\Phi}^{\omega}_{f},
\end{equation*}
where $\omega$ is the limiting structure displacement. We denote the set of all possible test functions $\hat{\bd{v}}$ obtained this way from all possible choices of $\bd{v} \in \mathcal{X}$ by $\mathcal{X}^{\omega}_{f}$, which is the \textbf{target test space}. The target test space $\mathcal{X}^{\omega}_{f}$ has the essential property that it is dense in the full test space $\mathcal{V}^{\omega}_{f}$ defined in \eqref{Vomegaf} for the limiting problem, so it suffices to verify that the limiting weak formulation in Definition \ref{plateweak} holds for arbitrary test functions $\hat{\bd{v}} \in \mathcal{X}^{\omega}_{f}$. 

Using the fluid domain ALE map, we also define admissible test functions for each $N$ for the corresponding semidiscrete formulation using the test function $\bd{v} \in \mathcal{X}$:
\begin{equation}\label{hatvN}
\hat{\bd{v}}_{N} = \bd{v} \circ \bd{\Phi}^{\tau_{\Delta t} \omega_{N}}_{f},
\end{equation}
and the key property is that $\hat{\bd{v}}_{N}$ are admissible test functions in the sense that $\nabla^{\tau_{\Delta t} \omega_{N}}_{f} \cdot \hat{\bd{v}}_{N} = 0$ for all $t \in [0, T]$. Furthermore, these approximate test functions $\hat{\bd{v}}_{N}$ converge strongly to the limiting target test function $\bd{v} \in \mathcal{X}^{\omega}_{f}$ in the following sense, which (combined with the strong convergences in Proposition \ref{strongsummary} and the weak convergences in Proposition \ref{uniformenergy}) allows us to pass to the limit as $N \to \infty$ in the semidiscrete formulation with the test function $\hat{\bd{v}}_{N}$ to obtain the limiting weak formulation with test function $\hat{\bd{v}}$. 

\begin{proposition}\label{testconvergence}
For an arbitrary $\bd{v} \in \mathcal{X}$, we have the following convergences:
\begin{equation*}
\hat{\bd{v}}_{N} \to \hat{\bd{v}} \text{ and } \nabla \hat{\bd{v}}_{N} \to \nabla \hat{\bd{v}},
\end{equation*}
both uniformly and pointwise on $[0, T] \times \Omega_{f}$. 
\end{proposition}

\begin{proof}
Since $\bd{v}$ is spatially smooth, it suffices to show that for $\epsilon > 0$ sufficiently small, chosen so that $5/2 - \epsilon > 2$:
\begin{equation}\label{vconv1}
\|\bd{\Phi}^{\omega}_{f} - \bd{\Phi}^{\tau_{\Delta t}\omega_{N}}_{f}\|_{L^{\infty}([0, T] \times \Omega_{f})} \le C\|\bd{\omega} - \tau_{\Delta t}\bd{\omega}_{N}\|_{C(0, T; H^{2 - \epsilon}(\Gamma))},
\end{equation}
\begin{equation}\label{vconv2}
\|\nabla \bd{\Phi}^{\omega}_{f} - \nabla \bd{\Phi}^{\tau_{\Delta t}\omega_{N}}_{f}\|_{L^{\infty}([0, T] \times \Omega_{f})} \le C\|\bd{\omega} - \tau_{\Delta t}\bd{\omega}_{N}\|_{C(0, T; H^{2 - \epsilon}(\Gamma))},
\end{equation}
as the result would follow by using that $\bd{v}$ is Lipschitz continuous in space (with a uniform Lipschitz constant in time) and from the convergence of $\bd{\omega}_{N} \to \bd{\omega}$ in $C(0, T; H^{s}(\Gamma))$ for all $0 < s < 2$. In addition, to show that $\|\bd{\omega}_{N} - \tau_{\Delta t} \bd{\omega}_{N}\|_{C(0, T; H^{2 - \epsilon}(\Gamma)} \to 0$ as $N \to \infty$, we use the uniform bound of $\bd{\omega}_{N}$ in $C^{\alpha}(0, T; H^{2 - \epsilon}(\Gamma))$ for some $\alpha > 0$, which follows from the interpolation and the uniform bounds of $\bd{\omega}_{N}$ in $L^{\infty}(0, T; H^{2}(\Gamma)) \cap W^{1, \infty}(0, T; L^{2}(\Gamma))$. 

To establish these inequalities \eqref{vconv1} and \eqref{vconv2}, note that $\bd{\Phi}^{\omega}_{f} - \bd{\Phi}^{\omega_{N}}_{f}$ satisfies $\Delta (\bd{\Phi}^{\omega}_{f} - \bd{\Phi}^{\omega_{N}}_{f}) = 0$. Furthermore, it is equal to $\bd{\omega} - \bd{\omega}_{N}$ along $\Gamma$ and has zero boundary data along the rest of $\partial \Omega_{f}$. Hence, by elliptic estimates, we conclude that
\begin{equation*}
\|\bd{\Phi}^{\omega}_{f} - \bd{\Phi}^{\omega_{N}}_{f}\|_{L^{\infty}(0, T; H^{5/2 - \epsilon}(\Omega_{f}))} \le C\|\bd{\omega} - \bd{\omega}_{N}\|_{L^{\infty}(0, T; H^{2 - \epsilon}(\Gamma))},
\end{equation*}
which establishes the desired estimates by Sobolev embedding, since $5/2 - \epsilon > 2$. 
\end{proof}

We can thus conclude the proof of existence of a regularized interface weak solution to the FPSI problem with a plate $h > 0$, with a regularized kinematic coupling condition, in the sense of Definition \ref{plateweak}, by considering $\hat{\bd{v}} \in \mathcal{X}^{\omega}_{f}$ in the target test space (which is dense in the full test space $\mathcal{V}^{\omega}_{f}$) and considering an approximating sequence of test functions $\hat{\bd{v}}_{N}$ as defined in \eqref{hatvN}, which are admissible test functions for the approximate semidiscrete weak formulation in \eqref{weakBiot1}, integrated over time on each subinterval $[t_{n}, t_{n + 1}]$ and summed over $n = 0, 1, ..., N - 1$. We can then pass to the limit as $N \to \infty$ in each of the terms in the resulting approximate weak formulations using the weak and weak-star convergences in Proposition \ref{uniformenergy}, the strong convergences in Proposition \ref{strongsummary}, and the strong convergence of the test functions $\hat{\bd{v}}_{N}$ to $\hat{\bd{v}}$ in Proposition \ref{testconvergence}. We then obtain the limiting weak formulation in Definition \ref{plateweak} in the limit as $N \to \infty$. For more details on the limit passage in the approximate weak formulations, we refer the reader to Section 9.3 in \cite{FPSIJMPA}, and related works on prototypical incompressible FSI models in Section 7.2 of \cite{MuhaCanic13} and Section 2.7.2 of \cite{CanicLectureNotes}. This establishes the existence of a regularized interface weak solution to the approximate system with a plate of thickness $h > 0$, see Theorem \ref{hthm}.

\section{Passage to the limit in the plate thickness, $h \to 0$}\label{finallimit}

Having obtained existence of weak regularized interface solutions for the approximate problem with a plate of thickness $h > 0$ in the sense of Definition \ref{plateweak}, we pass to the limit in the plate thickness $h \to 0$ with the objective of obtaining a weak (regularized interface) formulation for the coupled moving domain Biot-fluid problem without a plate interface separating the poroelastic Biot medium and the incompressible viscous fluid, in the sense of Definition \ref{regularweak}. Such a problem of a practical importance in applications of fluid-poroelastic structure interaction where the poroelastic medium is in direct contact with the fluid. 

In the previous Section \ref{hproof}, we have established existence of a regularized interface weak solution to a moving domain FPSI problem containing a plate of thickness $h > 0$ separating the Biot medium and the fluid, see Theorem \ref{hthm}. To emphasize the dependence on the thickness of the plate, we denote the solution to the problem with plate thickness $h > 0$ by $(\bd{u}_{h}, \bd{\eta}_{h}, p_{h}, \bd{\omega}_{h})$. We emphasize that the problem will still contain a regularization by $\delta$, and we will leave the parameter $\delta$ implicit in the notation for notational simplicity. While we take the limit as $h \to 0$ in the solutions $(\bd{u}_{h}, \bd{\eta}_{h}, p_{h}, \bd{\omega}_{h})$, we keep the regularization parameter $\delta > 0$ fixed but arbitrary throughout the limit passage as $h \to 0$.

From the existence proof, we have that the weak solutions to the regularized FPSI problem with plate thickness $h > 0$ satisfy the following energy inequality, as a consequence of the weak convergences in Proposition \ref{uniformenergy} and weak lower-semicontinuity:
\begin{multline}\label{energyhineq}
\frac{1}{2} \int_{\Omega_{f}(t)} |\bd{u}_{h}|^{2} + \frac{1}{2} \rho_{b} \int_{\Omega_{b}} |\partial_{t} \bd{\eta}_{h}|^{2} + \frac{1}{2} \int_{\Gamma} \left|h^{1/2} \partial_{t}\bd{\omega}_{h}\right|^{2} + \int_{0}^{t} \int_{\Gamma} \left|h^{1/2} \Delta \bd{\zeta}_{h}\right|^{2} + \frac{1}{2} \int_{\Gamma} \left|h^{1/2} \Delta \bd{\omega}_{h}\right|^{2} + \mu_{e} \int_{\Omega_{b}} |\bd{D}(\bd{\eta}_{h})|^{2} \\
+ \frac{1}{2} \lambda_{e} \int_{\Omega_{b}} |\nabla \cdot \bd{\eta}_{h}|^{2} + \frac{1}{2} c_{0} \int_{\Omega_{b}} |p_{h}|^{2} + 2\nu \int_{0}^{t} \int_{\Omega_{f}(t)} |\bd{D}(\bd{u}_{h})|^{2} + \beta \int_{0}^{t} \int_{\Gamma} \left(\mathcal{J}_{\Gamma}^{\eta_{h}}\right)^{-1} \left|\left(\bd{\zeta}_{h} - \bd{u}_{h}\right) \cdot \bd{\tau}^{\omega_{h}}\right|^{2} \\
+ 2\mu_{v} \int_{0}^{t} \int_{\Omega_{b}} |\bd{D}(\partial_{t}\bd{\eta}_{h})|^{2} + \lambda_{v} \int_{0}^{t} \int_{\Omega_{b}} |\nabla \cdot \partial_{t}\bd{\eta}_{h}|^{2} + \kappa \int_{0}^{t} \int_{\Omega_{b}} \mathcal{J}^{\eta^{\delta}_{h}}_{b} \left|\nabla^{\eta^{\delta}_{h}}_{b} p_{h}\right|^{2} \le E_{0},
\end{multline}
where $E_{0}$ is the initial energy, defined by:
\begin{multline*}
E_{0} := \frac{1}{2} \int_{\Omega_{f}(0)} |\bd{u}_{0}|^{2} + \frac{1}{2}\rho_{b} \int_{\Omega_{b}} |\partial_{t}\bd{\eta}_{0}|^{2} + \frac{1}{2} \int_{\Gamma} |h^{1/2} \partial_{t}\bd{\omega}_{0}|^{2} + \frac{1}{2} \int_{\Gamma} \left|h^{1/2} \Delta \bd{\omega}_{0}\right|^{2} + \mu_{e} \int_{\Omega_{b}} |\bd{D}(\bd{\eta}_{0})|^{2} \\
+ \frac{1}{2} \lambda_{e} \int_{\Omega_{b}} |\nabla \cdot \bd{\eta}_{0}|^{2} + \frac{1}{2} c_{0} \int_{\Omega_{b}} |p_{0}|^{2},
\end{multline*}
and the initial interface displacement $\bd{\omega}_{0}$ is defined by $\bd{\omega}_{0} := (\bd{\eta}_{0})^{\delta}|_{\Gamma}$ in agreement with the regularized kinematic coupling condition.

Before passing to the limit as the plate thickness $h \to 0$, we need some uniform control on geometric quantities involved in the problem. Namely, we want uniform control over the displacement of the fluid-Biot interface (which for $h > 0$ is the plate displacement) to avoid collision with the boundary of $\Omega$ (the entire combined domain) and geometric degeneracy. In addition, we want uniform control on the regularized Lagrangian map $\bd{\Phi}^{\eta^{\delta}}_{b}$ for the Biot domain and the ALE map for the fluid domain $\bd{\Phi}^{\omega}_{f}$, which will allow us to convert between time-dependent and fixed Biot/fluid domains uniformly in the plate thickness parameter $h$. This is the content of the following result, where we will assume without loss of generality that $h \le 1$, as we take the limit $h \to 0$ (compare to the corresponding uniform geometric bounds in $N$ for the $h$-level problem in Proposition \ref{coercivegeometry} and Proposition \ref{geometricN}). 

\begin{proposition}\label{geometryh}
There exists a time $T > 0$ independent of $0 < h \le 1$, such that there exists a regularized interface weak solution to the FPSI problem with plate thickness $h > 0$ on the time interval $[0, T]$. Furthermore, for this $T > 0$, there exist positive constants $R_{max}$, $c_{0}$, $c_{1}$, $c_{2}$, and $c_{3}$ (which are all independent of $0 < h \le 1$) such that the following geometric bounds hold uniformly in $0 < h \le 1$ for all $t \in [0, T]$:
\begin{enumerate}
\item \textbf{Non-degeneracy of the moving interface.} The map $\bd{\Phi}^{\omega_{h}}_{\Gamma}: \Gamma \to \Gamma(t)$ is injective, $\bd{\Phi}^{\omega_{h}}_{\Gamma}(\Gamma) \cap \partial \Omega = \varnothing$ for all $t \in [0, T]$, and for a positive constant $\alpha > 0$ that is independent of $0 < h \le 1$:
\begin{equation}\label{thbound}
|(-\sin(z), \cos(z)) + \partial_{z}\bd{\omega}_{h}(z)| \ge \alpha > 0, \qquad \text{ for all } z \in \Gamma.
\end{equation}
\item \textbf{Invertibility of the regularized Lagrangian map.} There exists positive constants $c_{0}, c_{1}, c_{2} > 0$ such that 
\begin{equation*}
\text{det}(\bd{I} + \nabla \bd{\eta}^{\delta}_{h}) \ge c_{0} > 0, \ \ |(\bd{I} + \nabla \bd{\eta}^{\delta}_{h})^{-1}| \ge c_{1} > 0, \ \  |\bd{I} + \nabla \bd{\eta}^{\delta}_{h}| \le c_{2} \quad \text{ on } \overline{\Omega_{b}}.
\end{equation*}
\item \textbf{Non-degeneracy of the (fluid domain) ALE map.} There exists a positive constant $c_{3} > 0$ such that 
\begin{equation*}
0 < c_{3}^{-1} \le \mathcal{J}^{\omega_{h}}_{f} \le c_{3}, \quad \text{ on } \overline{\Omega_{f}}, \quad |\nabla \bd{\Phi}^{\omega_{h}}_{f}| \le c_{3}, \quad \text{ on } \overline{\Omega_{f}}. 
\end{equation*}
\end{enumerate}
\end{proposition}

\begin{remark}
While we showed existence of a regularized interface weak solution in Section \ref{hproof}, the result in Proposition \ref{geometryh} completes the proof of Theorem \ref{hthm} on the approximate $h$-level problem, by verifying the final claim that the time of existence $T > 0$ of the regularized interface weak solution can be taken to be independent of the plate thickness parameter $0 < h \le 1$.
\end{remark}

\begin{proof}
We begin with the observation that since the regularization parameter $\delta > 0$ is fixed, we have that for every positive integer $k$ and for a constant $C$ independent of $h$ (but potentially depending on $\delta$ and $k$):
\begin{equation}\label{zetak}
\|\bd{\zeta}_{h}\|_{L^{\infty}(0, T; H^{k}(\Gamma))} \le C \|\partial_{t}\bd{\eta}^{\delta}_{h}\|_{L^{\infty}(0, T; H^{k + 1}(\Omega_{b}))} \le C(\delta, k) \|\partial_{t}\bd{\eta}_{h}\|_{L^{\infty}(0, T; L^{2}(\Omega_{b}))} \le C(\delta, k),
\end{equation}
since the uniform energy estimates give a uniform bound on $\partial_{t}\bd{\eta}_{h}$ in $L^{\infty}(0, T; L^{2}(\Omega_{b}))$ that is independent of $h$. We have also used the regularized kinematic coupling condition here, so that $\bd{\zeta}_{h} = \partial_{t}\bd{\eta}^{\delta}_{h}|_{\Gamma}$. Therefore, we have that for $t \in [0, T]$ for $T > 0$ to be determined and for $\bd{\omega}_{0} := \bd{\eta}_{0}|_{\Gamma}$:
\begin{equation}\label{omegahinc}
\|\bd{\omega}_{h}(t, \cdot) - \bd{\omega}_{0}\|_{H^{k}(\Gamma)} \le \int_{0}^{t} \|\bd{\zeta}_{h}(s, \cdot)\|_{H^{k}(\Gamma)} ds \le C(\delta, k) T.
\end{equation}
Therefore, by choosing $T$ sufficiently small and applying Proposition \ref{gammainjective}, we obtain the results on the Lagrangian map for the plate $\bd{\Phi}^{\omega_{h}}_{\Gamma}$. We also obtain the results on the Lagrangian map for the Biot medium, since the above estimate \eqref{zetak} holds for all $k$, and hence, we can estimate for $0 \le t \le T$:
\begin{equation*}
\|\bd{\eta}_{h}^{\delta}(t) - \bd{\eta}^{\delta}_{0}\|_{C^{1}(\Omega_{b})} \le \|\bd{\eta}^{\delta}_{h}(t) - \bd{\eta}^{\delta}_{0}\|_{H^{3}(\Omega_{b})} \le \int_{0}^{t} \|\partial_{t}\bd{\eta}^{\delta}_{h}(s)\|_{H^{3}(\Omega_{b})} ds \le C(\delta) T.
\end{equation*}
Finally, we can estimate the ALE map for the fluid domain using results on traces and the estimate \eqref{omegahinc} applied for $k = 2$ to obtain the following estimate which is independent of $h$:
\begin{equation*}
\|\bd{\Phi}^{\omega_{h}(t)}_{f} - \bd{\Phi}^{\omega_{0}}_{f}\|_{C^{1}(\Omega_{f})} \le C\|\bd{\Phi}^{\omega_{h}(t)}_{f} - \bd{\Phi}^{\omega_{0}}_{f}\|_{H^{5/2}(\Omega_{f})} \le C\|\bd{\omega}_{h}(t) - \bd{\omega}_{0}\|_{H^{2}(\Gamma)} \le C(\delta) T,
\end{equation*}
and hence by continuity, we can obtain the desired result on the ALE fluid domain map. 
\end{proof}

Therefore, we have the following uniform boundedness results, where these bounds are uniform in the plate thickness $h > 0$:

\begin{proposition}\label{uniformh}
For each fixed $\delta > 0$, we have the following bounds which are uniform in the plate thickness $0 < h \le 1$:
\begin{itemize}
\item $\{\bd{u}_{h}\}_{h > 0}$ is uniformly bounded in $L^{\infty}(0, T; L^{2}(\Omega_{f}(t)))$ and $L^{2}(0, T; H^{1}(\Omega_{f}(t)))$.
\item $\{\bd{\eta}_{h}\}_{h > 0}$ is uniformly bounded in $L^{\infty}(0, T; H^{1}(\Omega_{b}))$ and $W^{1, \infty}(0, T; L^{2}(\Omega_{b}))$. 
\item $\{p_{h}\}_{h > 0}$ is uniformly bounded in $L^{\infty}(0, T; L^{2}(\Omega_{b}))$ and $L^{2}(0, T; H^{1}(\Omega_{b}))$. 
\item $\{\bd{\omega}_{h}\}_{h > 0}$ is uniformly bounded in $L^{\infty}(0, T; H^{k}(\Gamma))$ for all positive integers $k$.
\item $\{\bd{\zeta}_{h}\}_{h > 0}$ is uniformly bounded in $L^{\infty}(0, T; H^{k}(\Gamma))$ for all positive integers $k$. 
\item $\{h^{1/2}\bd{\omega}_{h}\}_{h > 0}$ is uniformly bounded in $W^{1, \infty}(0, T; L^{2}(\Gamma))$.
\item $\{h^{1/2}\bd{\omega}_{h}\}_{h > 0}$ is uniformly bounded in $L^{\infty}(0, T; H^{2}(\Gamma))$. 
\item $\{h^{1/2}\bd{\zeta}_{h}\}_{h > 0}$ is uniformly bounded in $L^{2}(0, T; H^{2}(\Gamma))$. 
\end{itemize}
Hence, along a subsequence in $h \to 0$, we have the following convergences:
\begin{align*}
\bd{\eta}_{h} \rightharpoonup \bd{\eta}, \qquad &\text{weakly-star in $L^{\infty}(0, T; H^{1}(\Omega_{b}))$ and weakly-star in $W^{1, \infty}(0, T; L^{2}(\Omega_{b}))$}, \\
p_{h} \rightharpoonup p, \qquad &\text{weakly-star in $L^{\infty}(0, T; L^{2}(\Omega_{b}))$ and weakly in $L^{2}(0, T; H^{1}(\Omega_{b}))$}. \\
\end{align*}
\end{proposition}

\begin{proof}
This follows from the energy estimate in \eqref{energyhineq} which holds uniformly in $h$, since the initial energy $E_{0}$ is uniformly bounded in $0 < h \le 1$. We, however, emphasize that this result can only be obtained from the energy estimate \eqref{energyhineq} as a result of the previous Proposition \ref{geometryh}, since some of the uniform estimates are on the moving Biot domain, whereas the uniform estimates above are given on the reference Biot domain. As an example, the energy estimate implies that there exists a constant $C$ independent of $h$ such that for all $0 < h \le 1$:
\begin{equation*}
\int_{0}^{T} \int_{\Omega_{b}} \mathcal{J}^{\eta^{\delta}_{h}}_{b} \left|\nabla^{\eta^{\delta}_{h}}_{b} p_{h}\right|^{2} \le C.
\end{equation*}
Since this estimate involves $\mathcal{J}^{\eta^{\delta}_{h}}_{b} = \text{det}(\bd{I} + \nabla \bd{\eta}^{\delta}_{h})$ and the transformed gradient defined in \eqref{etanabla}:
\begin{equation*}
\nabla^{\eta^{\delta}_{h}}_{b} p_{h} = \nabla p_{h} \cdot (\bd{I} + \nabla \bd{\eta}^{\delta}_{h})^{-1},
\end{equation*}
we can only obtain an estimate for $\displaystyle \int_{0}^{T} \int_{\Omega_{b}} |\nabla p_{h}|^{2}$ that is uniform in $0 < h \le 1$ using the previous geometric bounds in Proposition \ref{geometryh}, which hold uniformly in $0 < h \le 1$ and $t \in [0, T]$, see \eqref{kappabound}.

We remark that the uniform bounds on the plate displacements and velocities $\{\bd{\omega}_{h}\}_{h > 0}$ and $\{\bd{\zeta}_{h}\}_{h > 0}$ follow from the uniform bounds on the Biot displacements $\{\bd{\eta}_{h}\}_{h > 0}$ in $W^{1, \infty}(0, T; L^{2}(\Omega_{b}))$ and the regularized kinematic coupling condition $\bd{\omega}_{h} = \bd{\eta}^{\delta}_{h}|_{\Gamma}$, where we use the regularizing properties of spatial convolution and the trace inequality. 
\end{proof}

With the uniform geometric bounds in Proposition \ref{geometryh}, we can similarly show the following uniform bounds on the trace of the fluid velocities and on the fluid velocities transferred to the fixed reference domain via the ALE map:
\begin{equation}\label{hatuh}
\hat{\bd{u}}_{h} := \bd{u}_{h} \circ \bd{\Phi}^{\omega_{h}}_{f}, \qquad \text{ on } \Omega_{f}.
\end{equation}

\begin{proposition}\label{tracehuniform}
The fluid velocities $\hat{\bd{u}}_{h}$ defined in \eqref{hatuh} are uniformly bounded in $L^{\infty}(0, T; L^{2}(\Omega_{f})) \cap L^{2}(0, T; H^{1}(\Omega_{f}))$ and the traces of the fluid velocities $\bd{u}_{h}|_{\Gamma}$ are uniformly bounded in $L^{2}(0, T; L^{2}(\Gamma))$, independently of $0 < h \le 1$. 
\end{proposition}

\subsection{Compactness arguments in the limit as $h \to 0$}

\noindent \textbf{Function spaces.} Because the weak formulation for the solutions with plate thickness $h$ now satisfy a fully continuous-in-time weak formulation in Definition \ref{plateweak} rather than a semidiscrete weak formulation, we must work with different function spaces, though the underlying ideas behind the compactness arguments remain the same, as for the $N \to \infty$ limit passage in Section \ref{compactN}. We define these function spaces here, and identify the appropriate compact embedding result that we will invoke in the singular limit passage as $h \to 0$. For a Banach space $B$ with associated norm $\|\cdot\|_{B}$, we define the fractional Sobolev space $W^{\alpha, p}(0, T; B)$ to be the set of all functions $f \in L^{p}(0, T; B)$ such that the following associated norm is finite:
\begin{equation}\label{fractionalnorm}
\|f\|_{W^{\alpha, p}(0, T; B)}^{p} = \|f\|_{L^{p}(0, T; B)}^{p} + \int_{0}^{T} \int_{0}^{T} \frac{\|f(t) - f(s)\|_{B}}{|t - s|^{1 + \alpha p}} ds dt.
\end{equation}
We then obtain the following compact embedding result, which can be found in Theorem 2.1 in \cite{FlandoliGatarek}.

\begin{proposition}\label{compactness}
Let $B_{0} \subset \subset B \subset B_{1}$ be reflexive Banach spaces, and let $0 < \alpha < 1$ and $1 < p < \infty$. Then,
\begin{equation*}
L^{p}(0, T; B_{0}) \cap W^{\alpha, p}(0, T; B_{1}) \subset \subset L^{p}(0, T; B).
\end{equation*}
\end{proposition}

\medskip

\noindent \textbf{Compactness of the Biot displacement.} We can use Arzela-Ascoli to obtain the following convergence result for the Biot displacements, and by using the regularized interface condition, we can also obtain a corresponding convergence result for the interface displacements.

\begin{proposition}\label{etahconv}
There exists a limiting Biot displacement $\bd{\eta}$ and a corresponding interface displacement $\bd{\omega} := \bd{\eta}^{\delta}|_{\Gamma}$, for which we have the following convergence results as $h \to 0$, along a subsequence:
\begin{equation*}
\bd{\eta}_{h} \to \bd{\eta}, \quad \text{ in } C(0, T; L^{2}(\Omega_{b})), \qquad \bd{\omega}_{h} \to \bd{\omega} \quad \text{ in } C(0, T; H^{k}(\Gamma)),
\end{equation*}
for all positive integers $k$.
\end{proposition}

\begin{proof}
This follows from the uniform boundedness results (uniform in $h$), the following compact embedding $L^{\infty}(0, T; H^{1}(\Omega_{b})) \cap W^{1, \infty}(0, T; L^{2}(\Omega_{b})) \subset \subset C(0, T; L^{2}(\Omega_{b}))$, and the regularized interface condition $\bd{\omega}_{h} = \bd{\eta}^{\delta}_{h}|_{\Gamma}$ combined with the regularizing properties of spatial convolution.
\end{proof}

\medskip

\noindent \textbf{Compactness for the Biot pore pressure.} By taking only the pressure test function to be nonzero in Definition \ref{plateweak}, we obtain that the pore pressure $p_{h}$ must satisfy the following weak formulation for all times $0 \le s \le t \le T$, for all test functions $r \in V_{p}$, where $V_{p}$ is defined in \eqref{Vp}:
\begin{multline}\label{poreincrement}
c_{0} \int_{\Omega_{b}} (p_{h}(t) - p_{h}(s)) r = \alpha \int_{s}^{t} \int_{\Omega_{b}} \mathcal{J}^{\eta_{h}^{\delta}}_{b} \partial_{t} \bd{\eta}^{\delta}_{h} \cdot \nabla^{\eta^{\delta}_{h}}_{b} r + \alpha \int_{s}^{t} \int_{\Gamma} (\bd{\zeta}_{h} \cdot \bd{n}^{\omega_{h}}) r \\
- \kappa \int_{s}^{t} \int_{\Omega_{b}} \mathcal{J}^{\eta^{\delta}_{h}}_{b} \nabla^{\eta^{\delta}_{h}}_{b} p_{h} \cdot \nabla^{\eta^{\delta}_{h}}_{b} r + \int_{s}^{t} \int_{\Gamma} ((\bd{u}_{h} - \bd{\zeta}_{h}) \cdot \bd{n}^{\omega_{h}}) r.
\end{multline}
Using this weak formulation, we can obtain the following strong convergence result for the pore pressures $p_{h}$ in the limit as $h \to 0$. 

\begin{proposition}\label{phconv}
Along a subsequence in $h$, $p_{h} \to p$ in $L^{2}(0, T; L^{2}(\Omega_{b}))$. 
\end{proposition}

\begin{proof}
We use the compactness criterion stated in Proposition \ref{compactness}. We first note that due to Proposition \ref{uniformh}, there exists a uniform constant $C$ that is independent of $h$ such that
\begin{equation*}
\|p_{h}\|_{L^{2}(0, T; H^{1}(\Omega_{b}))} \le C.
\end{equation*}

Next, we claim that 
\begin{equation}\label{incrementph}
\|p_{h}(t) - p_{h}(s)\|_{V_{p}'} \le C|t - s|^{1/2}, \text{ for all $s, t \in [0, T]$},
\end{equation}
for a constant $C$ that is independent of $h$. To obtain this increment estimate, we can use the weak formulation for the increments of the pore pressure $p_{h}(t) - p_{h}(s)$ in \eqref{poreincrement} and test by $r \in V_{p}$ with $\|r\|_{H^{1}(\Omega_{b})} \le 1$. We can estimate the resulting terms as follows:
\begin{itemize}
\item For the first term, we observe that $\|\partial_{t}\bd{\eta}^{\delta}_{h}\|_{L^{\infty}(0, T; L^{2}(\Omega_{b}))}$ is uniformly bounded in $h$, and by the geometric bounds (uniform in $h$) in Proposition \ref{geometryh} on $\mathcal{J}^{\eta^{\delta}_{h}}_{b}$ and $|(\bd{I} + \nabla \bd{\eta}^{\delta}_{h})^{-1}|$ and the assumption that $\|\nabla r\|_{L^{2}(\Omega_{b})} \le 1$, we deduce that $\displaystyle \left|\int_{s}^{t} \int_{\Omega_{b}} \mathcal{J}^{\eta^{\delta}_{h}}_{b} \partial_{t} \bd{\eta}^{\delta}_{h} \cdot \nabla^{\eta^{\delta}_{h}}_{b} r\right| \le C|t - s|$. 

\item The second term is also bounded by $\displaystyle \left|\int_{s}^{t} \int_{\Gamma} (\bd{\zeta}_{h} \cdot \bd{n}^{\omega_{h}}) r\right| \le C|t - s|$ once we make the following observations. First, $\|\bd{\zeta}_{h}\|_{L^{\infty}(0, T; L^{2}(\Gamma))}$ is uniformly bounded in $h$ by Proposition \ref{uniformh} and $\|r\|_{L^{2}(\Gamma)} \le C$ by the trace inequality. Since $|\partial_{z}\bd{\omega}_{h}| \le C$ pointwise along $\Gamma$ by the uniform high regularity bounds on $\bd{\omega}_{h}$ in $L^{\infty}(0, T; H^{k}(\Omega))$ for all positive integers $k$, we deduce that $|\bd{n}^{\omega_{h}}| \le C$ pointwise along $\Gamma$ also, where the rescaled normal vector is defined in \eqref{ntomega}, which gives the desired estimate. 

\item By the geometric bounds on $\mathcal{J}^{\eta^{\delta}_{h}}_{b}$ and $|(\bd{I} + \nabla \bd{\eta}^{\delta}_{h})^{-1}|$ in Proposition \ref{geometryh}, the uniform bound on $\|\nabla p_{h}\|_{L^{2}(0, T; L^{2}(\Omega_{b}))}$, and the fact that $\|\nabla r\|_{L^{2}(\Omega_{b})} \le 1$, we obtain
\begin{equation*}
\left|\int_{s}^{t} \int_{\Omega_{b}} \mathcal{J}^{\eta^{\delta}_{h}}_{b} p_{h} \cdot \nabla^{\eta^{\delta}_{h}}_{b} r\right| \le C|t - s|^{1/2}.
\end{equation*}
\item Finally, because $\bd{u}_{h}|_{\Gamma}$ is bounded uniformly in $h$ in $L^{2}(0, T; L^{2}(\Gamma))$ by Proposition \ref{tracehuniform}, and because $\bd{\zeta}_{h}$ is bounded uniformly in $h$ in $L^{2}(0, T; L^{2}(\Gamma))$ by Proposition \ref{uniformh}, and $|\bd{n}^{\omega_{h}}| \le C$ pointwise along $\Gamma$, we obtain the final bound that 
\begin{equation*}
\left|\int_{s}^{t} \int_{\Gamma} ((\bd{u}_{h} - \bd{\zeta}_{h}) \cdot \bd{n}^{\omega_{h}}) r\right| \le C|t - s|^{1/2}.
\end{equation*}
\end{itemize}
Combining these estimates gives the desired increment estimate \eqref{incrementph}, from which we deduce that for any choice of $0 < \alpha < 1$ and $1 < p < \infty$ such that $\alpha p < 1/2$:
\begin{equation*}
\|p_{h}\|_{W^{\alpha, p}(0, T; V_{p}')}^{p} \le \|p_{h}\|^{p}_{L^{p}(0, T; V_{p}')} + \int_{0}^{T} \int_{0}^{T} |t - s|^{-(\alpha p + 1/2)} ds dt < \infty,
\end{equation*}
since $\|p_{h}\|_{L^{p}(0, T: V_{p}')} \le \|p_{h}\|_{L^{p}(0, T: L^{2}(\Omega_{b}))} \le T^{1/p} \|p_{h}\|_{L^{\infty}(0, T; L^{2}(\Omega_{b}))}$. The strong convergence result then follows from Proposition \ref{compactness} using the chain of embeddings $H^{1}(\Omega_{b}) \subset \subset L^{2}(\Omega_{b}) \subset V_{p}'$. 
\end{proof}

\medskip

\noindent \textbf{Compactness for the Biot/interface velocity.} Next, we consider the limit of the Biot velocity and the plate velocity $(\bd{\xi}_{h}, \bd{\zeta}_{h})$ for $\bd{\xi}_{h} := \partial_{t}\bd{\eta}_{h}$. In the limit as $h \to 0$, there is no plate in the resulting problem, and we will hence instead refer to $\bd{\zeta}_{h}$ as the \textit{interface velocity}, since in the limit as $h \to 0$, $\bd{\zeta}_{h}$ will converge to the velocity of the regularized interface between the Biot medium and the fluid. We consider the following test space, as before in \eqref{Qv}:
\begin{equation*}
\mathcal{Q}_{v} = \{(\bd{\psi}, \bd{\varphi}) \in V_{d} \times H^{2}(\Gamma) : \bd{\psi}^{\delta}|_{\Gamma} = \bd{\varphi}\},
\end{equation*}
and we note (by using Definition \ref{plateweak}) that the Biot and interface velocities $(\bd{\xi}_{h}, \bd{\zeta}_{h})$ satisfy the following weak formulation for all test functions $(\bd{\psi}, \bd{\varphi}) \in \mathcal{Q}_{v}$ and for all times $0 \le s \le t \le T$:
\begin{small}
\begin{multline*}
\rho_{b} \int_{\Omega_{b}} (\bd{\xi}_{h}(t) - \bd{\xi}_{h}(s)) \cdot \bd{\psi} + h\int_{\Gamma} (\bd{\zeta}_{h}(t) - \bd{\zeta}_{h}(s)) \cdot \bd{\varphi} = \int_{s}^{t} \int_{\Gamma} \left(p_{h} - \frac{1}{2}|\bd{u}_{h}|^{2}\right) (\bd{\psi}^{\delta} \cdot \bd{n}^{\omega_{h}}) \\
+ \beta \int_{s}^{t} \int_{\Gamma} \Big(\mathcal{S}^{\omega_{h}}_{\Gamma}\Big)^{-1} (\bd{u}_{h} - \bd{\zeta}_{h}) \cdot \bd{\tau}^{\omega_{h}} \left(\bd{\psi}^{\delta} \cdot \bd{\tau}^{\omega_{h}}\right) - 2\mu_{e} \int_{s}^{t} \int_{\Omega_{b}} \bd{D}(\bd{\eta}_{h}) : \bd{D}(\bd{\psi}) - \lambda_{e} \int_{s}^{t} \int_{\Omega_{b}} (\nabla \cdot \bd{\eta}_{h}) (\nabla \cdot \bd{\psi}) \\
- 2\mu_{v} \int_{s}^{t} \int_{\Omega_{b}} \bd{D}(\bd{\xi}_{h}) : \bd{D}(\bd{\psi}) - 2\lambda_{v} \int_{s}^{t} \int_{\Omega_{b}} (\nabla \cdot \bd{\xi}_{h}) (\nabla \cdot \bd{\psi}) + \alpha \int_{s}^{t} \int_{\Omega_{b}} \mathcal{J}^{\eta^{\delta}_{h}}_{b} p_{h} \nabla^{\eta^{\delta}_{h}}_{b} \cdot \bd{\psi}^{\delta} + h \int_{s}^{t} \int_{\Gamma} \Delta \bd{\zeta}_{h} \cdot \Delta \bd{\varphi} + h \int_{s}^{t} \int_{\Gamma} \Delta \bd{\omega}_{h} \cdot \Delta \bd{\varphi}.
\end{multline*}
\end{small}
We have the following convergence result for the Biot velocity and interface velocity.
\begin{proposition}\label{velhconv}
Along a subsequence, $\bd{\xi}_{h} \to \bd{\xi}$ in $L^{2}(0, T; H^{-s}(\Omega_{b}))$ for $-1 < s < 0$ and $\bd{\zeta}_{h} \to \bd{\xi}^{\delta}|_{\Gamma}$ in $L^{2}(0, T; H^{k}(\Gamma))$ for all positive integers $k$. Furthermore, we can explicitly identify $\bd{\xi} = \partial_{t}\bd{\eta}$. 
\end{proposition}
\begin{proof}
Consider the ordered pair of Biot velocities and plate velocities $(\bd{\xi}_{h}, h\bd{\zeta}_{h})$ in the parameter $h > 0$, where we multiply the plate velocity by $h$ due to the weak formulation above. Note that $(\bd{\xi}_{h}, h\bd{\zeta}_{h})$ are uniformly bounded independently of $h$ in $L^{2}(0, T; L^{2}(\Omega_{b}) \times L^{2}(\Gamma))$ by Proposition \ref{uniformh} (and in fact $\|h\bd{\zeta}_{h}\|_{L^{2}(0, T; L^{2}(\Gamma))} \to 0$ as $h \to 0$). Furthermore, from the weak formulation, we can show that $(\bd{\xi}_{h}, h\bd{\zeta}_{h})$ is uniformly bounded in $W^{\alpha, 2}(0, T; \mathcal{Q}_{v}')$ for an appropriate choice of $\alpha > 0$, by establishing an increment estimate. To do this, we estimate each of the terms of the weak formulation for the Biot and plate velocity for a test function $(\bd{\psi}, \bd{\varphi}) \in \mathcal{Q}_{v}$, with $\|\bd{\psi}\|_{H^{1}(\Omega_{b})} \le 1$ and $\|\bd{\varphi}\|_{H^{2}_{0}(\Gamma)} \le 1$, as follows:
\begin{itemize}
\item Note that by the trace inequality, $p_{h}$ is uniformly bounded in $L^{2}(0, T; L^{2}(\Gamma))$ and by the trace inequality combined with the properties of spatial convolution, $\|\bd{\psi}^{\delta}\|_{L^{\infty}(\Gamma)} \le C$. Furthermore, we have the estimate $|\bd{n}^{\omega_{h}}| \le C$ uniformly in $h$, where the rescaled normal vector is defined in \eqref{ntomega}. We estimate:
\begin{align}\label{inc14}
\left|\int_{s}^{t} \int_{\Gamma} \left(p_{h} - \frac{1}{2}|\bd{u}_{h}|^{2}\right) (\bd{\psi}^{\delta} \cdot \bd{n}^{\omega^{\delta}_{h}})\right| &\le C \int_{s}^{t}\left(\|p_{h}\|_{L^{2}(\Gamma)} + \|\bd{u}_{h}\|_{L^{2}(\Gamma)}^{2}\right) \nonumber \\
&\le C \int_{s}^{t} \left(\|p_{h}\|_{L^{2}(\Gamma)} + \|\hat{\bd{u}}_{h}\|^{2}_{H^{3/4}(\Omega_{f})}\right) \le C|t - s|^{1/4},
\end{align}
where we used the fact that $\hat{\bd{u}}_{h} := \bd{u}_{h} \circ \bd{\Phi}^{\omega_{h}}_{f} \in L^{\infty}(0, T; L^{2}(\Omega_{f})) \cap L^{2}(0, T; H^{1}(\Omega_{f}))$ and hence is in $L^{8/3}(0, T; H^{3/4}(\Omega_{f}))$ uniformly in $h$, by the estimates in Proposition \ref{tracehuniform} and the geometric bounds in Proposition \ref{geometryh}.
\item Next, we use the facts that $|\bd{\tau}^{\omega^{\delta}_{h}}| \le C$ uniformly in $h$ where the rescaled tangent vector is defined in \eqref{ntomega}, the geometric estimate $\mathcal{J}^{\omega^{\delta}_{h}}_{\Gamma} \ge C$ by the definition in \eqref{ntdef2} and the uniform bound \eqref{thbound}, the assumption on the test function that $\|\bd{\psi}\|_{H^{1}(\Omega_{f})} \le 1$, and the trace inequality to obtain:
\begin{multline*}
\left|\beta \int_{s}^{t} \int_{\Gamma} \Big(\mathcal{S}^{\omega_{h}}_{\Gamma}\Big)^{-1} (\bd{u}_{h} - \bd{\zeta}_{h}) \cdot \bd{\tau}^{\omega_{h}} \left(\bd{\psi}^{\delta} \cdot \bd{\tau}^{\omega_{h}}\right)\right| \\
\le C\int_{s}^{t} \left(\|\bd{\zeta}_{h}\|_{L^{2}(\Gamma)} + \|\bd{u}_{h}\|_{L^{2}(\Gamma)}\right) \le C\int_{s}^{t} \left(\|\bd{\zeta}_{h}\|_{L^{2}(\Gamma)} + \|\hat{\bd{u}}_{h}\|_{H^{1}(\Omega_{f})}\right) \le C|t - s|^{1/2},
\end{multline*}
by the uniform bound of $\hat{\bd{u}}_{h}$ in $L^{2}(0, T; H^{1}(\Omega_{f}))$ in Proposition \ref{tracehuniform} and the uniform bound of $\bd{\zeta}_{h}$ in $L^{2}(0, T; H^{k}(\Gamma))$ for all positive integers $k$, see Proposition \ref{uniformh}. 
\item Next, we estimate using the uniform pointwise geometric bounds on $\mathcal{J}^{\eta^{\delta}_{h}}_{b}$ and $|(\bd{I} + \nabla \bd{\eta}^{\delta}_{h})^{-1}|$ in Proposition \ref{geometryh}, that
\begin{equation*}
\left|\int_{s}^{t} \int_{\Omega_{b}} \mathcal{J}^{\eta^{\delta}_{h}}_{b} p_{h} \nabla^{\eta^{\delta}_{h}}_{b} \cdot \bd{\psi}^{\delta}\right| \le C \int_{s}^{t} \|p_{h}\|_{L^{2}(\Omega_{b})} \le C|t - s|.
\end{equation*}
\end{itemize}
Hence, we deduce the following increment estimate:
\begin{equation*}
\|(\bd{\xi}_{h}, h\bd{\zeta}_{h})(t) - (\bd{\xi}_{h}, h\bd{\zeta}_{h})(s)\|_{\mathcal{Q}_{v}'} \le C|t - s|^{1/4},
\end{equation*}
and by the definition of the fractional Sobolev space in \eqref{fractionalnorm}, for any $\alpha$ such that $0 < \alpha < 1/8$, we obtain that
\begin{equation*}
\|(\bd{\xi}_{h}, h\bd{\zeta}_{h})\|_{W^{\alpha, 2}(0, T; \mathcal{Q}_{v}')} \le C, \quad \text{ uniformly in $h$}.
\end{equation*}
Therefore, using the compactness result in Proposition \ref{compactness} along with the chain of embeddings $L^{2}(\Omega_{b}) \times L^{2}(\Gamma) \subset \subset H^{-s}(\Omega_{b}) \times H^{-s}(\Gamma) \subset \mathcal{Q}_{v}'$, we obtain that $(\bd{\xi}_{h}, h\bd{\zeta}_{h})$ converges strongly in $L^{2}(0, T; L^{2}(\Omega_{b}) \times L^{2}(\Gamma))$ to some limiting ordered pair of functions.

We will hence denote the limit of $\bd{\xi}_{h}$ in $L^{2}(0, T; L^{2}(\Omega_{b}))$ by $\bd{\xi}$, and since $\bd{\xi}_{h} \rightharpoonup \partial_{t}\bd{\eta}$ weakly-star in $L^{\infty}(0, T; L^{2}(\Omega_{b}))$ by the weak-star convergence of $\bd{\eta}_{h} \rightharpoonup \bd{\eta}$ in $W^{1, \infty}(0, T; L^{2}(\Omega_{b}))$, we must have that $\bd{\xi} = \partial_{t}\bd{\eta}$ by uniqueness properties of the limit. By the regularized kinematic coupling condition, $\bd{\zeta}_{h} = \bd{\xi}_{h}^{\delta}|_{\Gamma}$, so by the regularizing properties of spatial convolution and the trace inequality, we obtain from the strong convergence $\bd{\xi}_{h} \to \bd{\xi}$ in $L^{2}(0, T; L^{2}(\Gamma))$ that $\bd{\zeta}_{h} \to \bd{\xi}^{\delta}|_{\Gamma}$ in $L^{2}(0, T; H^{k}(\Gamma))$ for any positive integer $k$. 
\end{proof}

\medskip

\noindent \textbf{Compactness arguments for the fluid velocity.} For the fluid velocities, we must handle the issue of moving domains when obtaining strong convergence results as in Proposition \ref{fluidcompact}, since the approximate fluid velocities $\bd{u}_{h}$ are all defined on different moving fluid domains. However, this will be quite simple compared to the semidiscrete case in the existence proof for the $h$-level problem in Proposition \ref{fluidcompact}. The main difficulty in the case of the semidiscrete scheme in Proposition \ref{fluidcompact} is that the time-discrete fluid velocities satisfy different weak formulations on each time interval, due to the interface update being staggered from the fluid/Biot update, which leads to requiring us to estimate the difference between consecutive ALE fluid maps $\bd{\Phi}^{\omega^{n}_{N}}_{f}$ and $\bd{\Phi}^{\omega^{n + 1}_{N}}_{f}$, see \eqref{tildeuN} and \eqref{fluidnN}. However, because our approximate solutions $\bd{u}_{h}$ in the limit passage as $h \to 0$ are time-continuous (rather than discretized in time), the compactness arguments for the fluid velocity are simpler, as we will be able to use an extension to by zero to express the weak formulation for $\bd{u}_{h}$ on a common maximal fluid domain, where we can use an increment estimate and compactness via fractional Sobolev spaces (Proposition \ref{compactness}) to conclude strong convergence. 

To do this, we extend the fluid velocities $\bd{u}_{h}$ from $\Omega_{h}(t) := \bd{\Phi}^{\omega_{h}(t)}_{f}(\Omega_{f})$ to the fixed maximal domain $\Omega := \{\bd{x} \in \R^{2} : |\bd{x}| < 2\}$, as was done in the proof of existence for fixed $h$ when passing to the limit as $N \to \infty$, see the discussion preceding Proposition \ref{compactfluid} and its statement. This is possible due to the non-degeneracy results on $\bd{\Phi}^{\omega_{h}}_{f}(\Gamma)$ stated in Proposition \ref{geometryh}. Because the fluid velocity in particular is decoupled from the plate velocity, we can treat the fluid velocities as functions on the maximal domain and directly estimate the fluid velocities in a dual norm for a fixed function space, in order to appeal to the compactness result in Proposition \ref{compactness} to directly obtain strong convergence. We emphasize that having the fluid velocity function decoupled from the other parts of the solution greatly simplifies this limit procedure, in contrast to for example having the fluid velocity kinematically coupled to the plate velocity via a kinematic coupling condition (such as the no-slip condition), as in past works on FSI in \cite{AubinLions,CanicLectureNotes, MS22}. 

We hence obtain the following strong convergence result for the fluid velocities $\bd{u}_{h}$ defined on the fixed maximal fluid domain:
\begin{proposition}
Along a subsequence in $h$, $\bd{u}_{h} \to \bd{u}$ in $L^{2}(0, T; L^{2}(\Omega))$ for some limiting function $\bd{u}$ that is equal to zero outside of the limiting moving fluid domain $\Omega_{f}(t)$ defined by the regularized interface displacement $\bd{\eta}^{\delta}|_{\Gamma}$ and the corresponding moving interface $\Gamma^{\eta^{\delta}}(t)$. 
\end{proposition}

\begin{proof}
We will appeal to the compactness result in Proposition \ref{compactness}, where we first note that the fluid velocities $\bd{u}_{h}$ extended to the maximal fluid domain $\Omega$ are uniformly bounded in $L^{2}(0, T; H^{s}(\Omega))$ for any $s$ such that $0 < s < 1/2$, due to the uniform bound of $\bd{u}_{h} \in L^{2}(0, T; H^{1}(\Omega_{f})$ and properties of extension by zero on Lipschitz domains (see Theorem 8.2 in \cite{KuanTawri}). We then define the following common test space on the maximal fluid domain:
\begin{equation*}
\mathcal{Q}_{M} = \{\bd{v} \in H^{1}(\Omega) : \nabla \cdot \bd{v} = 0 \text{ on } \Omega \text{ and } \bd{v} = 0 \text{ on } \partial \Omega_{f} \setminus \Gamma\}.
\end{equation*}
We claim that the fluid velocities $\bd{u}_{h}$ on the maximal fluid domain are bounded in $W^{\alpha, 2}(0, T; \mathcal{Q}_{M}')$ for some $\alpha > 0$, uniformly in the plate thickness parameter $h$. To show this, we note that for any $\bd{v} \in \mathcal{Q}_{M}$, the restriction to the moving fluid domains $1_{\Omega_{f, h}(t)} \bd{v}$ are admissible test functions for the weak formulation with parameter $h > 0$, and hence using Definition \ref{plateweak}, for any (spacetime) test function $\bd{v} \in C_{c}^{1}([0, T); \mathcal{Q}_{M})$, the following weak formulation for $\bd{u}_{h}$ holds on the fixed maximal domain:
\begin{small}
\begin{multline}\label{fluidhweak}
-\int_{0}^{T} \int_{\Omega} \bd{u}_{h} \cdot \partial_{t} \bd{v} - \int_{0}^{T} \int_{\Omega} \bd{u}_{h} \cdot (\bd{u}_{h} \cdot \nabla) \bd{v} + \int_{0}^{T} \int_{\Gamma(t)} \left(\bd{u}_{h} - \bd{\zeta}_{h} \right) \cdot \bd{n} (\bd{u}_{h} \cdot \bd{v}) + 2\nu \int_{0}^{T} \int_{\Omega_{f}(t)} \bd{D}(\bd{u}_{h}) : \bd{D}(\bd{v}) \\
- \int_{0}^{T} \int_{\Gamma(t)} \left(\frac{1}{2}|\bd{u}_{h}|^{2} - p_{h}\right) (\bd{v} \cdot \bd{n}) - \beta \int_{0}^{T} \int_{\Gamma(t)} (\bd{\zeta}_{h} - \bd{u}_{h}) \cdot \bd{\tau} (\bd{v} \cdot \bd{\tau}) = \int_{\Omega} \bd{u}_{0} \cdot \bd{v}(0).
\end{multline}
\end{small}
This follows by the support properties of the function $\bd{u}_{h}$ and the extension by zero, where we note that we also fully transferred the advective derivative in the second term of \eqref{fluidhweak} to the test function using integration by parts. Hence, using the weak continuity of $\bd{u}_{h}$, we obtain that for all (spatial) test functions $\bd{v} \in \mathcal{Q}_{M}$ and for all times $0 \le s \le t \le T$:
\begin{small}
\begin{multline*}
\int_{\Omega} (\bd{u}_{h}(t) - \bd{u}_{h}(s)) \cdot \bd{v} = \int_{s}^{t} \int_{\Omega} \bd{u}_{h} \cdot (\bd{u}_{h} \cdot \nabla) \bd{v} + \int_{s}^{t} \int_{\Gamma(t)} (\bd{\zeta}_{h} - \bd{u}_{h}) \cdot \bd{n} (\bd{u}_{h} \cdot \bd{v}) \\
- 2\nu \int_{s}^{t} \int_{\Omega_{f}(t)} \bd{D}(\bd{u}_{h}) : \bd{D}(\bd{v}) + \int_{s}^{t} \int_{\Gamma(t)} \left(\frac{1}{2} |\bd{u}_{h}|^{2} - p_{h}\right) (\bd{v} \cdot \bd{n}) + \beta \int_{s}^{t} \int_{\Gamma(t)} (\bd{\zeta}_{h} - \bd{u}_{h}) \cdot \bd{\tau} (\bd{v} \cdot \bd{\tau}). 
\end{multline*}
\end{small}
We estimate the terms for $\|\bd{v}\|_{Q_{M}} \le 1$ (so that $\|\bd{v}\|_{H^{1}(\Omega)} \le 1$) as follows:
\begin{itemize}
\item We estimate using geometric bounds on the Jacobian of the fluid ALE map (Proposition \ref{geometryh}):
\begin{align*}
\left|\int_{s}^{t} \int_{\Omega} \bd{u}_{h} \cdot (\bd{u}_{h} \cdot \nabla) \bd{v}\right| &\le \int_{s}^{t} \|\bd{u}_{h}\|_{L^{4}(\Omega_{f}(t))}^{2} \le C \int_{s}^{t} \|\hat{\bd{u}}_{h}\|_{L^{4}(\Omega_{f})}^{2} \\
&\le C\int_{s}^{t} \|\hat{\bd{u}}_{h}\|_{H^{1/2}(\Omega_{f})}^{2} \le \int_{s}^{t} \left(\|\hat{\bd{u}}_{h}\|_{L^{2}(\Omega_{f}(t))} \|\hat{\bd{u}}_{h}\|_{H^{1}(\Omega_{f}(t))}\right) \le C|t - s|^{1/2},
\end{align*}
since $\hat{\bd{u}}_{h}$ is uniformly bounded in $L^{\infty}(0, T; L^{2}(\Omega_{f})) \cap L^{2}(0, T; H^{1}(\Omega_{f}))$ by Proposition \ref{tracehuniform}.
\item Next, since $\|\bd{v}\|_{H^{1}(\Omega_{f})} \le 1$, by the trace inequality and Sobolev embedding, $\|\bd{v}\|_{L^{4}(\Gamma)} \le 1$. Therefore,
\begin{multline*}
\left|\int_{s}^{t} \int_{\Gamma(t)} (\bd{\zeta}_{h} - \bd{u}_{h}) \cdot \bd{n} (\bd{u}_{h} \cdot \bd{v})\right| \le C\int_{s}^{t} \left(\|\bd{\zeta}_{h}\|_{L^{2}(\Gamma)} + \|\bd{u}_{h}\|_{L^{4}(\Gamma)}\right) \|\bd{u}_{h}\|_{L^{4}(\Gamma)} \\
\le C\int_{s}^{t} \left(1 + \|\bd{u}_{h}\|_{L^{4}(\Gamma)}^{2}\right) \le C \int_{s}^{t} \Big(1 + \|\bd{u}_{h}\|_{H^{1/4}(\Gamma)}^{2}\Big) \le C \int_{s}^{t} \Big(1 + \|\bd{u}_{h}\|_{H^{3/4}(\Omega_{f})}^{2}\Big) \le C|t - s|^{1/4}
\end{multline*}
since $\hat{\bd{u}}_{h}$ is uniformly bounded in $L^{\infty}(0, T; L^{2}(\Omega_{f})) \cap L^{2}(0, T; H^{1}(\Omega_{f}))$ and hence also in $L^{8/3}(0, T; H^{3/4}(\Omega_{f}))$ by Proposition \ref{tracehuniform}. Similarly, by the uniform bound of $p_{h}$ in $L^{2}(0, T; H^{1}(\Omega_{b}))$ given in Proposition \ref{uniformh},
\begin{equation*}
\left|\int_{s}^{t} \int_{\Gamma(t)} \left(\frac{1}{2}|\bd{u}_{h}|^{2} - p_{h}\right) (\bd{v} \cdot \bd{n})\right| \le C\int_{s}^{t} \left(\|p_{h}\|_{L^{2}(\Gamma)} + \|\bd{u}_{h}\|_{L^{4}(\Gamma)}^{2}\right) \le C|t - s|^{1/4},
\end{equation*}
and by using the uniform bound of $\bd{\zeta}_{h}$ in $L^{2}(0, T; H^{k}(\Gamma))$ for all positive integers $k$ given in Proposition \ref{uniformh}, 
\begin{align*}
\left|\int_{s}^{t} \int_{\Gamma(t)} (\bd{\zeta}_{h} - \bd{u}_{h}) \cdot \bd{\tau} (\bd{v} \cdot \bd{\tau})\right| &\le C\int_{s}^{t} \left(\|\bd{\zeta}_{h}\|_{L^{2}(\Gamma)} + \|\bd{u}_{h}\|_{L^{2}(\Gamma)}\right) \\
&\le C\int_{s}^{t} \left(\|\bd{\zeta}_{h}\|_{L^{2}(\Gamma)} + \|\bd{u}_{h}\|_{H^{1}(\Gamma)}\right) \le C|t - s|^{1/2}.
\end{align*}
\item Finally, by using the uniform estimate of $\bd{u}_{h}$ in $L^{2}(0, T; H^{1}(\Omega_{f}(t)))$ in Proposition \ref{uniformh}, we estimate
\begin{equation*}
\left|2\nu \int_{s}^{t} \int_{\Omega_{f}(t)} \bd{D}(\bd{u}_{h}) : \bd{D}(\bd{v})\right| \le C|t - s|^{1/2},
\end{equation*}
\end{itemize}
So we obtain the final estimate that
\begin{equation*}
\sup_{\|v\|_{Q_{M}} \le 1} \left|\int_{\Omega} (\bd{u}_{h}(t) - \bd{u}_{h}(s)) \cdot \bd{v}\right| \le C|t - s|^{1/4},
\end{equation*}
and hence, $\|\bd{u}_{h}\|_{W^{\alpha, 2}(0, T; \mathcal{Q}_{M}')} \le C$ uniformly in $h$ for any $\alpha$ such that $0 < \alpha < 1/8$ by \eqref{fractionalnorm}. The desired strong convergence of $\bd{u}_{h} \to \bd{u}$ along a subsequence in $h$ follows by the compact embedding in Proposition \ref{compactness}. The support properties of the limit $\bd{u}$ follow from the uniform pointwise convergence of $\bd{\omega}_{h}$ to $\bd{\eta}|_{\Gamma}$ and the fact that $L^{2}(0, T; L^{2}(\Omega))$ convergence implies convergence for almost every $(t, \bd{x}) \in [0, T] \times \Omega$ along a subsequence in $h$, combined with the support properties of the extension by zero of $\bd{u}_{h}$. 
\end{proof}

\subsection{Passage to the limit}

We conclude the proof of the main result Theorem \ref{mainthm} by sketching how we can pass to the limit in the weak formulation for the approximate problem with plate of thickness $h > 0$ in Definition \ref{plateweak} to obtain the limiting weak formulation for the regularized interface problem in Definition \ref{regularweak}. We summarize the convergences we have obtained so far below, which follow from Proposition \ref{etahconv}, Proposition \ref{phconv}, Proposition \ref{velhconv}, and an argument similar to Proposition \ref{traceconvprop}. Using these convergences, we can pass to the limit in the regularized interface weak formulations for the approximate problem with a plate of thickness $h > 0$ (Definition \ref{plateweak}) to obtain a regularized interface weak solution to the original direct contact problem (Definition \ref{regularweak}).  

\begin{theorem}\label{singularlimit}
Let $(\bd{u}_{h}, \bd{\omega}_{h}, \bd{\eta}_{h}, p_{h})$ be regularized interface weak solutions for the approximate problem with plate of thickness $h > 0$ in the sense of Definition \ref{plateweak}, which all exist on a common time interval $[0, T]$, for the same initial data, satisfying the conditions of Theorem \ref{hthm}. We have the following convergences in the limit as $h \to 0$.
\begin{itemize}
\item $\bd{\eta}_{h} \to \bd{\eta}$ in $C(0, T; L^{2}(\Omega_{b}))$.
\item $p_{h} \to p$ in $L^{2}(0, T; L^{2}(\Omega_{b}))$.
\item $\bd{\xi}_{h} \to \bd{\xi}$ in $L^{2}(0, T; H^{-s}(\Omega_{b}))$ for $-1 < s < 0$. 
\item $\bd{u}_{h} \to \bd{u}$ in $L^{2}(0, T; L^{2}(\Omega))$ on the maximal domain.
\item $\bd{u}_{h}|_{\Gamma} \to \bd{u}|_{\Gamma}$ in $L^{2}(0, T; L^{2}(\Gamma))$. 
\item $\bd{\omega}_{h} \to \bd{\eta}^{\delta}|_{\Gamma}$ in $C(0, T; H^{k}(\Gamma))$ for all positive integers $k$.
\item $\bd{\zeta}_{h} \to \bd{\xi}^{\delta}|_{\Gamma}$ in $L^{2}(0, T; H^{k}(\Gamma))$ for all positive integers $k$.
\end{itemize}
Furthermore, the limiting solution $(\bd{u}, \bd{\eta}, p)$ is a regularized interface weak solution to the original direct contact problem in the sense of Definition \ref{regularweak}.
\end{theorem}

The passage to the limit from the $h$-level weak formulation to the limiting weak formulation features similar difficulties as to the passage in $N \to \infty$ for the approximate problem, in the sense that there are integrals over moving fluid domains. Hence, we can use the same technique of identifying target test functions and constructing test functions that are admissible for each of the approximate $h$-level weak formulations, which converge appropriate to the target test function. We can hence pass to the limit as $h \to 0$ in the approximate weak formulations in Definition \ref{plateweak} to obtain the limiting weak formulation in Definition \ref{regularweak} via the same arguments as in Section \ref{testfunctions}. We only make comments about the associated plate terms in the weak formulation in Definition \ref{plateweak} and how they vanish in the limit as $h \to 0$:
\begin{equation*}
-h \int_{0}^{T} \int_{\Gamma} \partial_{t}\bd{\omega} \cdot \partial_{t}\bd{\varphi} + h \int_{0}^{T} \int_{\Gamma} \Delta \partial_{t}\bd{\omega} \cdot \Delta \bd{\varphi} + h \int_{0}^{T} \int_{\Gamma} \Delta \bd{\omega} \cdot \Delta \bd{\varphi}.
\end{equation*}
Note that by the uniform bound of $h^{1/2}\bd{\omega}_{h}$ in $L^{\infty}(0, T; L^{2}(\Gamma))$, $h^{1/2}\Delta \bd{\zeta}_{h}$ in $L^{2}(0, T; L^{2}(\Gamma))$, and $h^{1/2}\Delta \bd{\omega}_{h}$ in $L^{\infty}(0, T; L^{2}(\Gamma))$, we have that these terms converge to zero, since $\bd{\varphi}$ is a fixed test function. For example,
\begin{equation*}
\left|h \int_{0}^{T} \int_{\Gamma} \partial_{t}\bd{\omega}_{h} \cdot \partial_{t}\bd{\varphi}\right| \le h^{1/2} \|h^{1/2}\partial_{t}\bd{\omega}_{h}\|_{L^{\infty}(0, T; L^{2}(\Gamma))} \|\partial_{t}\bd{\varphi}\|_{L^{1}(0, T; L^{2}(\Gamma))} \to 0,
\end{equation*}
as there is an extra $h^{1/2}$ factor, which causes the integral to converge to zero in the limit as $h \to 0$. Thus, we can pass to the limit in the weak formulation for the $h$-level problem to obtain a regularized interface weak solution to the direct contact FPSI problem with nonlinear geometric coupling in the sense of Definition \ref{plateweak}, which completes the proof of the main result in Theorem \ref{mainthm}.

\section{Conclusions}\label{conclusion}

In this manuscript, we introduced a {\emph{regularized interface method}} to study fluid-poroelastic structure interaction between an incompressible, viscous fluid and a poroelastic medium modeled by the Biot equations. The two systems are nonlinearly coupled across a moving interface defined by the trace of the Biot displacement. We focused on weak, finite-energy solutions to this problem, which is notoriously difficult to analyze due to the limited regularity of the interface. This lack of regularity complicates the rigorous definition of a weak solution: key geometric quantities--such as the Jacobian of the Lagrangian map on the Biot domain--lack sufficient regularity to ensure all terms in the formally derived weak formulation are well-defined.

The regularized interface method developed here is specifically designed to address this issue by providing a framework for defining a weak solution in the presence of low-regularity interface geometry. It consists of two main components: (1) spatial regularization of the Biot displacement to define a smooth, regularized moving interface, and (2) modification of the weak formulation to ensure that the associated energy is consistent with the original, non-regularized problem.

The regularization introduced via a fixed parameter $\delta > 0$ provides uniform control over higher-order spatial derivatives of the Biot displacement. This, in turn, enables uniform-in-time control of key geometric quantities and it allows us to control the onset of geometric degeneracies--at least locally in time. 
Although the regularized interface method was initially developed to address the ill-defined nature of the direct contact interface problems, the additional regularity gained from spatial convolution has a further important benefit: it allows for the treatment of \emph{vector-valued displacements} in the Biot medium. By controlling higher-order norms of the regularized Biot displacement, we can enforce injectivity uniformly in time for the approximate solutions, provided $\delta$ remains fixed. We emphasize that the inclusion of \textit{bulk poroelasticity with vector-valued displacements} in the context of weak solutions is a novel aspect of the present work. 

We established the existence of a direct contact regularized interface weak solution via an approximation scheme that introduces a plate of thickness $h > 0$ at the interface and then takes the singular limit as $h \to 0$. Existence for the approximate problem, involving  a plate,  was proven using a Lie operator splitting approach, which alternates between solving the plate subproblem and the Biot/fluid subproblem (see Section \ref{hproof}).
We emphasize that the existence analysis is rather technically involved, despite the spatial regularization of the Biot displacement providing uniform control over higher-order norms. In the approximate problem, the interface--modeled as a thin plate--undergoes vector-valued displacements from its reference configuration $\Gamma$. This necessitated the use of a harmonic extension ALE map for the fluid domain (see \eqref{alef1} and \eqref{alef2}) and required establishing uniform geometric control over mappings from fixed to reference domains (see Proposition \ref{geometricN}). 
These controls were essential for transitioning between the Lagrangian and Eulerian formulations and for avoiding geometric degeneracy.

Moreover, the more intricate geometry of the vector-valued interface displacements introduced substantial complexity in the compactness analysis. In particular, establishing compactness for the fluid velocities required nontrivial estimates comparing different time-discrete moving fluid domains, which vary due to the moving approximate geometries $\Gamma^{n}_{\Delta t}$. These technical challenges were addressed in Proposition \ref{tildecomp}.

Next, we addressed passing to the limit as $h \to 0$ in the ``approximate'' problems with a plate of thickness $h$  to obtain a limiting regularized interface weak solution for the direct contact problem, as stated in Theorem \ref{singularlimit}. 
A main component of this analysis was again, finding uniform geometric bounds independent of $h$, see Proposition \ref{geometryh}, and also showing that the regularized interface weak solutions for the approximate $h$-level problems all exist on a uniform time interval, which was accomplished by energy estimates. 
This completed the proof of the main result stated in Theorem \ref{mainthm}.

We remark that the regularized interface method introduced here, based on an explicit regularization of the direct contact interface,  may initially appear to be  a purely technical or artificial tool for mathematical analysis--especially since each choice of regularization parameter $\delta > 0$ leads to a distinct regularized interface weak solution.  However, the preliminary results of our current work indicate that the $\delta$-regularization of the problem introduced here, is a {\emph{``good approximation''}} to the original, non-regularized problem in the sense that our ``approximate'' solutions of the $\delta$-regularized problems converge, as $\delta \to 0$ to the classical solution of the non-regularized problem when such a solution exists. More precisely, our preliminary results indicate that we will be able to establish \textit{weak-classical consistency} -- given a smooth (classical) solution $(\bd{u}, \bd{\eta}, p)$ to the original direct contact Biot-fluid problem that exists on some time interval $[0, T]$, the regularized interface weak solutions  $(\bd{u}_{\delta}, \bd{\eta}_{\delta}, p_{\delta})$ all exist on that same uniform time interval $[0, T]$, and furthermore, these solutions converge in an appropriate norm as $\delta \to 0$ to the smooth (classical) solution $(\bd{u}, \bd{\eta}, p)$. 

We note that for a slightly different regularization scheme applied to a simpler FPSI problem with nonlinear geometric coupling and only transverse interface displacements, involving a separating reticular plate at the interface, a similar weak-classical consistency result has been established in \cite{FPSIJMPA}. Our experience with that analysis and our preliminary results so far  give us reason to be optimistic that a similar result can be established for the regularized interface method, applied to the more challenging setting of the direct contact Biot-fluid problem, studied here.

We conclude with the remark that, although the regularized interface method was developed specifically for the direct contact Biot-fluid problem considered in this manuscript, its advantages suggest that it may serve as a robust framework for analyzing a broader class of moving domain problems--particularly those in which the geometry depends on solutions with low regularity. The method's ability to provide uniform geometric control, even in the presence of vector-valued displacements in bulk poroelasticity, points to its potential applicability to direct contact FSI problem with bulk elasticity, and other systems involving complex geometric nonlinearities. Extending the regularized interface method to such broader settings represents a promising direction for future research.

\section*{Acknowledgements}

Jeffrey Kuan was supported by the National Science Foundation under the Mathematical Sciences Postdoctoral Research Fellowship (MSPRF) DMS-2303177.
Sun\v{c}ica \v{C}ani\'{c} was supported in part by the National Science Foundation under grants DMS-2408928, DMS-2247000, and DMS-2011319. Boris Muha was supported by the Croatian Science Foundation under the project number IP-2022-10-2962.

\section{Appendix}

\begin{figure}
\center
\includegraphics[scale=0.4]{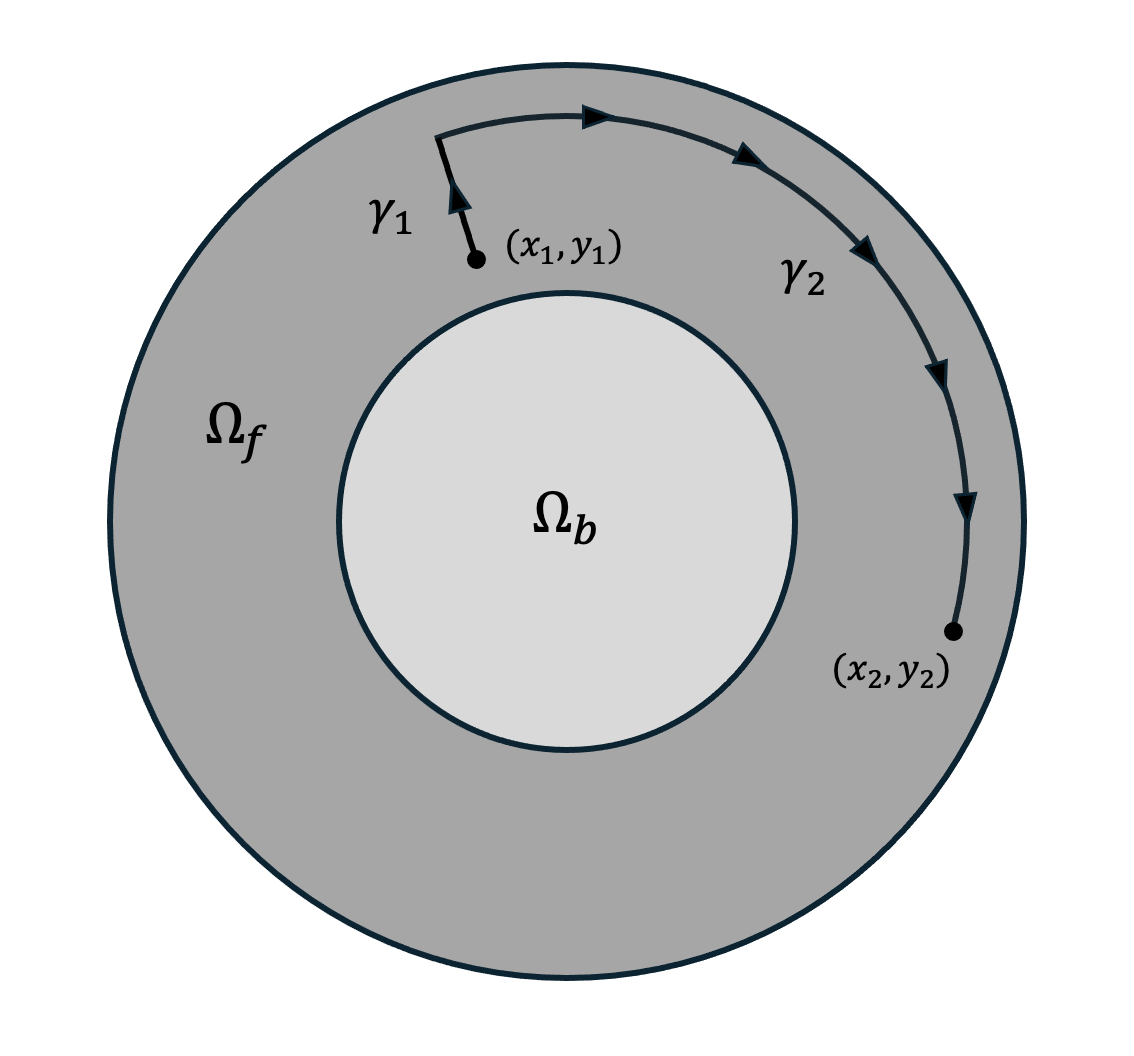}
\caption{The curves $\gamma_{1}$ and $\gamma_{2}$ connecting two points $(x_{1}, y_{1})$ and $(x_{2}, y_{2})$ in $\Omega_{f}$, where $\gamma_{1}$ is purely radial and $\gamma_{2}$ is purely circular (constant radius).}
\label{gammafig}
\end{figure}

In the appendix, we collect some auxiliary results that are used in the analysis. The first is a result related to the (reference) fluid domain geometry $\Omega_{f} := \{\bd{x} \in \R^{2} : 1 < |\bd{x}| < 2\}$. Because $\Omega_{f}$ is not convex, as the line segment connecting two points in $\Omega_{f}$ is not necessarily contained entirely within $\Omega_{f}$, we must define an alternative notion of ``distance" within $\Omega_{f}$ that is equivalent to the usual Euclidean norm. 

Let $(x_{1}, y_{1})$ and $(x_{2}, y_{2})$ be two points in $\Omega_{f}$, with associated polar coordinates $(r_{1}, \theta_{1})$ and $(r_{2}, \theta_{2})$ respectively, where $1 < r_{i} < 2$ for $i = 1, 2$ and $\theta_{1}, \theta_{2} \in (-2\pi, 2\pi)$ are chosen so that $|\theta_{1} - \theta_{2}| < \pi$. (While polar coordinates can be expressed modulo $2\pi$ in the angle, we choose representatives $\theta_{1}$ and $\theta_{2}$ that are the ``closest" to each other). Let $\gamma_{1}(x_{1}, y_{1}, x_{2}, y_{2})$ and $\gamma_{2}(x_{1}, y_{1}, x_{2}, y_{2})$ be the curves parametrized by $0 \le s \le 1$ by
\begin{equation*}
\gamma_{1}(x_{1}, y_{1}, x_{2}, y_{2}): s \to ((1 - s)r_{1} + sr_{2}, \theta_{1}),
\end{equation*}
\begin{equation*}
\gamma_{2}(x_{1}, y_{1}, x_{2}, y_{2}): (r_{2}, (1 - s)\theta_{1} + s\theta_{2}),
\end{equation*}
in polar coordinates. Then, let $\gamma(x_{1}, y_{1}, x_{2}, y_{2})$ be the composition of the two curves $\gamma_{1}(x_{1}, y_{1}, x_{2}, y_{2})$ followed by $\gamma_{2}(x_{1}, y_{1}, x_{2}, y_{2})$. See Figure \ref{gammafig}. Since $\gamma(x_{1}, y_{1}, x_{2}, y_{2})$ is a curve contained entirely within $\Omega_{f}$ connecting $(x_{1}, y_{1})$ and $(x_{2}, y_{2})$, its length can be thought of as a natural measure of distance within $\Omega_{f}$. We claim that the length of this connecting curve $\gamma(x_{1}, y_{1}, x_{2}, y_{2})$ is equivalent to the usual Euclidean distance.

\begin{proposition}\label{gammaprop}
Consider points $(x_{1}, y_{1})$ and $(x_{2}, y_{2})$ in $\Omega_{f}$, connected by the curve $\gamma(x_{1}, y_{1}, x_{2}, y_{2}) \subset \Omega$ defined above. Then,
\begin{equation*}
|(x_{1}, y_{1}) - (x_{2}, y_{2})| \le \text{Length}\Big(\gamma(x_{1}, y_{1}, x_{2}, y_{2})\Big) \le 5|(x_{1}, y_{1}) - (x_{2}, y_{2})|,
\end{equation*}
for all $(x_{1}, y_{1}), (x_{2}, y_{2}) \in \Omega_{f}$. 
\end{proposition} 

\begin{proof}
We first show the assertion when $r_{1} = r_{2}$, namely, when $|(x_{1}, y_{1})| = |(x_{2}, y_{2})|$. In this case, the curve $\gamma(x_{1}, y_{1}, x_{2}, y_{2})$ is just the curve $\gamma_{2}(x_{1}, y_{1}, x_{2}, y_{2})$. Without loss of generality, by rotation and reflection, we can assume that $\theta_{1} = 0$ and $0 < \theta_{2} \le \pi$, in which case, we can compute, where $r$ is the common value of $r_{1} = r_{2}$:
\begin{equation*}
\text{Length}\Big(\gamma(x_{1}, y_{1}, x_{2}, y_{2})\Big) = r\theta_{2}, 
\end{equation*}
\begin{equation*}
|(x_{1}, y_{1}) - (x_{2}, y_{2})| = |(r, 0) - (r\cos(\theta_{2}), r\sin(\theta_{2}))| = r\sqrt{2}\Big(\sqrt{1 - \cos(\theta_{2})}\Big).
\end{equation*}
Since $0 < \theta_{2} \le \pi$, we can estimate $\displaystyle \frac{\sqrt{1 - \cos(\theta_{2})}}{\theta_{2}} \ge \frac{1}{\sqrt{2}}$ for all $0 < \theta_{2} \le \pi$. Therefore, we conclude that
\begin{multline}\label{gammalengthr}
\text{Length}\Big(\gamma(x_{1}, y_{1}, x_{2}, y_{2})\Big) \le 2|(x_{1}, y_{1}) - (x_{2}, y_{2})|, \qquad \text{ for all } (x_{1}, y_{1}), (x_{2}, y_{2}) \in \Omega_{f} \\
\text{ with } |(x_{1}, y_{1})| = |(x_{2}, y_{2})|.
\end{multline}

Now consider general $(x_{1}, y_{1})$ and $(x_{2}, y_{2})$ in $\Omega_{f}$ and let $(x^{*}, y^{*}) \in \Omega_{f}$ be the intermediate point, namely the endpoint of $\gamma_{1}(x_{1}, y_{1}, x_{2}, y_{2})$ and the starting point of $\gamma_{2}(x_{1}, y_{1}, x_{2}, y_{2})$, which has polar coordinates $(r_{2}, \theta_{1})$. Then, we calculate that
\begin{equation*}
\text{Length}\Big(\gamma(x_{1}, y_{1}, x_{2}, y_{2})\Big) = \text{Length}\Big(\gamma_{1}(x_{1}, y_{1}, x_{2}, y_{2})\Big) + \text{Length}\Big(\gamma_{2}(x_{1}, y_{1}, x_{2}, y_{2})\Big).
\end{equation*}
By the definition of $\gamma_{1}(x_{1}, y_{1}, x_{2}, y_{2})$, we estimate that
\begin{equation*}
\text{Length}\Big(\gamma_{1}(x_{1}, y_{1}, x_{2}, y_{2})\Big) = \Big||(x_{1}, y_{1})| - |(x_{2}, y_{2})|\Big| \le |(x_{1}, y_{1}) - (x_{2}, y_{2})|. 
\end{equation*}
Since $\gamma_{2}(x_{1}, y_{1}, x_{2}, y_{2})$ connects the points $(x^{*}, y^{*})$ and $(x_{2}, y_{2})$ along a constant radius, by \eqref{gammalengthr},
\begin{align*}
\text{Length}\Big(\gamma_{2}(x_{1}, y_{1}, x_{2}, y_{2})\Big) &\le 2\|(x^{*}, y^{*}) - (x_{2}, y_{2})\| \\
&\le 2\Big(|(x_{1}, y_{1}) - (x_{2}, y_{2})| + |(x^{*}, y^{*}) - (x_{1}, y_{1})|\Big) \\
&\le 2\Big(|(x_{1}, y_{1}) - (x_{2}, y_{2})| + \text{Length}\Big(\gamma_{1}(x_{1}, y_{1}, x_{2}, y_{2})\Big)\Big) \\
&\le 4\Big(|(x_{1}, y_{1}) - (x_{2}, y_{2})|\Big).
\end{align*}
So we conclude that
\begin{equation*}
|(x_{1}, y_{1}) - (x_{2}, y_{2})| \le \text{Length}\Big(\gamma(x_{1}, y_{1}, x_{2}, y_{2})\Big) \le 5|(x_{1}, y_{1}) - (x_{2}, y_{2})|,
\end{equation*}
with the lower bound being immediate and the upper bound following from the preceding calculations.
\end{proof}

Finally, we state a result about the transformation of lengths of curves under composition by a differentiable map.

\begin{proposition}\label{lengthtransform}
Let $f: U \to \R^{2}$ be a $C^{1}$ map, where $U \subset \R^{2}$ is an open set. If $\gamma$ is a piecewise smooth curve in $U$, then 
\begin{equation*}
\text{Length}\Big(f(\gamma)\Big) \le \|\nabla f\|_{L^{\infty}(U)} \cdot \text{Length}(\gamma).
\end{equation*}
\end{proposition}

\begin{proof}
Consider the case where $\gamma$ is smooth (the piecewise smooth case follows by considering each smooth piece separately), and let $\bd{r}(s)$ for $a \le s \le b$ be a parametrization of $\gamma$. Then, $f(\bd{r}(s))$ is a parametrization of $f(\gamma)$. We compute
\begin{equation*}
\text{Length}\Big(f(\gamma)\Big) = \int_{a}^{b} |\nabla f(\bd{r}(s)) \bd{r}'(s)| ds \le \|\nabla f\|_{L^{\infty}(U)} \int_{a}^{b} |\bd{r}'(s)| ds \le \|\nabla f\|_{L^{\infty}(U)} \text{Length}(\gamma).
\end{equation*}
\end{proof}

{\bf{Conflict of interest statment:}} The authors declare that they have no conflicts of interest, financial or otherwise, related to the subject matter of this manuscript.

\bibliographystyle{plain}
\bibliography{FPSIBibliography}

{\bf{Data availability statement:}} This manuscript has no associated data. 

\end{document}